\documentclass[twoside,11pt,reqno]{amsart}

\usepackage{amsmath}
\usepackage{amssymb}
\usepackage{verbatim}
\usepackage{dsfont}
\usepackage{enumerate}
\usepackage{enumitem}
\usepackage[cmtip,arrow,all]{xy}
\usepackage[mathscr]{eucal}
\usepackage{amsthm}
\usepackage{footmisc}
\usepackage{stmaryrd}
\usepackage{amscd}
\usepackage{amsfonts}
\usepackage{latexsym}
\usepackage{amscd}
\usepackage{graphicx}
\usepackage{tikz}
\usepackage{enumitem}
\usepackage{txfonts}
\usepackage{setspace}
\usepackage{wasysym}
\usepackage[breaklinks=true]{hyperref}
\usepackage{pdfpages}
\usepackage{pb-diagram,pb-xy}
\usetikzlibrary{decorations.pathmorphing}

\numberwithin{equation}{section}
\newcommand{\arxiv}[1]{{\tt arXiv:#1}}
\newtheorem{Proposition}{Proposition}[section]
\newtheorem{Lemma}[Proposition]{Lemma}
\newtheorem{Theorem}[Proposition]{Theorem}
\newtheorem{Corollary}[Proposition]{Corollary}
\newtheorem{iTheorem}{Theorem}

\theoremstyle{definition}

\theoremstyle{remark}
\newtheorem{Remark}[Proposition]{Remark}
\renewcommand{\k}{\Bbbk}
\newcommand{\K}{\mathbb{K}}
\newcommand{\unit}{\mathds{1}}
\newcommand{\qc}{p} 
\def\up{\uparrow}
\def\down{\downarrow}
\def\thickdown{\pmb\downarrow}
\def\diam{{\color{white}\scriptstyle\vardiamondsuit}\hspace{-1.493mm}\scriptstyle\diamondsuit}
\def\thickup{\pmb\uparrow}
\def\Kar{\operatorname{Kar}}
\def\Add{\operatorname{Add}}
\def\signs{\sigma}
\def\Sym{\operatorname{Sym}}

\newcommand{\rh}{\mathrm{Heis}}
\def\H{\mathcal{H}}
\def\Heis{\mathcal{H}eis}
\def\HEIS{\mathcal{H}\mathfrak{eis}}
\def\h{\mathbb{O}}
\renewcommand{\sl}{\mathfrak{sl}}

\newcommand{\g}{\mathfrak{g}}
\newcommand{\hh}{\mathfrak{h}}

\def\op{{\operatorname{op}}}

\def\Hom{\operatorname{Hom}}

\def\wt{{\operatorname{wt}}}
\def\End{{\operatorname{End}}}

\def\bi{\text{\boldmath$i$}}
\def\SG{\mathfrak{S}}
\def\bj{\text{\boldmath$j$}}
\def\UU{{\mathfrak U}}

\def\C{{\mathbb C}}
\def\Vec{\mathcal{V}ec_{\operatorname{fd}}}
\def\Z{{\mathbb Z}}
\def\N{{\mathbb N}}
\def\R{{\mathcal R}}
\def\A{{\mathcal A}}
\def\AA{{\mathfrak A}}
\def\eps{{\varepsilon}}
\def\B{\mathbf{B}}

\def\fdrcomod{\operatorname{comod}_{\operatorname{fd}}\!\operatorname{-\!}}

\def\fdlmod{\operatorname{\!-\!}\operatorname{mod}_{\operatorname{fd}}}
\def\lfdlmod{\operatorname{\!-\!}\operatorname{mod}_{\operatorname{lfd}}}

\def\lfdrmod{\operatorname{mod}_{\operatorname{lfd}}\operatorname{\!-\!}}
\newcommand{\cev}[1]{\reflectbox{\ensuremath{\vec{\reflectbox{\ensuremath{#1}}}}}}

\def\dt{{\color{white}\bullet}\!\!\!\circ}
\def\bull{{\scriptstyle\bullet}}

\def\anticlock{\begin{tikzpicture}[baseline=-.9mm]
\filldraw[white] (0,0) circle (1.72mm);
\draw[-] (0,-0.18) to[out=180,in=-102] (-.178,0.02);
\draw[-] (-0.18,0) to[out=90,in=180] (0,0.18);
\draw[-] (0.18,0) to[out=-90,in=0] (0,-0.18);
\draw[<-] (0,0.18) to[out=0,in=90] (0.18,0);
\end{tikzpicture}\,}
\def\thickanticlock{\begin{tikzpicture}[baseline=-.9mm]
\draw[-,thick] (0,-0.18) to[out=180,in=-102] (-.178,0.02);
\draw[-,thick] (-0.18,0) to[out=90,in=180] (0,0.18);
\draw[-,thick] (0.18,0) to[out=-90,in=0] (0,-0.18);
\draw[<-,thick] (0,0.18) to[out=0,in=90] (0.18,0);
\end{tikzpicture}\,}
\def\anticlockplus{\begin{tikzpicture}[baseline=-.9mm]
\filldraw[white] (0,0) circle (1.72mm);
\draw[-] (0,-0.18) to[out=180,in=-102] (-.178,0.02);
\draw[-] (-0.18,0) to[out=90,in=180] (0,0.18);
\draw[-] (0.18,0) to[out=-90,in=0] (0,-0.18);
\draw[<-] (0,0.18) to[out=0,in=90] (0.18,0);
   \node at (0,0) {$+$};
\end{tikzpicture}\,}

\def\anticlocki{\begin{tikzpicture}[baseline=-.9mm]
\filldraw[white] (0,0) circle (1.72mm);
\draw[-,thick] (0,-0.18) to[out=180,in=-102] (-.178,0.02);
\draw[-,thick] (-0.18,0) to[out=90,in=180] (0,0.18);
\draw[-,thick] (0.18,0) to[out=-90,in=0] (0,-0.18);
\draw[<-,thick] (0,0.18) to[out=0,in=90] (0.18,0);
   \node at (0,0) {$\scriptstyle{i}$};
\end{tikzpicture}\,}
\def\anticlockj{\begin{tikzpicture}[baseline=-.9mm]
\filldraw[white] (0,0) circle (1.72mm);
\draw[-,thick] (0,-0.18) to[out=180,in=-102] (-.178,0.02);
\draw[-,thick] (-0.18,0) to[out=90,in=180] (0,0.18);
\draw[-,thick] (0.18,0) to[out=-90,in=0] (0,-0.18);
\draw[<-,thick] (0,0.18) to[out=0,in=90] (0.18,0);
   \node at (-0.02,-0.02) {$\scriptstyle{j}$};
\end{tikzpicture}\,}

\def\clock{\begin{tikzpicture}[baseline=-.9mm]
\filldraw[white] (0,0) circle (1.72mm);
\draw[-] (0,-0.18) to[out=180,in=-90] (-.18,0);
\draw[->] (-0.18,0) to[out=90,in=180] (0,0.18);
\draw[-] (-0.02,0.178) to[out=12,in=90] (0.18,0);
\draw[-] (0.18,0) to[out=-90,in=0] (0,-0.18);
\end{tikzpicture}\,}
\def\clockplus{\begin{tikzpicture}[baseline=-.9mm]
\filldraw[white] (0,0) circle (1.72mm);
\draw[-] (0,-0.18) to[out=180,in=-90] (-.18,0);
\draw[->] (-0.18,0) to[out=90,in=180] (0,0.18);
\draw[-] (-0.02,0.178) to[out=12,in=90] (0.18,0);
\draw[-] (0.18,0) to[out=-90,in=0] (0,-0.18);
   \node at (0,0) {$+$};
\end{tikzpicture}\,}

\def\clocki{\begin{tikzpicture}[baseline=-.9mm]
\filldraw[white] (0,0) circle (1.72mm);
\draw[-,thick] (0,-0.18) to[out=180,in=-90] (-.18,0);
\draw[->,thick] (-0.18,0) to[out=90,in=180] (0,0.18);
\draw[-,thick] (-0.02,0.178) to[out=12,in=90] (0.18,0);
\draw[-,thick] (0.18,0) to[out=-90,in=0] (0,-0.18);
   \node at (0,0) {$\scriptstyle{i}$};
\end{tikzpicture}\,}
\def\clockj{\begin{tikzpicture}[baseline=-.9mm]
\filldraw[white] (0,0) circle (1.72mm);
\draw[-,thick] (0,-0.18) to[out=180,in=-90] (-.18,0);
\draw[->,thick] (-0.18,0) to[out=90,in=180] (0,0.18);
\draw[-,thick] (-0.02,0.178) to[out=12,in=90] (0.18,0);
\draw[-,thick] (0.18,0) to[out=-90,in=0] (0,-0.18);
   \node at (0,0) {$\scriptstyle{j}$};
\end{tikzpicture}\,}

\newcommand{\darkg}{\color{green!70!black}}
\tikzset{darkg/.style={green!70!black}}

\begin{document}
\title[Heisenberg and Kac-Moody]{Heisenberg and Kac-Moody Categorification}

\author[J.~Brundan]{Jonathan Brundan}
\address[J.B.]{Department of Mathematics\\
University of Oregon\\ Eugene\\ OR 97403\\ USA}
\email{brundan@uoregon.edu}

\author[A.~Savage]{Alistair Savage}

\address[A.S]{
  Department of Mathematics and Statistics \\
  University of Ottawa\\
  Ottawa, ON\\ Canada
}
\email{alistair.savage@uottawa.ca}
\urladdr{\href{https://alistairsavage.ca}{alistairsavage.ca}, \textrm{\textit{ORCiD}:} \href{https://orcid.org/0000-0002-2859-0239}{orcid.org/0000-0002-2859-0239}}

\author[B.~Webster]{Ben Webster}
\address[B.W.]{Department of Pure Mathematics, University of Waterloo \&
  Perimeter Institute for Theoretical Physics\\
Waterloo, ON\\ Canada}
\email{ben.webster@uwaterloo.ca}


\thanks{J.B. supported in part by NSF grant
  DMS-1700905. A.S. supported by Discovery Grant RGPIN-2017-03854 from
  the Natural Sciences and Engineering Research Council of Canada.
B.W. supported by Discovery Grant RGPIN-2018-03974 from the
  Natural Sciences and Engineering Research Council of Canada.
This research was also supported by Perimeter Institute for Theoretical Physics. Research at Perimeter Institute is supported by the Government of Canada through the Department of Innovation, Science and Economic Development and by the Province of Ontario through the Ministry of Research and Innovation.}

\begin{abstract}
We show that any Abelian
module category over the (degenerate or quantum) Heisenberg category
satisfying suitable finiteness conditions
may be viewed as a 2-representation over a corresponding 
Kac-Moody 2-category (and vice versa). 
This gives a way to construct Kac-Moody actions in many representation-theoretic examples 
which is independent of Rouquier's original approach via ``control by
$K_0$.'' As an application, we prove an isomorphism theorem for generalized cyclotomic
quotients of these categories, extending the known isomorphism between
cyclotomic quotients of type A affine Hecke algebras and quiver Hecke algebras.
\end{abstract}

\maketitle

\section{Introduction}

The field of higher representation theory has both benefitted and
suffered from a multiplicity of perspectives.  One such juncture is in
the definition of a categorical action of a Kac-Moody algebra, which was
developed independently by Rouquier \cite{Rou} and
Khovanov and Lauda \cite{KL3}.
Both of these works introduced a remarkable new 2-category, 
the {\em Kac-Moody 2-category} $\UU(\g)$
associated to a
symmetrizable Kac-Moody algebra $\g$, although it took several more
years before the distinct approaches taken in \cite{Rou, KL3}
were reconciled with one another; see \cite{Bkm}.
The object set of $\UU(\g)$ is the weight lattice $X$ of the
underlying Kac-Moody algebra. Then a {\em categorical action} 
of $\g$ on a family of categories
$(\R_\lambda)_{\lambda \in X}$ is the data of a strict 2-functor from
$\UU(\g)$ to the 2-category $\mathfrak{Cat}$ of 
categories sending $\lambda$ to
$\R_\lambda$ for each $\lambda \in X$.
This means that 
there are functors $E_i:\R_\lambda \rightarrow
\R_{\lambda+\alpha_i}$, $F_i:\R_{\lambda+\alpha_i}\rightarrow
\R_{\lambda}$ 
corresponding to the Chevalley generators $e_i, f_i\:(i \in I)$ of
$\g$
(where
  $\alpha_i$ is the $i$th simple root), and 
there are natural transformations between these functors
satisfying relations paralleling the 2-morphisms in $\UU(\g)$.
These relations are recorded in $\S$\ref{songs} below. They imply that
\begin{enumerate}
\item[(KM1)] there are prescribed adjunctions
$(E_i, F_i)$ for all $i \in I$;
\item[(KM2)]
for $d \geq 0$ there is an action of the
{\em quiver Hecke algebra} $QH_d$
of the same Cartan type as $\mathfrak{g}$
on the $d$th power of the functor $E := \bigoplus_{i \in I} E_i$;
\item[(KM3)]
there is an explicit isomorphism of functors lifting the
familiar Chevalley relation $[e_i, f_j] = \delta_{i,j} h_i$ in the Lie algebra $\mathfrak{g}$;
see (\ref{day1})--(\ref{day3}).
\end{enumerate}
In this article, we will only consider categorical actions on {Abelian} categories
satisfying certain finiteness properties, which are needed to ensure 
that the relevant morphism spaces are 
finite-dimensional vector spaces. More precisely, all categories considered will either be
{\em locally finite Abelian} or {\em Schurian} $\k$-linear
categories for a fixed {\em algebraically closed field} $\k$; see $\S$\ref{scars}
for these definitions.
All functors between such categories will be assumed to be
$\k$-linear without further mention.

In Cartan type A, 
Rouquier also introduced a related notion of 
{\em $\mathfrak{sl}'_I$-categorification}, which was based in part on his previous work with
Chuang \cite{CR} treating the case of $\mathfrak{sl}_2$. Instead of the
tower of quiver Hecke algebras mentioned in the previous paragraph, the
definition of $\mathfrak{sl}'_I$-categorification involves a tower
of {\em affine Hecke algebras} of type A (either quantum or degenerate). 
In more detail, assume that we are given $z=q-q^{-1} \in \k$.
Let $AH_d$ be the affine Hecke algebra corresponding to the
symmetric group $\SG_d$ with defining parameter $q$
if $z \neq 0$, or its degenerate analog if $z=0$.
Let $I$ be a subset of $\k$ closed under the automorphisms
$i \mapsto i^\pm$ defined by
\begin{equation*}
    i^{\pm} := \left\{\begin{array}{ll}
    q^{\pm 2} i&\text{in the quantum case ($z\neq 0$),}\\
    i\pm 1&\text{in the degenerate case ($z = 0$),}
    \end{array}\right.
    \end{equation*}
    assuming moreover that $0 \notin I$ in the quantum case. The map $i \mapsto i^+$ defines edges 
    making the set $I$
into a quiver with connected components of type $\mathrm A_\infty$ if $\qc = 0$ or
$\mathrm A_{\qc-1}^{(1)}$ if $\qc \neq 0$, where $\qc$ is the (not necessarily prime!) {\em quantum characteristic}, that is, the smallest positive integer such that $q^{\qc-1}+q^{\qc-3}+\cdots +q^{1-\qc}=0$ or $0$ if no such integer exists.
Let $\mathfrak{g} = \mathfrak{sl}_I'$ be the corresponding
(derived) Kac-Moody algebra. To have an $\mathfrak{sl}'_I$-categorification
on a locally finite Abelian or Schurian $\k$-linear category $\R$, one needs:
\begin{enumerate}
\item[(SL1)]
an adjoint pair $(E,F)$ of endofunctors of $\R$ such that $F$ is also left adjoint to $E$;
\item[(SL2)] endomorphisms $x:E \Rightarrow E$ and $\tau:E^2 \Rightarrow E^2$
inducing an action of $AH_d$ on the $d$th
power $E^d$ for all $d \geq 0$.
\end{enumerate}
Assume moreover that all eigenvalues of $x:E \Rightarrow E$
belong to the given set $I$ so that,
by taking generalized eigenspaces, one obtains decompositions of $E$ and its adjoint $F$ into 
eigenfunctors: $E = \bigoplus_{i \in I} E_i$, $F = \bigoplus_{i \in I}
F_i$. Then, we require that
\begin{enumerate}
\item[(SL3)]
 the induced maps $e_i :=
[E_i]$ and $f_i := [F_i]$ make the
complexified Grothendieck group $\C \otimes_{\Z} K_0(\R)$ into an integrable
representation of the Lie algebra $\g$, with the Grothendieck group of
each block of $\R$ giving rise to an isotypic representation of
the Cartan subalgebra $\hh$ of $\g$.
\end{enumerate}
Under these hypotheses, 
there is an induced categorical action of $\g$ on
$(\R_\lambda)_{\lambda \in X}$ in the sense defined in the previous paragraph,
for Serre subcategories $\R_\lambda$ of $\R$ defined so that
$\C\otimes_\Z K_0(\R_\lambda)$ is the $\lambda$-weight space of $\C\otimes_\Z K_0(\R)$.
This fundamental result is known as ``control by $K_0$;'' 
see \cite[Theorem 5.30]{Rou} in the 
locally finite Abelian case and \cite[Theorem 4.27]{BD} for the extension to the Schurian case.
In its proof, the property
(SL1) obviously implies (KM1), and (SL2) implies (KM2) due to
the isomorphism $\widehat{AH}_d \cong \widehat{QH}_d$
between completions of affine Hecke algebras 
and quiver Hecke algebras discovered in \cite{Rou, BK}. 
Finally, and most interesting, to pass from (SL3)
(which involves relations at the level of the Grothendieck group) to
(KL3) (which involves ``higher'' relations),
Rouquier applies the sophisticated structure theory developed in \cite{CR}, thereby
reducing the proof to minimal $\mathfrak{sl}_2$-categorifications which are analyzed explicitly.

In the current literature, almost all examples of categorical actions
of Kac-Moody algebras of Cartan type A on Abelian categories are
constructed via this ``control by
$K_0$'' theorem. In this article, we develop a new approach 
to constructing such Kac-Moody actions based instead on the {\em Heisenberg category}
$\Heis_k$ of central charge $k \in \Z$.
This is a monoidal category that is constructed from affine Hecke
algebras in a way that is entirely analogous to the construction of the
Kac-Moody 2-category from quiver Hecke algebras.
It comes in two forms, degenerate or quantum, depending on the
parameter $z=q-q^{-1}$ as fixed above. In the special case $k=-1$, the
Heisenberg category was
defined originally in the degenerate case by Khovanov \cite{K}
and in the quantum case by
Licata and the second author \cite{LS}. The appropriate extension of
the definition to arbitrary
central charge was worked out much more recently; see \cite{MS18, Bheis} in
the degenerate case  and \cite{BSWqheis} in the quantum case.
A {\em categorical Heisenberg action} on a category $\R$ is the data
of a strict monoidal functor $\Heis_k \rightarrow \mathcal{E}nd(\R)$, where
$\mathcal{E}nd(\R)$ is the strict monoidal category consisting of
endofunctors and natural transformations.
In view of
the defining relations of $\Heis_k$ recorded in 
$\S\S$\ref{dh}--\ref{qh} below, this means that there are endofunctors
$E, F:\R \rightarrow \R$ and natural transformations such that
\begin{enumerate}
\item[(H1)] there is a prescribed adjunction $(E,F$);
\item[(H2)] for $d \geq 0$ there is an action of $AH_d$ on $E^d$;
\item[(H3)] there is an explicit isomorphism of functors lifting
  the relation $[e,f] = k$ in the Heisenberg algebra of central
  charge $k$; see (\ref{invrela})--(\ref{invrelb}) in the degenerate case and
  (\ref{invrel1a})--(\ref{invrel1b}) in the quantum case\footnote{In the quantum case
    there is one additional relation recorded just after (\ref{invrel1b}).}.
\end{enumerate}
The properties (H1)--(H3) exactly parallel (KM1)--(KM3),
unlike (SL1)--(SL3).
Now we can formulate our first main theorem; see Theorem~\ref{thm1} below for
a more precise statement. The idea of the proof is to upgrade the
homomorphism $\widehat{QH}_d \rightarrow \widehat{AH}_d$
constructed in \cite{Rou, BK} to the entire 2-category $\UU(\g)$.

\begin{iTheorem}\label{th:A}
Let $\R$ be either a locally finite Abelian or a Schurian
$\k$-linear category equipped with a categorical Heisenberg action.
Let $I$ be the {\em spectrum} of $\R$, that is, the set of eigenvalues
of the given endomorphism $x:E \Rightarrow E$.
This set is closed under the maps $i \mapsto i^{\pm}$ defined above.
Let $\mathfrak{g} = \mathfrak{sl}_I'$ be the corresponding Kac-Moody algebra with weight lattice $X$.
For each $\lambda \in X$, there is a Serre subcategory $\R_\lambda$ of
$\R$ defined explicitly in $\S$\ref{swd} below in terms of 
the action of $\End_{\Heis_k}(\unit)$ (``bubbles'').
Moreover, there is a canonically induced categorical action of $\g$ on
$(\R_\lambda)_{\lambda \in X}$
in the sense of (KM1)--(KM3).
\end{iTheorem}

This theorem considerably simplifies the construction of the most important
examples of categorical Kac-Moody actions. In these examples,
the existence of a Heisenberg action is straightforward to
demonstrate, so that Theorem \ref{th:A} can be applied without any need to develop the theory to the
point of being able to check relations on the Grothendieck
group. Of course it is still important to 
investigate such aspects, but it is helpful to have the rich
structure theory of a categorical Kac-Moody action in place from the outset.
For example, one often wants to compute the spectrum $I$ exactly, or to
find an explicit combinatorial description
of the underlying crystal structure on the
set $\B$ of isomorphism classes of irreducible objects.
The answers to these sorts of more intricate combinatorial questions
tend to vary in a discontinuous fashion as parameters change, whereas
the existence of a Heisenberg action is more robust.

\subsection*{Representations of symmetric groups and related Hecke algebras}
The original motivating example comes from the representation theory of the
symmetric groups $\SG_d$.
As observed in \cite[$\S$7.1]{CR}, the classical representation theory of symmetric groups (Specht modules, branching rules, blocks, etc...) implies that
$\R := \bigoplus_{d \geq 0} \k\SG_d\fdlmod$
admits the structure of an $\mathfrak{sl}_I'$-categorification
with $E$ given by induction and $F$ by restriction. The set $I$ (which is the spectrum
in our language) is
the image of $\Z$ in $\k$, so that 
$\mathfrak{sl}_I'$
is $\mathfrak{sl}_\infty(\C)$ if
$\qc = 0$ or 
$\widehat{\mathfrak{sl}}_\qc(\C)'$ if $\qc > 0$.
Applying ``control by $K_0$" it follows that there is an induced categorical action of $\mathfrak{g}=\mathfrak{sl}_I'$; the Grothendieck group $K_0(\R)$ is a $\Z$-form for the basic representation of $\mathfrak{g}$ (e.g., see \cite[Theorem 4.18]{BK2}).
Subsequently, Khovanov \cite{K} used this example to motivate his definition of the degenerate Heisenberg category $\Heis_{-1}$, making the existence of a categorical Heisenberg action on $\R$ almost tautological:
the conditions (H1) and (H2) are immediate while (H3) follows from the 
Mackey isomorphism $F \circ E \cong E \circ F\oplus
\operatorname{Id}$. So now Theorem \ref{th:A} gives a new proof of the existence of a categorical action of $\mathfrak{g}$,
without any need to appeal to combinatorial facts from the representation theory of symmetric groups. (See also \cite{QSY} for a different point of view.)

There are many much-studied variations on this example, in which one replaces $\k \SG_d$ by
higher level cyclotomic quotients of (degenerate or quantum) affine Hecke algebras or
quiver Hecke algebras; see \cite{Ariki, BK2, KK}.
The Grothendieck groups in these cases give $\Z$-forms for the other integrable highest or lowest weight representations.
Another closely related situation is the category $\mathcal{O}$ for rational Cherednik algebras of types
  $G(\ell,1,d)$ for $d \geq 0$, which categorifies Fock space; see \cite{GGOR, Shan}.
  This also includes categories of modules over cyclotomic $q$-Schur
  algebras as a special case. 
We refer the reader to \cite[$\S\S$6--7]{BSWqheis} for further
  discussion of this from the perspective of the quantum Heisenberg category;
our approach 
does not require any integrality assumptions unlike much of the existing literature.

\subsection*{Representations of the general linear group and related algebras}
There are many variants of the representation theory of the general
linear group, including
\begin{itemize}
\item rational representations of the algebraic group $GL_n$ over $\k$;
\item representations of the Lie algebra 
$\mathfrak{gl}_n(\C)$
in the BGG category $\mathcal O$;
\item analogous categories for the general linear supergroup
  $GL_{m|n}$ and its Lie superalgebra;
\item finite-dimensional representations of restricted enveloping algebras arising from the Lie
algebra $\mathfrak{gl}_n(\k)$ over a field of positive characteristic;
\item analogous categories for the quantized enveloping algebra
  $U_q(\mathfrak{gl}_n)$, including situations in which $q$ is a root
  of unity.
\end{itemize}
Each of these gives rise to a locally finite Abelian category $\R$
admitting a categorical Heisenberg action of central charge zero, either degenerate
in the classical cases or quantum when $U_q(\mathfrak{gl}_n)$ is involved.
The endofunctors $E$ and $F$ are defined by tensoring
with the $n$-dimensional defining representation $V$ and its dual
$V^*$, respectively.
The endomorphism $x:E \Rightarrow E$ arises from the action of the
Casimir tensor, while $\tau:E^2 \Rightarrow E^2$ comes from the tensor flip classically,
or its braided analog defined by the $R$-matrix in the quantum case. 
The relations (H1)--(H3) are all easy to check, with (H3)
amounting to the existence of a particular isomorphism $V \otimes V^* \cong V^*
\otimes V$. On applying Theorem \ref{th:A}, we obtain a uniform proof of the
existence of a categorical Kac-Moody action on each of these
categories. In most cases, this action has already been constructed in the
literature via ``control by $K_0$;'' 
e.g. see \cite[$\S$7.4]{CR} and \cite[$\S$6.4]{RW}
for rational representations of $GL_n$,
\cite[$\S$7.5]{CR} and \cite[$\S$4.4]{BKrep}
for category $\mathcal O$,
\cite[$\S\S$6--7]{V} and \cite[$\S$5]{BSWqheis} for the quantum analogs,
and \cite[$\S$5.1]{CC} and \cite[$\S$3.2]{BLW}
for the super analogs. In particular, for rational representations of $GL_n$, the complexified Grothendieck group may be identified with $\bigwedge^n \operatorname{Nat}_p$, where $\operatorname{Nat}_p$ is a natural level zero representation of $\widehat{\mathfrak{sl}}_p(\mathbb{C})'$,
while for integral blocks of category $\mathcal O$ for $\mathfrak{gl}_{m|n}(\mathbb{C})$
the complexified Grothendieck group is $\operatorname{Nat}_+^{\otimes m} \otimes \operatorname{Nat}_-^{\otimes n}$ where $\operatorname{Nat}_{\pm}$ are the natural and dual natural representations of $\mathfrak{sl}_\infty(\mathbb{C})$.
As well as establishing the existence of a categorical Kac-Moody action in all of these previously known cases, our approach encompasses several new situations involving quantum groups at roots of unity and restricted enveloping algebras in positive characteristic; for the latter we are only aware of \cite[Theorem 3.12]{NZ} in the existing literature, which treats a special case by explicitly checking relations at the level of $K_0$.

\subsection*{Generalized cyclotomic quotients}
We have already mentioned cyclotomic quotients of the affine Hecke
algebras $AH_d$. The
isomorphism theorem of \cite{BK} shows that these are isomorphic
to corresponding cyclotomic quotients of the quiver Hecke algebras
$QH_d$.
Both of these families of cyclotomic
quotients can also be obtained in a Morita equivalent form by taking
cyclotomic quotients of the Heisenberg category $\Heis_k$ or 
of the Kac-Moody 2-category $\UU(\g)$. This was first realized by
Rouquier in the Kac-Moody setting, indeed, it is the key to Rouquier's
definition of universal categorifications of integrable highest
weight modules; see \cite[Theorem 4.25]{R2}. The analogous theorem in the Heisenberg
setting is \cite[Theorem 1.7]{Bheis}. 
This point of view leads naturally to many more examples which we
refer to as {\em generalized cyclotomic quotients}; these were first
considered in the Kac-Moody setting in \cite[Proposition 5.6]{Wcan} and categorify tensor products of an integrable lowest weight and an integrable highest weight representation of $\g$.
In the final section of this article, we apply Theorem \ref{th:A}, this time with $\R$
being a Schurian category, to prove the following result (see Theorem~\ref{thm2}).

\begin{iTheorem}\label{th:B}
Consider
the generalized cyclotomic quotients $H_Z(\mu|\nu)$
of the Kac-Moody 2-category as defined in $\S$\ref{kmgcq}
and $H_Z(m|n)$ of the (degenerate or quantum) Heisenberg
category as defined in $\S$\ref{massive}.
Assuming the defining parameters are chosen so that
(\ref{boring})--(\ref{music}) hold,  these algebras are isomorphic via an explicit isomorphism.
\end{iTheorem}

The data needed to define generalized cyclotomic
quotients in the most general form includes a finite-dimensional,
commutative, local algebra $Z$, but generalized cyclotomic
quotients are already interesting
when $Z$ is simply taken to be equal to the
ground field $\k$.
Assuming this and taking the parameter $n$ (which in general is a monic
polynomial $n(u) \in Z[u]$) to be of degree zero, 
the generalized
cyclotomic quotient $H_Z(m|n)$ is the usual cyclotomic quotient of
$\Heis_k$ associated to $m$ (which in general is a monic polynomial $m(u) \in Z[u]$) for $k = -\deg m(u)$. Then Theorem \ref{th:B} specializes to the isomorphism theorem between cyclotomic
quotients of affine Hecke algebras and quiver Hecke algebras of type
A already mentioned.

Another example of a generalized cyclotomic quotient ``in nature"
arises by taking $Z= \k = \C$, and either
$m(u) = u$ and $n(u) = u+d$ in the degenerate case, or
$m(u) = u-1$ and $n(u) = u-q^{-2d}$ in the quantum case for $q$ that
is not a root of unity.
Under these assumptions, the generalized cyclotomic quotient $H_Z(m|n)$ is the locally unital algebra underlying the
{\em oriented Brauer category} denoted 
$\mathcal{OB}(d)$ in \cite{BCNR} in the degenerate case, or the 
{\em HOMFLY-PT skein category} denoted $\mathcal{OS}(z,q^d)$
in \cite{Bskein} in the quantum case.
The additive Karoubi envelopes of these monoidal categories are the Deligne
categories
$\underline{\operatorname{Re}}\!\operatorname{p} GL_d$
and $\underline{\operatorname{Re}}\!\operatorname{p}
U_q(\mathfrak{gl}_d)$, respectively.
Assuming that $d \in \Z$ (so that the spectrum $I$ is $\Z$ in the degenerate case or $q^{2\Z}$ in the quantum case),
Theorem \ref{th:B} implies that both of these categories are equivalent as $\k$-linear categories to the additive Karoubi envelope of the corresponding
generalized cyclotomic quotient of $\UU(\mathfrak{sl}_\infty)$. This
was proved originally using ``control by $K_0$" in \cite{Bskein}. (See also \cite[Theorem 10.2.7]{E-A} for a
related uniqueness result.)

Due to their universal nature, 
generalized cyclotomic quotients also play an important role in the proof
of the final theorem of the article, Theorem~\ref{thm3}, which explains how to construct a categorical Heisenberg action starting from a suitable Kac-Moody action.
This result gives a converse to Theorem \ref{th:A},
further clarifying the relationship between the three formulations (KM1)--(KM3),
(SL1)--(SL3) and (H1)--(H3) of the notion of categorical action discussed in this introduction.

The equivalence of Heisenberg and Kac-Moody actions revealed by this paper seems to be a feature of categorical actions which does not persist at the decategorified level. For the degenerate Heisenberg category and assuming that the ground field $\k$ is of characteristic zero, \cite[Theorem 1.1]{BSWk0} shows that the Grothendieck ring $K_0(\Kar(\Heis_k))$ of the additive Karoubi envelope of $\Heis_k$ is isomorphic to a certain $\Z$-form for the universal enveloping algebra of the infinite-dimensional Heisenberg Lie algebra specialized at central charge $k$. In this case, we expect that the passage from categorical Kac-Moody action to categorical Heisenberg action arising from Theorem~\ref{thm3} is related at the level of complexified Grothendieck groups to restriction from  $\mathfrak{sl}_\infty(\mathbb{C})$ (suitably completed) to its principal Heisenberg subalgebra.

\section{Preliminaries}

Throughout the article, $\k$ is an algebraically closed field and $z \in \k$ is a parameter.
We refer to the cases $z\neq 0$ and $z = 0$ as the
{\em quantum} and {\em degenerate} cases, respectively.
For use in the quantum case, we choose a root $q$ of the polynomial
$x^2-zx-1$, so that $z=q - q^{-1}$.
We also have in mind some fixed integer $k$, which we call the {\em central charge}. 

\subsection{Generating functions}\label{sgf}
We will often use generating functions when working with elements of
an algebra $A$. This means that we will work with formal
Laurent series $f(u) \in A(\!(u^{-1})\!)$ in an indeterminate $u$ (or
$v$, $w$, \dots).
We write $\left[f(u)\right]_{u^r}$ for the $u^r$-coefficient of such a
series,
$\left[f(u)\right]_{u^{< 0}}$ for $\sum_{r < 0} \left[f(u)\right]_{u^r} u^r$,
$\left[f(u)\right]_{u^{\geq 0}}$ for $\sum_{r \geq 0} \left[f(u)\right]_{u^r} u^r$ (which is a polynomial),
and so on.
To give an example, suppose that $$f(u) = \sum_{r \geq 0} f_r
u^{k-r} \in u^k 1_A + u^{k-1} A[\![u^{-1}]\!]$$
for some $f_r \in A$. Then
we can define new elements $g_r \in A$ by declaring that
$$
g(u) = \sum_{r \geq 0} g_r u^{-k-r} \in u^{-k}1_A + u^{-k-1}
A[\![u^{-1}]\!]
$$
is the inverse of the formal Laurent series $f(u)$.
In fact, setting $f_r := 0$ for $r < 0$, we have that
\begin{equation}\label{psid}
g_r = \det\left(-f_{s-t+1}\right)_{s,t=1,\dots,r}.
\end{equation}
This identity is valid even if
$A$ is non-commutative providing
the determinant is interpreted
as a suitably ordered Laplace expansion.
The best known instance of it arises in the algebra of
symmetric functions $\Sym$, in which the generating functions
$e(u) = \sum_{r \geq 0} e_r u^{-r}$ and
$h(u) = \sum_{r \geq 0} h_r {u^{-r}}$
for the
elementary and complete symmetric
functions
are related by the identity $e(u) h(-u) = 1$.
The determinantal formula from \cite[(I.2.6)]{Mac} is then
equivalent to (\ref{psid}).

\subsection{Locally finite Abelian and Schurian categories}\label{scars}
We will be studying categorical actions on $\k$-linear Abelian categories $\R$ satisfying certain
finiteness conditions, following \cite[$\S$2]{BS}.
The nicest condition to impose is that $\R$ is a
{\em locally finite Abelian category}. This means that
$\R$ is Abelian, all objects are of finite length, and the
space of morphisms between any two objects is finite-dimensional.
By a theorem of Takeuchi, an (essentially small) $\k$-linear category
$\R$ is a locally finite Abelian category in this sense if and only if it is
equivalent to the category $\fdrcomod C$ of finite-dimensional right
$C$-comodules for a
coalgebra $C$; e.g., see  \cite[Theorem 1.9.15]{EGNO}.

Special cases of locally finite Abelian categories
include {\em finite Abelian categories}, that is,
categories equivalent to $A\fdlmod$ for a finite-dimensional algebra
$A$,
and {\em essentially finite Abelian categories} in the sense of \cite[$\S$2.4]{BS}\footnote{In \cite[$\S$2.1]{BLW}, essentially finite Abelian categories were called ``Schurian categories" but we will use the latter terminology for a slightly different notion.},
that is, the locally finite Abelian categories that have enough
projectives and injectives.
An (essentially small) $\k$-linear category $\R$ is an essentially finite Abelian category if and only if it is equivalent to the category
$A\fdlmod$ of finite-dimensional left $A$-modules
for some essentially finite-dimensional locally unital algebra $A$.
Here, a {\em locally unital algebra} is an associative algebra equipped with
a {\em local unit}, that is, 
a system $\{1_a\:|\:a \in \mathbb A\}$ of mutually orthogonal
idempotents such that
\begin{equation}
A = \bigoplus_{a,a' \in \mathbb A} 1_a A 1_{a'}.
\end{equation}
We say that $A$ is {\em essentially finite-dimensional} if both $\dim 1_a A < \infty$ and
$\dim A 1_a< \infty$ for all $a \in \mathbb A$.
A left $A$-module means a left module $V$ as usual such
that $V = \bigoplus_{a \in \mathbb A} 1_a V$.

The other sort of Abelian categories with which we will be concerned
are the so-called Schurian categories. Although a well-known concept,
the language is not standard. The idea was discussed in detail
in \cite[$\S$2]{BD}\footnote{In \cite{BD} the terminology ``locally Schurian" was used instead of ``Schurian."}, but actually we will follow the conventions of
\cite[$\S$2.3]{BS}, according to which a {\em Schurian category} is a
category $\R$ that is equivalent to the category $A\lfdlmod$ of {locally
finite-dimensional} left $A$-modules for a {locally
  finite-dimensional}
locally unital algebra $A$. Here, a locally unital algebra $A$ (resp., a left $A$-module $V$) is called {\em locally finite-dimensional}
if $\dim 1_a A 1_{a'} < \infty$ (resp., $\dim 1_a V < \infty$) for all $a,a'
\in \mathbb A$.
Care is needed since an object $V$ in a Schurian category
$\R$
is not
necessarily of finite length, although all such modules have finite
composition multiplicities. Also for $V,W \in \R$ the morphism space $\Hom_{\R}(V,W)$
is not necessarily finite-dimensional,
although it is if $V$ is finitely generated.
We refer the reader to \cite{BD,BS} for further discussion.

To give a sense of the difference between locally finite Abelian categories and
Schurian categories, we formulate the appropriate notion of
Grothendieck group which should be used in the two settings.
If $\R$ is a locally finite Abelian category, the Grothendieck
group
$K_0(\R)$ is the free Abelian group generated by isomorphism classes $[V]$
of modules
subject to relations
$[V]=[V_1]+[V_2]$ for all short exact
sequences
$0 \rightarrow V_1 \rightarrow V \rightarrow V_2 \rightarrow 0$.
If $\R$ is a Schurian category, the Grothendieck group
$K_0(\R)$ is the free Abelian group generated by isomorphism
classes $[P]$ of finitely-generated projective modules subject to
relations $[P]=[P_1]+[P_2]$ if $P \cong P_1 \oplus P_2$.

Suppose that $\R$ is either locally finite Abelian or Schurian.
As our ground field is algebraically closed, we have that $\End_\R(L)\cong \k$
for any irreducible object $L \in \R$.
By a {\em sweet endofunctor} of $\R$, we mean a $\k$-linear functor $F:\R
\rightarrow \R$ that possesses both a left adjoint and
a right adjoint, with the two adjoints being isomorphic functors.
Such a functor is automatically additive and exact, so it induces an
endomorphism $[F]$
of the Grothendieck ring $K_0(\R)$.
Also, such a functor sends finitely generated objects to finitely
generated objects.
In the Schurian case, some further properties of sweet
endofunctors are discussed in \cite[$\S$2.4]{BD},
including the following:

\begin{Lemma}[{\cite[Lemma 2.12]{BD}}]\label{bdl}
Suppose that $F$ and $G$ are sweet endofunctors of a Schurian
category $\R$, and that $\eta:F \Rightarrow G$ is a natural
transformation such that $\eta_L:FL \rightarrow GL$ is an isomorphism
for each irreducible $L \in \R$. Then $\eta$ is an isomorphism.
\end{Lemma}

For finitely generated
 $V \in \R$, the functor
$\Hom_\R(V, -):\R \rightarrow \Vec$
has a left adjoint 
\begin{equation}\label{tensorproduct}
V \otimes-:\mathcal{V}ec_{\operatorname{fd}} \rightarrow \R.
\end{equation}
To make an explicit choice for this functor, one needs to pick a
basis for each finite-dimensional vector space $W$.
Then $V \otimes W := V^{\oplus \dim W}$ and, for a linear map $f:W
\rightarrow W'$, the morphism
$V \otimes f:V \otimes W \rightarrow V \otimes W'$ is 
the morphism $V^{\oplus \dim W} \rightarrow V^{\oplus \dim W'}$ defined
by the matrix of $f$ with respect to the fixed bases.

\subsection{Diagrammatics}
We will use the string calculus for strict monoidal categories
and strict 2-categories
as explained in \cite[Chapter 2]{TV}.
We will also use 
analogous
diagrammatic formalism 
when working with module categories and 2-representations.

To give a brief review,
let $\A$ be a strict $\k$-linear monoidal category.
A (strict) {\em module category} over $\A$
is a
$\k$-linear category $\R$ plus a $\k$-linear functor
$-\otimes-:\A \boxtimes \R \rightarrow \R$ satisfying
associativity and unity axioms.
Here, $\A \boxtimes \R$ is the $\k$-linearization of
the Cartesian product $\A \times \R$.
Equivalently, a module category
is a $\k$-linear category $\R$ together with a strict $\k$-linear monoidal functor
$\mathtt{R}:\A \rightarrow \mathcal{E}nd_\k(\R)$, where
$\mathcal{E}nd_\k(\R)$ denotes the strict $\k$-linear monoidal
category with objects that are $\k$-linear endofunctors of $\R$ and morphisms
that are natural transformations.
We usually suppress the monoidal functor $\mathtt{R}$, using the same notation
$f:E \rightarrow F$ both for
a morphism in $\A$ and for
the natural transformation
between endofunctors of $\R$ that is its image under $\mathtt{R}$.
The evaluation $f_V:EV \rightarrow FV$ 
of this natural transformation
on an object $V \in \R$ will be represented diagrammatically by drawing a 
 line labelled by
$V$ on the right-hand side of the usual string diagram
for $f$:
$$
\mathord{
\begin{tikzpicture}[baseline = 0]
	\draw[-,darkg,thick] (0.08,-.3) to (0.08,.3);
	\draw[-] (-.4,-.3) to (-.4,-.14);
	\draw[-] (-.4,.3) to (-.4,.14);
      \draw (-.4,0) circle (4pt);
   \node at (-.4,0) {$\scriptstyle{f}$};
   \node at (-.4,-.47) {$\scriptstyle{E}$};
   \node at (.08,-.47) {$\darkg\scriptstyle{V}$};
   \node at (-.4,.47) {$\scriptstyle{F}$};
\end{tikzpicture}
}\:.
$$
This line represents the identity endomorphism of the object $V$. Another
morphism $g:V
\rightarrow W$ in $\A$ can be represented by placing a coupon labelled
by $g$ on it.
For example, the following depicts
$(f \otimes W) \circ (E \otimes g) = f \otimes g = (F
\otimes g) \circ (f \otimes V)$:
\begin{equation*}
\mathord{
\begin{tikzpicture}[baseline = 0]
	\draw[-,darkg,thick] (0.08,-.4) to (0.08,-.24);
	\draw[-,darkg,thick] (0.08,.4) to (0.08,.04);
      \draw[darkg,thick] (0.08,-0.1) circle (4pt);
   \node at (0.08,-0.1) {$\darkg\scriptstyle{g}$};
	\draw[-] (-.8,-.4) to (-.8,-.04);
	\draw[-] (-.8,.4) to (-.8,.24);
      \draw (-.8,0.1) circle (4pt);
   \node at (-.8,.1) {$\scriptstyle{f}$};
   \node at (-.8,-.57) {$\scriptstyle{E}$};
   \node at (.08,-.57) {$\darkg\scriptstyle{V}$};
   \node at (-.8,.57) {$\scriptstyle{F}$};
   \node at (.08,.57) {$\darkg\scriptstyle{W}$};
\end{tikzpicture}
}
\quad=\quad
\mathord{
\begin{tikzpicture}[baseline = 0]
	\draw[-,darkg,thick] (0.08,-.4) to (0.08,-.14);
	\draw[-,darkg,thick] (0.08,.4) to (0.08,.14);
      \draw[darkg,thick] (0.08,0) circle (4pt);
   \node at (0.08,0) {$\darkg\scriptstyle{g}$};
	\draw[-] (-.8,-.4) to (-.8,-.14);
	\draw[-] (-.8,.4) to (-.8,.14);
      \draw (-.8,0) circle (4pt);
   \node at (-.8,0) {$\scriptstyle{f}$};
   \node at (-.8,-.57) {$\scriptstyle{E}$};
   \node at (.08,-.57) {$\darkg\scriptstyle{V}$};
   \node at (-.8,.57) {$\scriptstyle{F}$};
   \node at (.08,.57) {$\darkg\scriptstyle{W}$};
\end{tikzpicture}
}
\quad=\quad
\mathord{
\begin{tikzpicture}[baseline = 0]
	\draw[-,darkg,thick] (0.08,-.4) to (0.08,-.04);
	\draw[-,darkg,thick] (0.08,.4) to (0.08,.24);
      \draw[darkg,thick] (0.08,0.1) circle (4pt);
   \node at (0.08,0.1) {$\darkg\scriptstyle{g}$};
	\draw[-] (-.8,-.4) to (-.8,-.24);
	\draw[-] (-.8,.4) to (-.8,.04);
      \draw (-.8,-0.1) circle (4pt);
   \node at (-.8,-.1) {$\scriptstyle{f}$};
   \node at (-.8,-.57) {$\scriptstyle{E}$};
   \node at (.08,-.57) {$\darkg\scriptstyle{V}$};
   \node at (-.8,.57) {$\scriptstyle{F}$};
   \node at (.08,.57) {$\darkg\scriptstyle{W}$};
\end{tikzpicture}
}\:.
\end{equation*}
The equality of these morphisms is
the {\em interchange law} for module categories.

Suppose instead that $\AA$ is a strict $\k$-linear 2-category.
A (strict) {\em 2-representation}
of $\AA$ is a family $(\R_\lambda)_{\lambda \in \AA}$ of $\k$-linear
categories indexed by the objects of $\AA$, plus 
$\k$-linear functors $\mathcal{H}om_{\AA}(\lambda,\mu) \boxtimes
\R_\lambda \rightarrow \R_\mu$ for $\lambda,\mu \in \AA$ satisfying
associativity and unity axioms.
Equivalently, letting
$\mathfrak{Cat}_\k$ be the strict $\k$-linear 2-category of $\k$-linear
categories,
a 2-representation is a family
$(\R_\lambda)_{\lambda \in \AA}$ of $\k$-linear categories together with
a  strict $\k$-linear
2-functor $\mathbf{R}:\AA \rightarrow \mathfrak{Cat}_\k$
such that $\mathbf{R}(\lambda) = \R_\lambda$ for each $\lambda \in
\AA$.
As with module categories, when working with a 2-representation we
will usually drop the 2-functor
$\mathbf{R}$ from our notation.
The string calculus can be used in this setting too.
For example, a 1-morphism $F 1_\lambda:\lambda \rightarrow \mu$ 
in $\AA$ gives rise to a functor
$F|_{\R_\lambda}:\R_\lambda \rightarrow \R_\mu$; the diagram
$$
\mathord{
\begin{tikzpicture}[baseline = 0]
	\draw[-,darkg,thick] (0.08,-.3) to (0.08,-.14);
	\draw[-,darkg,thick] (0.08,.3) to (0.08,.14);
	\draw[-] (-0.5,.3) to (-0.5,-.3);
      \draw[darkg,thick] (0.08,0) circle (4pt);
   \node at (0.08,0) {$\darkg\scriptstyle{g}$};
   \node at (.08,-.47) {$\darkg\scriptstyle{V}$};
   \node at (-.5,-.47) {$\scriptstyle{F}$};
   \node at (.08,.47) {$\darkg\scriptstyle{W}$};
   \node at (-.3,0.02) {$\color{gray}\scriptstyle{\lambda}$};
\end{tikzpicture}
}
$$
depicts the morphism in $\R_\mu$ 
obtained by applying this 
to
morphism $g:V \rightarrow W$ in $\R_\lambda$.
We say that $(\R_\lambda)_{\lambda \in \AA}$ is a {\em locally finite
  Abelian} or a {\em Schurian} 2-representation if each of the categories $\R_\lambda$ is
a locally finite Abelian category or a Schurian category, respectively.

\subsection{A version of Hensel's lemma}
Let $Z$ be a finite-dimensional, commutative, local
$\k$-algebra with unique maximal ideal $J = J(Z)$.
As $\k$ is algebraically closed, we may naturally identify
the quotient $Z / J$ with $\k$. Note that two polynomials $g(u),
h(u) \in Z[u]$ are relatively prime if and only if their
images in $\k[u]$ are relatively prime. Equivalently, there exist
$a(u), b(u) \in Z[u]$ such that $a(u)g(u)+b(u)h(u) = 1$.
The following is well known but we could not
find a suitable reference.

\begin{Lemma}
Suppose that $I$ is a proper ideal of $Z$ such that $I^2 = 0$.
Let $\bar{Z} := Z / I$.
\begin{itemize}
\item[(1)]
Suppose that we are given a monic polynomial $\bar f(u) \in
\bar{Z}[u]$ and some choice of $\hat{f}(u) \in Z[u]$ lifting $\bar
f(u)$.
Then there is a unique {monic} lift $f(u)$ of $\bar f(u)$
such that $\hat f(u) = f(u) q(u)$
for $q(u) \in 1 + I[u]$. Moreover, $\deg f(u) = \deg \bar f(u)$.
\item[(2)]
For a monic lift $f(u)$ of $\bar f(u)$ as in (1), 
suppose in addition that we are given relatively prime monic polynomials 
$\bar g(u), \bar h(u) \in \bar{Z}[u]$
such that $\bar f(u) = \bar g(u) \bar h(u)$.
There exist monic lifts $g(u)$ of $\bar g(u)$ and $h(u)$
of $\bar h(u)$ such that $f(u) = g(u) h(u)$.
\item[(3)] The monic lifts $g(u)$ and $h(u)$ in (2) are unique.
\end{itemize}
\end{Lemma}

\begin{proof}
\begin{enumerate}[wide]
    \item 
Let $p(u)$ be any monic lift of $\bar f(u)$. It is automatically of the
same degree. 
By the division algorithm, we have that $\hat f(u) = p(u)
q(u)+r(u)$ for $r(u)$ with $\deg r(u) < \deg p(u)$.
On reducing coefficients modulo $I$, we
see that $q(u) \in 1 + I[u]$ and $r(u) \in I[u]$.
Since $I^2 = 0$ it follows that $r(u) = r(u) q(u)$. Hence, 
we have that $\hat f(u) = f(u)q(u)$ for $f(u) := p(u)+r(u)$, which
is another monic lift of $\bar f(u)$.
Uniqueness is obvious.
\item
Let $\hat g(u)$ and $\hat h(u)$ be any lifts of $\bar g(u)$ and
$\bar h(u)$.
Since $\bar g(u), \bar h(u)$ are relatively prime, there exist
$a(u), b(u) \in Z[u]$ such that $a(u)
\hat g(u) + b(u) \hat h(u)=1$.
Applying (1) to the lift $\hat g(u) + (f(u)-\hat g(u) \hat h(u))b(u)$
of $\bar g(u)$, we see
that there exists
a monic lift $g(u)$ of $\bar g(u)$ and $p(u)\in 1+I[u]$
such that $g(u) p(u) = \hat g(u)+(f(u)-\hat g(u)\hat h(u))b(u)$.
Similarly there is a monic lift $h(u)$ of $\bar h(u)$
and $q(u) \in 1+I[u]$ such that $h(u) q(u) = \hat h(u) + (f(u)-\hat g(u)\hat h(u)) a(u)$.
Using the assumption $I^2=0$, it is easy to check that
$f(u) = g(u) h(u) p(u) q(u)$.
Moreover since $f(u)$ and $g(u)h(u)$ are monic and $p(u) q(u) \in
1+I[u]$, we must actually have that $p(u) q(u) = 1$.
\item Suppose that we have two such factorizations
$f(u) = g(u) h(u) = g'(u) h'(u)$. 
Then $g'(u) = g(u) + s(u)$ and
$h'(u) = h(u) + t(u)$ for $s(u), t(u) \in I[u]$, and we deduce that
$g(u) t(u) + h(u) s(u) = 0$.
Again we choose $a(u), b(u) \in Z[u]$ so that $a(u) g(u) + b(u) h(u) = 1$. 
Then we have that 
$$
(1-b(u)h(u)) t(u) + a(u) h(u) s(u) = a(u) g(u) t(u) + a(u) h(u) s(u)
= 0.
$$
Hence, $t(u) = (b(u)t(u) - a(u) s(u)) h(u)$ and
$h'(u) = (1 + b(u)t(u)-a(u)s(u)) h(u)$.
But $h(u)$ and $h'(u)$ are both monic and 
$b(u)t(u)-a(u)s(u) \in I[u]$, which implies that $b(u)t(u)-a(u)s(u) =
0$, i.e., $h'(u) = h(u)$.
Similarly, $g'(u) = g(u)$.
\end{enumerate}
\end{proof}

\begin{Corollary}
Suppose that $f(u) \in Z[u]$ is a monic polynomial whose reduction
modulo $J$ is $\bar f(u) \in \k[u]$.
Suppose that we are given a factorization $\bar f(u) = \bar g(u) \bar
h(u)$ for relatively prime monic polynomials $\bar g(u), \bar h(u) \in
\k[u]$.
There exist unique monic lifts $g(u), h(u) \in Z[u]$ of $\bar
g(u), \bar h(u)$ such that $f(u) = g(u) h(u)$.
\end{Corollary}

\begin{proof}
This follows from the lemma by induction on the nilpotency degree of
$J$.
\end{proof}

\begin{Corollary}\label{hens}
Suppose that $f(u) \in Z[u]$ is a monic polynomial. 
Let $p_i$ be the multiplicity of $i \in \k$ as a root of 
$\bar f(u) \in \k[u]$, i.e.,
$\bar f(u) = \prod_{i \in I} (u-i)^{p_i}$ for some subset $I$ of $\k$.
Then there exist unique monic polynomials 
$f_i(u) \in u^{p_i}+J[u]$
 such that $f(u) = \prod_{i \in I} f_i(u-i)$.
\end{Corollary}

\begin{proof}
This follows from the previous corollary by induction on $\deg f(u)$.
(Note the assumption that $f_i(u)$ is monic and belongs to $u^{p_i}+J[u]$ is equivalent
to the assertion that $f_i(u-i)$ is a monic lift of $(u-i)^{p_i} \in \k[u]$.)
\end{proof}

\section{Three diagrammatic categories}

In this section, we review the definitions of the three diagrammatic categories that
are the subject of the paper:
the degenerate Heisenberg category, the quantum Heisenberg
category,
and the Kac-Moody 2-category of type A.
We also explain how to recast the defining relations in terms of generating functions.

\subsection{The degenerate Heisenberg category}\label{dh}
The Heisenberg category $\Heis_k$ is a strict $\k$-linear 
monoidal category defined by generators and relations.
In this subsection, we review the definition in the degenerate case
$z=0$. This was worked out originally by
Khovanov \cite{K} for central charge $k=-1$ (our
convention), then extended to all negative central charges in \cite{MS18}.
We instead follow the approach of
\cite[Theorem 1.2]{Bheis}, which simplified the presentation and incorporated
also the non-negative central charges, with $\Heis_0$
being the affine oriented Brauer category from \cite{BCNR}.
In this approach, the {\em degenerate Heisenberg category}
$\Heis_k$ is the strict $\k$-linear monoidal category
generated by objects
$E=\up$ and $F=\down$
and
morphisms
\begin{align}\label{dHgens}
\mathord{

}&:F \otimes E\rightarrow\unit
\end{align}
subject to certain relations.
To record these, we denote $n\geq 0$ dots on a string instead by labelling a single dot
with the multiplicity $n$.
Also introduce
the sideways crossings
\begin{align*}
\mathord{
%
}\:,
\end{align*}
and
the negatively dotted bubbles
\begin{align*}
\mathord{
}
&:=
\left\{\begin{array}{ll}
\det\left(
\mathord{
\begin{tikzpicture}[baseline = 1.25mm]
  \draw[<-] (0,0.4) to[out=180,in=90] (-.2,0.2);
  \draw[-] (0.2,0.2) to[out=90,in=0] (0,.4);
 \draw[-] (-.2,0.2) to[out=-90,in=180] (0,0);
  \draw[-] (0,0) to[out=0,in=-90] (0.2,0.2);
   \node at (-0.2,0.2) {$\dt$};
   \node at (-0.7,0.2) {$\scriptstyle{r-s+k}$};
\end{tikzpicture}
}\:
\right)_{r,s=1,\dots,n}
\phantom{(-1)^{n+1}}&\text{if $k \geq n > 0$,}\\
1_\unit&\text{if $k \geq n=0$,}\\
0&\text{if $k \geq n < 0$,}
\end{array}%
\right.
\end{align*}
Then the relations are as follows:
\begin{align}\label{dAHA}
\mathord{
%
}
\:.
\end{align}
In fact, one only needs to impose {\em one} of the adjunction
relations (\ref{rightadj}) or (\ref{leftadj}), then the other one
follows automatically. Moreover, $\Heis_k$ is strictly pivotal, i.e.,
the following relations hold:
\begin{align*}
\mathord{
%
}\:.
\end{align*}
These assertions are established in \cite[Theorem 1.3]{Bheis}.

The category $\Heis_k$ has an alternative presentation which is often useful
when constructing $\Heis_k$-module categories since it involves fewer
generators and relations. In this approach, which is \cite[Definition 1.1]{Bheis},
one just needs the generating morphisms
$%
\:$ (hence, we also have the rightwards crossing defined
as above),
subject to the relations (\ref{dAHA})--(\ref{rightadj}) together with the omnipotent
{\em inversion relation}, namely, that the 
following 
is an isomorphism in the additive envelope:
\begin{align}
\label{invrela}
\left[
%
}
\right]
&:E \otimes F \oplus
\unit^{\oplus (-k)}
\:\rightarrow\:
 F \otimes E&&\text{if $k \leq
  0$}.\label{invrelb}
\end{align}
The resulting category then contains unique morphisms
$\:%
\:$
such that the other relations
(\ref{rightcurl})--(\ref{sideways}) hold; see \cite[Lemma 5.2]{BSWk0}.

The following additional relations 
are also
derived in \cite[Theorem 1.3]{Bheis}:
the {\em infinite Grassmannian relation}
\begin{align*}
\sum_{r\in \Z}
\begin{array}{l}
\mathord{
\begin{tikzpicture}[baseline = 1.25mm]
  \draw[<-] (0,0.4) to[out=180,in=90] (-.2,0.2);
  \draw[-] (0.2,0.2) to[out=90,in=0] (0,.4);
 \draw[-] (-.2,0.2) to[out=-90,in=180] (0,0);
  \draw[-] (0,0) to[out=0,in=-90] (0.2,0.2);
   \node at (-0.2,0.2) {$\dt$};
   \node at (-.38,0.2) {$\scriptstyle{r}$};
\end{tikzpicture}
}\\
\mathord{
\begin{tikzpicture}[baseline = 1.25mm]
  \draw[->] (0.2,0.2) to[out=90,in=0] (0,.4);
  \draw[-] (0,0.4) to[out=180,in=90] (-.2,0.2);
\draw[-] (-.2,0.2) to[out=-90,in=180] (0,0);
  \draw[-] (0,0) to[out=0,in=-90] (0.2,0.2);
   \node at (0.2,0.2) {$\dt$};
   \node at (.7,0.2) {$\scriptstyle{n-r-2}$};
\end{tikzpicture}
}\end{array}\!\!\!
&=-\delta_{n,0} 1_\unit
\end{align*}
for any $n \in \Z$,
the {\em alternating braid relation}
\begin{align*}
\mathord{
%
}
\end{align*}
for $n \in \Z$.
It seems to be most convenient to work with these relations in
terms of generating functions as in $\S$\ref{sgf}.
In order to do this, we switch henceforth to using the notation
$\mathord{
%
}$ to denote a dot of multiplicity $n$; we do this also
for negatively
dotted bubbles using negative values of $n$. Then we can 
represent linear combinations of monomials 
by labelling dots by {polynomials} in $x$ too. 
Viewing the power series $$(u-x)^{-1} = u^{-1} + u^{-2} x + u^{-3} x^2 +
  \cdots \in \k[x][\![u^{-1}]\!]$$ 
as a generating function for multiple dots on a string,
the dot sliding relation implies the following:
\begin{equation}
\mathord{
%
}\:.
\end{equation}
To write the other relations in this form, we use the following
generating functions for the dotted bubbles:
\begin{align}\label{igq}
\anticlock(u) &:= \sum_{r\in\Z}
\mathord{
%
} \:u^{-r-1}\in u^{-k} 1_\unit +u^{-k-1}\End_{\Heis_k}(\unit)\llbracket u^{-1}\rrbracket.\label{igp}
\end{align}
Then the infinite Grassmannian relation implies that
\begin{equation}\label{igproper}
\anticlock(u)\; \clock(u) = 1_\unit.
\end{equation}
This puts us in the situation of (\ref{psid}), which explains
the origin of the determinantal formulae used to define the negatively dotted
bubbles above.
In terms of generating functions, the other relations involving bubbles translate into
the following:
\begin{align}
\mathord{
%
\label{moregin}}
\:.
\end{align}
To understand the last relation, it is helpful to note that $1-(u-x)^{-2} = \frac{(u-(x+1))(u-(x-1))}{(u-x)^2}$.

\begin{Lemma}\label{sure}
For a polynomial $p(u) \in \k[u]$, we have that
\begin{align}\label{sure1}
\mathord{
%
}p(u)\right]_{u^{-1}}\!\!\!.
\end{align}
\end{Lemma}

\begin{proof}
By linearity, it suffices to prove (\ref{sure1})--(\ref{sure2}) in the case that $p(u) =
u^r$ for $r \geq 0$, and in that case they follow easily on computing the $u^{-1}$-coefficient
on the right-hand side, recalling
also the definitions
(\ref{igq})--(\ref{igp}).
To deduce (\ref{sure3}), rewrite the left-hand side using
(\ref{sure1}), then apply the curl relation
(\ref{mostgin}).
\end{proof}

Finally, let us justify the terminology ``Heisenberg category''
in more detail.
Let $\Kar(\Heis_k)$ be the additive Karoubi envelope of $\Heis_k$,
and $K_0(\Kar(\Heis_k))$ be the Grothendieck ring of that monoidal category.
When the characteristic of the ground field is zero, $K_0(\Kar(\Heis_k))$
is isomorphic to the {\em Heisenberg ring}
$\rh_k$, that is, the ring
generated by elements $\{h_n^+, e_n^-\:|\:n \geq 0\}$ subject to
the relations
\begin{align}\label{upper}
h_0^+&=e_0^-=1,&
h_m^+ h_n^+ &= h_n^+ h_m^+,&
e_m^- e_n^- &= e_n^- e_m^-,&
  h_m^+ e_n^- &=
\sum_{r=0}^{\min(m,n)}
                  \binom{\,k\,}{\,r\,}\:e_{n-r}^- h_{m-r}^+.
\end{align}
This ring is a $\Z$-form
for the universal enveloping algebra of the infinite-dimensional
Heisenberg Lie algebra specialized at central charge $k$.
The existence of an isomorphism
$K_0(\Kar(\Heis_k)) \cong \rh_k$ was conjectured originally by Khovanov
in \cite{K} for $k=-1$ and it was proved in general in \cite[Theorem 1.1]{BSWk0}.
Under the isomorphism,
the classes $[E], [F] \in K_0(\Kar(\Heis_k))$
correspond to $h_1^+, e_1^- \in \rh_k$;
more generally
$h_n^+, e_n^-$ correspond to summands of $E^n$ and $F^n$ defined by
idempotents that correspond to the trivial and sign representations
of the symmetric group $\SG_n$.
When $\k$ is of positive characteristic, the category $\Kar(\Heis_k)$ does not have enough
indecomposable objects
for there to be any chance of an analogous isomorphism; in this case, we expect that one
should really work with a ``thickened'' version of $\Heis_k$ which incorporates generators of the affine Schur algebra.
However, for the purposes of the present article, the category $\Heis_k$ as defined
above is exactly the right object.

\subsection{The quantum Heisenberg category}\label{qh}
In the quantum case $z \neq 0$,
the category $\Heis_k$ was introduced in
\cite[Definition 4.1]{BSWqheis}, building on the earlier work \cite{LS} which produced a
different (but closely related) deformation of Khovanov's Heisenberg
category.
In fact, in the quantum case, there is an additional invertible parameter $t$ which
we will treat here as an indeterminate (although in applications one usually
specializes $t$ to a scalar in $\k^\times$).
Thus, in the quantum case, 
we will work over the 
ground ring \begin{equation}\label{groundring}
\K := \k[t,t^{-1}],
\end{equation} 
and
define the {\em quantum Heisenberg category} $\Heis_k$ to be the
strict $\K$-linear monoidal category
generated by objects
$E=\up$ and $F=\down$
and the following
morphisms:
\begin{align}\label{qHgens}
\mathord{
\begin{tikzpicture}[baseline = 0]
	\draw[->] (0.08,-.3) to (0.08,.4);
      \node at (0.08,0.05) {$\dt$};
\end{tikzpicture}
}
&:E\rightarrow E,
&\mathord{
\begin{tikzpicture}[baseline = 1mm]
	\draw[<-] (0.4,0.4) to[out=-90, in=0] (0.1,0);
	\draw[-] (0.1,0) to[out = 180, in = -90] (-0.2,0.4);
\end{tikzpicture}
}&:\unit\rightarrow F\otimes E
\:,
&\mathord{
\begin{tikzpicture}[baseline = 1mm]
	\draw[<-] (0.4,0) to[out=90, in=0] (0.1,0.4);
	\draw[-] (0.1,0.4) to[out = 180, in = 90] (-0.2,0);
\end{tikzpicture}
}&:E\otimes F\rightarrow\unit\:,\\
\mathord{
\begin{tikzpicture}[baseline = 0]
	\draw[->] (0.28,-.3) to (-0.28,.4);
	\draw[-,white,line width=4pt] (-0.28,-.3) to (0.28,.4);
	\draw[->] (-0.28,-.3) to (0.28,.4);
\end{tikzpicture}
}&:E\otimes E \rightarrow E\otimes E
\:,&
\mathord{
\begin{tikzpicture}[baseline = 1mm]
	\draw[-] (0.4,0.4) to[out=-90, in=0] (0.1,0);
	\draw[->] (0.1,0) to[out = 180, in = -90] (-0.2,0.4);
\end{tikzpicture}
}&:\unit\rightarrow E\otimes F
\:,
&
\mathord{
\begin{tikzpicture}[baseline = 1mm]
	\draw[-] (0.4,0) to[out=90, in=0] (0.1,0.4);
	\draw[->] (0.1,0.4) to[out = 180, in = 90] (-0.2,0);
\end{tikzpicture}
}&:F \otimes E\rightarrow\unit.\label{mango}
\end{align}
The generators on the left of (\ref{qHgens})--(\ref{mango}), the dot and the positive crossing, are required to be invertible.
The invertibility of the dot
means that now it makes sense to label dots by an arbitrary integer,
rather than just by $n \in \N$.
We denote the inverse of the positive crossing
by
\begin{align*}
\mathord{
\begin{tikzpicture}[baseline = -.5mm]
	\draw[->] (-0.28,-.3) to (0.28,.4);
	\draw[-,line width=4pt,white] (0.28,-.3) to (-0.28,.4);
	\draw[->] (0.28,-.3) to (-0.28,.4);
\end{tikzpicture}
}&:E\otimes E\rightarrow E \otimes E
\:,
\end{align*}
and call this the negative crossing.
Thus, we have that
\begin{align}
\mathord{
\begin{tikzpicture}[baseline = -1mm]
	\draw[-] (0.28,-.6) to[out=90,in=-90] (-0.28,0);
	\draw[->] (-0.28,0) to[out=90,in=-90] (0.28,.6);
	\draw[-,line width=4pt,white] (-0.28,-.6) to[out=90,in=-90] (0.28,0);
	\draw[-] (-0.28,-.6) to[out=90,in=-90] (0.28,0);
	\draw[-,line width=4pt,white] (0.28,0) to[out=90,in=-90] (-0.28,.6);
	\draw[->] (0.28,0) to[out=90,in=-90] (-0.28,.6);
\end{tikzpicture}
}&=
\mathord{
\begin{tikzpicture}[baseline = -1mm]
	\draw[->] (0.18,-.6) to (0.18,.6);
	\draw[->] (-0.18,-.6) to (-0.18,.6);
\end{tikzpicture}
}
=
\mathord{
\begin{tikzpicture}[baseline = -1mm]
	\draw[->] (0.28,0) to[out=90,in=-90] (-0.28,.6);
	\draw[-,line width=4pt,white] (-0.28,0) to[out=90,in=-90] (0.28,.6);
	\draw[->] (-0.28,0) to[out=90,in=-90] (0.28,.6);
	\draw[-] (-0.28,-.6) to[out=90,in=-90] (0.28,0);
	\draw[-,line width=4pt,white] (0.28,-.6) to[out=90,in=-90] (-0.28,0);
	\draw[-] (0.28,-.6) to[out=90,in=-90] (-0.28,0);
\end{tikzpicture}
}\:.\label{AHA1}
\end{align}
We also introduce the sideways crossings, both positive and negative,
\begin{align*}
\mathord{
\begin{tikzpicture}[baseline = -.5mm]
	\draw[->] (-0.28,-.3) to (0.28,.4);
	\draw[line width=4pt,white,-] (0.28,-.3) to (-0.28,.4);
	\draw[<-] (0.28,-.3) to (-0.28,.4);
\end{tikzpicture}
}&:=
\mathord{
\begin{tikzpicture}[baseline = 0]
	\draw[->] (0.3,-.5) to (-0.3,.5);
	\draw[line width=4pt,-,white] (-0.2,-.2) to (0.2,.3);
	\draw[-] (-0.2,-.2) to (0.2,.3);
        \draw[-] (0.2,.3) to[out=50,in=180] (0.5,.5);
        \draw[->] (0.5,.5) to[out=0,in=90] (0.8,-.5);
        \draw[-] (-0.2,-.2) to[out=230,in=0] (-0.6,-.5);
        \draw[-] (-0.6,-.5) to[out=180,in=-90] (-0.85,.5);
\end{tikzpicture}
}\:,
&\mathord{
\begin{tikzpicture}[baseline = -.5mm]
	\draw[<-] (0.28,-.3) to (-0.28,.4);
	\draw[line width=4pt,white,-] (-0.28,-.3) to (0.28,.4);
	\draw[->] (-0.28,-.3) to (0.28,.4);
\end{tikzpicture}
}&:=
\mathord{
\begin{tikzpicture}[baseline = 0]
	\draw[-] (-0.2,-.2) to (0.2,.3);
	\draw[-,line width=4pt,white] (0.3,-.5) to (-0.3,.5);
	\draw[->] (0.3,-.5) to (-0.3,.5);
        \draw[-] (0.2,.3) to[out=50,in=180] (0.5,.5);
        \draw[->] (0.5,.5) to[out=0,in=90] (0.8,-.5);
        \draw[-] (-0.2,-.2) to[out=230,in=0] (-0.6,-.5);
        \draw[-] (-0.6,-.5) to[out=180,in=-90] (-0.85,.5);
\end{tikzpicture}
}\:,\\
\mathord{
\begin{tikzpicture}[baseline = -.5mm]
	\draw[<-] (-0.28,-.3) to (0.28,.4);
	\draw[line width=4pt,white,-] (0.28,-.3) to (-0.28,.4);
	\draw[->] (0.28,-.3) to (-0.28,.4);
\end{tikzpicture}
}&:=
\mathord{
\begin{tikzpicture}[baseline = 0]
	\draw[-] (-0.2,.2) to (0.2,-.3);
	\draw[-,line width=4pt,white] (0.3,.5) to (-0.3,-.5);
	\draw[<-] (0.3,.5) to (-0.3,-.5);
        \draw[-] (0.2,-.3) to[out=130,in=180] (0.5,-.5);
        \draw[-] (0.5,-.5) to[out=0,in=270] (0.8,.5);
        \draw[-] (-0.2,.2) to[out=130,in=0] (-0.5,.5);
        \draw[->] (-0.5,.5) to[out=180,in=-270] (-0.8,-.5);
\end{tikzpicture}
}\:,&
\mathord{
\begin{tikzpicture}[baseline = -.5mm]
	\draw[->] (0.28,-.3) to (-0.28,.4);
	\draw[line width=4pt,white,-] (-0.28,-.3) to (0.28,.4);
	\draw[<-] (-0.28,-.3) to (0.28,.4);
\end{tikzpicture}
}&:=
\mathord{
\begin{tikzpicture}[baseline = 0]
	\draw[<-] (0.3,.5) to (-0.3,-.5);
	\draw[-,line width=4pt,white] (-0.2,.2) to (0.2,-.3);
	\draw[-] (-0.2,.2) to (0.2,-.3);
        \draw[-] (0.2,-.3) to[out=130,in=180] (0.5,-.5);
        \draw[-] (0.5,-.5) to[out=0,in=270] (0.8,.5);
        \draw[-] (-0.2,.2) to[out=130,in=0] (-0.5,.5);
        \draw[->] (-0.5,.5) to[out=180,in=-270] (-0.8,-.5);
\end{tikzpicture}
}\:,
\end{align*}
and the $(+)$-bubbles\footnote{In \cite{BSWqheis}, one also finds
  {\em $(-)$-bubbles} which will not be needed here.}
\begin{align*}
\mathord{
\begin{tikzpicture}[baseline = 1.25mm]
  \draw[->] (0.2,0.2) to[out=90,in=0] (0,.4);
  \draw[-] (0,0.4) to[out=180,in=90] (-.2,0.2);
\draw[-] (-.2,0.2) to[out=-90,in=180] (0,0);
  \draw[-] (0,0) to[out=0,in=-90] (0.2,0.2);
   \node at (0.5,0.2) {$\scriptstyle{n-k}$};
   \node at (0,0.2) {$+$};
\end{tikzpicture}}
&:=
\left\{\begin{array}{ll}
\mathord{
\begin{tikzpicture}[baseline = 1.25mm]
  \draw[->] (0.2,0.2) to[out=90,in=0] (0,.4);
  \draw[-] (0,0.4) to[out=180,in=90] (-.2,0.2);
\draw[-] (-.2,0.2) to[out=-90,in=180] (0,0);
  \draw[-] (0,0) to[out=0,in=-90] (0.2,0.2);
   \node at (0.2,0.2) {$\dt$};
   \node at (0.6,0.2) {$\scriptstyle{n-k}$};
\end{tikzpicture}}
\hspace{51.8mm}&\text{if $k<n$,}\\
t^{n+1} z^{n-1}\det\left(
\!\mathord{
\begin{tikzpicture}[baseline = 1.25mm]
  \draw[<-] (0,0.4) to[out=180,in=90] (-.2,0.2);
  \draw[-] (0.2,0.2) to[out=90,in=0] (0,.4);
 \draw[-] (-.2,0.2) to[out=-90,in=180] (0,0);
  \draw[-] (0,0) to[out=0,in=-90] (0.2,0.2);
   \node at (-0.2,0.2) {$\dt$};
   \node at (-0.85,0.2) {$\scriptstyle{r-s+k+1}$};
\end{tikzpicture}
}\,
\right)_{r,s=1,\dots,n},&\text{if $k \geq n > 0$,}\\
tz^{-1} 1_\unit&\text{if $k \geq n=0$,}\\
0&\text{if $k \geq n < 0$,}
\end{array}\right.\\
\mathord{
\begin{tikzpicture}[baseline = 1.25mm]
  \draw[<-] (0,0.4) to[out=180,in=90] (-.2,0.2);
  \draw[-] (0.2,0.2) to[out=90,in=0] (0,.4);
 \draw[-] (-.2,0.2) to[out=-90,in=180] (0,0);
  \draw[-] (0,0) to[out=0,in=-90] (0.2,0.2);
   \node at (0,0.2) {$+$};
   \node at (-0.48,0.2) {$\scriptstyle{n+k}$};
\end{tikzpicture}
}&:=
\left\{
\begin{array}{ll}
\mathord{
\begin{tikzpicture}[baseline = 1.25mm]
  \draw[<-] (0,0.4) to[out=180,in=90] (-.2,0.2);
  \draw[-] (0.2,0.2) to[out=90,in=0] (0,.4);
 \draw[-] (-.2,0.2) to[out=-90,in=180] (0,0);
  \draw[-] (0,0) to[out=0,in=-90] (0.2,0.2);
   \node at (-0.2,0.2) {$\dt$};
   \node at (-0.6,0.2) {$\scriptstyle{n+k}$};
\end{tikzpicture}
}&\text{if $-k < n$,}\\
(-1)^{n+1} t^{-n-1} z^{n-1}\det\left(\:\mathord{
\begin{tikzpicture}[baseline = 1.25mm]
  \draw[->] (0.2,0.2) to[out=90,in=0] (0,.4);
  \draw[-] (0,0.4) to[out=180,in=90] (-.2,.2);
\draw[-] (-.2,0.2) to[out=-90,in=180] (0,0);
  \draw[-] (0,0) to[out=0,in=-90] (0.2,0.2);
   \node at (0.2,0.2) {$\dt$};
   \node at (0.85,0.2) {$\scriptstyle{r-s-k+1}$};
\end{tikzpicture}
}\right)_{r,s=1,\dots,n}&\text{if $-k\geq n > 0$,}\\
-t^{-1}z^{-1} 1_\unit
&\text{if $-k \geq n=0$,}\\
0&\text{if $-k\geq n < 0$.}
\end{array}\right.
\end{align*}
The other defining relations are as follows:
\begin{align}\label{AHA2}
\mathord{
\begin{tikzpicture}[baseline = -.5mm]
	\draw[->] (0.28,-.3) to (-0.28,.4);
	\draw[line width=4pt,white,-] (-0.28,-.3) to (0.28,.4);
	\draw[->] (-0.28,-.3) to (0.28,.4);
\end{tikzpicture}
}-\mathord{
\begin{tikzpicture}[baseline = -.5mm]
	\draw[->] (-0.28,-.3) to (0.28,.4);
	\draw[line width=4pt,white,-] (0.28,-.3) to (-0.28,.4);
	\draw[->] (0.28,-.3) to (-0.28,.4);
\end{tikzpicture}
}&=
z\:\mathord{
\begin{tikzpicture}[baseline = -.5mm]
	\draw[->] (0.18,-.3) to (0.18,.4);
	\draw[->] (-0.18,-.3) to (-0.18,.4);
\end{tikzpicture}
}\:,
&\mathord{
\begin{tikzpicture}[baseline = -1mm]
	\draw[->] (0.45,-.6) to (-0.45,.6);
        \draw[-] (0,-.6) to[out=90,in=-90] (-.45,0);
        \draw[-,line width=4pt,white] (-0.45,0) to[out=90,in=-90] (0,0.6);
        \draw[->] (-0.45,0) to[out=90,in=-90] (0,0.6);
	\draw[-,line width=4pt,white] (0.45,.6) to (-0.45,-.6);
	\draw[<-] (0.45,.6) to (-0.45,-.6);
\end{tikzpicture}
}
&=
\mathord{
\begin{tikzpicture}[baseline = -1mm]
	\draw[->] (0.45,-.6) to (-0.45,.6);
        \draw[-,line width=4pt,white] (0,-.6) to[out=90,in=-90] (.45,0);
        \draw[-] (0,-.6) to[out=90,in=-90] (.45,0);
        \draw[->] (0.45,0) to[out=90,in=-90] (0,0.6);
	\draw[-,line width=4pt,white] (0.45,.6) to (-0.45,-.6);
	\draw[<-] (0.45,.6) to (-0.45,-.6);
\end{tikzpicture}
}
\:,
\qquad\qquad
\mathord{
\begin{tikzpicture}[baseline = -.5mm]
	\draw[->] (-0.28,-.3) to (0.28,.4);
      \node at (-0.16,-0.15) {$\dt$};
	\draw[-,line width=4pt,white] (0.28,-.3) to (-0.28,.4);
	\draw[->] (0.28,-.3) to (-0.28,.4);
\end{tikzpicture}
}=
\mathord{
\begin{tikzpicture}[baseline = -.5mm]
	\draw[->] (0.28,-.3) to (-0.28,.4);
	\draw[line width=4pt,white,-] (-0.28,-.3) to (0.28,.4);
	\draw[->] (-0.28,-.3) to (0.28,.4);
      \node at (0.145,0.23) {$\dt$};
\end{tikzpicture}
}\:,\\
\mathord{
\begin{tikzpicture}[baseline = -.8mm]
  \draw[->] (0.3,0) to (0.3,.4);
	\draw[-] (0.3,0) to[out=-90, in=0] (0.1,-0.4);
	\draw[-] (0.1,-0.4) to[out = 180, in = -90] (-0.1,0);
	\draw[-] (-0.1,0) to[out=90, in=0] (-0.3,0.4);
	\draw[-] (-0.3,0.4) to[out = 180, in =90] (-0.5,0);
  \draw[-] (-0.5,0) to (-0.5,-.4);
\end{tikzpicture}
}
&=
\mathord{\begin{tikzpicture}[baseline=-.8mm]
  \draw[->] (0,-0.4) to (0,.4);
\end{tikzpicture}
}\:,
&\mathord{
\begin{tikzpicture}[baseline = -.8mm]
  \draw[->] (0.3,0) to (0.3,-.4);
	\draw[-] (0.3,0) to[out=90, in=0] (0.1,0.4);
	\draw[-] (0.1,0.4) to[out = 180, in = 90] (-0.1,0);
	\draw[-] (-0.1,0) to[out=-90, in=0] (-0.3,-0.4);
	\draw[-] (-0.3,-0.4) to[out = 180, in =-90] (-0.5,0);
  \draw[-] (-0.5,0) to (-0.5,.4);
\end{tikzpicture}
}
&=
\mathord{\begin{tikzpicture}[baseline=-.8mm]
  \draw[<-] (0,-0.4) to (0,.4);
\end{tikzpicture}
}\:,\label{rightadj2}\\
\mathord{
\begin{tikzpicture}[baseline = -.8mm]
  \draw[-] (0.3,0) to (0.3,-.4);
	\draw[-] (0.3,0) to[out=90, in=0] (0.1,0.4);
	\draw[-] (0.1,0.4) to[out = 180, in = 90] (-0.1,0);
	\draw[-] (-0.1,0) to[out=-90, in=0] (-0.3,-0.4);
	\draw[-] (-0.3,-0.4) to[out = 180, in =-90] (-0.5,0);
  \draw[->] (-0.5,0) to (-0.5,.4);
\end{tikzpicture}
}
&=
\mathord{\begin{tikzpicture}[baseline=-.8mm]
  \draw[->] (0,-0.4) to (0,.4);
\end{tikzpicture}
}\:,
&\mathord{
\begin{tikzpicture}[baseline = -.8mm]
  \draw[-] (0.3,0) to (0.3,.4);
	\draw[-] (0.3,0) to[out=-90, in=0] (0.1,-0.4);
	\draw[-] (0.1,-0.4) to[out = 180, in = -90] (-0.1,0);
	\draw[-] (-0.1,0) to[out=90, in=0] (-0.3,0.4);
	\draw[-] (-0.3,0.4) to[out = 180, in =90] (-0.5,0);
  \draw[->] (-0.5,0) to (-0.5,-.4);
\end{tikzpicture}
}
&=
\mathord{\begin{tikzpicture}[baseline=-.8mm]
  \draw[<-] (0,-0.4) to (0,.4);
\end{tikzpicture}
}\:,\label{leftadj2}\\\label{saturday}
\mathord{
\begin{tikzpicture}[baseline = -0.5mm]
	\draw[<-] (0,0.6) to (0,0.3);
	\draw[-] (0.3,-0.2) to [out=0,in=-90](.5,0);
	\draw[-] (0.5,0) to [out=90,in=0](.3,0.2);
	\draw[-] (0.3,.2) to [out=180,in=90](0,-0.3);
	\draw[-] (0,-0.3) to (0,-0.6);
	\draw[-,line width=4pt,white] (0,0.3) to [out=-90,in=180] (.3,-0.2);
	\draw[-] (0,0.3) to [out=-90,in=180] (.3,-0.2);
\end{tikzpicture}
}&=\delta_{k,0}
t^{-1}\:\mathord{
\begin{tikzpicture}[baseline = -0.5mm]
	\draw[<-] (0,0.6) to (0,-0.6);
\end{tikzpicture}
}
\:\:\text{if $k \geq 0$,}
&\mathord{
\begin{tikzpicture}[baseline = -0.5mm]
	\draw[<-] (0,0.6) to (0,0.3);
	\draw[-] (-0.3,-0.2) to [out=180,in=-90](-.5,0);
	\draw[-] (-0.5,0) to [out=90,in=180](-.3,0.2);
	\draw[-] (-0.3,.2) to [out=0,in=90](0,-0.3);
	\draw[-] (0,-0.3) to (0,-0.6);
	\draw[-,line width=4pt,white] (0,0.3) to [out=-90,in=0] (-.3,-0.2);
	\draw[-] (0,0.3) to [out=-90,in=0] (-.3,-0.2);
\end{tikzpicture}
}&=
\delta_{k,0}
t\:\mathord{
\begin{tikzpicture}[baseline = -0.5mm]
	\draw[<-] (0,0.6) to (0,-0.6);
\end{tikzpicture}
}
\:\:\quad\text{if $k \leq 0$},\end{align}\begin{align}
\mathord{
\begin{tikzpicture}[baseline = 1.25mm]
  \draw[<-] (0,0.4) to[out=180,in=90] (-.2,0.2);
  \draw[-] (0.2,0.2) to[out=90,in=0] (0,.4);
 \draw[-] (-.2,0.2) to[out=-90,in=180] (0,0);
  \draw[-] (0,0) to[out=0,in=-90] (0.2,0.2);
   \node at (-0.2,0.2) {$\dt$};
   \node at (-0.6,0.2) {$\scriptstyle{n+k}$};
\end{tikzpicture}
}&=
{\textstyle\frac{\delta_{n,-k} t - \delta_{n,0} t^{-1}}{z}}
1_\unit
\:\text{if $-k \leq n\leq 0$,}
&
\mathord{
\begin{tikzpicture}[baseline = 1.25mm]
  \draw[->] (0.2,0.2) to[out=90,in=0] (0,.4);
  \draw[-] (0,0.4) to[out=180,in=90] (-.2,0.2);
\draw[-] (-.2,0.2) to[out=-90,in=180] (0,0);
  \draw[-] (0,0) to[out=0,in=-90] (0.2,0.2);
   \node at (0.2,0.2) {$\dt$};
   \node at (0.6,0.2) {$\scriptstyle{n-k}$};
\end{tikzpicture}
}\!&= 
{\textstyle\frac{\delta_{n,0} t - \delta_{n,k} t^{-1}}{z}}
1_\unit
\:\:\:\text{if $k \leq n \leq 0$,}\end{align}\begin{align}
\mathord{
\begin{tikzpicture}[baseline = -.9mm]
	\draw[<-] (-0.28,-.6) to[out=90,in=-90] (0.28,0);
	\draw[-,white,line width=4pt] (0.28,-.6) to[out=90,in=-90] (-0.28,0);
	\draw[-] (0.28,-.6) to[out=90,in=-90] (-0.28,0);
	\draw[->] (-0.28,0) to[out=90,in=-90] (0.28,.6);
	\draw[-,line width=4pt,white] (0.28,0) to[out=90,in=-90] (-0.28,.6);
	\draw[-] (0.28,0) to[out=90,in=-90] (-0.28,.6);
\end{tikzpicture}
}
&=
\mathord{
\begin{tikzpicture}[baseline = -.9mm]
	\draw[->] (0.08,-.6) to (0.08,.6);
	\draw[<-] (-0.28,-.6) to (-0.28,.6);
\end{tikzpicture}
}
+tz
\mathord{
\begin{tikzpicture}[baseline=-.9mm]
	\draw[<-] (0.3,0.6) to[out=-90, in=0] (0,.1);
	\draw[-] (0,.1) to[out = 180, in = -90] (-0.3,0.6);
	\draw[-] (0.3,-.6) to[out=90, in=0] (0,-0.1);
	\draw[->] (0,-0.1) to[out = 180, in = 90] (-0.3,-.6);
\end{tikzpicture}}
\!+z^2\!\sum_{r,s > 0}
\!\!\!\mathord{
\begin{tikzpicture}[baseline = 1mm]
  \draw[<-] (0,0.4) to[out=180,in=90] (-.2,0.2);
  \draw[-] (0.2,0.2) to[out=90,in=0] (0,.4);
 \draw[-] (-.2,0.2) to[out=-90,in=180] (0,0);
  \draw[-] (0,0) to[out=0,in=-90] (0.2,0.2);
   \node at (0,0.2) {$+$};
   \node at (-.57,0.2) {$\scriptstyle{-r-s}$};
\end{tikzpicture}
}
\mathord{
\begin{tikzpicture}[baseline=-.9mm]
	\draw[<-] (0.3,0.6) to[out=-90, in=0] (0,.1);
	\draw[-] (0,.1) to[out = 180, in = -90] (-0.3,0.6);
      \node at (0.44,-0.3) {$\scriptstyle{s}$};
	\draw[-] (0.3,-.6) to[out=90, in=0] (0,-0.1);
	\draw[->] (0,-0.1) to[out = 180, in = 90] (-0.3,-.6);
   \node at (0.27,0.3) {$\dt$};
      \node at (0.27,-0.3) {$\dt$};
   \node at (.43,.3) {$\scriptstyle{r}$};
\end{tikzpicture}}\:,&
\mathord{
\begin{tikzpicture}[baseline = -.9mm]
	\draw[->] (0.28,0) to[out=90,in=-90] (-0.28,.6);
	\draw[<-] (0.28,-.6) to[out=90,in=-90] (-0.28,0);
	\draw[-,line width=4pt,white] (-0.28,0) to[out=90,in=-90] (0.28,.6);
	\draw[-] (-0.28,0) to[out=90,in=-90] (0.28,.6);
	\draw[-,line width=4pt,white] (-0.28,-.6) to[out=90,in=-90] (0.28,0);
	\draw[-] (-0.28,-.6) to[out=90,in=-90] (0.28,0);
\end{tikzpicture}
}
&=\mathord{
\begin{tikzpicture}[baseline = -0.9mm]
	\draw[<-] (0.08,-.6) to (0.08,.6);
	\draw[->] (-0.28,-.6) to (-0.28,.6);
\end{tikzpicture}
}
-t^{-1}z
\mathord{
\begin{tikzpicture}[baseline=-0.9mm]
	\draw[-] (0.3,0.6) to[out=-90, in=0] (0,0.1);
	\draw[->] (0,0.1) to[out = 180, in = -90] (-0.3,0.6);
	\draw[<-] (0.3,-.6) to[out=90, in=0] (0,-0.1);
	\draw[-] (0,-0.1) to[out = 180, in = 90] (-0.3,-.6);
\end{tikzpicture}}
\!+z^2\!\sum_{r,s > 0}\!\!\!
\mathord{
\begin{tikzpicture}[baseline=-0.9mm]
	\draw[-] (0.3,0.6) to[out=-90, in=0] (0,0.1);
	\draw[->] (0,0.1) to[out = 180, in = -90] (-0.3,0.6);
      \node at (-0.4,0.3) {$\scriptstyle{r}$};
      \node at (-0.25,0.3) {$\dt$};
	\draw[<-] (0.3,-.6) to[out=90, in=0] (0,-0.1);
	\draw[-] (0,-0.1) to[out = 180, in = 90] (-0.3,-.6);
   \node at (-0.27,-0.4) {$\dt$};
   \node at (-.45,-.35) {$\scriptstyle{s}$};
\end{tikzpicture}}
\mathord{
\begin{tikzpicture}[baseline = 1mm]
  \draw[->] (0.2,0.2) to[out=90,in=0] (0,.4);
  \draw[-] (0,0.4) to[out=180,in=90] (-.2,0.2);
\draw[-] (-.2,0.2) to[out=-90,in=180] (0,0);
  \draw[-] (0,0) to[out=0,in=-90] (0.2,0.2);
   \node at (0,0.2) {$+$};
   \node at (.55,0.2) {$\scriptstyle{-r-s}$};
\end{tikzpicture}
}.\label{limb}\end{align}
As in 
the degenerate case, one actually only needs to impose one of the
adjunction relations (\ref{rightadj2}) or (\ref{leftadj2}), after which the other one may be
deduced as a consequence of the other relations.
Moreover, the quantum Heisenberg category is strictly
pivotal, so that one 
can introduce the downward dot and the downward
positive and negative crossings by
taking left and/or right mates of the upward ones.

Again like the degenerate case, there are also some alternative
presentations involving an 
inversion
relation; see \cite[Definitions 2.2 and 3.1]{BSWqheis}.
To formulate a version of this, one just needs the generating morphisms
$\begin{tikzpicture}[baseline = -1mm]
	\draw[->] (0.08,-.2) to (0.08,.2);
      \node at (0.08,0) {$\dt$};
\end{tikzpicture}$,
$\begin{tikzpicture}[baseline = -1mm]
	\draw[->] (0.2,-.2) to (-0.2,.2);
	\draw[-,white,line width=4pt] (-0.2,-.2) to (0.2,.2);
	\draw[->] (-0.2,-.2) to (0.2,.2);
\end{tikzpicture}\:$,
$\begin{tikzpicture}[baseline = .75mm]
	\draw[<-] (0.3,0.3) to[out=-90, in=0] (0.1,0);
	\draw[-] (0.1,0) to[out = 180, in = -90] (-0.1,0.3);
\end{tikzpicture}\:$ and
$\begin{tikzpicture}[baseline = .75mm]
	\draw[<-] (0.3,0) to[out=90, in=0] (0.1,0.3);
	\draw[-] (0.1,0.3) to[out = 180, in = 90] (-0.1,0);
\end{tikzpicture}\:$, the first two of which are required to be
invertible (hence, we also 
get negative upwards and positive/negative rightwards crossings as above),
subject to the relations (\ref{AHA2})--(\ref{rightadj2}) plus 
the {\em inversion relation} asserting that the following is
invertible:
\begin{align}
\label{invrel1a}
\left[\!\!\!\!\!
\begin{array}{r}
\mathord{
\begin{tikzpicture}[baseline = 0]
	\draw[->] (-0.28,-.3) to (0.28,.3);
	\draw[-,line width=4pt,white] (0.28,-.3) to (-0.28,.3);
	\draw[<-] (0.28,-.3) to (-0.28,.3);
   \end{tikzpicture}
}\\
\mathord{
\begin{tikzpicture}[baseline = 1mm]
	\draw[<-] (0.4,0) to[out=90, in=0] (0.1,0.4);
      \node at (-0.15,0.45) {$\phantom\bullet$};
	\draw[-] (0.1,0.4) to[out = 180, in = 90] (-0.2,0);
\end{tikzpicture}
}\\
\mathord{
\begin{tikzpicture}[baseline = 1mm]
	\draw[<-] (0.4,0) to[out=90, in=0] (0.1,0.4);
	\draw[-] (0.1,0.4) to[out = 180, in = 90] (-0.2,0);
      \node at (-0.15,0.45) {$\phantom\bullet$};
      \node at (-0.15,0.2) {$\dt$};
\end{tikzpicture}
}\\\vdots\:\:\;\\
\mathord{
\begin{tikzpicture}[baseline = 1mm]
	\draw[<-] (0.4,0) to[out=90, in=0] (0.1,0.4);
	\draw[-] (0.1,0.4) to[out = 180, in = 90] (-0.2,0);
     \node at (-0.52,0.2) {$\scriptstyle{k-1}$};
      \node at (-0.15,0.42) {$\phantom\bullet$};
      \node at (-0.15,0.2) {$\dt$};
\end{tikzpicture}
}
\end{array}
\right]
&:
E \otimes F \:\rightarrow\:
F \otimes E \oplus \unit^{\oplus k}
&&\text{if $k \geq
  0$,}\\
\left[\:
\mathord{
\begin{tikzpicture}[baseline = 0]
	\draw[<-] (0.28,-.3) to (-0.28,.3);
	\draw[-,line width=4pt,white] (-0.28,-.3) to (0.28,.3);
	\draw[->] (-0.28,-.3) to (0.28,.3);
\end{tikzpicture}
}\:\:\:
\mathord{
\begin{tikzpicture}[baseline = -0.9mm]
	\draw[<-] (0.4,0.2) to[out=-90, in=0] (0.1,-.2);
	\draw[-] (0.1,-.2) to[out = 180, in = -90] (-0.2,0.2);
\end{tikzpicture}
}
\:\:\:
\mathord{
\begin{tikzpicture}[baseline = -0.9mm]
	\draw[<-] (0.4,0.2) to[out=-90, in=0] (0.1,-.2);
	\draw[-] (0.1,-.2) to[out = 180, in = -90] (-0.2,0.2);
      \node at (0.38,0) {$\dt$};
\end{tikzpicture}
}
\:\:\:\cdots
\:\:\:
\mathord{
\begin{tikzpicture}[baseline = -0.9mm]
	\draw[<-] (0.4,0.2) to[out=-90, in=0] (0.1,-.2);
	\draw[-] (0.1,-.2) to[out = 180, in = -90] (-0.2,0.2);
     \node at (0.83,0) {$\scriptstyle{-k-1}$};
      \node at (0.38,0) {$\dt$};
\end{tikzpicture}
}
\right]
&:E \otimes F \oplus
\unit^{\oplus (-k)}
\:\rightarrow\:
 F \otimes  E&&
\text{if $k \leq 0$.}\label{invrel1b}
\end{align}
The situation is slightly more delicate than in the degenerate case as
it is also necessary to impose one additional 
relation:
\begin{itemize}
\item
If $k > 0$ we require that
$\begin{tikzpicture}[baseline = 1mm]
  \draw[<-] (0,0.4) to[out=180,in=90] (-.2,0.2);
  \draw[-] (0.2,0.2) to[out=90,in=0] (0,.4);
 \draw[-] (-.2,0.2) to[out=-90,in=180] (0,0);
      \node at (-0.2,0.2) {$\dt$};
      \node at (-.5,.2) {$\scriptstyle{-1}$};
  \draw[-] (0,0) to[out=0,in=-90] (0.2,0.2);
      \node at (0,0) {$\diam$};
\end{tikzpicture}
=-t^2 1_\unit$
where $\:\begin{tikzpicture}[baseline = .75mm]
	\draw[-] (0.3,0.3) to[out=-90, in=0] (0.1,0);
	\draw[->] (0.1,0) to[out = 180, in = -90] (-0.1,0.3);
      \node at (.1,0) {$\diam$};
\end{tikzpicture}\:$ is the last entry of the
inverse of the matrix (\ref{invrel1a}).
\item
If $k < 0$ we require that
$\begin{tikzpicture}[baseline = 1mm]
  \draw[-] (0,0.4) to[out=180,in=90] (-.2,0.2);
  \draw[-] (0.2,0.2) to[out=90,in=0] (0,.4);
 \draw[-] (-.2,0.2) to[out=-90,in=180] (0,0);
  \draw[->] (0,0) to[out=0,in=-90] (0.2,0.2);
      \node at (-0.2,0.2) {$\dt$};
      \node at (-.5,.2) {$\scriptstyle{-1}$};
      \node at (0,.4) {$\diam$};
\end{tikzpicture}
=-t^{-2} 1_\unit$
where $\:\begin{tikzpicture}[baseline = .75mm]
	\draw[-] (0.3,0) to[out=90, in=0] (0.1,0.3);
	\draw[->] (0.1,0.3) to[out = 180, in = 90] (-0.1,0);
      \node at (0.1,.3) {$\diam$};
\end{tikzpicture}\:$
is the last entry of the inverse of the matrix (\ref{invrel1b}).
\item
If $k=0$ there are two equivalent presentations
here:
if one picks
(\ref{invrel1a}) the additional relation is
$\mathord{
\begin{tikzpicture}[baseline = 1.5mm]
	\draw[<-] (-0.2,.5) to[out=-90,in=90] (0.2,.05);
	\draw[-] (0.2,.05) to[out=-90, in=0] (0,-0.15);
	\draw[-] (0,-0.15) to[out = 180, in = -90] (-0.2,.05);
	\draw[-,line width=4pt,white] (0.25,.5) to[out=-90,in=90] (-0.2,.05);
	\draw[-] (0.2,.5) to[out=-90,in=90] (-0.2,.05);
\draw[-](-.2,.5) to [out=90,in=180] (0,.7);
\draw[-](.2,.5) to [out=90,in=0] (0,.7);
\end{tikzpicture}
}= \frac{1-t^{-2}}{z}\unit$
where
$\begin{tikzpicture}[baseline = -1mm]
	\draw[->] (0.2,-.2) to (-0.2,.2);
	\draw[-,line width=4pt,white] (-0.2,-.2) to (0.2,.2);
	\draw[<-] (-0.2,-.2) to (0.2,.2);
\end{tikzpicture}
:= \Big(
\begin{tikzpicture}[baseline = -1mm]
	\draw[->] (-0.2,-.2) to (0.2,.2);
	\draw[-,line width=4pt,white] (0.2,-.2) to (-0.2,.2);
	\draw[<-] (0.2,-.2) to (-0.2,.2);
   \end{tikzpicture}
\Big)^{-1}$, while for
(\ref{invrel1b}) it is
$\mathord{
\begin{tikzpicture}[baseline = 1.5mm]
	\draw[-] (0.2,.05) to[out=-90, in=0] (0,-0.15);
	\draw[-] (0,-0.15) to[out = 180, in = -90] (-0.2,.05);
	\draw[->] (0.2,.5) to[out=-90,in=90] (-0.2,.05);
\draw[-](-.2,.5) to [out=90,in=180] (0,.7);
\draw[-](.2,.5) to [out=90,in=0] (0,.7);
	\draw[-,line width=4pt,white] (-0.2,.5) to[out=-90,in=90] (0.2,.05);
	\draw[-] (-0.2,.5) to[out=-90,in=90] (0.2,.05);
\end{tikzpicture}
}
=\frac{t^2-1}{z} 1_\unit$
where
$\mathord{
\begin{tikzpicture}[baseline = -1mm]
	\draw[<-] (-0.2,-.2) to (0.2,.2);
	\draw[-,line width=4pt,white] (0.2,-.2) to (-0.2,.2);
	\draw[->] (0.2,-.2) to (-0.2,.2);
   \end{tikzpicture}
} := \Big(\begin{tikzpicture}[baseline = -1mm]
	\draw[<-] (0.2,-.2) to (-0.2,.2);
	\draw[-,line width=4pt,white] (-0.2,-.2) to (0.2,.2);
	\draw[->] (-0.2,-.2) to (0.2,.2);
\end{tikzpicture}
\Big)^{-1}$.
\end{itemize}
The resulting category then contains unique morphisms
$\:\begin{tikzpicture}[baseline = .75mm]
	\draw[-] (0.3,0.3) to[out=-90, in=0] (0.1,0);
	\draw[->] (0.1,0) to[out = 180, in = -90] (-0.1,0.3);
\end{tikzpicture}\:$
and
$\:\begin{tikzpicture}[baseline = .75mm]
	\draw[-] (0.3,0) to[out=90, in=0] (0.1,0.3);
	\draw[->] (0.1,0.3) to[out = 180, in = 90] (-0.1,0);
\end{tikzpicture}\:$
such that the other relations
(\ref{saturday})--(\ref{limb}) hold; see \cite[Lemma 4.3]{BSWqheis}.

\begin{Remark}\label{normalization}
The alternative
presentation of $\Heis_k$ just formulated only involves even powers of $t$,
so that using it the category could be defined over $\k[t^2,t^{-2}]$
rather than the algebra 
$\K$ from (\ref{groundring}). 
The square root $t$ of $t^2$ is needed in 
order for there to exist 
leftwards cups and caps satisfying the earlier relations.
The specific normalization of these leftwards cups and caps 
was chosen originally in \cite{BSWqheis} 
so as to match the usual normalization in the HOMFLY-PT skein
category; see \cite{Bskein}.
\end{Remark}

In \cite[$\S\S$2--4]{BSWqheis}, many additional relations are derived from
the defining relations, 
including
counterparts of the infinite Grassmannian, alternating braid, curl,
and bubble slide relations.
Again, all of these relations can be reformulated quite compactly in terms of generating functions.
We do this here just for the infinite Grassmannian relation, the curl
relation and the bubble slides, since actually those are the only ones
we will need later on.
Like we did in the previous subsection, we switch from now onwards to labelling dots by
polynomials, now possibly in $\k[x,x^{-1}]$, instead of by integers.
We also assemble the $(+)$-bubbles
into the following
generating functions:
\begin{align}\label{summer1}
\anticlock(u) &:= t^{-1}z \sum_{r\in\Z}
\mathord{
\begin{tikzpicture}[baseline = 1.25mm]
  \draw[-] (0,0.4) to[out=180,in=90] (-.2,0.2);
  \draw[->] (0.2,0.2) to[out=90,in=0] (0,.4);
 \draw[-] (-.2,0.2) to[out=-90,in=180] (0,0);
  \draw[-] (0,0) to[out=0,in=-90] (0.2,0.2);
   \node at (0,.21) {$+$};
   \node at (0.33,0.2) {$\scriptstyle{r}$};
\end{tikzpicture}
}\: u^{-r}
\in u^k 1_\unit + u^{k-1} \End_{\Heis_k}(\unit)\llbracket
u^{-1} \rrbracket
,\\\label{summer2}
\clock(u)&:= -tz \sum_{r\in\Z}
\mathord{
\begin{tikzpicture}[baseline = 1.25mm]
  \draw[<-] (0,0.4) to[out=180,in=90] (-.2,0.2);
  \draw[-] (0.2,0.2) to[out=90,in=0] (0,.4);
 \draw[-] (-.2,0.2) to[out=-90,in=180] (0,0);
  \draw[-] (0,0) to[out=0,in=-90] (0.2,0.2);
   \node at (0,.21) {$+$};
   \node at (-0.33,0.2) {$\scriptstyle{r}$};
\end{tikzpicture}
} \:u^{-r} \in u^{-k}1_\unit +
u^{-k-1}\End_{\Heis_k}(\unit)\llbracket
u^{-1} \rrbracket.
\end{align}
Here, we are using slightly different notation from \cite{BSWqheis}, where these were denoted
$\anticlockplus(u)$ and $\clockplus(u)$.
Then we have the following, which are equivalent to \cite[Lemmas 3.4, 4.4 and 4.6]{BSWqheis}:
\begin{equation}\label{igproper2}
\anticlock(u)\; \clock(u)
= 1_\unit,
\end{equation}
\begin{align}
\mathord{
\begin{tikzpicture}[baseline = -1mm]
	\draw[<-] (0,0.6) to (0,0.3);
	\draw[-] (-0.5,0) to [out=90,in=180](-.3,0.2);
	\draw[-] (-0.3,.2) to [out=0,in=90](0,-0.3);
	\draw[-] (-0.3,-0.2) to [out=180,in=-90](-.5,0);
	\draw[-,white,line width=4pt] (0,0.3) to [out=-90,in=0] (-.3,-0.2);
	\draw[-] (0,0.3) to [out=-90,in=0] (-.3,-0.2);
	\draw[-] (0,-0.3) to (0,-0.6);
   \node at (-1.06,0) {$\scriptstyle{(u-x)^{-1}}$};
      \node at (-0.5,0) {$\dt$};
\end{tikzpicture}
}&=
t\left[\mathord{\begin{tikzpicture}[baseline = -1mm]
\node at (-.6,0.05) {$\anticlock \scriptstyle (u)$};
	\draw[->] (0.08,-.5) to (0.08,.6);
      \node at (.08,0.05) {$\dt$};
      \node at (.64,0.05) {$\scriptstyle{(u-x)^{-1}}$};
\end{tikzpicture}
}\right]_{u^{< 0}}\!\!,
&\mathord{
\begin{tikzpicture}[baseline = -1mm]
	\draw[<-] (0,0.6) to (0,0.3);
	\draw[-] (0.3,-0.2) to [out=0,in=-90](.5,0);
	\draw[-] (0.5,0) to [out=90,in=0](.3,0.2);
	\draw[-] (0.3,.2) to [out=180,in=90](0,-0.3);
	\draw[-,white,line width=4pt] (0,0.3) to [out=-90,in=180] (.3,-0.2);
	\draw[-] (0,0.3) to [out=-90,in=180] (.3,-0.2);
	\draw[-] (0,-0.3) to (0,-0.6);
   \node at (1.06,0) {$\scriptstyle{(u-x)^{-1}}$};
      \node at (0.5,0) {$\dt$};
\end{tikzpicture}
}&=
t^{-1}\left[\mathord{
\begin{tikzpicture}[baseline = -1mm]
	\draw[->] (0.08,-.5) to (0.08,.6);
   \node at (-0.46,0.05) {$\scriptstyle{(u-x)^{-1}}$};
      \node at (0.08,0.05) {$\dt$};
      \node at (.8,0.05) {$\clock\scriptstyle (u)$};
\end{tikzpicture}
}\right]_{u^{< 0}}\!,\label{somegin}\end{align}
\begin{align}
\mathord{
\begin{tikzpicture}[baseline = -1mm]
	\draw[->] (0.08,-.4) to (0.08,.4);
      \node at (0.8,0) {$\anticlock \scriptstyle (u)$};
\end{tikzpicture}
}
&=
\mathord{\begin{tikzpicture}[baseline = -1mm]
	\draw[->] (0.08,-.4) to (0.08,.4);
      \node at (-0.6,0) {$\anticlock \scriptstyle (u)$};
      \node at (.08,0) {$\dt$};
      \node at (1.1,0) {$\scriptstyle{1-z^2xu(u-x)^{-2}}$};
\end{tikzpicture}
}
\:,
&\mathord{
\begin{tikzpicture}[baseline = -1mm]
	\draw[->] (0.08,-.4) to (0.08,.4);
      \node at (-0.6,0) {$\clock \scriptstyle (u)$};
\end{tikzpicture}
  }&=
\mathord{
\begin{tikzpicture}[baseline = -1mm]
	\draw[->] (0.08,-.4) to (0.08,.4);
      \node at (0.8,0) {$\clock \scriptstyle (u)$};
      \node at (.08,0) {$\dt$};
      \node at (-.95,0) {$\scriptstyle{1-z^2xu(u-x)^{-2}}$};
\end{tikzpicture}
}
\:.\label{gin}
\end{align}
For the last relation, we note that $1-z^2xu(u-x)^{-2} = \frac{(u-q^2x)(u-q^{-2}x)}{(u-x)^2}$.

\begin{Lemma}\label{surely}
For a polynomial $p(u) \in \k[u]$, we have that
\begin{align}\label{surely1}
\mathord{
\begin{tikzpicture}[baseline = -1mm]
 	\draw[->] (0.08,-.4) to (0.08,.4);
     \node at (0.08,0) {$\dt$};
     \node at (0.45,0) {$\scriptstyle p(x)$};
\end{tikzpicture}
}
&=
\left[
\mathord{
\begin{tikzpicture}[baseline = -1mm]
 	\draw[->] (0.08,-.4) to (0.08,.4);
     \node at (0.08,0) {$\dt$};
     \node at (-0.55,0) {$\scriptstyle (u-x)^{-1}$};
\end{tikzpicture}
}
p(u)
\right]_{u^{-1}},
&
\mathord{
\begin{tikzpicture}[baseline = -1mm]
 	\draw[<-] (0.08,-.4) to (0.08,.4);
     \node at (0.08,0.03) {$\dt$};
     \node at (0.45,0.03) {$\scriptstyle p(x)$};
\end{tikzpicture}
}
&=
\left[
\mathord{
\begin{tikzpicture}[baseline = -1mm]
 	\draw[<-] (0.08,-.4) to (0.08,.4);
     \node at (0.08,0.03) {$\dt$};
     \node at (-0.55,0.03) {$\scriptstyle (u-x)^{-1}$};
\end{tikzpicture}
}
p(u)
\right]_{u^{-1}},\\\label{surely2}
\mathord{
\begin{tikzpicture}[baseline = 1.25mm]
  \draw[<-] (0,0.4) to[out=180,in=90] (-.2,0.2);
  \draw[-] (0.2,0.2) to[out=90,in=0] (0,.4);
 \draw[-] (-.2,0.2) to[out=-90,in=180] (0,0);
  \draw[-] (0,0) to[out=0,in=-90] (0.2,0.2);
   \node at (0.2,0.2) {$\dt$};
   \node at (0.56,0.2) {$\scriptstyle{p(x)}$};
\end{tikzpicture}
}
\!&=
\frac{t p(0) 1_\unit\!-t^{-1}\! \left[{\scriptstyle\clock(u)}\: p(u)\right]_{u^0}}{z},&
\mathord{
\begin{tikzpicture}[baseline = 1.25mm]
  \draw[-] (0,0.4) to[out=180,in=90] (-.2,0.2);
  \draw[->] (0.2,0.2) to[out=90,in=0] (0,.4);
 \draw[-] (-.2,0.2) to[out=-90,in=180] (0,0);
  \draw[-] (0,0) to[out=0,in=-90] (0.2,0.2);
   \node at (0.2,0.2) {$\dt$};
   \node at (0.56,0.2) {$\scriptstyle{p(x)}$};
\end{tikzpicture}
}
\!&=
\frac{t\left[ {\scriptstyle\anticlock(u)}\: p(u)\right]_{u^0}\!-t^{-1} p(0)1_\unit}{z},\\\label{surely3}
\mathord{
\begin{tikzpicture}[baseline = -1mm]
	\draw[<-] (0,0.6) to (0,0.3);
	\draw[-] (0.3,-0.2) to [out=0,in=-90](.5,0);
	\draw[-] (0.5,0) to [out=90,in=0](.3,0.2);
	\draw[-] (0.3,.2) to [out=180,in=90](0,-0.3);
	\draw[-] (0,-0.3) to (0,-0.6);
	\draw[-,white,line width=4pt] (0,0.3) to [out=-90,in=180] (.3,-0.2);
	\draw[-] (0,0.3) to [out=-90,in=180] (.3,-0.2);
   \node at (.87,0) {$\scriptstyle{p(x)}$};
      \node at (0.5,0) {$\dt$};
\end{tikzpicture}
}&=
t^{-1}
\left[\!\mathord{
\begin{tikzpicture}[baseline = -1mm]
	\draw[->] (0.08,-.6) to (0.08,.6);
   \node at (-0.46,0.05) {$\scriptstyle{(u-x)^{-1}}$};
      \node at (0.08,0.05) {$\dt$};
      \node at (.7,0.05) {$\clock \scriptstyle (u)$};
\end{tikzpicture}
}p(u)\right]_{u^{-1}}\!\!\!,
&
\mathord{
\begin{tikzpicture}[baseline = -1mm]
	\draw[-] (0,0.6) to (0,0.3);
	\draw[-] (0.3,-0.2) to [out=0,in=-90](.5,0);
	\draw[-] (0.5,0) to [out=90,in=0](.3,0.2);
	\draw[-] (0,0.3) to [out=-90,in=180] (.3,-0.2);
	\draw[-,white,line width=4pt] (0.3,.2) to [out=180,in=90](0,-0.3);
	\draw[-] (0.3,.2) to [out=180,in=90](0,-0.3);
	\draw[->] (0,-0.3) to (0,-0.6);
   \node at (.87,0) {$\scriptstyle{p(x)}$};
      \node at (0.5,0) {$\dt$};
\end{tikzpicture}
}&=
t\left[\!\mathord{
\begin{tikzpicture}[baseline = -1mm]
	\draw[<-] (0.08,-.6) to (0.08,.6);
   \node at (-0.46,0.05) {$\scriptstyle{(u-x)^{-1}}$};
      \node at (0.08,0.05) {$\dt$};
      \node at (.7,0.05) {$\anticlock \scriptstyle (u)$};
\end{tikzpicture}
}p(u)\right]_{u^{-1}}\!\!\!.
\end{align}
\end{Lemma}

\begin{proof}
This is almost the same as the proof of Lemma~\ref{sure}, using
(\ref{summer1})--(\ref{summer2}) and (\ref{somegin}) instead of
(\ref{igq})--(\ref{igp})
and (\ref{mostgin}).
For (\ref{surely2}), one also needs to know that
$\clock = tz^{-1} 1_\unit+ \scriptstyle{0}\clockplus$
and
$\anticlock = \anticlockplus\!{\scriptstyle{0}}- t^{-1}z^{-1}1_\unit$ 
due to \cite[(2.18), (3.12)]{BSWqheis}.
\end{proof}

In the quantum case for $q$ not a root
of unity, it is conjectured
that 
$K_0(\Kar(\Heis_k))$
is isomorphic to the Heisenberg ring $\rh_k$,
just like in the degenerate case.

\subsection{The Kac-Moody 2-category}\label{songs}
Last, but by no means least, we have the Kac-Moody 2-category.
This was defined by Khovanov and Lauda \cite{KL3} and
Rouquier \cite{Rou}.
In fact,
there is such a category associated to any symmetrizable Cartan
matrix, but in this paper we are only interested in the ones of Cartan type A,
so we specialize to that right away.
Our exposition is based on \cite{Bkm}, which unified the different approaches of
Khovanov-Lauda and Rouquier, and \cite[$\S$3]{BD}, which
incorporated some
renormalizations of the bubbles following the idea of \cite{BHLW}
in order to make the strictly pivotal structure apparent.

Assume that $I$ is a set equipped with a
fixed-point-free automorphism $I \rightarrow I, i \mapsto i^+$.
Let $i \mapsto i^-$ be the inverse function.
This can also be interpreted as the data of a quiver whose connected components
are of types $A_\infty$ or $A_{\qc-1}^{(1)}$ for $\qc \geq 2$.
There is an associated generalized Cartan matrix
$(a_{i,j})_{i,j \in I}$ with $a_{i,i} := 2$ for each $i \in I$, and
$a_{i,j}:=-\delta_{i^+,j}-\delta_{i,j^+}$
for each $i \neq j$.
Let $\g$ be the Kac-Moody Lie algebra over $\C$
generated by $\{e_i, f_i, h_i\:|\:i \in I\}$ subject to the Serre
relations defined from
the Cartan matrix $(a_{i,j})_{i, j \in I}$.
Note $\g$ is a direct sum of Kac-Moody Lie algebras of
types
$\sl_\infty(\C)$ (the infinite components in the quiver)
or $\widehat{\sl}_{\qc}(\C)'$ (finite components with $\qc$ vertices).

Let $\hh$ be the Cartan subalgebra of $\g$ with
basis
$\{h_i\:|\:i \in I\}$.
The {\em weight lattice} $X$ of $\g$ is the Abelian
subgroup of $\hh^*$ generated by
the fundamental weights $\{\Lambda_j\:|\:j \in I\}$ defined from
$\langle h_i,\Lambda_j \rangle = \delta_{i,j}$.
We have the set of {\em dominant weights}
$$
X^+ := \bigoplus_{i \in I} \N \Lambda_i
= \big\{\lambda \in X\:\big|\:\langle h_i, \lambda \rangle \geq 0
\text{ for all $i \in I$ and $\textstyle\sum_{i \in I}\langle
  h_i,\lambda\rangle < \infty$}\big\}.
$$
Let $\alpha_i := \sum_{j \in I} a_{i,j} \Lambda_j \in X$ be the
$i$th simple root.
Unlike the fundamental weights, these are not necessarily linearly independent, indeed,
we have that $\sum_{i \in I_0} \alpha_i = 0$ for each finite component
$I_0$ of $I$, due to the fact that we have not extended by
scaling elements.
Let $Y := \sum_{i \in I} \Z \alpha_i \subseteq X$.

Finally, we choose signs
$\{\signs_{i}(\lambda)\:|\:\lambda\in X, i \in I\}$ so that
$\signs_{i}(\lambda)\signs_{i}(\lambda+\alpha_j) = (-1)^{\delta_{i,j^+}}$
for each $j \in I$. There is a unique such choice
satisfying $\signs_{i}(\lambda) = 1$ for each $i \in I$ and each
$\lambda$ lying in a set of $X / Y$-coset
representatives.

Then the {\em Kac-Moody 2-category}
 $\UU(\g)$
is the strict $\k$-linear 2-category
with objects $X$,
generating
1-morphisms
$E_i 1_\lambda = \substack{\thickup \\ {\scriptscriptstyle i}}{\scriptstyle \color{gray}\lambda}:\lambda \rightarrow \lambda+\alpha_i$ and
$F_i 1_\lambda= \substack{{\scriptscriptstyle
    i}\\\thickdown}{\scriptstyle \color{gray}\lambda}:\lambda \rightarrow \lambda-\alpha_i$ for
$i \in I$ and $\lambda \in X$,
and generating
2-morphisms
\begin{align}\label{QHgens}
\mathord{

}
\!{\color{gray}\scriptstyle\lambda}
&:F_i E_i 1_\lambda \Rightarrow 1_\lambda.
\end{align}
This time, the sideways crossings are
defined from
\begin{align*}
\mathord{
%
}\:,
\end{align*}
and there are negatively dotted bubbles defined by
\begin{align*}
\mathord{
}
&:=
\left\{
\begin{array}{ll}
(-1)^n \signs_{i}(\lambda)^{n+1}
\det\left(
\:\mathord{
\begin{tikzpicture}[baseline = 1.25mm]
  \draw[<-,thick] (0,0.4) to[out=180,in=90] (-.2,0.2);
  \draw[-,thick] (0.2,0.2) to[out=90,in=0] (0,.4);
 \draw[-,thick] (-.2,0.2) to[out=-90,in=180] (0,0);
  \draw[-,thick] (0,0) to[out=0,in=-90] (0.2,0.2);
   \node at (-0.2,0.2) {$\bull$};
   \node at (-.95,0.2) {$\scriptstyle{r-s+\langle h_i, \lambda\rangle}$};
   \node at (0.05,-.15) {$\scriptstyle{i}$};
   \node at (.38,0.2) {$\color{gray}\scriptstyle{\lambda}$};
\end{tikzpicture}
}\:
\right)_{r,s=1,\dots,n}
&\text{if $\langle h_i,\lambda\rangle \geq n > 0$},\\
\signs_{i}(\lambda)1_{1_\lambda}&\text{if $\langle h_i,\lambda\rangle \geq n  = 0$,}\\
0&\text{if $\langle h_i,\lambda\rangle \geq n < 0$,}
\end{array}%
\right.
\end{align*}
The generating 2-morphisms are subject to the following
relations:
\begin{align}\label{QHA1}
\mathord{
%
}.\hspace{10mm}\label{flight2}
\end{align}
As with the Heisenberg category, one only needs to impose one of the relations
(\ref{rightadj3}) or (\ref{leftadj3}), then the other follows as a
consequence. Moreover, $\UU(\g)$ is strictly pivotal, so that
we can again introduce downward dots and crossings by taking right
and/or left mates of the upward ones.

The presentation described in the previous paragraph is similar
to the original approach of Khovanov and Lauda.
Rouquier's approach was based instead on an inversion relation. 
To formulate it,
we need the generating morphisms
$%
\:\,{\color{gray}\scriptstyle\lambda}$
 (hence, we also have the rightwards crossings),
subject to the relations (\ref{QHA1})--(\ref{rightadj3}) plus the
{\em inversion relation} asserting that the following are isomorphisms:
\begin{align}
\mathord{
%
\right]
&:
E_i F_i 1_\lambda \Rightarrow
F_i E_i 1_\lambda \oplus 1_\lambda^{\oplus \langle h_i,\lambda\rangle}
&&\text{if $\langle h_i,\lambda\rangle \geq
  0$}.\label{day3}
\end{align}

\begin{Lemma}\label{lateaddition}
Let $\AA$ be a strict $\k$-linear 2-category
containing objects $\{o_\lambda\:|\:\lambda \in X\}$, 1-morphisms
$E_i 1_\lambda:o_\lambda \rightarrow o_{\lambda+\alpha_i}$ and $F_i
1_\lambda:o_\lambda\rightarrow o_{\lambda-\alpha_i}$,
and
2-morphisms
$\mathord{
%
\!{\color{gray}\scriptstyle\lambda}\:$ for all
$i\in I$ and $\lambda\in X$
such that the relations (\ref{flight1})--(\ref{flight2}) all hold
(for the sideways crossings and negatively dotted bubbles defined as
above),
then these 2-morphisms are uniquely determined.
\end{Lemma}

\begin{proof}
Fix $i \in I$ and $\lambda \in X$.
Let $M$ be the matrix (\ref{day2}) if $\langle h_i, \lambda\rangle \geq 0$
or the matrix (\ref{day3}) if $\langle h_i, \lambda\rangle < 0$, viewed as a
2-morphism in the additive envelope $\Add(\AA)$.
The assumed relations 
 (\ref{QHA1})--(\ref{rightadj3}) and  (\ref{flight1})--(\ref{flight2})
 imply that $M$ is invertible.
Moreover the first entry of the inverse matrix $M^{-1}$ is 
$-\mathord{
\begin{tikzpicture}[baseline = -.5]
	\draw[->,thick] (0.18,-.15) to (-0.18,.3);
	\draw[<-,thick] (-0.18,-.15) to (0.18,.3);
   \node at (0.22,-.2) {$\scriptstyle{i}$};
   \node at (0.22,.35) {$\scriptstyle{i}$};
\end{tikzpicture}
}
{\color{gray}\scriptstyle\lambda}$. Thus, this 2-morphism is uniquely
determined in $\AA$ independent of the choices of the leftwards cups
and caps. Also if $\langle h_i, \lambda \rangle > 0$ (resp., $\langle
h_i, \lambda \rangle < 0$) then the last
entry of $M^{-1}$ is $\signs_i(\lambda)
\:\:\begin{tikzpicture}[baseline = 0]
	\draw[-,thick] (0.3,0.2) to[out=-90, in=0] (0.1,-0.1);
	\draw[->,thick] (0.1,-0.1) to[out = 180, in = -90] (-0.1,0.2);
    \node at (0.3,.35) {$\scriptstyle{i}$};
\end{tikzpicture}
\!{\color{gray}\scriptstyle\lambda}\:$
(resp., 
$\signs_i(\lambda)\:\:\begin{tikzpicture}[baseline = -2]
	\draw[-,thick] (0.3,-0.1) to[out=90, in=0] (0.1,0.2);
	\draw[->,thick] (0.1,0.2) to[out = 180, in = 90] (-0.1,-0.1);
    \node at (0.3,-.25) {$\scriptstyle{i}$};
\end{tikzpicture}
\!{\color{gray}\scriptstyle\lambda}\:$).
So these 2-morphisms are uniquely determined.
Finally, using (\ref{flight1})--(\ref{bus1}), one sees that
\begin{align*}
\mathord{
\begin{tikzpicture}[baseline = 1.5mm]
	\draw[-,thick] (0.4,0.4) to[out=-90, in=0] (0.1,0);
	\draw[->,thick] (0.1,0) to[out = 180, in = -90] (-0.2,0.4);
    \node at (0.4,.55) {$\scriptstyle{i}$};
  \node at (0.5,0.15) {$\color{gray}\scriptstyle{\lambda}$};
 \end{tikzpicture}
}
&=
\signs_i(\lambda)\:
\mathord{
\begin{tikzpicture}[baseline = 0]
	\draw[-,thick] (0.28,.6) to[out=240,in=90] (-0.28,-.1);
	\draw[<-,thick] (-0.28,.6) to[out=300,in=90] (0.28,-0.1);
   \node at (0.28,.75) {$\scriptstyle{i}$};
   \node at (.45,.2) {$\color{gray}\scriptstyle{\lambda}$};
	\draw[-,thick] (0.28,-0.1) to[out=-90, in=0] (0,-0.4);
	\draw[-,thick] (0,-0.4) to[out = 180, in = -90] (-0.28,-0.1);
      \node at (0.89,-0.1) {$\scriptstyle{-\langle h_i,\lambda\rangle}$};
      \node at (0.27,-0.1) {$\bull$};
\end{tikzpicture}
}\text{ if $\langle h_i, \lambda\rangle \leq 0$},&
\mathord{
\begin{tikzpicture}[baseline = 2]
	\draw[-,thick] (0.4,0) to[out=90, in=0] (0.1,0.4);
	\draw[->,thick] (0.1,0.4) to[out = 180, in = 90] (-0.2,0);
    \node at (0.4,-.15) {$\scriptstyle{i}$};
  \node at (0.62,0.15) {$\scriptstyle{\lambda}$};
\end{tikzpicture}
}
&=
-\signs_i(\lambda)
\mathord{
\begin{tikzpicture}[baseline = -5]
	\draw[-,thick] (0.28,-.6) to[out=120,in=-90] (-0.28,.1);
	\draw[<-,thick] (-0.28,-.6) to[out=60,in=-90] (0.28,.1);
   \node at (0.28,-.75) {$\scriptstyle{i}$};
   \node at (.4,-.2) {$\color{gray}\scriptstyle{\lambda}$};
	\draw[-,thick] (0.28,0.1) to[out=90, in=0] (0,0.4);
	\draw[-,thick] (0,0.4) to[out = 180, in = 90] (-0.28,0.1);
      \node at (-0.74,0.15) {$\scriptstyle{\langle h_i,\lambda\rangle}$};
      \node at (-0.27,0.15) {$\bull$};
\end{tikzpicture}
}\text{ if $\langle h_i, \lambda\rangle \geq 0$}.
\end{align*}
This means these morphisms are uniquely determined too.
\end{proof}

Again, one can introduce generating functions and work with
the defining relations in those terms; this technique was pioneered in \cite{Wmemoir}.
We just write down the counterparts of
(\ref{igproper}), (\ref{mostgin})--(\ref{moregin})  and
(\ref{igproper2})--(\ref{gin}).
Let
\begin{align}\label{chink}
{\color{gray}\scriptstyle\lambda}\:\anticlocki(u) &:= \signs_i(\lambda)\sum_{r \in \Z}
\mathord{
\begin{tikzpicture}[baseline = 1.25mm]
  \draw[->,thick] (0.2,0.2) to[out=90,in=0] (0,.4);
  \draw[-,thick] (0,0.4) to[out=180,in=90] (-.2,0.2);
\draw[-,thick] (-.2,0.2) to[out=-90,in=180] (0,0);
  \draw[-,thick] (0,0) to[out=0,in=-90] (0.2,0.2);
   \node at (0.2,0.2) {$\bull$};
   \node at (.4,0.2) {$\scriptstyle{r}$};
   \node at (-0.05,-.15) {$\scriptstyle{i}$};
   \node at (-.38,0.2) {$\color{gray}\scriptstyle{\lambda}$};
\end{tikzpicture}}u^{-r-1}\in
u^{\langle h_i,\lambda\rangle} 1_{1_\lambda}+
u^{\langle h_i,\lambda\rangle-1} \End(1_\lambda)[\![u^{-1}]\!]
,
\\\label{chunk}
{\color{gray}\scriptstyle\lambda}\:\clocki(u)
&:=
\signs_i(\lambda)\sum_{r \in\Z}
\mathord{
\begin{tikzpicture}[baseline = 1.25mm]
  \draw[<-,thick] (0,0.4) to[out=180,in=90] (-.2,0.2);
  \draw[-,thick] (0.2,0.2) to[out=90,in=0] (0,.4);
 \draw[-,thick] (-.2,0.2) to[out=-90,in=180] (0,0);
  \draw[-,thick] (0,0) to[out=0,in=-90] (0.2,0.2);
   \node at (0,-.15) {$\scriptstyle{i}$};
   \node at (0.36,0.2) {$\color{gray}\scriptstyle{\lambda}$};
   \node at (-0.2,0.2) {$\bull$};
   \node at (-.4,0.2) {$\scriptstyle{r}$};
\end{tikzpicture}
}
u^{-r-1} \in
u^{-\langle h_i,\lambda\rangle} 1_{1_\lambda}+
u^{-\langle h_i,\lambda\rangle-1} \End(1_\lambda)[\![u^{-1}]\!].
\end{align}
Switching from now on to labelling dots by polynomials rather than
integers in the same way as we
did when working with the Heisenberg category, but using the variable $y$ in place of $x$ to
avoid possible confusion later on, we have that
\begin{align}\label{infgras}
\clocki(u) \:\anticlocki(u)\:{\scriptstyle\color{gray}\lambda} = 1_{1_\lambda},
\end{align}
\begin{align}
\mathord{
\begin{tikzpicture}[baseline = -1mm]
	\draw[<-,thick] (0,0.6) to (0,0.3);
	\draw[-,thick] (0,0.3) to [out=-90,in=0] (-.3,-0.2);
	\draw[-,thick] (-0.3,-0.2) to [out=180,in=-90](-.5,0);
	\draw[-,thick] (-0.5,0) to [out=90,in=180](-.3,0.2);
	\draw[-,thick] (-0.3,.2) to [out=0,in=90](0,-0.3);
	\draw[-,thick] (0,-0.3) to (0,-0.5);
   \node at (-1.06,0) {$\scriptstyle{(u-y)^{-1}}$};
      \node at (-0.5,0) {$\bull$};
   \node at (0,-.6) {$\scriptstyle{i}$};
   \node at (-0.8,-0.3) {$\color{gray}\scriptstyle{\lambda}$};
\end{tikzpicture}
}&=
\signs_i(\lambda)\left[\mathord{\begin{tikzpicture}[baseline = -1mm]
\node at (-.6,0.05) {$\anticlocki \scriptstyle (u)$};
	\draw[->,thick] (0.08,-.5) to (0.08,.6);
      \node at (.08,0.05) {$\bull$};
      \node at (.64,0.05) {$\scriptstyle{(u-y)^{-1}}$};
   \node at (0.08,-.6) {$\scriptstyle{i}$};
   \node at (-0.5,-0.3) {$\color{gray}\scriptstyle{\lambda}$};
\end{tikzpicture}
}\right]_{u^{< 0}}\!\!\!\!\!,
&\mathord{
\begin{tikzpicture}[baseline = -1mm]
	\draw[<-,thick] (0,0.6) to (0,0.3);
	\draw[-,thick] (0,0.3) to [out=-90,in=180] (.3,-0.2);
	\draw[-,thick] (0.3,-0.2) to [out=0,in=-90](.5,0);
	\draw[-,thick] (0.5,0) to [out=90,in=0](.3,0.2);
	\draw[-,thick] (0.3,.2) to [out=180,in=90](0,-0.3);
	\draw[-,thick] (0,-0.3) to (0,-0.5);
   \node at (1.06,0) {$\scriptstyle{(u-y)^{-1}}$};
      \node at (0.5,0) {$\bull$};
   \node at (0,-.6) {$\scriptstyle{i}$};
   \node at (0.5,-0.3) {$\color{gray}\scriptstyle{\lambda}$};
\end{tikzpicture}
}&=
-
\signs_i(\lambda)\left[\mathord{
\begin{tikzpicture}[baseline = -1mm]
	\draw[->,thick] (0.08,-.5) to (0.08,.6);
   \node at (-0.46,0.05) {$\scriptstyle{(u-y)^{-1}}$};
      \node at (0.08,0.05) {$\bull$};
      \node at (.8,0.05) {$\clocki \scriptstyle (u)$};
   \node at (.08,-.6) {$\scriptstyle{i}$};
   \node at (0.5,-0.3) {$\color{gray}\scriptstyle{\lambda}$};
\end{tikzpicture}
}\right]_{u^{< 0}}\!\!\!\!\!,\\
\mathord{\begin{tikzpicture}[baseline = -1mm]
	\draw[->,thick] (0.08,-.4) to (0.08,.4);
      \node at (-0.6,0) {$\anticlockj\scriptstyle (u)$};
   \node at (.08,-.55) {$\scriptstyle{i}$};
   \node at (0.4,0) {$\color{gray}\scriptstyle{\lambda}$};
\end{tikzpicture}
}
&=
\mathord{
\begin{tikzpicture}[baseline = -1mm]
	\draw[->,thick] (0.08,-.4) to (0.08,.4);
      \node at (0.8,0) {$\anticlockj\scriptstyle(u)$};
   \node at (.08,-.55) {$\scriptstyle{i}$};
   \node at (0.5,-0.3) {$\color{gray}\scriptstyle{\lambda}$};
      \node at (.08,0) {$\bull$};
      \node at (-.65,0.05) {$\scriptstyle{(u-y)^{\langle h_i, \alpha_j\rangle}}$};
\end{tikzpicture}
},&
\mathord{
\begin{tikzpicture}[baseline = -1mm]
	\draw[->,thick] (0.08,-.4) to (0.08,.4);
      \node at (0.8,0) {$\clockj\scriptstyle(u)$};
   \node at (.08,-.55) {$\scriptstyle{i}$};
   \node at (0.5,-0.3) {$\color{gray}\scriptstyle{\lambda}$};
\end{tikzpicture}
}
&=
\mathord{
\begin{tikzpicture}[baseline = -1mm]
	\draw[->,thick] (0.08,-.4) to (0.08,.4);
      \node at (-0.6,0) {$\clockj\scriptstyle(u)$};
      \node at (.08,0) {$\bull$};
      \node at (.85,0.05) {$\scriptstyle{(u-y)^{\langle h_i,\alpha_j\rangle}}$};
   \node at (.08,-.55) {$\scriptstyle{i}$};
   \node at (0.4,-.3) {$\color{gray}\scriptstyle{\lambda}$};
\end{tikzpicture}
}.\label{bubslides}
\end{align}
The following is proved in exactly the same way as Lemma~\ref{sure}.

\begin{Lemma}\label{sure2lemma}
For a polynomial $p(u) \in \k[u]$, we have that
\begin{align}\label{sure1a}
\mathord{
\begin{tikzpicture}[baseline = -1.5mm]
 	\draw[->,thick] (0.08,-.4) to (0.08,.4);
   \node at (.08,-.55) {$\scriptstyle{i}$};
     \node at (0.08,0) {$\bull$};
     \node at (-0.33,0) {$\scriptstyle p(y)$};
   \node at (0.35,0) {$\color{gray}\scriptstyle{\lambda}$};
\end{tikzpicture}
}
&=
\left[
\mathord{
\begin{tikzpicture}[baseline = -1.5mm]
 	\draw[->,thick] (0.08,-.4) to (0.08,.4);
   \node at (.08,-.55) {$\scriptstyle{i}$};
     \node at (0.08,0) {$\bull$};
     \node at (-0.52,0) {$\scriptstyle (u-y)^{-1}$};
   \node at (0.4,.3) {$\color{gray}\scriptstyle{\lambda}$};
\end{tikzpicture}
}
\!\!\!\!\!p(u)
\right]_{u^{-1}},
&
\mathord{
\begin{tikzpicture}[baseline = 0mm]
 	\draw[<-,thick] (0.08,-.4) to (0.08,.4);
   \node at (.08,.55) {$\scriptstyle{i}$};
     \node at (0.08,0.04) {$\bull$};
     \node at (-0.33,0.04) {$\scriptstyle p(y)$};
   \node at (0.35,0.05) {$\color{gray}\scriptstyle{\lambda}$};
\end{tikzpicture}
}
&=
\left[
\mathord{
\begin{tikzpicture}[baseline = 0mm]
 	\draw[<-,thick] (0.08,-.4) to (0.08,.4);
   \node at (.08,.55) {$\scriptstyle{i}$};
     \node at (0.08,0.04) {$\bull$};
     \node at (-0.5,0.04) {$\scriptstyle (u-y)^{-1}$};
   \node at (0.4,-.3) {$\color{gray}\scriptstyle{\lambda}$};
\end{tikzpicture}
}
\!\!\!\!\!p(u)
\right]_{u^{-1}},\\\label{sure2a}
\!\mathord{
\begin{tikzpicture}[baseline = 1.25mm]
  \draw[<-,thick] (0,0.4) to[out=180,in=90] (-.2,0.2);
  \draw[-,thick] (0.2,0.2) to[out=90,in=0] (0,.4);
 \draw[-,thick] (-.2,0.2) to[out=-90,in=180] (0,0);
  \draw[-,thick] (0,0) to[out=0,in=-90] (0.2,0.2);
   \node at (0.2,0.2) {$\bull$};
   \node at (0.55,0.2) {$\scriptstyle{p(y)}$};
   \node at (0,-.15) {$\scriptstyle{i}$};
   \node at (-0.4,.2) {$\color{gray}\scriptstyle{\lambda}$};
\end{tikzpicture}
}
&=
\sigma_i(\lambda)
\left[{\scriptstyle{\color{gray}\lambda\:\:}\clocki(u)}\:p(u)\right]_{u^{-1}},&
\mathord{
\begin{tikzpicture}[baseline = 1.25mm]
  \draw[-,thick] (0,0.4) to[out=180,in=90] (-.2,0.2);
  \draw[->,thick] (0.2,0.2) to[out=90,in=0] (0,.4);
 \draw[-,thick] (-.2,0.2) to[out=-90,in=180] (0,0);
  \draw[-,thick] (0,0) to[out=0,in=-90] (0.2,0.2);
   \node at (0.2,0.2) {$\bull$};
   \node at (0.55,0.2) {$\scriptstyle{p(y)}$};
   \node at (0,-.15) {$\scriptstyle{i}$};
   \node at (-0.4,.2) {$\color{gray}\scriptstyle{\lambda}$};
\end{tikzpicture}
}
&=
\sigma_i(\lambda)
\left[{\scriptstyle{\color{gray}\lambda\:\:}\anticlocki(u)}\:p(u)\right]_{u^{-1}},\\\label{sure3a}
\mathord{
\begin{tikzpicture}[baseline = -1.5mm]
	\draw[<-,thick] (0,0.6) to (0,0.3);
	\draw[-,thick] (0,0.3) to [out=-90,in=180] (.3,-0.2);
	\draw[-,thick] (0.3,-0.2) to [out=0,in=-90](.5,0);
	\draw[-,thick] (0.5,0) to [out=90,in=0](.3,0.2);
	\draw[-,thick] (0.3,.2) to [out=180,in=90](0,-0.3);
	\draw[-,thick] (0,-0.3) to (0,-0.5);
   \node at (0,-.65) {$\scriptstyle{i}$};
   \node at (0.3,-.4) {$\color{gray}\scriptstyle{\lambda}$};
   \node at (.87,0) {$\scriptstyle{p(y)}$};
      \node at (0.5,0) {$\bull$};
\end{tikzpicture}
}&=
-
\sigma_i(\lambda)\left[\!\mathord{
\begin{tikzpicture}[baseline = -1.5mm]
	\draw[->,thick] (0.08,-.5) to (0.08,.6);
   \node at (0.08,-.65) {$\scriptstyle{i}$};
   \node at (0.4,-.4) {$\color{gray}\scriptstyle{\lambda}$};
   \node at (-0.46,0.05) {$\scriptstyle{(u-y)^{-1}}$};
      \node at (0.08,0.05) {$\bull$};
      \node at (.7,0.05) {$\clocki \scriptstyle (u)$};
\end{tikzpicture}
}p(u)\right]_{u^{-1}}\!\!\!\!\!,
&
\mathord{
\begin{tikzpicture}[baseline = 0mm]
	\draw[-,thick] (0,0.5) to (0,0.3);
	\draw[-,thick] (0,0.3) to [out=-90,in=180] (.3,-0.2);
	\draw[-,thick] (0.3,-0.2) to [out=0,in=-90](.5,0);
	\draw[-,thick] (0.5,0) to [out=90,in=0](.3,0.2);
	\draw[-,thick] (0.3,.2) to [out=180,in=90](0,-0.3);
	\draw[->,thick] (0,-0.3) to (0,-0.6);
   \node at (.87,0) {$\scriptstyle{p(y)}$};
      \node at (0.5,0) {$\bull$};
   \node at (0,.65) {$\scriptstyle{i}$};
   \node at (0.4,-.4) {$\color{gray}\scriptstyle{\lambda}$};
\end{tikzpicture}
}&=
\sigma_i(\lambda)\left[\!\mathord{
\begin{tikzpicture}[baseline = 0mm]
	\draw[<-,thick] (0.08,-.6) to (0.08,.5);
   \node at (-0.46,0.05) {$\scriptstyle{(u-y)^{-1}}$};
      \node at (0.08,0.05) {$\bull$};
      \node at (.7,0.05) {$\anticlocki \scriptstyle (u)$};
   \node at (0.08,.65) {$\scriptstyle{i}$};
   \node at (0.4,-.4) {$\color{gray}\scriptstyle{\lambda}$};
\end{tikzpicture}
}p(u)\right]_{u^{-1}}\!\!\!\!\!\!.
\end{align}
\end{Lemma}

Finally, we outline the precise connection between $\UU(\g)$ and the quantized
enveloping algebra $U_q(\g)$ associated to $\g$. To do this, one needs to introduce a $\Z$-grading on
2-morphisms, thereby making $\UU(\g)$ into a 2-category enriched in
graded vector spaces. From that, one obtains a graded 2-category
$\UU_q(\g)$
by formally adjoining grading shift operators to the $1$-morphism categories.
The Grothendieck ring $K_0(\Kar(\UU_q(\g)))$ of the additive Karoubi
envelope of this graded 2-category
is then a $\Z[q,q^{-1}]$-algebra with $q$ acting by the grading shift.
This Grothendieck ring
 is isomorphic to the $\Z[q,q^{-1}]$-form of Lusztig's
idempotented form for the quantized enveloping algebra of
$\g$.
This was proved for $\sl_\infty(\C)$ in \cite{KL3}, and in
general in \cite{Wunfurling}.
Since we will not need these results here, we omit the detailed constructions.

\section{Heisenberg module categories}

This section is the heart of the article.
Let
$\Heis_k$ be the Heisenberg category, either
degenerate 
or quantum 
according to the choice of $z \in \k$.
Suppose that we are given a ($\k$-linear)
$\Heis_k$-module
category $\R$ which is either locally finite Abelian or
Schurian.
We are going to show that $\R$ can be given the structure of a
Kac-Moody 2-representation.

\subsection{Eigenfunctors}\label{sendo}
The endofunctors $E$ and $F$ of $\R$ defined by the generating
objects of $\Heis_k$ are biadjoint, with adjunctions $(E,F)$ and
$(F,E)$ defined by the rightwards cups/caps and the leftwards
cups/caps, respectively. Hence, both $E$ and $F$ are sweet
endofunctors.
 For $i \in \k$, let $E_i$ and $F_i$ be the subfunctors of $E$ and $F$
defined on $V \in \R$
by declaring that
$E_i V$ and $F_i V$ are the generalized $i$-eigenspaces of
the endomorphisms
$\mathord{\begin{tikzpicture}[baseline = -1mm]
 	\draw[->] (0.08,-.2) to (0.08,.2);
     \node at (0.08,0) {$\dt$};
 	\draw[-,darkg,thick] (0.38,.2) to (0.38,-.2);
     \node at (0.55,0) {$\darkg\scriptstyle{V}$};
\end{tikzpicture}
}$ and
$\mathord{\begin{tikzpicture}[baseline = -1mm]
 	\draw[<-] (0.08,-.2) to (0.08,.2);
     \node at (0.08,0.02) {$\dt$};
 	\draw[-,darkg,thick] (0.38,.2) to (0.38,-.2);
     \node at (0.55,0) {$\darkg\scriptstyle{V}$};
\end{tikzpicture}
}$, respectively.

Let us spell this definition out in more detail. In the Schurian case, any
object is the
direct limit of its compact ($=$ finitely presented) subobjects by \cite[Lemma 2.6]{BS}, so in view of the
exactness of $E$ and $F$ it suffices to
define $E_i V$ and $F_i V$ under the assumption that $V$ is finitely
generated. Assuming this (which is no restriction at all in the
locally finite Abelian case), the objects
$EV$ and $FV$ are finitely generated too,
hence, their
endomorphism algebras
$\End_{\R}(EV)$ and $\End_{\R}(FV)$ are
finite-dimensional.
So we can define $m_V(u), n_V(u) \in \k[u]$ to
be the (monic) {\em minimal polynomials} of
the endomorphisms
$\mathord{\begin{tikzpicture}[baseline = -1mm]
 	\draw[->] (0.08,-.2) to (0.08,.2);
     \node at (0.08,0) {$\dt$};
 	\draw[-,darkg,thick] (0.38,.2) to (0.38,-.2);
     \node at (0.55,0) {$\darkg\scriptstyle{V}$};
\end{tikzpicture}
}$ and
$\mathord{\begin{tikzpicture}[baseline = -1mm]
 	\draw[<-] (0.08,-.2) to (0.08,.2);
     \node at (0.08,0.02) {$\dt$};
 	\draw[-,darkg,thick] (0.38,.2) to (0.38,-.2);
     \node at (0.55,0) {$\darkg\scriptstyle{V}$};
\end{tikzpicture}
}$, respectively.
Then
there are {injective} homomorphisms
\begin{align}\label{CRT1}
\k[u] / (m_V(u)) &\hookrightarrow \End_{\R} (EV),
&
\k[u] / (n_V(u)) &\hookrightarrow \End_{\R} (FV),\\
p(u) &
\mapsto \mathord{
\begin{tikzpicture}[baseline = -1mm]
 	\draw[->] (0.08,-.3) to (0.08,.4);
     \node at (0.08,0.05) {$\dt$};
     \node at (-0.3,0.05) {$\scriptstyle p(x)$};
 	\draw[-,darkg,thick] (0.45,.4) to (0.45,-.3);
     \node at (0.45,-.45) {$\darkg\scriptstyle{V}$};
\end{tikzpicture}
},
&p(u) &\mapsto \mathord{
\begin{tikzpicture}[baseline = -1mm]
 	\draw[<-] (0.08,-.3) to (0.08,.4);
     \node at (0.08,0.1) {$\dt$};
     \node at (-0.3,0.1) {$\scriptstyle p(x)$};
 	\draw[-,darkg,thick] (0.45,.4) to (0.45,-.3);
     \node at (0.45,-.45) {$\darkg\scriptstyle{V}$};
\end{tikzpicture}
}.\notag
\end{align}
Also let $\eps_i(V)$ and $\phi_i(V)$ denote the multiplicities
of $i \in \k$ as a root of the polynomials $m_V(u)$ and $n_V(u)$,
respectively.
By the Chinese remainder theorem, we have that
\begin{align}\label{CRTdef}
\k[u] / (m_V(u)) &\cong \bigoplus_{i \in \k} \k[u] \big/ \big((u-i)^{\eps_i(V)}\big),&
\k[u] / (n_V(u)) &\cong \bigoplus_{i \in \k} \k[u] \big/
                   \big((u-i)^{\phi_i(V)}\big).
\end{align}
There are corresponding decompositions $1 = \sum_{i \in \k} e_i$ of
and $1 = \sum_{i \in \k} f_i$ of
the identity elements of these algebras as a sum of mutually orthogonal
idempotents.
We define $E_i V$ and $F_i V$ to be the
summands of $EV$ and $FV$, respectively, defined by the images of the idempotents
$e_i$ and $f_i$ under
(\ref{CRT1}).

We will represent the identity endomorphisms of the functors $E_i$ and $F_i$ by
vertical strings colored by $i$, see the first pair of diagrams
below.
The inclusions
$E_i \hookrightarrow E$ and $F_i \hookrightarrow F$ are depicted
by the second pair of diagrams below. The projections $E
\twoheadrightarrow E_i$ and $F \twoheadrightarrow F_i$ are the final pair.
\begin{align*}
&\mathord{
\begin{tikzpicture}[baseline = 0]
	\draw[->] (0.08,-.3) to (0.08,.4);
      \node at (0.08,-0.45) {$\scriptstyle{i}$};
\end{tikzpicture}
}
\!:E_i \Rightarrow E_i,
&&
\mathord{
\begin{tikzpicture}[baseline = 0]
	\draw[<-] (0.08,-.3) to (0.08,.4);
      \node at (0.08,0.55) {$\scriptstyle{i}$};
\end{tikzpicture}
} \!:F_i \Rightarrow F_i,
&&
\mathord{
\begin{tikzpicture}[baseline = 0]
	\draw[->] (0.08,-.3) to (0.08,.4);
	\draw[-] (0.04,.05) to (0.12,.05);
      \node at (0.08,-0.45) {$\scriptstyle{i}$};
\end{tikzpicture}
} \!:E_i \Rightarrow E,
&&
\mathord{
\begin{tikzpicture}[baseline = 0]
	\draw[<-] (0.08,-.3) to (0.08,.4);
	\draw[-] (0.04,.05) to (0.12,.05);
      \node at (0.08,-0.45) {$\scriptstyle{i}$};
\end{tikzpicture}
} \!:F_i \Rightarrow F,
&&
\mathord{
\begin{tikzpicture}[baseline = 0mm]
	\draw[->] (0.08,-.3) to (0.08,.4);
	\draw[-] (0.04,.05) to (0.12,.05);
      \node at (0.08,0.55) {$\scriptstyle{i}$};
\end{tikzpicture}
} \!:E \Rightarrow E_i,
&&
\mathord{
\begin{tikzpicture}[baseline = 0mm]
s	\draw[<-] (0.08,-.3) to (0.08,.4);
	\draw[-] (0.04,.05) to (0.12,.05);
      \node at (0.08,0.55) {$\scriptstyle{i}$};
\end{tikzpicture}
} \!:F \Rightarrow F_i.
\end{align*}
To illustrate the notation, the natural transformation
$\:\mathord{
\begin{tikzpicture}[baseline = -0.8mm]
	\draw[->] (0.08,-.25) to (0.08,.35);
	\draw[-] (0.04,.15) to (0.12,.15);
	\draw[-] (0.04,-.05) to (0.12,-.05);
      \node at (0.2,0.05) {$\scriptstyle{i}$};
\end{tikzpicture}
}:E \Rightarrow E
$
is the projection of $E$ onto its summand $E_i$, while
\begin{equation}\label{othogonality}
\mathord{
\begin{tikzpicture}[baseline = 0]
	\draw[->] (0.08,-.3) to (0.08,.4);
	\draw[-] (0.04,.15) to (0.12,.15);
	\draw[-] (0.04,-0.05) to (0.12,-0.05);
      \node at (0.08,0.55) {$\scriptstyle{j}$};
      \node at (0.08,-.45) {$\scriptstyle{i}$};
\end{tikzpicture}
}
= \delta_{i,j}
\mathord{
\begin{tikzpicture}[baseline = 0]
	\draw[->] (0.08,-.3) to (0.08,.4);
        \node at (0.08,-.45) {$\scriptstyle{i}$};
\end{tikzpicture}
}\:\:.
\end{equation}
It is also clear from the definition that the endomorphisms of
$E$ and $F$
defined by the dots restrict to endomorphisms of the
summands $E_i$
and $F_i$.
Representing these restrictions simply by drawing the dots on a string
colored by $i$, we have that
\begin{align}\label{sizzle}
\mathord{
\begin{tikzpicture}[baseline = -1mm]
	\draw[->] (0.08,-.3) to (0.08,.4);
      \node at (0.08,-0.05) {$\dt$};
	\draw[-] (0.04,.15) to (0.12,.15);
      \node at (0.08,-0.45) {$\scriptstyle{i}$};
\end{tikzpicture}
}&=
\mathord{
\begin{tikzpicture}[baseline = -1mm]
	\draw[->] (0.08,-.3) to (0.08,.4);
      \node at (0.08,0.15) {$\dt$};
	\draw[-] (0.04,-.05) to (0.12,-.05);
      \node at (0.08,-0.45) {$\scriptstyle{i}$};
\end{tikzpicture}
}\:\:,
&
\mathord{
\begin{tikzpicture}[baseline = -1mm]
	\draw[<-] (0.08,-.3) to (0.08,.4);
      \node at (0.08,-0.05) {$\dt$};
	\draw[-] (0.04,.15) to (0.12,.15);
      \node at (0.08,-0.45) {$\scriptstyle{i}$};
\end{tikzpicture}
}&=
\mathord{
\begin{tikzpicture}[baseline = -1mm]
	\draw[<-] (0.08,-.3) to (0.08,.4);
      \node at (0.08,0.15) {$\dt$};
	\draw[-] (0.04,-.05) to (0.12,-.05);
      \node at (0.08,-0.45) {$\scriptstyle{i}$};
\end{tikzpicture}
}\:\:,&
\mathord{
\begin{tikzpicture}[baseline = 1mm]
	\draw[->] (0.08,-.3) to (0.08,.4);
      \node at (0.08,-0.05) {$\dt$};
	\draw[-] (0.04,.15) to (0.12,.15);
      \node at (0.08,0.55) {$\scriptstyle{i}$};
\end{tikzpicture}
}&=
\mathord{
\begin{tikzpicture}[baseline = 1mm]
	\draw[->] (0.08,-.3) to (0.08,.4);
      \node at (0.08,0.15) {$\dt$};
	\draw[-] (0.04,-.05) to (0.12,-.05);
      \node at (0.08,0.55) {$\scriptstyle{i}$};
\end{tikzpicture}
}\:\:,
&
\mathord{
\begin{tikzpicture}[baseline = 1mm]
	\draw[<-] (0.08,-.3) to (0.08,.4);
      \node at (0.08,-0.05) {$\dt$};
	\draw[-] (0.04,.15) to (0.12,.15);
      \node at (0.08,0.55) {$\scriptstyle{i}$};
\end{tikzpicture}
}&=
\mathord{
\begin{tikzpicture}[baseline = 1mm]
	\draw[<-] (0.08,-.3) to (0.08,.4);
      \node at (0.08,0.15) {$\dt$};
	\draw[-] (0.04,-.05) to (0.12,-.05);
      \node at (0.08,0.55) {$\scriptstyle{i}$};
\end{tikzpicture}
}
\:\:.
\end{align}

Since the downwards dot is both the left and right mate of the upwards dot,
the adjunctions $(E,F)$ and $(F,E)$ induce adjunctions $(E_i, F_i)$
and $(F_i, E_i)$ for all $i \in I$.
We draw the units and counits of these adjunctions using cups and caps
colored by $i$.
Again, the various inclusions and projections commute with these morphisms:
\begin{align}
\begin{array}{llll}
\mathord{
\begin{tikzpicture}[baseline = 1mm]
	\draw[<-] (0.4,0.4) to[out=-90, in=0] (0.1,0);
	\draw[-] (-0.2,.188) to (-0.123,.22);
	\draw[-] (0.1,0) to[out = 180, in = -90] (-0.2,0.4);
      \node at (-0.2,0.55) {$\scriptstyle{i}$};
\end{tikzpicture}
}
\:=
\mathord{
\begin{tikzpicture}[baseline = 1mm]
	\draw[<-] (0.4,0.4) to[out=-90, in=0] (0.1,0);
	\draw[-] (.402,.188) to (.325,.22);
	\draw[-] (0.1,0) to[out = 180, in = -90] (-0.2,0.4);
      \node at (-0.2,0.55) {$\scriptstyle{i}$};
\end{tikzpicture}
}\:
,
\quad&
\mathord{
\begin{tikzpicture}[baseline = 1mm]
	\draw[-] (0.4,0.4) to[out=-90, in=0] (0.1,0);
	\draw[-] (-0.2,.188) to (-0.123,.22);
	\draw[->] (0.1,0) to[out = 180, in = -90] (-0.2,0.4);
      \node at (-0.2,0.55) {$\scriptstyle{i}$};
\end{tikzpicture}
}
\;=\;
\mathord{
\begin{tikzpicture}[baseline = 1mm]
	\draw[-] (0.4,0.4) to[out=-90, in=0] (0.1,0);
	\draw[-] (.402,.188) to (.325,.22);
	\draw[->] (0.1,0) to[out = 180, in = -90] (-0.2,0.4);
      \node at (0.4,0.55) {$\scriptstyle{i}$};
\end{tikzpicture}
}\:
,
\quad&
\mathord{
\begin{tikzpicture}[baseline = 1mm]
	\draw[<-] (0.4,0.4) to[out=-90, in=0] (0.1,0);
	\draw[-] (-0.2,.188) to (-0.123,.22);
	\draw[-] (0.1,0) to[out = 180, in = -90] (-0.2,0.4);
      \node at (0.4,0.55) {$\scriptstyle{i}$};
\end{tikzpicture}
}
=
\:\mathord{
\begin{tikzpicture}[baseline = 1mm]
	\draw[<-] (0.4,0.4) to[out=-90, in=0] (0.1,0);
	\draw[-] (.402,.188) to (.325,.22);
	\draw[-] (0.1,0) to[out = 180, in = -90] (-0.2,0.4);
      \node at (0.4,0.55) {$\scriptstyle{i}$};
\end{tikzpicture}
}\:
,
\quad&
\mathord{
\begin{tikzpicture}[baseline = 1mm]
	\draw[-] (0.4,0.4) to[out=-90, in=0] (0.1,0);
	\draw[-] (-0.2,.188) to (-0.123,.22);
	\draw[->] (0.1,0) to[out = 180, in = -90] (-0.2,0.4);
      \node at (0.4,0.55) {$\scriptstyle{i}$};
\end{tikzpicture}
}
=\:
\mathord{
\begin{tikzpicture}[baseline = 1mm]
	\draw[-] (0.4,0.4) to[out=-90, in=0] (0.1,0);
	\draw[-] (.402,.188) to (.325,.22);
	\draw[->] (0.1,0) to[out = 180, in = -90] (-0.2,0.4);
      \node at (0.4,0.55) {$\scriptstyle{i}$};
\end{tikzpicture}
}\:
,
\\\\
\mathord{
\begin{tikzpicture}[baseline = 1mm]
	\draw[<-] (0.4,0) to[out=90, in=0] (0.1,0.4);
	\draw[-] (-0.2,.22) to (-0.123,.188);
	\draw[-] (0.1,0.4) to[out = 180, in = 90] (-0.2,0);
      \node at (-0.2,-0.15) {$\scriptstyle{i}$};
\end{tikzpicture}
}\:=\:
\mathord{
\begin{tikzpicture}[baseline = 1mm]
	\draw[<-] (0.4,0) to[out=90, in=0] (0.1,0.4);
	\draw[-] (.402,.22) to (.325,.188);
	\draw[-] (0.1,0.4) to[out = 180, in = 90] (-0.2,0);
      \node at (-0.2,-0.15) {$\scriptstyle{i}$};
\end{tikzpicture}
}\:,
\quad&
\mathord{
\begin{tikzpicture}[baseline = 1mm]
	\draw[-] (0.4,0) to[out=90, in=0] (0.1,0.4);
	\draw[-] (-0.2,.22) to (-0.123,.188);
	\draw[->] (0.1,0.4) to[out = 180, in = 90] (-0.2,0);
      \node at (-0.2,-0.15) {$\scriptstyle{i}$};
\end{tikzpicture}
}\;=\;
\mathord{
\begin{tikzpicture}[baseline = 1mm]
	\draw[-] (0.4,0) to[out=90, in=0] (0.1,0.4);
	\draw[-] (.402,.22) to (.325,.188);
	\draw[->] (0.1,0.4) to[out = 180, in = 90] (-0.2,0);
      \node at (-0.2,-0.15) {$\scriptstyle{i}$};
\end{tikzpicture}
}\:,
\quad&
\mathord{
\begin{tikzpicture}[baseline = 1mm]
	\draw[<-] (0.4,0) to[out=90, in=0] (0.1,0.4);
	\draw[-] (-0.2,.22) to (-0.123,.188);
	\draw[-] (0.1,0.4) to[out = 180, in = 90] (-0.2,0);
      \node at (0.4,-0.15) {$\scriptstyle{i}$};
\end{tikzpicture}
}=\:
\mathord{
\begin{tikzpicture}[baseline = 1mm]
	\draw[<-] (0.4,0) to[out=90, in=0] (0.1,0.4);
	\draw[-] (.402,.22) to (.325,.188);
	\draw[-] (0.1,0.4) to[out = 180, in = 90] (-0.2,0);
      \node at (0.4,-0.15) {$\scriptstyle{i}$};
\end{tikzpicture}
}\:,
\quad&
\mathord{
\begin{tikzpicture}[baseline = 1mm]
	\draw[-] (0.4,0) to[out=90, in=0] (0.1,0.4);
	\draw[-] (-0.2,.22) to (-0.123,.188);
	\draw[->] (0.1,0.4) to[out = 180, in = 90] (-0.2,0);
      \node at (0.4,-0.15) {$\scriptstyle{i}$};
\end{tikzpicture}
}=\:
\mathord{
\begin{tikzpicture}[baseline = 1mm]
	\draw[-] (0.4,0) to[out=90, in=0] (0.1,0.4);
	\draw[-] (.402,.22) to (.325,.188);
	\draw[->] (0.1,0.4) to[out = 180, in = 90] (-0.2,0);
      \node at (0.4,-0.15) {$\scriptstyle{i}$};
\end{tikzpicture}
}\:.
\end{array}\label{incpro}
\end{align}

The situation with crossings is more interesting.
For $i,j,i',j' \in \k$, define
\begin{align}
\left.
\begin{array}{ll}
\mathord{
\begin{tikzpicture}[baseline = -1mm]
	\draw[->] (0.28,-.28) to (-0.28,.28);
	\draw[->] (-0.28,-.28) to (0.28,.28);
      \node at (-0.33,-0.43) {$\scriptstyle{j}$};
      \node at (0.33,-0.43) {$\scriptstyle{i}$};
      \node at (-0.33,0.43) {$\scriptstyle{j'}$};
      \node at (0.33,0.43) {$\scriptstyle{i'}$};
\node at (0,-0.01) {$\diamond$};
\end{tikzpicture}
}:=
\mathord{
\begin{tikzpicture}[baseline = -1mm]
	\draw[->] (0.28,-.28) to (-0.28,.28);
	\draw[->] (-0.28,-.28) to (0.28,.28);
      \node at (-0.33,-0.43) {$\scriptstyle{j}$};
      \node at (0.33,-0.43) {$\scriptstyle{i}$};
      \node at (-0.33,0.43) {$\scriptstyle{j'}$};
      \node at (0.33,0.43) {$\scriptstyle{i'}$};
	\draw[-] (.14,.22) to (.22,.14);
	\draw[-] (.14,-.22) to (.22,-.14);
	\draw[-] (-.14,.22) to (-.22,.14);
	\draw[-] (-.14,-.22) to (-.22,-.14);
\end{tikzpicture}
}
&\text{in the degenerate case,}\\
\mathord{
\begin{tikzpicture}[baseline = -1mm]
	\draw[->] (0.28,-.28) to (-0.28,.28);
	\draw[-,white,line width=3pt] (-0.28,-.28) to (0.28,.28);
	\draw[->] (-0.28,-.28) to (0.28,.28);
      \node at (-0.33,-0.43) {$\scriptstyle{j}$};
      \node at (0.33,-0.43) {$\scriptstyle{i}$};
      \node at (-0.33,0.43) {$\scriptstyle{j'}$};
      \node at (0.33,0.43) {$\scriptstyle{i'}$};
\node at (0,-0.01) {$\diamond$};
\end{tikzpicture}
}:=
\mathord{
\begin{tikzpicture}[baseline = -1mm]
	\draw[->] (0.28,-.28) to (-0.28,.28);
	\draw[-,white,line width=4pt] (-0.28,-.28) to (0.28,.28);
	\draw[->] (-0.28,-.28) to (0.28,.28);
      \node at (-0.33,-0.43) {$\scriptstyle{j}$};
      \node at (0.33,-0.43) {$\scriptstyle{i}$};
      \node at (-0.33,0.43) {$\scriptstyle{j'}$};
      \node at (0.33,0.43) {$\scriptstyle{i'}$};
	\draw[-] (.14,.22) to (.22,.14);
	\draw[-] (.14,-.22) to (.22,-.14);
	\draw[-] (-.14,.22) to (-.22,.14);
	\draw[-] (-.14,-.22) to (-.22,-.14);
\end{tikzpicture}
}\:,\quad
\mathord{
\begin{tikzpicture}[baseline = 0]
	\draw[->] (-0.28,-.28) to (0.28,.28);
	\draw[-,white,line width=3pt] (0.28,-.28) to (-0.28,.28);
	\draw[->] (0.28,-.28) to (-0.28,.28);
      \node at (-0.33,-0.43) {$\scriptstyle{j}$};
      \node at (0.33,-0.43) {$\scriptstyle{i}$};
      \node at (-0.33,0.43) {$\scriptstyle{j'}$};
      \node at (0.33,0.43) {$\scriptstyle{i'}$};
\node at (0,-0.01) {$\diamond$};
\end{tikzpicture}
}:=
\mathord{
\begin{tikzpicture}[baseline = 0]
	\draw[->] (-0.28,-.28) to (0.28,.28);
	\draw[-,line width=4pt,white] (0.28,-.28) to (-0.28,.28);
	\draw[->] (0.28,-.28) to (-0.28,.28);
      \node at (-0.33,-0.43) {$\scriptstyle{j}$};
      \node at (0.33,-0.43) {$\scriptstyle{i}$};
      \node at (-0.33,0.43) {$\scriptstyle{j'}$};
      \node at (0.33,0.43) {$\scriptstyle{i'}$};
	\draw[-] (.14,.22) to (.22,.14);
	\draw[-] (.14,-.22) to (.22,-.14);
	\draw[-] (-.14,.22) to (-.22,.14);
	\draw[-] (-.14,-.22) to (-.22,-.14);
\end{tikzpicture}
}&
\text{in the quantum case.}
\end{array}\right.\label{interesting}\end{align}
Thus, these natural transformation are defined by first including
the summand $E_j E_i$ into $E E$, then
applying natural transformation $EE \Rightarrow EE$ defined by
the usual crossing (positive or negative in the quantum case), then
projecting
$E E$ onto the summand $E_{j'} E_{i'}$.
The defining relations plus (\ref{sizzle}) imply
that
\begin{align}\label{newdotslide}
\mathord{
\begin{tikzpicture}[baseline = -1mm]
	\draw[->] (0.28,-.28) to (-0.28,.28);
	\draw[->] (-0.28,-.28) to (0.28,.28);
      \node at (-0.33,-0.43) {$\scriptstyle{j}$};
      \node at (0.33,-0.43) {$\scriptstyle{i}$};
      \node at (-0.33,0.43) {$\scriptstyle{j'}$};
      \node at (0.33,0.43) {$\scriptstyle{i'}$};
\node at (0,-.01) {$\diamond$};
\node at (-.17,-.19) {$\dt$};
\end{tikzpicture}
}&=
\mathord{
\begin{tikzpicture}[baseline = -1mm]
	\draw[->] (0.28,-.28) to (-0.28,.28);
	\draw[->] (-0.28,-.28) to (0.28,.28);
      \node at (-0.33,-0.43) {$\scriptstyle{j}$};
      \node at (0.33,-0.43) {$\scriptstyle{i}$};
      \node at (-0.33,0.43) {$\scriptstyle{j'}$};
      \node at (0.33,0.43) {$\scriptstyle{i'}$};
\node at (0,-.01) {$\diamond$};
\node at (.17,.15) {$\dt$};
\end{tikzpicture}
}+
\delta_{i,i'} \delta_{j,j'}
\begin{tikzpicture}[baseline = -1mm]
	\draw[->] (0.18,-.28) to (0.18,.28);
	\draw[->] (-0.18,-.28) to (-0.18,.28);
      \node at (-0.18,-0.43) {$\scriptstyle{j}$};
      \node at (0.18,-0.43) {$\scriptstyle{i}$};
\end{tikzpicture},&
\mathord{
\begin{tikzpicture}[baseline = -1mm]
	\draw[->] (0.28,-.28) to (-0.28,.28);
	\draw[->] (-0.28,-.28) to (0.28,.28);
      \node at (-0.33,-0.43) {$\scriptstyle{j}$};
      \node at (0.33,-0.43) {$\scriptstyle{i}$};
      \node at (-0.33,0.43) {$\scriptstyle{j'}$};
      \node at (0.33,0.43) {$\scriptstyle{i'}$};
\node at (0,-.01) {$\diamond$};
\node at (-.17,.15) {$\dt$};
\end{tikzpicture}
}&=
\mathord{
\begin{tikzpicture}[baseline = -1mm]
	\draw[->] (0.28,-.28) to (-0.28,.28);
	\draw[->] (-0.28,-.28) to (0.28,.28);
      \node at (-0.33,-0.43) {$\scriptstyle{j}$};
      \node at (0.33,-0.43) {$\scriptstyle{i}$};
      \node at (-0.33,0.43) {$\scriptstyle{j'}$};
      \node at (0.33,0.43) {$\scriptstyle{i'}$};
\node at (0,-.01) {$\diamond$};
\node at (.18,-.19) {$\dt$};
\end{tikzpicture}
}+
\delta_{i,i'} \delta_{j,j'}
\begin{tikzpicture}[baseline = -1mm]
	\draw[->] (0.18,-.28) to (0.18,.28);
	\draw[->] (-0.18,-.28) to (-0.18,.28);
      \node at (-0.18,-0.43) {$\scriptstyle{j}$};
      \node at (0.18,-0.43) {$\scriptstyle{i}$};
\end{tikzpicture}
\end{align}
in the degenerate case, or
\begin{align}\label{newdotslide2}
\mathord{
\begin{tikzpicture}[baseline = -1mm]
	\draw[->] (-0.28,-.28) to (0.28,.28);
	\draw[-,white,line width=3pt] (0.28,-.28) to (-0.28,.28);
	\draw[->] (0.28,-.28) to (-0.28,.28);
      \node at (-0.33,-0.43) {$\scriptstyle{j}$};
      \node at (0.33,-0.43) {$\scriptstyle{i}$};
      \node at (-0.33,0.43) {$\scriptstyle{j'}$};
      \node at (0.33,0.43) {$\scriptstyle{i'}$};
\node at (0,-.01) {$\diamond$};
\node at (-.17,-.19) {$\dt$};
\end{tikzpicture}
}&=
\mathord{
\begin{tikzpicture}[baseline = -1mm]
	\draw[->] (0.28,-.28) to (-0.28,.28);
	\draw[-,line width=3pt, white] (-0.28,-.28) to (0.28,.28);
	\draw[->] (-0.28,-.28) to (0.28,.28);
      \node at (-0.33,-0.43) {$\scriptstyle{j}$};
      \node at (0.33,-0.43) {$\scriptstyle{i}$};
      \node at (-0.33,0.43) {$\scriptstyle{j'}$};
      \node at (0.33,0.43) {$\scriptstyle{i'}$};
\node at (0,-.01) {$\diamond$};
\node at (.17,.15) {$\dt$};
\end{tikzpicture}
},&
\mathord{
\begin{tikzpicture}[baseline = -1mm]
	\draw[->] (-0.28,-.28) to (0.28,.28);
	\draw[-,white,line width=3pt] (0.28,-.28) to (-0.28,.28);
	\draw[->] (0.28,-.28) to (-0.28,.28);
      \node at (-0.33,-0.43) {$\scriptstyle{j}$};
      \node at (0.33,-0.43) {$\scriptstyle{i}$};
      \node at (-0.33,0.43) {$\scriptstyle{j'}$};
      \node at (0.33,0.43) {$\scriptstyle{i'}$};
\node at (0,-.01) {$\diamond$};
\node at (-.17,.15) {$\dt$};
\end{tikzpicture}
}&=
\mathord{
\begin{tikzpicture}[baseline = -1mm]
	\draw[->] (0.28,-.28) to (-0.28,.28);
	\draw[-,line width=3pt, white] (-0.28,-.28) to (0.28,.28);
	\draw[->] (-0.28,-.28) to (0.28,.28);
      \node at (-0.33,-0.43) {$\scriptstyle{j}$};
      \node at (0.33,-0.43) {$\scriptstyle{i}$};
      \node at (-0.33,0.43) {$\scriptstyle{j'}$};
      \node at (0.33,0.43) {$\scriptstyle{i'}$};
\node at (0,-.01) {$\diamond$};
\node at (.18,-.19) {$\dt$};
\end{tikzpicture}
},
&
\mathord{
\begin{tikzpicture}[baseline = -1mm]
	\draw[->] (0.28,-.28) to (-0.28,.28);
	\draw[-,line width=3pt, white] (-0.28,-.28) to (0.28,.28);
	\draw[->] (-0.28,-.28) to (0.28,.28);
      \node at (-0.33,-0.43) {$\scriptstyle{j}$};
      \node at (0.33,-0.43) {$\scriptstyle{i}$};
      \node at (-0.33,0.43) {$\scriptstyle{j'}$};
      \node at (0.33,0.43) {$\scriptstyle{i'}$};
\node at (0,-.01) {$\diamond$};
\end{tikzpicture}
}
-
\mathord{
\begin{tikzpicture}[baseline = -1mm]
	\draw[->] (-0.28,-.28) to (0.28,.28);
	\draw[-,white,line width=3pt] (0.28,-.28) to (-0.28,.28);
	\draw[->] (0.28,-.28) to (-0.28,.28);
      \node at (-0.33,-0.43) {$\scriptstyle{j}$};
      \node at (0.33,-0.43) {$\scriptstyle{i}$};
      \node at (-0.33,0.43) {$\scriptstyle{j'}$};
      \node at (0.33,0.43) {$\scriptstyle{i'}$};
\node at (0,-.01) {$\diamond$};
\end{tikzpicture}
}
&=
\delta_{i,i'} \delta_{j,j'}
z
\begin{tikzpicture}[baseline = -1mm]
	\draw[->] (0.18,-.28) to (0.18,.28);
	\draw[->] (-0.18,-.28) to (-0.18,.28);
      \node at (-0.18,-0.43) {$\scriptstyle{j}$};
      \node at (0.18,-0.43) {$\scriptstyle{i}$};
\end{tikzpicture}
\end{align}
in the quantum case.
There are also sideways and downwards versions of the new crossings
which may be defined in a similar way, or equivalently by ``rotating'' the upwards
ones using (\ref{incpro}).
The following lemma is well known but essential.

\begin{Lemma}\label{essential}
If $\{i,j\}\neq\{i',j'\}$ then the natural transformation (\ref{interesting})
is zero.
The same
holds for the rotated versions of these crossings.
\end{Lemma}

\begin{proof}
For the rotated crossings the lemma follows from the upwards case
using also (\ref{incpro}).
To prove the result for the upwards crossing, we just explain in the
degenerate case; the quantum case is similar using (\ref{newdotslide2}) in place
of (\ref{newdotslide}).
If $\{i,j\}\neq\{i',j'\}$ then one of the following holds:
$i \notin \{i',j'\},
j \notin \{i',j'\}, i' \notin \{i,j\}$ or $j' \notin \{i,j\}$.
Suppose first
that $j \notin \{i',j'\}$ or $i' \notin \{i,j\}$.
It suffices to show that the natural transformation vanishes on every
finitely generated $V \in \R$.
We can find polynomials $f(u), g(u) \in \k[u]$ so that
$f(u) (u-j)^{\eps_j(E_i V)} + g(u) (u-i')^{\eps_{i'}(V)}= 1$.
Letting $p(u) := g(u) (u-i')^{\eps_{i'}(V)}$,
we then use (\ref{newdotslide}) to see that
$$
\mathord{
\begin{tikzpicture}[baseline = -1mm]
	\draw[->] (0.28,-.28) to (-0.28,.28);
	\draw[->] (-0.28,-.28) to (0.28,.28);
\node at (0,-0.01) {$\diamond$};
      \node at (-0.33,-0.43) {$\scriptstyle{j}$};
      \node at (0.33,-0.43) {$\scriptstyle{i}$};
      \node at (-0.33,0.43) {$\scriptstyle{j'}$};
      \node at (0.33,0.43) {$\scriptstyle{i'}$};
	\draw[-,darkg,thick] (0.7,-.28) to (0.7,.28);
      \node at (0.7,-.43) {$\darkg\scriptstyle{V}$};
\end{tikzpicture}
}
=
\mathord{
\begin{tikzpicture}[baseline = -1mm]
	\draw[->] (0.28,-.28) to (-0.28,.28);
	\draw[->] (-0.28,-.28) to (0.28,.28);
\node at (0,-0.01) {$\diamond$};
      \node at (-0.33,-0.43) {$\scriptstyle{j}$};
      \node at (0.33,-0.43) {$\scriptstyle{i}$};
      \node at (-0.33,0.43) {$\scriptstyle{j'}$};
      \node at (0.33,0.43) {$\scriptstyle{i'}$};
	\draw[-,darkg,thick] (0.7,-.28) to (0.7,.28);
\node at (-.17,-.19) {$\dt$};
\node at (-.5,-.1) {$\scriptstyle p(x)$};
      \node at (0.7,-.43) {$\darkg\scriptstyle{V}$};
\end{tikzpicture}
}
=
\mathord{
\begin{tikzpicture}[baseline = -1mm]
	\draw[->] (0.28,-.28) to (-0.28,.28);
	\draw[->] (-0.28,-.28) to (0.28,.28);
\node at (0,-0.01) {$\diamond$};
      \node at (-0.33,-0.43) {$\scriptstyle{j}$};
      \node at (0.33,-0.43) {$\scriptstyle{i}$};
      \node at (-0.33,0.43) {$\scriptstyle{j'}$};
      \node at (0.33,0.43) {$\scriptstyle{i'}$};
\node at (.17,.15) {$\dt$};
\node at (.5,0.1) {$\scriptstyle p(x)$};
	\draw[-,darkg,thick] (1,-.28) to (1,.28);
      \node at (1,-.43) {$\darkg\scriptstyle{V}$};
\end{tikzpicture}
}
=
0.
$$ 
A similar argument with the dot on the other string treats the cases
$i \notin \{i',j'\}$ or $j' \notin \{i,j\}$.
\end{proof}

Now we come to an extremely useful diagrammatic convention.
On any finitely generated $V \in \R$, the endomorphism 
$
\begin{tikzpicture}[baseline = -1.7mm]
	\draw[->] (0.08,-.2) to (0.08,.2);
      \node at (0.08,-0.35) {$\scriptstyle{i}$};
      \node at (0.08,-0.02) {$\dt$};
      \node at (-0.25,-0.02) {$\scriptstyle{x-i}$};
 	\draw[-,darkg,thick] (0.4,-.2) to (0.4,.2);
     \node at (0.55,0) {$\darkg\scriptstyle{V}$};
\end{tikzpicture}
$
is nilpotent, hence, the notation
$
\begin{tikzpicture}[baseline = -1.7mm]
	\draw[->] (0.08,-.2) to (0.08,.2);
      \node at (0.08,-0.35) {$\scriptstyle{i}$};
      \node at (0.08,-0.02) {$\dt$};
      \node at (-0.3,-0.02) {$\scriptstyle{p(x)}$};
 	\draw[-,darkg,thick] (0.4,-.2) to (0.4,.2);
     \node at (0.55,0) {$\darkg\scriptstyle{V}$};
\end{tikzpicture}
$
makes sense for power series $p(x) \in
\k[\![x-i]\!]$ rather than merely for polynomials.
Since any object of $\R$ is a direct limit of finitely generated
objects, it follows that there is a well-defined natural
transformation
\begin{equation}\label{brexit}
\begin{tikzpicture}[baseline = -0.8mm]
	\draw[->] (0.08,-.25) to (0.08,.3);
      \node at (0.08,-0.4) {$\scriptstyle{i}$};
      \node at (0.08,0.02) {$\dt$};
      \node at (-0.3,0.02) {$\scriptstyle{p(x)}$};
\end{tikzpicture}:
E_i \Rightarrow E_i
\end{equation}
for any $i \in \k$ and any $p(x) \in \k[\![x-i]\!]$. 
The same definition can be made for dots on downward strings too.
More generally, suppose that we are given some more complicated string
diagram for a natural
transformation between some endofunctors of $\R$,
together with a sequence of $n$ 
points $P_1,\dots,P_n$ on strings colored $i_1,\dots,i_n \in \k$ in this diagram.
Then for any $p(x_1,\dots,x_n) \in \k[\![x_1-i_1,\dots,x_n-i_n]\!]$
there is a well-defined natural transformation represented
diagrammatically by drawing a dot on each of the given points in the
given diagram then
joining them up with a dotted arrow directed from $P_1$ to $P_n$
labelled
by the power series $p(x_1,\dots,x_n)$.
Thus, $x_1$ indicates $x$ labelling the first dot (the one nearest the tail of the arrow) and $x_n$ indicates $x$ labelling the last dot (the one nearest the head).
To give an example, suppose that $n = 2$ and $i_1 \neq i_2$.
Set $c :=
(i_2-i_1)^{-1}$ so that $(x_2-x_1)^{-1} \in \k[\![x_1-i_1,x_2-i_2]\!]$
has power series expansion
$c
-c^2(x_1-i_1)
+ c^2(x_2-i_2) +$ (higher order terms).
Then we have defined the natural transformations
\begin{align*}
\begin{tikzpicture}[baseline = -0.8mm]
	\draw[->,densely dotted] (0.55,.03) to (0.15,.03);
	\draw[->] (0.58,-.25) to (0.58,.3);
	\draw[->] (0.08,-.25) to (0.08,.3);
      \node at (0.08,-0.4) {$\scriptstyle{i_2}$};
      \node at (0.08,0.02) {$\dt$};
      \node at (0.58,-0.4) {$\scriptstyle{i_1}$};
      \node at (0.58,0.02) {$\dt$};
      \node at (1.3,0.04) {$\scriptstyle{(x_2-x_1)^{-1}}$};
\end{tikzpicture}&=
c\begin{tikzpicture}[baseline = -0.8mm]
	\draw[->] (0.58,-.25) to (0.58,.3);
	\draw[->] (0.08,-.25) to (0.08,.3);
      \node at (0.08,-0.4) {$\scriptstyle{i_2}$};
      \node at (0.58,-0.4) {$\scriptstyle{i_1}$};
\end{tikzpicture}
-c^2\begin{tikzpicture}[baseline = -0.8mm]
	\draw[->] (0.58,-.25) to (0.58,.3);
	\draw[->] (0.08,-.25) to (0.08,.3);
      \node at (0.08,-0.4) {$\scriptstyle{i_2}$};
      \node at (0.58,-0.4) {$\scriptstyle{i_1}$};
      \node at (0.58,0.02) {$\dt$};
      \node at (.93,0.04) {$\scriptstyle x-i_1$};
\end{tikzpicture}
+c^2\begin{tikzpicture}[baseline = -0.8mm]
	\draw[->] (0.58,-.25) to (0.58,.3);
	\draw[->] (0.08,-.25) to (0.08,.3);
      \node at (0.08,-0.4) {$\scriptstyle{i_2}$};
      \node at (0.58,-0.4) {$\scriptstyle{i_1}$};
      \node at (0.08,0.02) {$\dt$};
      \node at (-.27,0.04) {$\scriptstyle x-i_2$};
\end{tikzpicture}
+\cdots,\\
\begin{tikzpicture}[baseline = -0.8mm]
	\draw[<-,densely dotted] (0.52,.03) to (0.15,.03);
	\draw[<-] (0.58,-.25) to (0.58,.3);
	\draw[<-] (0.08,-.25) to (0.08,.3);
      \node at (0.08,-0.4) {$\scriptstyle{i_1}$};
      \node at (0.08,0.02) {$\dt$};
      \node at (0.58,-0.4) {$\scriptstyle{i_2}$};
      \node at (0.58,0.02) {$\dt$};
      \node at (-.57,0.04) {$\scriptstyle{(x_2-x_1)^{-1}}$};
\end{tikzpicture}&=
c\begin{tikzpicture}[baseline = -0.8mm]
	\draw[<-] (0.58,-.25) to (0.58,.3);
	\draw[<-] (0.08,-.25) to (0.08,.3);
      \node at (0.08,-0.4) {$\scriptstyle{i_1}$};
      \node at (0.58,-0.4) {$\scriptstyle{i_2}$};
\end{tikzpicture}
-c^2\begin{tikzpicture}[baseline = -0.8mm]
	\draw[<-] (0.58,-.25) to (0.58,.3);
	\draw[<-] (0.08,-.25) to (0.08,.3);
      \node at (0.08,-0.4) {$\scriptstyle{i_1}$};
      \node at (0.58,-0.4) {$\scriptstyle{i_2}$};
      \node at (0.08,0.02) {$\dt$};
      \node at (-.27,0.04) {$\scriptstyle x-i_1$};
\end{tikzpicture}
+c^2\begin{tikzpicture}[baseline = -0.8mm]
	\draw[<-] (0.58,-.25) to (0.58,.3);
	\draw[<-] (0.08,-.25) to (0.08,.3);
      \node at (0.08,-0.4) {$\scriptstyle{i_1}$};
      \node at (0.58,-0.4) {$\scriptstyle{i_2}$};
      \node at (0.58,0.02) {$\dt$};
      \node at (.93,0.04) {$\scriptstyle x-i_2$};
\end{tikzpicture}
+\cdots.
\end{align*}
These natural transformations appear in the following lemma.

\begin{Lemma}\label{essential2}
For $j \neq i$, we have that
$$
\left.
\begin{array}{ll}
\begin{tikzpicture}[baseline = -1mm]
	\draw[->] (0.28,-.28) to (-0.28,.28);
	\draw[->] (-0.28,-.28) to (0.28,.28);
      \node at (-0.33,-0.43) {$\scriptstyle{j}$};
      \node at (0.33,-0.43) {$\scriptstyle{i}$};
      \node at (-0.33,0.43) {$\scriptstyle{j}$};
      \node at (0.33,0.43) {$\scriptstyle{i}$};
\node at (0,-.01) {$\diamond$};
\end{tikzpicture}
=
\begin{tikzpicture}[baseline = -0.8mm]
	\draw[->,densely dotted] (0.55,.03) to (0.15,.03);
	\draw[->] (0.58,-.25) to (0.58,.3);
	\draw[->] (0.08,-.25) to (0.08,.3);
      \node at (0.08,-0.4) {$\scriptstyle{j}$};
      \node at (0.08,0.02) {$\dt$};
      \node at (0.58,-0.4) {$\scriptstyle{i}$};
      \node at (0.58,0.02) {$\dt$};
      \node at (1.3,0.04) {$\scriptstyle{(x_2-x_1)^{-1}}$};
\end{tikzpicture}
&\text{in the degenerate case,}\\
\begin{tikzpicture}[baseline = -1mm]
	\draw[->] (0.28,-.28) to (-0.28,.28);
	\draw[-,white,line width=3pt] (-0.28,-.28) to (0.28,.28);
	\draw[->] (-0.28,-.28) to (0.28,.28);
      \node at (-0.33,-0.43) {$\scriptstyle{j}$};
      \node at (0.33,-0.43) {$\scriptstyle{i}$};
      \node at (-0.33,0.43) {$\scriptstyle{j}$};
      \node at (0.33,0.43) {$\scriptstyle{i}$};
\node at (0,-.01) {$\diamond$};
\end{tikzpicture}
=
z\begin{tikzpicture}[baseline = -0.8mm]
	\draw[->,densely dotted] (0.55,.03) to (0.15,.03);
	\draw[->] (0.58,-.25) to (0.58,.3);
	\draw[->] (0.08,-.25) to (0.08,.3);
      \node at (0.08,-0.4) {$\scriptstyle{j}$};
      \node at (0.08,0.02) {$\dt$};
      \node at (0.58,-0.4) {$\scriptstyle{i}$};
      \node at (0.58,0.02) {$\dt$};
      \node at (1.4,0.04) {$\scriptstyle{x_2(x_2-x_1)^{-1}}$};
\end{tikzpicture},
\qquad
\begin{tikzpicture}[baseline = -1mm]
	\draw[->] (-0.28,-.28) to (0.28,.28);
	\draw[-,white,line width=3pt] (0.28,-.28) to (-0.28,.28);
	\draw[->] (0.28,-.28) to (-0.28,.28);
      \node at (-0.33,-0.43) {$\scriptstyle{j}$};
      \node at (0.33,-0.43) {$\scriptstyle{i}$};
      \node at (-0.33,0.43) {$\scriptstyle{j}$};
      \node at (0.33,0.43) {$\scriptstyle{i}$};
\node at (0,-.01) {$\diamond$};
\end{tikzpicture}
=
z\begin{tikzpicture}[baseline = -0.8mm]
	\draw[->,densely dotted] (0.55,.03) to (0.15,.03);
	\draw[->] (0.58,-.25) to (0.58,.3);
	\draw[->] (0.08,-.25) to (0.08,.3);
      \node at (0.08,-0.4) {$\scriptstyle{j}$};
      \node at (0.08,0.02) {$\dt$};
      \node at (0.58,-0.4) {$\scriptstyle{i}$};
      \node at (0.58,0.02) {$\dt$};
      \node at (1.4,0.04) {$\scriptstyle{x_1(x_2-x_1)^{-1}}$};
\end{tikzpicture}
&\text{in the quantum case}.
\end{array}\right.
$$
\end{Lemma}

\begin{proof}
It suffices to prove this when the natural transformations are
evaluated on a finitely generated object $V \in \R$. 
We have to prove that 
$$
\phi := 
\left\{
\begin{array}{ll}
\begin{tikzpicture}[baseline = -1mm]
	\draw[->] (0.28,-.28) to (-0.28,.28);
	\draw[->] (-0.28,-.28) to (0.28,.28);
      \node at (-0.33,-0.43) {$\scriptstyle{j}$};
      \node at (0.33,-0.43) {$\scriptstyle{i}$};
      \node at (-0.33,0.43) {$\scriptstyle{j}$};
      \node at (0.33,0.43) {$\scriptstyle{i}$};
\node at (0,-.01) {$\diamond$};
	\draw[-,darkg,thick] (0.8,-.25) to (0.8,.3);
      \node at (0.8,-0.4) {$\darkg\scriptstyle{V}$};
\end{tikzpicture}
-\begin{tikzpicture}[baseline = -0.8mm]
	\draw[->,densely dotted] (0.55,.03) to (0.15,.03);
	\draw[->] (0.58,-.25) to (0.58,.3);
	\draw[->] (0.08,-.25) to (0.08,.3);
      \node at (0.08,-0.4) {$\scriptstyle{j}$};
      \node at (0.08,0.02) {$\dt$};
      \node at (0.58,-0.4) {$\scriptstyle{i}$};
      \node at (0.58,0.02) {$\dt$};
      \node at (1.3,0.04) {$\scriptstyle{(x_2-x_1)^{-1}}$};
	\draw[-,darkg,thick] (2.1,-.25) to (2.1,.3);
      \node at (2.1,-0.4) {$\darkg\scriptstyle{V}$};
\end{tikzpicture}&\text{in the degenerate case,}\\
\begin{tikzpicture}[baseline = -1mm]
	\draw[->] (0.28,-.28) to (-0.28,.28);
	\draw[-,line width=3pt,white] (-0.28,-.28) to (0.28,.28);
	\draw[->] (-0.28,-.28) to (0.28,.28);
      \node at (-0.33,-0.43) {$\scriptstyle{j}$};
      \node at (0.33,-0.43) {$\scriptstyle{i}$};
      \node at (-0.33,0.43) {$\scriptstyle{j}$};
      \node at (0.33,0.43) {$\scriptstyle{i}$};
\node at (0,-.01) {$\diamond$};
	\draw[-,darkg,thick] (0.8,-.25) to (0.8,.3);
      \node at (0.8,-0.4) {$\darkg\scriptstyle{V}$};
\end{tikzpicture}
-z\begin{tikzpicture}[baseline = -0.8mm]
	\draw[->,densely dotted] (0.55,.03) to (0.15,.03);
	\draw[->] (0.58,-.25) to (0.58,.3);
	\draw[->] (0.08,-.25) to (0.08,.3);
      \node at (0.08,-0.4) {$\scriptstyle{j}$};
      \node at (0.08,0.02) {$\dt$};
      \node at (0.58,-0.4) {$\scriptstyle{i}$};
      \node at (0.58,0.02) {$\dt$};
      \node at (1.4,0.04) {$\scriptstyle{x_2(x_2-x_1)^{-1}}$};
	\draw[-,darkg,thick] (2.3,-.25) to (2.3,.3);
      \node at (2.3,-0.4) {$\darkg\scriptstyle{V}$};
\end{tikzpicture}
&\text{in the quantum case}
\end{array}\right.
$$
is zero in the finite-dimensional algebra 
$A := \End_{\R}(E_j E_i V)$.
Let $L:A \rightarrow A$ be the linear map defined by left
multiplication (diagrammatically, this is vertical composition on the top) by 
$\begin{tikzpicture}[baseline = -1.7mm]
	\draw[->] (0.08,-.2) to (0.08,.2);
	\draw[->] (0.38,-.2) to (0.38,.2);
      \node at (0.08,-0.35) {$\scriptstyle{j}$};
      \node at (0.38,-0.35) {$\scriptstyle{i}$};
      \node at (0.08,-0.02) {$\dt$};
 	\draw[-,darkg,thick] (0.68,-.2) to (0.68,.2);
     \node at (0.83,0) {$\darkg\scriptstyle{V}$};
\end{tikzpicture}
$,
let $R:A \rightarrow A$ be the linear map defined by right
multiplication (diagrammatically, this is vertical composition on the bottom) by 
$\begin{tikzpicture}[baseline = -1.7mm]
	\draw[->] (0.08,-.2) to (0.08,.2);
	\draw[->] (0.38,-.2) to (0.38,.2);
      \node at (0.08,-0.35) {$\scriptstyle{j}$};
      \node at (0.38,-0.35) {$\scriptstyle{i}$};
      \node at (0.38,-0.02) {$\dt$};
 	\draw[-,darkg,thick] (0.68,-.2) to (0.68,.2);
     \node at (0.83,0) {$\darkg\scriptstyle{V}$};
\end{tikzpicture}
$, and let
$I:A \rightarrow A$ be the identity map.
We have that $(L-jI)^{\eps_j(E_i V)} = 0$ 
and $(R - i I)^{\eps_i(V)} = 0$.
Hence, for sufficiently large $N$, we have that
$$
((L-R) + (i-j) I)^N = ((L-j I)-(R-iI))^N = 0.
$$
Now observe that $(L-R)(\phi) = 0$ by the relations (\ref{newdotslide})--(\ref{newdotslide2}).
Hence, we have shown that $(i-j)^N \phi = 0$. Since $i \neq j$ this
implies that $\phi = 0$.
\end{proof}

\subsection{Bubbles and central characters}\label{swd}
Any dotted bubble in $\Heis_k$
defines an endomorphism of the identity functor
$\operatorname{Id}_{\R}$, i.e., an element of the center of
the category $\R$.
In particular, for $V \in \R$, dotted bubbles evaluate to
elements of the center $Z_V$ of the endomorphism
algebra $\End_\R(V)$.
It is convenient to
work with all of these endomorphisms at once in terms of the
generating function
\begin{align}\label{jjj}
\h_V(u)&:=
\mathord{
\begin{tikzpicture}[baseline = -1mm]
     \node at (0.08,0) {$\scriptstyle\anticlock(u)$};
 	\draw[-,darkg,thick] (0.68,.2) to (0.68,-.22);
     \node at (0.68,-.37) {$\darkg\scriptstyle{V}$};
\end{tikzpicture}
}=
\left(\mathord{
\begin{tikzpicture}[baseline = -1mm]
     \node at (0.08,0) {$\scriptstyle\clock(u)$};
 	\draw[-,darkg,thick] (0.68,.2) to (0.68,-.22);
     \node at (0.68,-.37) {$\darkg\scriptstyle{V}$};
\end{tikzpicture}
}\right)^{-1}.
\end{align}
Recalling (\ref{igq}) and
(\ref{summer1}),
we have $\h_V(u) \in u^k + u^{k-1} Z_V[\![u^{-1}]\!]$.
In the quantum case, there is also a distinguished element $t_V \in
Z_V^\times$ define by the action of $t 1_\unit$.
In the following lemma, 
given a polynomial
$p(u) = \sum_{s=0}^r z_s u^{r-s} \in Z_V[u]$,
we let
\begin{align*}
\mathord{
\begin{tikzpicture}[baseline = -1mm]
 	\draw[->] (0.08,-.3) to (0.08,.3);
     \node at (0.08,0.01) {$\dt$};
     \node at (-0.3,0.01) {$\scriptstyle p(x)$};
 	\draw[-,darkg,thick] (0.45,.3) to (0.45,-.3);
     \node at (0.45,-0.45) {$\darkg\scriptstyle{V}$};
\end{tikzpicture}
}
&:=
\sum_{s=0}^r 
\mathord{
\begin{tikzpicture}[baseline = -1mm]
 	\draw[->] (0.08,-.3) to (0.08,.3);
     \node at (0.08,0.01) {$\dt$};
     \node at (-0.27,0.01) {$\scriptstyle x^{r-s}$};
 	\draw[-,darkg,thick] (0.45,.15) to (0.45,.3);
 	\draw[-,darkg,thick] (0.45,-.15) to (0.45,-.3);
     \node at (0.45,-.45) {$\darkg\scriptstyle{V}$};
      \draw[darkg,thick] (0.45,0.0) circle (4.5pt);
   \node at (0.45,0) {$\darkg\scriptstyle{z_s}$};
\end{tikzpicture}
},
&
\mathord{
\begin{tikzpicture}[baseline = -1mm]
 	\draw[<-] (0.08,-.3) to (0.08,.3);
     \node at (0.08,0.01) {$\dt$};
     \node at (-0.3,0.01) {$\scriptstyle p(x)$};
 	\draw[-,darkg,thick] (0.45,.3) to (0.45,-.3);
     \node at (0.45,-.45) {$\darkg\scriptstyle{V}$};
\end{tikzpicture}
}
&:=
\sum_{s=0}^r
\mathord{
\begin{tikzpicture}[baseline = -1mm]
 	\draw[<-] (0.08,-.3) to (0.08,.3);
     \node at (0.08,0.01) {$\dt$};
     \node at (-0.27,0.01) {$\scriptstyle x^{r-s}$};
 	\draw[-,darkg,thick] (0.45,.15) to (0.45,.3);
 	\draw[-,darkg,thick] (0.45,-.15) to (0.45,-.3);
      \draw[darkg,thick] (0.45,0.0) circle (4.5pt);
   \node at (0.45,0) {$\darkg\scriptstyle{z_s}$};
     \node at (0.45,-.45) {$\darkg\scriptstyle{V}$};
\end{tikzpicture}
}.
\end{align*}
Lemmas~\ref{sure} and \ref{sure2lemma} obviously extend to the setting 
of coefficients in $Z_V$.

\begin{Lemma}\label{preimpy}
Let $V \in \R$ be any object. 
\begin{itemize}
\item[(1)]
If $f(u) \in Z_V[u]$ is a monic polynomial such that
$\mathord{
\begin{tikzpicture}[baseline = -1mm]
 	\draw[->] (0.08,-.2) to (0.08,.2);
     \node at (0.08,0.01) {$\dt$};
     \node at (-0.3,0.01) {$\scriptstyle f(x)$};
 	\draw[-,darkg,thick] (0.32,.2) to (0.32,-.2);
     \node at (0.49,0) {$\darkg\scriptstyle{V}$};
\end{tikzpicture}
}=0$,
then $g(u) := \h_V(u) f(u)$ is a monic polynomial
in $Z_V[u]$
of degree $\deg f(u)+k$ 
such that
${\begin{tikzpicture}[baseline = -1mm]
 	\draw[<-] (0.08,-.2) to (0.08,.2);
     \node at (-0.3,0.01) {$\scriptstyle g(x)$};
     \node at (0.08,0.01) {$\dt$};
 	\draw[-,darkg,thick] (0.32,.2) to (0.32,-.2);
     \node at (0.49,0) {$\darkg\scriptstyle{V}$};
\end{tikzpicture}
}=0$.
\item[(2)]
If $g(u) \in Z_V[u]$ is a monic polynomial such that
$\mathord{
\begin{tikzpicture}[baseline = -1mm]
 	\draw[<-] (0.08,-.2) to (0.08,.2);
     \node at (-0.3,0.01) {$\scriptstyle g(x)$};
     \node at (0.08,0.01) {$\dt$};
 	\draw[-,darkg,thick] (0.32,.2) to (0.32,-.2);
     \node at (0.49,0) {$\darkg\scriptstyle{V}$};
\end{tikzpicture}
}=0$, then $f(u) := \h_V(u)^{-1} g(u)$ is a monic polynomial
in $Z_V[u]$
of degree $\deg g(u)-k$ 
such that
$\mathord{
\begin{tikzpicture}[baseline = -1mm]
 	\draw[->] (0.08,-.2) to (0.08,.2);
     \node at (0.08,0.01) {$\dt$};
     \node at (-0.3,0.01) {$\scriptstyle f(x)$};
 	\draw[-,darkg,thick] (0.32,.2) to (0.32,-.2);
     \node at (0.49,0) {$\darkg\scriptstyle{V}$};
\end{tikzpicture}
}=0$.
\end{itemize}
In the quantum case, we also have that
$f(0) = t_V^2 g(0)$ in both situations.
\end{Lemma}

\begin{proof}
We just consider (1), since (2) is similar.
To show that $g(u)$ is a polynomial,
we must show that $[g(u)]_{u^{-r-1}}=0$ for $r \geq 0$.
Let $p(u) := u^r f(u)$
in the degenerate case or $p(u) := t_V^{-1} z u^{r+1} f(u)$ in the quantum case.
Applying (\ref{sure2}) or (\ref{surely2}),
we have that
$$
\big[
g(u)\big]_{u^{-r-1}}
=\big[
\h_V (u) f(u)\big]_{u^{-r-1}}
=
\left[
\mathord{
\begin{tikzpicture}[baseline = 1.25mm]
\node at (-.05,.2){$\scriptstyle\anticlock(u)$};
 	\draw[-,darkg,thick] (.8,.6) to (.8,.43);
 	\draw[-,darkg,thick] (.8,-.03) to (.8,-.2);
     \node at (.8,-.35) {$\darkg\scriptstyle{V}$};
 \draw[darkg,thick] (.8,0.2) ellipse (9pt and 6.5pt);   
\node at (.82,0.2) {$\darkg\scriptstyle{f(u)}$};
\end{tikzpicture}
}\right]_{u^{-r-1}}=
\mathord{
\begin{tikzpicture}[baseline = 1.25mm]
  \draw[-] (0,0.4) to[out=180,in=90] (-.2,0.2);
  \draw[->] (0.2,0.2) to[out=90,in=0] (0,.4);
 \draw[-] (-.2,0.2) to[out=-90,in=180] (0,0);
  \draw[-] (0,0) to[out=0,in=-90] (0.2,0.2);
   \node at (0.2,0.2) {$\dt$};
   \node at (0.55,0.2) {$\scriptstyle{p(x)}$};
 	\draw[-,darkg,thick] (1,.6) to (1,-.2);
     \node at (1,-.35) {$\darkg\scriptstyle{V}$};
\end{tikzpicture}
}.
$$
This is zero as 
$\mathord{
\begin{tikzpicture}[baseline = -1mm]
 	\draw[->] (0.08,-.2) to (0.08,.2);
     \node at (0.08,0.01) {$\dt$};
     \node at (-0.3,0.01) {$\scriptstyle f(x)$};
 	\draw[-,darkg,thick] (0.32,.2) to (0.32,-.2);
     \node at (0.49,0) {$\darkg\scriptstyle{V}$};
\end{tikzpicture}
}=0$.
Hence, $g(u)$ is a polynomial in $u$.
Moreover, in the quantum case the same argument with $r=-1$
gives that $g(0) = t_V^{-2} f(0)$.

It remains to show that
$\mathord{
\begin{tikzpicture}[baseline = -1mm]
 	\draw[<-] (0.08,-.2) to (0.08,.2);
     \node at (0.08,0.03) {$\dt$};
     \node at (-0.3,0.03) {$\scriptstyle g(x)$};
 	\draw[-,darkg,thick] (0.32,-.2) to (0.32,.2);
     \node at (0.57,0) {$\darkg\scriptstyle{V}$};
\end{tikzpicture}
}=0$.
In the degenerate case, this follows by
(\ref{sure1}) and (\ref{sure3}):
\begin{align*}
\mathord{
\begin{tikzpicture}[baseline = -1mm]
 	\draw[<-] (0.08,-.4) to (0.08,.4);
     \node at (0.08,0) {$\dt$};
     \node at (0.45,0) {$\scriptstyle g(x)$};
 	\draw[-,darkg,thick] (.9,.4) to (.9,-.4);
     \node at (.9,-.55) {$\darkg\scriptstyle{V}$};
\end{tikzpicture}
}
&=\left[\mathord{
\begin{tikzpicture}[baseline = -1mm]
	\draw[<-] (0.08,-.4) to (0.08,.4);
   \node at (-0.46,0) {$\scriptstyle{(u-x)^{-1}}$};
      \node at (0.08,0) {$\dt$};
 	\draw[-,darkg,thick] (.6,.4) to (.6,.23);
 	\draw[-,darkg,thick] (.6,.-.23) to (.6,-.4);
     \node at (.6,-.55) {$\darkg\scriptstyle{V}$};
 \draw[darkg,thick] (0.6,0) ellipse (9pt and 6.5pt);   
\node at (0.62,0) {$\darkg\scriptstyle{g(u)}$};
\end{tikzpicture}
} \right]_{u^{-1}}
=
\left[\mathord{
\begin{tikzpicture}[baseline = -1mm]
	\draw[<-] (0.08,-.4) to (0.08,.4);
   \node at (-0.46,0) {$\scriptstyle{(u-x)^{-1}}$};
      \node at (0.08,0) {$\dt$};
      \node at (.65,0) {$\anticlock \scriptstyle (u)$};
 	\draw[-,darkg,thick] (1.5,.4) to (1.5,.23);
 	\draw[-,darkg,thick] (1.5,.-.23) to (1.5,-.4);
     \node at (1.5,-.55) {$\darkg\scriptstyle{V}$};
 \draw[darkg,thick] (1.5,0) ellipse (9pt and 6.5pt);   
\node at (1.52,0) {$\darkg\scriptstyle{f(u)}$};
\end{tikzpicture}
}\right]_{u^{-1}}
=\:
\mathord{
\begin{tikzpicture}[baseline = -1mm]
	\draw[-] (0,0.4) to (0,0.3);
	\draw[-] (0,0.3) to [out=-90,in=180] (.3,-0.2);
	\draw[-] (0.3,-0.2) to [out=0,in=-90](.5,0);
	\draw[-] (0.5,0) to [out=90,in=0](.3,0.2);
	\draw[-] (0.3,.2) to [out=180,in=90](0,-0.3);
	\draw[->] (0,-0.3) to (0,-0.4);
   \node at (.85,0) {$\scriptstyle{f(x)}$};
      \node at (0.5,0) {$\dt$};
 	\draw[-,darkg,thick] (1.3,.4) to (1.3,-.4);
     \node at (1.3,-.55) {$\darkg\scriptstyle{V}$};
\end{tikzpicture}
}=0.
\end{align*}
The proof in the quantum case is similar, using (\ref{surely1}) and
(\ref{surely3}) instead.
\end{proof}

If $L \in \R$ is irreducible then of course 
$\h_L(u) \in \k(\!(u^{-1})\!)$.
The following relates the central character
information encoded in this generating function
to the minimal polynomials $m_L(u)$ and $n_L(u)$ introduced earlier.

\begin{Lemma}\label{impy}
For an irreducible object $L \in \R$,
we have that
$$
\h_L(u) = n_L(u)/m_L(u).
$$
Moreover, in the quantum case, the (invertible!) constant terms of
the polynomials $m_L(u)$ and $n_L(u)$ satisfy $t_L^2 = m_L(0) / n_L(0)$.
\end{Lemma}

\begin{proof}
Applying Lemma~\ref{preimpy}(1) with $f(u) = m_L(u)$ shows that
$\h_L(u) m_L(u)$ is a monic polynomial of degree $\deg m_L(u) + k$ 
which is divisible by $n_L(u)$. Hence, $\deg n_L(u) \leq \deg m_L(u)+k$.
Applying Lemma~\ref{preimpy}(2) with $g(u) = n_L(u)$ shows that
$\h_L(u)^{-1} n_L(u)$ is a monic polynomial of degree $\deg n_L(u)-k$
that is divisible by $m_L(u)$. Hence, $\deg m_L(u) \leq \deg n_L(u)-k$.
We deduce that both inequalities are equalities,
and we actually have that $n_L(u) = \h_L(u) m_L(u)$.
The assertion about the constant terms follows from the final part of Lemma~\ref{preimpy}.
\end{proof}

For $i \in \k$, define $i^\pm$ as in the introduction.

\begin{Lemma}\label{choose}
Suppose that $L \in \R$ is an irreducible object and
let $K$ be an irreducible subquotient of $E_i L$ for some $i \in \k$.
Then
\begin{equation} \label{choose1}
    \h_K(u)
    = \frac{\h_L(u) (u-i)^2}{(u-i^+)(u-i^-)}.
\end{equation}
\end{Lemma}

\begin{proof}
This follows from the bubble slides (\ref{moregin}) and (\ref{gin}).
For example, in the degenerate case, we have by (\ref{moregin})
that
$$
\mathord{
\begin{tikzpicture}[baseline = -1mm]
     \node at (0.08,0) {$\scriptstyle\anticlock(u)$};
 	\draw[-,darkg,thick] (1.1,.3) to (1.1,-.3);
 	\draw[<-] (0.68,.3) to (0.68,-.32);
     \node at (.68,-.45) {$\scriptstyle i$};
     \node at (1.1,-.45) {$\darkg\scriptstyle{L}$};
\end{tikzpicture}
}=
\mathord{
\begin{tikzpicture}[baseline = -1mm]
     \node at (0.6,0) {$\scriptstyle\anticlock(u)$};
 	\draw[-,darkg,thick] (1.1,.3) to (1.1,-.3);
 	\draw[<-] (0.08,.3) to (0.08,-.3);
     \node at (-.843,0.03) {$\scriptstyle \frac{(u-x)^2}{(u-(x+1))(u-(x-1))}\:\dt$};
     \node at (1.1,-.45) {$\darkg\scriptstyle{L}$};
     \node at (.08,-.45) {$\scriptstyle i$};
\end{tikzpicture}
}
=
\mathord{
\begin{tikzpicture}[baseline = -1mm]
 	\draw[-,darkg,thick] (.5,.3) to (.5,-.3);
 	\draw[<-] (0.08,.3) to (0.08,-.3);
     \node at (-.843,0.03) {$
\scriptstyle \frac{\h_L(u) (u-x)^2}{(u-(x+1))(u-(x-1))}\:\dt$};
     \node at (.5,-.45) {$\darkg\scriptstyle{L}$};
     \node at (.08,-.45) {$\scriptstyle i$};
\end{tikzpicture}
}.
$$
When we pass to the irreducible subquotient $K$ of $E_i L$, 
we can replace the occurences of $x$ in the expression on the right-hand side
here with $i$, and the lemma follows.
\end{proof}

Now we define the {\em spectrum} $I$ of $\R$
to be the union of the sets of roots of the minimal polynomials
$m_L(u)$ for all irreducible $L \in \R$.
Noting that $i$ is a root of $m_L(u)$ if and only if $E_i L \neq 0$,
we have equivalently that $I$ is the set of all $i \in \k$ such that
$E_i L \neq 0$ for some irreducible
$L \in \R$. In view of the exactness of $E_i$, we can
drop the word ``irreducible'' in this characterization:
the spectrum $I$ is the set of all $i \in \k$ such that
$E_i$ is a non-zero endofunctor of $\R$.
By adjunction, it follows that $I$ is the set of all $i \in \k$
such that the endofunctor $F_i$ is non-zero, hence, $I$ could also
be defined as
the union of the sets of roots of the polynomials
$n_L(u)$ for all irreducible $L \in \R$.
This discussion shows that
\begin{align}\label{pape}
E &=
\bigoplus_{i \in I} E_i,&
F &= \bigoplus_{i \in I} F_i,
\end{align}
with each of the endofunctors $E_i$ and $F_i$ written here being non-zero.

\begin{Lemma}\label{closed}
We have that $i \in I$ if and only if $i^+ \in I$.
Moreover, in the quantum case, we have that $0 \notin I$.
\end{Lemma}

\begin{proof}
The fact that $0 \notin I$ in the quantum case follows from the invertibility of the dot.
For the first part, 
it suffices to show for $i \in I$ that $i^+$ and $i^-$ both belong to
$I$.
Let $j:= i^\pm$ for some choice of the sign.
As $i \in I$, there is an irreducible $L \in \R$ such that
$E_i L \neq 0$.
Let $K$ be an irreducible subquotient of $E_i L$.
By (\ref{choose1}), we have that
$\h_K(u) (u-i^+)(u-i^-) = \h_L(u) (u-i)^2$. Using
Lemma~\ref{impy}, we deduce that
$$
m_L(u) n_K(u) (u-i^+)(u-i^-) = m_K(u) n_L(u) (u-i)^2.
$$
Thus $(u-j)$ divides either $m_K(u)$ or $n_L(u)$, so either $E_j K
\neq 0$ or $F_j L \neq 0$.
This shows that $E_j \neq 0$ or $F_j \neq 0$, hence, $j \in I$.
\end{proof}

In view of Lemma~\ref{closed}, the map $i \mapsto i^+$ defines a fixed-point-free automorphism of $I$.
This puts us in the situation of $\S$\ref{songs}, so we can associate
a Kac-Moody Lie algebra $\g$ with weight lattice $X$,
fundamental weights $\{\Lambda_i\:|\:i \in I\}$, etc.
For an irreducible object $L \in \R$, let
\begin{equation}\label{chickens}
\wt(L) := \sum_{i \in I} (\phi_i(L)-\eps_i(L)) \Lambda_i \in X.
\end{equation}
In other words, due to the definition preceeding (\ref{CRTdef}) and
Lemma~\ref{impy},
$\langle h_i, \wt(L)\rangle \in \Z$ is the multiplicity of $u=i$ as a zero
or pole of the rational function $\h_L(u) \in \k(u)$ for each $i \in
I$.
Then for $\lambda \in X$ we let $\R_\lambda$ be the Serre
subcategory of $\R$ consisting of the objects $V$
such that every irreducible subquotient $L$ of $V$
satisfies $\wt(L) = \lambda$.
The point of this definition is
that irreducible objects $K, L \in \R$ with $\wt(K) \neq \wt(L)$
have different central characters.
Using also the general theory of blocks in our two sorts of Abelian category,
it follows that
\begin{equation}\label{one}
    \textstyle
    \R =
    \begin{cases}
        \bigoplus_{\lambda \in X} \R_\lambda & \text{if $\R$ is locally finite Abelian}, \\
        \prod_{\lambda \in X} R_\lambda & \text{if $\R$ is Schurian}.
    \end{cases}
\end{equation}
We refer to this as the {\em weight space}
decomposition of $\R$.

\begin{Lemma}\label{pg}
For $\lambda \in X$ and $i \in I$, the
 restrictions of $E_i$ and $F_i$ to $\R_\lambda$ give functors
\begin{align*}
E_i|_{\R_\lambda}&:\R_\lambda \rightarrow \R_{\lambda+\alpha_i},&
F_i|_{\R_\lambda}&:\R_\lambda \rightarrow \R_{\lambda-\alpha_i},
\end{align*}
\end{Lemma}

\begin{proof}
For $E_i$, this follows from Lemma~\ref{choose}.
Then it follows for $F_i$ by adjunction.
\end{proof}

\subsection{The main isomorphism}
The next lemma is quite trivial but serves as a good warm-up exercise for
the one that follows.

\begin{Lemma}\label{banach1}
For $i,j \in I$ with $j \neq i$,
the natural transformations
\begin{align*}
\begin{tikzpicture}[baseline = -1mm]
	\draw[->] (0.28,-.28) to (-0.28,.28);
	\draw[<-] (-0.28,-.28) to (0.28,.28);
      \node at (-0.33,-0.43) {$\scriptstyle{j}$};
      \node at (0.33,-0.43) {$\scriptstyle{i}$};
      \node at (-0.33,0.43) {$\scriptstyle{i}$};
      \node at (0.33,0.43) {$\scriptstyle{j}$};
\node at (0,-.01) {$\diamond$};
\end{tikzpicture}
:F_j E_i \Rightarrow E_i F_j&
,&
\begin{tikzpicture}[baseline = -1mm]
	\draw[<-] (0.28,-.28) to (-0.28,.28);
	\draw[->] (-0.28,-.28) to (0.28,.28);
      \node at (-0.33,-0.43) {$\scriptstyle{i}$};
      \node at (0.33,-0.43) {$\scriptstyle{j}$};
      \node at (-0.33,0.43) {$\scriptstyle{j}$};
      \node at (0.33,0.43) {$\scriptstyle{i}$};
\node at (0,-.01) {$\diamond$};
\end{tikzpicture}
: E_i F_j \Rightarrow F_j E_i&
\end{align*}
are mutually inverse
isomorphisms.
(Here, we have drawn the crossings in the degenerate
case; in the quantum case they should be interpreted as 
positive or negative crossings, it does not matter which is chosen.)
\end{Lemma}

\begin{proof}
Check that the compositions both ways around are the
identities. For example, one way in the degenerate case gives
\begin{align*}
\mathord{
\begin{tikzpicture}[baseline = -1mm]
\node at (-.23,.75){$\scriptstyle j$};
\node at (.23,.75){$\scriptstyle i$};
\node at (-.23,-.75){$\scriptstyle j$};
\node at (.23,-.75){$\scriptstyle i$};
\node at (-.35,0){$\scriptstyle i$};
\node at (.35,0){$\scriptstyle j$};
	\draw[-] (0.23,0) to[out=90,in=-90] (-0.23,.6);
	\draw[->] (-0.23,0) to[out=90,in=-90] (0.23,.6);
	\draw[-] (0.23,-.6) to[out=90,in=-90] (-0.23,0);
	\draw[<-] (-0.23,-.6) to[out=90,in=-90] (0.23,0);
\node at (0,.3){$\diamond$};
\node at (0,-.3){$\diamond$};
\end{tikzpicture}
}
&
=
\mathord{
\begin{tikzpicture}[baseline = -1mm]
\node at (-.23,.75){$\scriptstyle j$};
\node at (.23,.75){$\scriptstyle i$};
\node at (-.23,-.75){$\scriptstyle j$};
\node at (.23,-.75){$\scriptstyle i$};
	\draw[-] (0.23,0) to[out=90,in=-90] (-0.23,.6);
	\draw[->] (-0.23,0) to[out=90,in=-90] (0.23,.6);
	\draw[-] (0.23,-.6) to[out=90,in=-90] (-0.23,0);
	\draw[<-] (-0.23,-.6) to[out=90,in=-90] (0.23,0);
	\draw[-] (-.2,.4) to (-.13,.45);
	\draw[-] (-.2,-.4) to (-.13,-.45);
	\draw[-] (.2,.4) to (.13,.45);
	\draw[-] (.2,-.4) to (.13,-.45);
\end{tikzpicture}
}
=
\mathord{
\begin{tikzpicture}[baseline = -1mm]
\node at (-.28,.75){$\scriptstyle j$};
\node at (.08,-.75){$\scriptstyle i$};
	\draw[->] (0.08,-.6) to (0.08,.6);
	\draw[<-] (-0.28,-.6) to (-0.28,.6);
\end{tikzpicture}
}
+\sum_{r,s \geq 0}
\mathord{
\begin{tikzpicture}[baseline = 1mm]
  \draw[<-] (0,0.4) to[out=180,in=90] (-.2,0.2);
  \draw[-] (0.2,0.2) to[out=90,in=0] (0,.4);
 \draw[-] (-.2,0.2) to[out=-90,in=180] (0,0);
  \draw[-] (0,0) to[out=0,in=-90] (0.2,0.2);
   \node at (-0.2,0.2) {$\dt$};
   \node at (-.78,0.2) {$\scriptstyle{-r-s-2}$};
\end{tikzpicture}
}
\mathord{
\begin{tikzpicture}[baseline=-1mm]
	\draw[<-] (0.2,0.6) to[out=-90, in=0] (0,.1);
	\draw[-] (0,.1) to[out = 180, in = -90] (-0.2,0.6);
	\draw[-] (0.2,-.6) to[out=90, in=0] (0,-0.1);
	\draw[->] (0,-0.1) to[out = 180, in = 90] (-0.2,-.6);
\node at (-0.2,.75){$\scriptstyle j$};
\node at (.2,-.75){$\scriptstyle i$};
\node at (-0.2,-.75){$\scriptstyle j$};
\node at (.2,.75){$\scriptstyle i$};
      \node at (0.16,-0.2) {$\dt$};
      \node at (0.33,-0.2) {$\scriptstyle{s}$};
   \node at (0.16,0.2) {$\dt$};
   \node at (.33,.2) {$\scriptstyle{r}$};
	\draw[-] (-.24,.415) to (-.155,.42);
	\draw[-] (.24,.415) to (.155,.42);
	\draw[-] (-.24,-.415) to (-.155,-.42);
	\draw[-] (.24,-.415) to (.155,-.42);
\end{tikzpicture}}
=
\mathord{
\begin{tikzpicture}[baseline = -1mm]
\node at (-.28,.75){$\scriptstyle j$};
\node at (.08,-.75){$\scriptstyle i$};
	\draw[->] (0.08,-.6) to (0.08,.6);
	\draw[<-] (-0.28,-.6) to (-0.28,.6);
\end{tikzpicture}
},
\end{align*}
using Lemma~\ref{essential} (the sideways crossing version!) for the first equality, the relation \eqref{sideways} for
the second, and (\ref{othogonality})--(\ref{incpro}) for the final one.
The other cases are similar.
\end{proof}

Now we come to what is really the main step. In the statement 
of the following two lemmas, the restrictions $F_i E_i|_{\R_\lambda}$ and $E_i F_i|_{\R_\lambda}$
are endofunctors of $\R_\lambda$ due to Lemma~\ref{pg}.

\begin{Lemma}\label{banach2}
Given $\lambda \in X$ and $i \in I$ such that
 $\langle h_i, \lambda \rangle \leq 0$, the natural
transformation
$$
\left[
\begin{tikzpicture}[baseline = -1mm]
	\draw[<-] (0.28,-.28) to (-0.28,.28);
	\draw[-,white,line width=3pt] (-0.28,-.28) to (0.28,.28);
	\draw[->] (-0.28,-.28) to (0.28,.28);
      \node at (-0.33,-0.43) {$\scriptstyle{i}$};
      \node at (0.33,-0.43) {$\scriptstyle{i}$};
      \node at (-0.33,0.43) {$\scriptstyle{i}$};
      \node at (0.33,0.43) {$\scriptstyle{i}$};
\node at (0,-.01) {$\diamond$};
\end{tikzpicture}
\quad
\mathord{
\begin{tikzpicture}[baseline = 1mm]
	\draw[<-] (0.4,.4) to[out=-90, in=0] (0.1,0);
	\draw[-] (0.1,0) to[out = 180, in = -90] (-0.2,.4);
      \node at (-0.2,0.55) {$\scriptstyle{i}$};
\end{tikzpicture}
}
\quad
\mathord{
\begin{tikzpicture}[baseline = 1mm]
	\draw[<-] (0.4,.4) to[out=-90, in=0] (0.1,0);
	\draw[-] (0.1,0) to[out = 180, in = -90] (-0.2,.4);
      \node at (-0.2,0.55) {$\scriptstyle{i}$};
\node at (0.37,.2) {$\dt$};
\node at (0.72,.2) {$\scriptstyle x-i$};
\end{tikzpicture}
}
\quad
\cdots
\quad
\mathord{
\begin{tikzpicture}[baseline = 1mm]
	\draw[<-] (0.4,.4) to[out=-90, in=0] (0.1,0);
	\draw[-] (0.1,0) to[out = 180, in = -90] (-0.2,.4);
      \node at (-0.2,0.55) {$\scriptstyle{i}$};
\node at (0.37,.2) {$\dt$};
\node at (1.25,.2) {$\scriptstyle(x-i)^{-\langle
    h_i,\lambda\rangle-1}$};
\end{tikzpicture}
}
\right]
:E_i F_i|_{\R_\lambda}
\oplus \operatorname{Id}_{\R_\lambda}^{\oplus (-\langle h_i, \lambda \rangle)}\Rightarrow
F_i E_i|_{\R_\lambda}
$$
is an isomorphism.
 (This time, we have drawn the crossing in the quantum case; in the
degenerate case it should be replaced by the degenerate
crossing.)
\end{Lemma}

\begin{proof}
We just prove this in the quantum case; the degenerate case is similar.
It suffices to prove that the natural transformation in the statement
of the lemma defines an isomorphism on every irreducible object $L \in
\R_\lambda$; in the Schurian case one needs to apply
Lemma~\ref{bdl} to make this reduction.
So take an irreducible $L \in \R_\lambda$.
We have that $m:= \eps_i(L)-\phi_i(L)=-\langle h_i,\lambda\rangle \geq 0$.
Let $P := \k[u] / (m_L(u))$ and
$Q := \k[u] / (n_L(u))$.
Let $P_i$ and $Q_i$ be the summands of $P$ and $Q$ that are
isomorphic to
$\k[u] \big/ \big((u-i)^{\eps_i(L)}\big)$
and
$\k[u] \big/ \big((u-i)^{\phi_i(L)}\big)$
in the CRT decomposition (\ref{CRTdef}).
To be explicit, let $f(u)$ and $g(u)$ be polynomials
such that
\begin{align*}
f(u) m_L(u) / (u-i)^{\eps_i(u)}&\equiv 1 \pmod{(u-i)^{\eps_i(L)}},\\
g(u) n_L(u) / (u-i)^{\phi_i(u)} &\equiv 1 \pmod{(u-i)^{\phi_i(L)}}.
\end{align*}
Then the identity elements $e_i \in P_i$ and $f_i \in Q_i$
are the images of
$f(u) m_L(u) / (u-i)^{\eps_i(L)}$
and $g(u) n_L(u) / (u-i)^{\phi_i(L)}$
in $P$ and $Q$, respectively.
Moreover, $f(u)$ is invertible in $P_i$, so $P_i$ can be described
equivalently as the ideal of
$P$ generated by $m_L(u) / (u-i)^{\eps_i(L)}$.
Similarly, $Q_i$ is the ideal of $Q$ generated by
$n_L(u) / (u-i)^{\phi_i(L)}$.
There is an injective $\k[u]$-module homomorphism
$$
\mu:Q_i \hookrightarrow P_i,
\qquad
n_L(u) / (u-i)^{\phi_i(L)} \mapsto t_L^{-1} m_L(u) / (u-i)^{\phi_i(L)}.
$$
Its image has basis
$(u-i)^m e_i,
(u-i)^{m+1} e_i,\dots,
(u-i)^{\eps_i(L)-1} e_i$.
Let $C_i$ be the subspace
of $P_i$ with basis
$e_i, (u-i) e_i,\dots, (u-i)^{m-1} e_i$. This is a linear
complement to $\mu(Q_i)$ in $P_i$.

The composition of the algebra embeddings (\ref{CRT1}) with the
adjunction isomorphisms
$\End_{\R}(EL) \cong \Hom_{\R}(L, FEL)$
and
$\End_{\R}(FL) \cong \Hom_{\R}(L, EFL)$
give us linear embeddings $\vec{\beta}:P \hookrightarrow \Hom_{\R}(L, FEL)$
and $\cev{\beta}:Q \hookrightarrow \Hom_{\R}(L, EFL)$,
respectively.
So:
\begin{align*}
\vec{\beta}(p(u)) &=
\mathord{
\begin{tikzpicture}[baseline = 1mm]
	\draw[<-] (0.4,.4) to[out=-90, in=0] (0.1,0);
	\draw[-] (0.1,0) to[out = 180, in = -90] (-0.2,.4);
      \node at (0.7,0.18) {$\scriptstyle{p(x)}$};
      \node at (.36,0.18) {$\dt$};
	\draw[-,darkg,thick] (1.1,.4) to (1.1,-.2);
      \node at (1.1,-0.35) {$\darkg\scriptstyle{L}$};
\end{tikzpicture}
},
&\cev{\beta}(p(u)) &=
\mathord{
\begin{tikzpicture}[baseline = 1mm]
	\draw[-] (0.4,.4) to[out=-90, in=0] (0.1,0);
	\draw[->] (0.1,0) to[out = 180, in = -90] (-0.2,.4);
      \node at (0.7,0.18) {$\scriptstyle{p(x)}$};
      \node at (.36,0.18) {$\dt$};
	\draw[-,darkg,thick] (1.1,.4) to (1.1,-.2);
      \node at (1.1,-0.35) {$\darkg\scriptstyle{L}$};
\end{tikzpicture}
}.
\end{align*}
Recalling (\ref{tensorproduct}), the linear maps $\vec{\beta}$ and $\cev{\beta}$ induce morphisms
\begin{align*}
\vec{\gamma}:L \otimes P &\rightarrow FEL,
&\cev{\gamma}:L \otimes Q
&\rightarrow EFL.
\end{align*}
For example, if $v_1,\dots,v_n$ is
the fixed basis for $P$ then $\vec{\gamma}$ is the morphism $L^{\oplus n}
\rightarrow FEL$ defined by the matrix
$\left[\vec{\beta}(v_1)\:\cdots\:\vec{\beta}(v_n)\right]$. 
As 
the morphisms $\vec{\beta}(v_1),\dots,\vec{\beta}(v_n):L \rightarrow FEL$ are
linearly independent and $L$ is irreducible,
$\vec{\gamma}$ is a monomorphism. Similarly, so is $\cev{\gamma}$.
As $\vec{\beta}(e_i)$ maps $L$ into the summand $F_i E_i L$ of
$FEL$, we have that
$\vec{\gamma}(L \otimes P_i) \subseteq F_i E_i L$. Similarly,
$\cev{\gamma}(L \otimes Q_i) \subseteq E_i F_i L$.
Finally, let
\begin{align*}
\vec{\chi}&:=
\begin{tikzpicture}[baseline = -1mm]
	\draw[<-] (0.28,-.28) to (-0.28,.28);
	\draw[-,line width=3pt,white] (-0.28,-.28) to (0.28,.28);
	\draw[->] (-0.28,-.28) to (0.28,.28);
      \node at (-0.33,-0.43) {$\scriptstyle{i}$};
      \node at (0.33,-0.43) {$\scriptstyle{i}$};
      \node at (-0.33,0.43) {$\scriptstyle{i}$};
      \node at (0.33,0.43) {$\scriptstyle{i}$};
\node at (0,-.01) {$\diamond$};
	\draw[-,darkg,thick] (0.7,-.28) to (0.7,.28);
	\node at (0.7,-.43) {$\darkg\scriptstyle{L}$};
\end{tikzpicture}:E_i F_i L \rightarrow F_i E_i L,
&\cev{\chi}&:=
\begin{tikzpicture}[baseline = -1mm]
	\draw[<-] (-0.28,-.28) to (0.28,.28);
	\draw[-,line width=3pt,white] (0.28,-.28) to (-0.28,.28);
	\draw[->] (0.28,-.28) to (-0.28,.28);
      \node at (-0.33,-0.43) {$\scriptstyle{i}$};
      \node at (0.33,-0.43) {$\scriptstyle{i}$};
      \node at (-0.33,0.43) {$\scriptstyle{i}$};
      \node at (0.33,0.43) {$\scriptstyle{i}$};
\node at (0,-.01) {$\diamond$};
	\draw[-,darkg,thick] (0.7,-.28) to (0.7,.28);
	\node at (0.7,-.43) {$\darkg\scriptstyle{L}$};
\end{tikzpicture}:F_i E_i L \rightarrow E_i F_i L.
\end{align*}
We are trying to prove that the morphism
$$
\left[
\vec{\chi}\quad
\vec{\beta}\left(e_i\right) \quad
\vec{\beta}\left((u-i)e_i\right)\quad\cdots\quad
\vec{\beta}\big((u-i)^{m-1} e_i\big)
\right]
:E_i F_i L \oplus L^{\oplus m} \rightarrow F_i E_i L
$$
is an isomorphism.
Equivalently,
using the basis $e_i, (u-i)
e_i,\dots,(u-i)^{m-1} e_i$ for $C_i$ to identify $L \otimes C_i$ with 
$L^{\oplus m}$,
we must show that 
$$
\theta := \left[\vec{\chi} \quad\vec{\gamma}|_{L\otimes C_i}\right]:E_i F_i L
\:\oplus\: L\otimes C_i
\rightarrow F_i E_i L
$$
is an isomorphism. This follows from the following series of claims.

\vspace{2mm}
\noindent
\underline{Claim 1}:
{\em
$\cev{\chi}(\vec{\gamma}(L\otimes P_i)) \subseteq \cev{\gamma}(L\otimes Q_i)$.}
To justify this, take $p(u) \in P_i$,
we have that
$$
\cev{\chi}(\vec{\beta}(p(u))) =
\mathord{
\begin{tikzpicture}[baseline = 1.5mm]
	\draw[-] (0.25,-0) to[out=-90, in=0] (0,-0.25);
	\draw[-] (0.25,.6) to[out=240,in=90] (-0.25,-0);
	\draw[-] (0,-0.25) to[out = 180, in = -90] (-0.25,-0);
	\draw[-,white,line width=3pt] (-0.25,.6) to[out=300,in=90] (0.25,-0);
	\draw[<-] (-0.25,.6) to[out=300,in=90] (0.25,-0);
      \node at (0.3,.2) {$\scriptstyle{i}$};
      \node at (-0.3,.2) {$\scriptstyle{i}$};
      \node at (0.58,-0.12) {$\scriptstyle{p(x)}$};
      \node at (.23,-.13) {$\dt$};
      \node at (-.28,.75) {$\scriptstyle i$};
      \node at (.28,.75) {$\scriptstyle i$};
      \node at (0,.35) {$\diamond$};
	\draw[-,darkg,thick] (1,.6) to (1,-.25);
      \node at (1,-0.4) {$\darkg\scriptstyle{L}$};
\end{tikzpicture}
}
=
\mathord{
\begin{tikzpicture}[baseline = 1.5mm]
	\draw[-] (0.25,-0) to[out=-90, in=0] (0,-0.25);
	\draw[-] (0,-0.25) to[out = 180, in = -90] (-0.25,-0);
	\draw[-] (0.25,.6) to[out=240,in=90] (-0.25,-0);
	\draw[-,white,line width=4pt] (-0.25,.6) to[out=300,in=90] (0.25,-0);
	\draw[<-] (-0.25,.6) to[out=300,in=90] (0.25,-0);
	\draw[-] (-.11,0.51) to (-0.18,.43);
	\draw[-] (.11,0.51) to (0.18,.43);
      \node at (-.28,.75) {$\scriptstyle i$};
      \node at (.28,.75) {$\scriptstyle i$};
      \node at (0.58,-0.12) {$\scriptstyle{p(x)}$};
      \node at (.23,-.13) {$\dt$};
	\draw[-] (.2,0.04) to (0.3,.05);
      \node at (.3,.22) {$\scriptstyle i$};
	\draw[-] (.11,0.2) to (0.18,.28);
	\draw[-,darkg,thick] (1,.6) to (1,-.25);
      \node at (1,-0.4) {$\darkg\scriptstyle{L}$};
\end{tikzpicture}
}=
\mathord{
\begin{tikzpicture}[baseline = 1.5mm]
	\draw[-] (0.25,-0) to[out=-90, in=0] (0,-0.25);
	\draw[-] (0,-0.25) to[out = 180, in = -90] (-0.25,-0);
	\draw[-] (0.25,.6) to[out=240,in=90] (-0.25,-0);
	\draw[-,white,line width=4pt] (-0.25,.6) to[out=300,in=90] (0.25,-0);
	\draw[<-] (-0.25,.6) to[out=300,in=90] (0.25,-0);
      \node at (0.6,.03) {$\scriptstyle{p(x)}$};
	\draw[-] (-.11,0.51) to (-0.18,.43);
	\draw[-] (.11,0.51) to (0.18,.43);
      \node at (-.28,.75) {$\scriptstyle i$};
      \node at (.28,.75) {$\scriptstyle i$};
      \node at (.25,0) {$\dt$};
	\draw[-,darkg,thick] (1,.6) to (1,-.25);
      \node at (1,-0.4) {$\darkg\scriptstyle{L}$};
\end{tikzpicture}
}\:.
$$
Using the defining relations, $p(x)$ can now be commuted past the
crossing and the curl can be ``straightened.''
The resulting morphism clearly has image
in $\cev{\gamma}(L \otimes Q_i)$.

\vspace{2mm}
\noindent
\underline{Claim 2}:
{\em $\vec{\chi} \circ \cev{\gamma}=\vec{\gamma} \circ (L \otimes \mu)$.}
Take a polynomial $p(u) \in \k[u]$ representing an element of $Q_i$, i.e., a
polynomial
divisible by $n_L(u) / (u-i)^{\phi_i(L)}$.
Let $q(u) := t_L^{-1} p(u) m_L(u) / n_L(u) \in \k[u]$. This is a representative
for the image of $p(u)$ under $\mu:Q_i \rightarrow P_i$.
Using Lemmas~\ref{surely} and \ref{impy}, we have
that
\begin{align*}
\vec{\chi}(\cev{\beta}(p(u))) &=
\mathord{
\begin{tikzpicture}[baseline = 1.5mm]
	\draw[-] (-0.25,.6) to[out=300,in=90] (0.25,-0);
	\draw[-] (0.25,-0) to[out=-90, in=0] (0,-0.25);
	\draw[-] (0,-0.25) to[out = 180, in = -90] (-0.25,-0);
	\draw[-,white,line width=3pt] (0.25,.6) to[out=240,in=90] (-0.25,-0);
	\draw[<-] (0.25,.6) to[out=240,in=90] (-0.25,-0);
      \node at (0.3,.2) {$\scriptstyle{i}$};
      \node at (-0.3,.2) {$\scriptstyle{i}$};
      \node at (0.58,-0.12) {$\scriptstyle{p(x)}$};
      \node at (.23,-.13) {$\dt$};
      \node at (-.28,.75) {$\scriptstyle i$};
      \node at (.28,.75) {$\scriptstyle i$};
      \node at (0,.35) {$\diamond$};
	\draw[-,darkg,thick] (1,.6) to (1,-.25);
      \node at (1,-0.4) {$\darkg\scriptstyle{L}$};
\end{tikzpicture}
}
=
\mathord{
\begin{tikzpicture}[baseline = 1.5mm]
	\draw[-] (-0.25,.6) to[out=300,in=90] (0.25,-0);
	\draw[-] (0.25,-0) to[out=-90, in=0] (0,-0.25);
	\draw[-] (0,-0.25) to[out = 180, in = -90] (-0.25,-0);
	\draw[-,line width=4pt,white] (0.25,.6) to[out=240,in=90] (-0.25,-0);
	\draw[<-] (0.25,.6) to[out=240,in=90] (-0.25,-0);
	\draw[-] (-.11,0.51) to (-0.18,.43);
	\draw[-] (.11,0.51) to (0.18,.43);
      \node at (-.28,.75) {$\scriptstyle i$};
      \node at (.28,.75) {$\scriptstyle i$};
      \node at (0.58,-0.12) {$\scriptstyle{p(x)}$};
      \node at (.23,-.13) {$\dt$};
	\draw[-] (.2,0.04) to (0.3,.05);
      \node at (.3,.22) {$\scriptstyle i$};
	\draw[-] (.11,0.2) to (0.18,.28);
	\draw[-,darkg,thick] (1,.6) to (1,-.25);
      \node at (1,-0.4) {$\darkg\scriptstyle{L}$};
\end{tikzpicture}
}=
\mathord{
\begin{tikzpicture}[baseline = 1.5mm]
	\draw[-] (-0.25,.6) to[out=300,in=90] (0.25,-0);
	\draw[-] (0.25,-0) to[out=-90, in=0] (0,-0.25);
	\draw[-] (0,-0.25) to[out = 180, in = -90] (-0.25,-0);
	\draw[-,line width=4pt,white] (0.25,.6) to[out=240,in=90] (-0.25,-0);
	\draw[<-] (0.25,.6) to[out=240,in=90] (-0.25,-0);
      \node at (0.6,.03) {$\scriptstyle{p(x)}$};
	\draw[-] (-.11,0.51) to (-0.18,.43);
	\draw[-] (.11,0.51) to (0.18,.43);
      \node at (-.28,.75) {$\scriptstyle i$};
      \node at (.28,.75) {$\scriptstyle i$};
      \node at (.25,0) {$\dt$};
	\draw[-,darkg,thick] (1,.6) to (1,-.25);
      \node at (1,-0.4) {$\darkg\scriptstyle{L}$};
\end{tikzpicture}
}\\&=
t^{-1}\left[
p(u)
\begin{tikzpicture}[baseline = -1.5mm]
	\draw[<-] (0.3,.4) to[out=-90, in=0] (0.1,0);
	\draw[-] (0.1,0) to[out = 180, in = -90] (-0.1,.4);
      \node at (-0.1,0.55) {$\scriptstyle{i}$};
      \node at (0.85,0.18) {$\scriptstyle{(u-x)^{-1}}$};
      \node at (.3,0.18) {$\dt$};
\node at (.27,-.35){$\scriptstyle\clock(u)$};
	\draw[-,darkg,thick] (1.5,.4) to (1.5,-.55);
      \node at (1.5,-0.7) {$\darkg\scriptstyle{L}$};
\end{tikzpicture}
\right]_{u^{-1}}
=\left[
t_L^{-1} p(u) \h_L(u)^{-1}
\begin{tikzpicture}[baseline = 1mm]
	\draw[<-] (0.3,.4) to[out=-90, in=0] (0.1,0);
	\draw[-] (0.1,0) to[out = 180, in = -90] (-0.1,.4);
      \node at (-0.1,0.55) {$\scriptstyle{i}$};
      \node at (0.83,0.18) {$\scriptstyle{(u-x)^{-1}}$};
      \node at (.3,0.18) {$\dt$};
	\draw[-,darkg,thick] (1.4,.4) to (1.4,-.1);
      \node at (1.4,-0.25) {$\darkg\scriptstyle{L}$};
\end{tikzpicture}\right]_{u^{-1}}\\
&=
\left[
q(u)
\begin{tikzpicture}[baseline = 1mm]
	\draw[<-] (0.3,.4) to[out=-90, in=0] (0.1,0);
	\draw[-] (0.1,0) to[out = 180, in = -90] (-0.1,.4);
      \node at (-0.1,0.55) {$\scriptstyle{i}$};
      \node at (0.83,0.18) {$\scriptstyle{(u-x)^{-1}}$};
      \node at (.3,0.18) {$\dt$};
	\draw[-,darkg,thick] (1.4,.4) to (1.4,-.1);
      \node at (1.4,-0.25) {$\darkg\scriptstyle{L}$};
\end{tikzpicture}\right]_{u^{-1}}
=
\begin{tikzpicture}[baseline = 1mm]
	\draw[<-] (0.3,.4) to[out=-90, in=0] (0.1,0);
	\draw[-] (0.1,0) to[out = 180, in = -90] (-0.1,.4);
      \node at (-0.1,0.55) {$\scriptstyle{i}$};
      \node at (0.63,0.18) {$\scriptstyle{q(x)}$};
      \node at (.3,0.18) {$\dt$};
	\draw[-,darkg,thick] (1.1,.4) to (1.1,-.1);
      \node at (1.1,-0.25) {$\darkg\scriptstyle{L}$};
\end{tikzpicture}
=\vec{\beta}(\mu(p(u))).
\end{align*}
The claim follows from this using the definitions of $\vec{\gamma}$ and $\cev{\gamma}$.

\vspace{2mm}
\noindent
\underline{Claim 3}: {\em We have that $\vec{\chi} \circ \cev{\chi} = 1_{F_i E_i
    L} + \phi$ for some morphism $\phi:F_i E_i L \rightarrow F_i E_i
  L$ whose image is contained in $\vec{\gamma}(L \otimes P_i)$.
Similarly, $\cev{\chi} \circ \vec{\chi} =
  1_{E_i F_i L} + \phi$ for
some morphism $\phi:E_i F_i L \rightarrow E_i F_i L$ whose image is
contained in $\cev{\gamma}(L \otimes Q_i)$.}
We just
explain in the first case. We have that
\begin{align*}
\vec{\chi} \circ \cev{\chi} &=
\begin{tikzpicture}[baseline = -1mm]
\node at (-.23,.75){$\scriptstyle i$};
\node at (.23,.75){$\scriptstyle i$};
\node at (-.23,-.75){$\scriptstyle i$};
\node at (.23,-.75){$\scriptstyle i$};
\node at (-.35,0){$\scriptstyle i$};
\node at (.35,0){$\scriptstyle i$};
	\draw[-,darkg,thick] (0.6,.6) to (0.6,-.6);
\node at (.6,-.75){$\darkg\scriptstyle{L}$};
	\draw[<-] (-0.23,-.6) to[out=90,in=-90] (0.23,0);
	\draw[-] (0.23,0) to[out=90,in=-90] (-0.23,.6);
	\draw[-,line width=3pt, white] (-0.23,0) to[out=90,in=-90] (0.23,.6);
	\draw[->] (-0.23,0) to[out=90,in=-90] (0.23,.6);
	\draw[-,line width=3pt, white] (0.23,-.6) to[out=90,in=-90] (-0.23,0);
	\draw[-] (0.23,-.6) to[out=90,in=-90] (-0.23,0);
\node at (0,.3){$\diamond$};
\node at (0,-.3){$\diamond$};
\end{tikzpicture}
=
\begin{tikzpicture}[baseline = -1mm]
\node at (-.23,.75){$\scriptstyle i$};
\node at (.23,.75){$\scriptstyle i$};
\node at (-.23,-.75){$\scriptstyle i$};
\node at (.23,-.75){$\scriptstyle i$};
	\draw[-,darkg,thick] (0.6,.6) to (0.6,-.6);
\node at (.6,-.75){$\darkg\scriptstyle{L}$};
	\draw[<-] (-0.23,-.6) to[out=90,in=-90] (0.23,0);
	\draw[-] (0.23,0) to[out=90,in=-90] (-0.23,.6);
	\draw[-,line width=4pt,white] (-0.23,0) to[out=90,in=-90] (0.23,.6);
	\draw[->] (-0.23,0) to[out=90,in=-90] (0.23,.6);
	\draw[-,line width=4pt, white] (0.23,-.6) to[out=90,in=-90] (-0.23,0);
	\draw[-] (0.23,-.6) to[out=90,in=-90] (-0.23,0);
	\draw[-] (-.11,0.45) to (-0.18,.37);
	\draw[-] (.11,0.45) to (0.18,.37);
	\draw[-] (-.11,-0.45) to (-0.18,-.37);
	\draw[-] (.11,-0.45) to (0.18,-.37);
\end{tikzpicture}
\!-
\sum_{j \neq i}
\begin{tikzpicture}[baseline = -1mm]
\node at (-.23,.75){$\scriptstyle i$};
\node at (.23,.75){$\scriptstyle i$};
\node at (-.23,-.75){$\scriptstyle i$};
\node at (.23,-.75){$\scriptstyle i$};
\node at (-.35,0){$\scriptstyle j$};
\node at (.35,0){$\scriptstyle j$};
	\draw[-,darkg,thick] (0.6,.6) to (0.6,-.6);
\node at (.6,-.75){$\darkg\scriptstyle{L}$};
	\draw[<-] (-0.23,-.6) to[out=90,in=-90] (0.23,0);
	\draw[-] (0.23,0) to[out=90,in=-90] (-0.23,.6);
	\draw[-,white,line width=3pt] (-0.23,0) to[out=90,in=-90] (0.23,.6);
	\draw[->] (-0.23,0) to[out=90,in=-90] (0.23,.6);
	\draw[-,line width=3pt, white] (0.23,-.6) to[out=90,in=-90] (-0.23,0);
	\draw[-] (0.23,-.6) to[out=90,in=-90] (-0.23,0);
\node at (0,.3){$\diamond$};
\node at (0,-.3){$\diamond$};
\end{tikzpicture}=
\begin{tikzpicture}[baseline = -1mm]
\node at (-.23,.75){$\scriptstyle i$};
\node at (.23,.75){$\scriptstyle i$};
\node at (-.23,-.75){$\scriptstyle i$};
\node at (.23,-.75){$\scriptstyle i$};
	\draw[-,darkg,thick] (0.6,.6) to (0.6,-.6);
\node at (.6,-.75){$\darkg\scriptstyle{L}$};
	\draw[<-] (-0.23,-.6) to[out=90,in=-90] (0.23,0);
	\draw[->] (-0.23,0) to[out=90,in=-90] (0.23,.6);
	\draw[-,line width=4pt,white] (0.23,0) to[out=90,in=-90] (-0.23,.6);
	\draw[-] (0.23,0) to[out=90,in=-90] (-0.23,.6);
	\draw[-,line width=4pt, white] (0.23,-.6) to[out=90,in=-90] (-0.23,0);
	\draw[-] (0.23,-.6) to[out=90,in=-90] (-0.23,0);
	\draw[-] (-.11,0.45) to (-0.18,.37);
	\draw[-] (.11,0.45) to (0.18,.37);
	\draw[-] (-.11,-0.45) to (-0.18,-.37);
	\draw[-] (.11,-0.45) to (0.18,-.37);
\end{tikzpicture}
\!-z\begin{tikzpicture}[baseline = -1mm]
\node at (-.23,.75){$\scriptstyle i$};
\node at (-.23,-.75){$\scriptstyle i$};
\node at (.23,-.75){$\scriptstyle i$};
	\draw[-,darkg,thick] (0.6,.6) to (0.6,-.6);
\node at (.6,-.75){$\darkg\scriptstyle{L}$};
	\draw[<-] (-0.23,-.6) to[out=90,in=-90] (0.23,0);
	\draw[-] (0.23,0) to[out=90,in=0] (0,.2);
	\draw[-] (-0.23,0) to[out=90,in=180] (0,.2);
	\draw[-,line width=4pt, white] (0.23,-.6) to[out=90,in=-90] (-0.23,0);
	\draw[-] (0.23,-.6) to[out=90,in=-90] (-0.23,0);
	\draw[-] (-.11,-0.45) to (-0.18,-.37);
	\draw[-] (.11,-0.45) to (0.18,-.37);
	\draw[<-] (0.2,0.6) to[out=-90, in=0] (0,.3);
	\draw[-] (0,.3) to[out = 180, in = -90] (-0.2,0.6);
\end{tikzpicture}
\!-
\sum_{j \neq i}
\begin{tikzpicture}[baseline = -1mm]
\node at (-.23,.75){$\scriptstyle i$};
\node at (.23,.75){$\scriptstyle i$};
\node at (-.23,-.75){$\scriptstyle i$};
\node at (.23,-.75){$\scriptstyle i$};
\node at (-.35,0){$\scriptstyle j$};
\node at (.35,0){$\scriptstyle j$};
	\draw[-,darkg,thick] (0.6,.6) to (0.6,-.6);
\node at (.6,-.75){$\darkg\scriptstyle{L}$};
	\draw[<-] (-0.23,-.6) to[out=90,in=-90] (0.23,0);
	\draw[-] (0.23,0) to[out=90,in=-90] (-0.23,.6);
	\draw[-,line width=3pt,white] (-0.23,0) to[out=90,in=-90] (0.23,.6);
	\draw[->] (-0.23,0) to[out=90,in=-90] (0.23,.6);
	\draw[-,line width=3pt, white] (0.23,-.6) to[out=90,in=-90] (-0.23,0);
	\draw[-] (0.23,-.6) to[out=90,in=-90] (-0.23,0);
\node at (0,.3){$\diamond$};
\node at (0,-.3){$\diamond$};
\end{tikzpicture}
\\&=
\mathord{
\begin{tikzpicture}[baseline = -.9mm]
	\draw[->] (0.08,-.6) to (0.08,.6);
	\draw[<-] (-0.28,-.6) to (-0.28,.6);
\node at (-.28,.75){$\scriptstyle i$};
\node at (.08,-.75){$\scriptstyle i$};
	\draw[-,darkg,thick] (0.45,.6) to (0.45,-.6);
\node at (.45,-.75){$\darkg\scriptstyle{L}$};
\end{tikzpicture}
}
+tz
\mathord{
\begin{tikzpicture}[baseline=-1mm]
	\draw[<-] (0.2,0.6) to[out=-90, in=0] (0,.1);
	\draw[-] (0,.1) to[out = 180, in = -90] (-0.2,0.6);
	\draw[-] (0.2,-.6) to[out=90, in=0] (0,-0.1);
	\draw[->] (0,-0.1) to[out = 180, in = 90] (-0.2,-.6);
\node at (-0.2,.75){$\scriptstyle i$};
\node at (-0.2,-.75){$\scriptstyle i$};
	\draw[-,darkg,thick] (0.55,.6) to (0.55,-.6);
\node at (.55,-.75){$\darkg\scriptstyle{L}$};
\end{tikzpicture}}
+z^2\sum_{r,s > 0}
\mathord{
\begin{tikzpicture}[baseline = 1mm]
  \draw[<-] (0,0.4) to[out=180,in=90] (-.2,0.2);
  \draw[-] (0.2,0.2) to[out=90,in=0] (0,.4);
 \draw[-] (-.2,0.2) to[out=-90,in=180] (0,0);
  \draw[-] (0,0) to[out=0,in=-90] (0.2,0.2);
   \node at (0,0.2) {$+$};
   \node at (-.57,0.2) {$\scriptstyle{-r-s}$};
\end{tikzpicture}
}
\mathord{
\begin{tikzpicture}[baseline=-1mm]
	\draw[<-] (0.2,0.6) to[out=-90, in=0] (0,.1);
	\draw[-] (0,.1) to[out = 180, in = -90] (-0.2,0.6);
	\draw[-] (0.2,-.6) to[out=90, in=0] (0,-0.1);
	\draw[->] (0,-0.1) to[out = 180, in = 90] (-0.2,-.6);
\node at (-0.2,.75){$\scriptstyle i$};
\node at (-0.2,-.75){$\scriptstyle i$};
      \node at (0.16,-0.2) {$\dt$};
      \node at (0.33,-0.2) {$\scriptstyle{s}$};
   \node at (0.16,0.2) {$\dt$};
   \node at (.33,.2) {$\scriptstyle{r}$};
	\draw[-,darkg,thick] (0.6,.6) to (0.6,-.6);
\node at (.6,-.75){$\darkg\scriptstyle{L}$};
\end{tikzpicture}}
-z\begin{tikzpicture}[baseline = -1mm]
\node at (-.23,.75){$\scriptstyle i$};
\node at (-.23,-.75){$\scriptstyle i$};
\node at (.23,-.75){$\scriptstyle i$};
	\draw[-,darkg,thick] (0.6,.6) to (0.6,-.6);
\node at (.6,-.75){$\darkg\scriptstyle{L}$};
	\draw[<-] (-0.23,-.6) to[out=90,in=-90] (0.23,0);
	\draw[-] (0.23,0) to[out=90,in=0] (0,.2);
	\draw[-] (-0.23,0) to[out=90,in=180] (0,.2);
	\draw[-,line width=4pt, white] (0.23,-.6) to[out=90,in=-90] (-0.23,0);
	\draw[-] (0.23,-.6) to[out=90,in=-90] (-0.23,0);
	\draw[-] (-.11,-0.45) to (-0.18,-.37);
	\draw[-] (.11,-0.45) to (0.18,-.37);
	\draw[<-] (0.2,0.6) to[out=-90, in=0] (0,.3);
	\draw[-] (0,.3) to[out = 180, in = -90] (-0.2,0.6);
\end{tikzpicture}
-
\sum_{j \neq i}
\begin{tikzpicture}[baseline = -1mm]
\node at (-.23,.75){$\scriptstyle i$};
\node at (.23,.75){$\scriptstyle i$};
\node at (-.23,-.75){$\scriptstyle i$};
\node at (.23,-.75){$\scriptstyle i$};
\node at (-.35,0){$\scriptstyle j$};
\node at (.35,0){$\scriptstyle j$};
	\draw[-,darkg,thick] (0.6,.6) to (0.6,-.6);
\node at (.6,-.75){$\darkg\scriptstyle{L}$};
	\draw[<-] (-0.23,-.6) to[out=90,in=-90] (0.23,0);
	\draw[-] (0.23,0) to[out=90,in=-90] (-0.23,.6);
	\draw[-,line width=3pt,white] (-0.23,0) to[out=90,in=-90] (0.23,.6);
	\draw[->] (-0.23,0) to[out=90,in=-90] (0.23,.6);
	\draw[-,line width=3pt, white] (0.23,-.6) to[out=90,in=-90] (-0.23,0);
	\draw[-] (0.23,-.6) to[out=90,in=-90] (-0.23,0);
\node at (0,.3){$\diamond$};
\node at (0,-.3){$\diamond$};
\end{tikzpicture}.
\end{align*}
The second, third and fourth terms on the right-hand side are morphisms whose
image is contained
in $\vec{\gamma}(L \otimes P_i)$.
It just remains to see that the final term consists of such morphisms too. Take $j
\neq i$.
Like in the proof of Lemma~\ref{essential}, we can find a
polynomial $p(u) \in \k[u]$ divisible by $(u-j)^{\phi_j(L)}$ so that
$p(u)\equiv 1 \pmod{(u-i)^{\phi_i(E_i L)}}$.
We have that
$$
\begin{tikzpicture}[baseline = -1mm]
\node at (-.23,.75){$\scriptstyle i$};
\node at (.23,.75){$\scriptstyle i$};
\node at (-.23,-.75){$\scriptstyle i$};
\node at (.23,-.75){$\scriptstyle i$};
\node at (-.35,0){$\scriptstyle j$};
\node at (.35,0){$\scriptstyle j$};
	\draw[-,darkg,thick] (0.6,.6) to (0.6,-.6);
\node at (.6,-.75){$\darkg\scriptstyle{L}$};
	\draw[<-] (-0.23,-.6) to[out=90,in=-90] (0.23,0);
	\draw[-] (0.23,0) to[out=90,in=-90] (-0.23,.6);
	\draw[-,white,line width=3pt] (-0.23,0) to[out=90,in=-90] (0.23,.6);
	\draw[->] (-0.23,0) to[out=90,in=-90] (0.23,.6);
	\draw[-,line width=3pt, white] (0.23,-.6) to[out=90,in=-90] (-0.23,0);
	\draw[-] (0.23,-.6) to[out=90,in=-90] (-0.23,0);
\node at (0,.3){$\diamond$};
\node at (0,-.3){$\diamond$};
\end{tikzpicture}
=
\begin{tikzpicture}[baseline = -1mm]
\node at (-.23,.75){$\scriptstyle i$};
\node at (.23,.75){$\scriptstyle i$};
\node at (-.23,-.75){$\scriptstyle i$};
\node at (.23,-.75){$\scriptstyle i$};
\node at (-.35,0){$\scriptstyle j$};
\node at (.35,0){$\scriptstyle j$};
	\draw[-,darkg,thick] (0.6,.6) to (0.6,-.6);
\node at (.6,-.75){$\darkg\scriptstyle{L}$};
	\draw[<-] (-0.23,-.6) to[out=90,in=-90] (0.23,0);
	\draw[-] (0.23,0) to[out=90,in=-90] (-0.23,.6);
	\draw[-,white,line width=3pt] (-0.23,0) to[out=90,in=-90] (0.23,.6);
	\draw[->] (-0.23,0) to[out=90,in=-90] (0.23,.6);
	\draw[-,line width=3pt, white] (0.23,-.6) to[out=90,in=-90] (-0.23,0);
	\draw[-] (0.23,-.6) to[out=90,in=-90] (-0.23,0);
\node at (0,.3){$\diamond$};
\node at (0,-.3){$\diamond$};
\node at (-.43,.45){$\scriptstyle p(x)\dt$};
\end{tikzpicture}.
$$
Now using the commutation relations (\ref{newdotslide2}),  we commute
$p(x)$ past the crossing to produce a term that is zero as $p(x)$ is
divisible by a sufficiently large power of $(x-j)$, plus correction
terms all of which are morphisms with image lying in $\vec{\gamma}(L \otimes
P_i)$.

\vspace{2mm}
\noindent
\underline{Claim 4}: {\em $\theta$ is an epimorphism.}
Note by the first assertion from Claim 3
that
$F_i E_i L \subseteq \vec{\chi}(E_i F_i L) + \vec{\gamma}(L \otimes P_i)$.
Claim 2 implies that
$\vec{\gamma}(L \otimes \mu(Q_i)) =
\vec{\chi}(\cev{\gamma}(L \otimes Q_i)) \subseteq \vec{\chi}(E_i F_i L)$.
Since $P_i = \mu(Q_i) \oplus C_i$, we deduce that
$F_i E_i L \subseteq \vec{\chi}(E_i F_i L) + \vec{\gamma}(L \otimes C_i)$
as required.

\vspace{2mm}
\noindent
\underline{Claim 5}: {\em $\theta$ is a monomorphism.}
Let $K$ be its kernel.
Of course, $K$ is contained in the kernel of the composition
$\cev{\chi} \circ \theta =
\left[\cev{\chi}\circ \vec{\chi} \quad \cev{\chi}\circ \vec{\gamma}|_{L \otimes C_i}\right].$
Using the second assertion from Claim 3 together with Claim 1,
we deduce that $K \subseteq
\cev{\gamma}(L \otimes Q_i) \:\oplus\: L \otimes C_i$.
Hence, it suffices to show that
$\left[\vec{\chi} \circ \cev{\gamma} \quad\vec{\gamma}|_{L \otimes C_i}\right]:
L \otimes Q_i\:\oplus\:L \otimes C_i \rightarrow F_i E_i L$ is
a monomorphism.
Using
Claim 2 again, this follows
because both
$L \otimes \mu$
and
$\vec{\gamma} = \left[\vec{\gamma}|_{\mu(Q_i)\otimes L} \quad \vec{\gamma}|_{L \otimes C_i}\right]$
are monomorphisms.
\end{proof}

\begin{Lemma}\label{banach3}
Given $\lambda \in X$ and $i \in I$ such that
 $\langle h_i, \lambda \rangle \geq 0$, the natural
transformation
$$
\left[\!\!\!\!\!
\begin{array}{r}
\mathord{
\begin{tikzpicture}[baseline = 0]
	\draw[->] (-0.28,-.3) to (0.28,.3);
	\draw[-,white,line width=3pt] (0.28,-.3) to (-0.28,.3);
	\draw[<-] (0.28,-.3) to (-0.28,.3);
      \node at (-0.33,-0.43) {$\scriptstyle{i}$};
      \node at (0.33,-0.43) {$\scriptstyle{i}$};
      \node at (-0.33,0.43) {$\scriptstyle{i}$};
      \node at (0.33,0.43) {$\scriptstyle{i}$};
\node at (0,-.01) {$\diamond$};
   \end{tikzpicture}
}\!\!\!\\
\mathord{
\begin{tikzpicture}[baseline = 1mm]
	\draw[<-] (0.4,0) to[out=90, in=0] (0.1,0.4);
      \node at (-0.15,0.45) {$\phantom\bullet$};
	\draw[-] (0.1,0.4) to[out = 180, in = 90] (-0.2,0);
      \node at (-0.2,-0.14) {$\scriptstyle{i}$};
\end{tikzpicture}
}\\
\mathord{
\begin{tikzpicture}[baseline = 1mm]
	\draw[<-] (0.4,0) to[out=90, in=0] (0.1,0.4);
	\draw[-] (0.1,0.4) to[out = 180, in = 90] (-0.2,0);
      \node at (-0.15,0.45) {$\phantom\bullet$};
      \node at (-0.15,0.2) {$\dt$};s
\node at (-0.52,.2) {$\scriptstyle x-i$};
      \node at (-0.2,-0.14) {$\scriptstyle{i}$};
\end{tikzpicture}
}\\\vdots\:\:\;\\
\mathord{
\begin{tikzpicture}[baseline = 1mm]
	\draw[<-] (0.4,0) to[out=90, in=0] (0.1,0.4);
	\draw[-] (0.1,0.4) to[out = 180, in = 90] (-0.2,0);
\node at (-0.9,.2) {$\scriptstyle(x-i)^{\langle h_i,\lambda\rangle-1}$};
      \node at (-0.15,0.42) {$\phantom\bullet$};
      \node at (-0.15,0.2) {$\dt$};
      \node at (-0.2,-0.14) {$\scriptstyle{i}$};
\end{tikzpicture}
}
\end{array}
\right]
:E_i F_i|_{\R_\lambda}\Rightarrow
F_i E_i|_{\R_\lambda}
\oplus \operatorname{Id}_{\R_\lambda}^{\oplus \langle h_i, \lambda \rangle}
$$
is an isomorphism. (Again, we have just drawn the crossing in the quantum case.)
\end{Lemma}

\begin{proof}
Let $(\Heis_k)^{\op}$ be the opposite category viewed as a monoidal
category with the same horizontal composition law as in $\Heis_k$.
Let $\Heis_{-k}'$ be the $\k$-linear category $\Heis_{-k}$ in the degenerate case,
or the $\K$-linear category defined in the same way as $\Heis_{-k}$ but with $t$ replaced
by $t^{-1}$ in the quantum case.
By \cite[Lemma 2.1]{Bheis} or \cite[Theorem 3.2]{BSWqheis},
there is a $\k$-linear isomorphism
$\Omega:\Heis_{-k}\stackrel{\sim}{\rightarrow} (\Heis'_k)^{\op}$
defined by reflecting diagrams in a horizontal plane, then multiplying
by $(-1)^{x+y}$ where $x$ is the total number of crossings and $y$ is
the total number of leftward cups and caps in the diagram.
Saying that $\R$ is a module category over $\Heis_k$ is equivalent to
saying that $\R^{\op}$ is a module category over $(\Heis_k)^{\op}$.
Note moreover that $\R^{\op}$ is an Abelian category of the same type
(locally finite Abelian or Schurian) as $\R$ itself due to
\cite[(2.2), (2.10)]{BS}.
Its pull-back through the isomorphism $\Omega$ gives us a $\Heis_{-k}'$-module category
$\Omega^*(\R^{\op})$.
Moreover 
$$
(\Omega^*(\R^\op))_{-\lambda}=(\R_\lambda)^{\op}.
$$ 
This follows from (\ref{chickens}) since $\Omega$ switches $E$ and $F$. 
Now we take 
$\lambda\in X$ with $\langle h_i, \lambda\rangle \geq 0$
and consider
the natural transformation
between endofunctors of $\R_\lambda$
from the statement of the lemma.
This natural transformation can be viewed instead as a natural transformation
$F_i E_i|_{(\R_\lambda)^\op} \oplus
\operatorname{Id}_{(\R_\lambda)^\op}^{\oplus \langle h_i, \lambda \rangle}
\Rightarrow E_i F_i|_{(\R_\lambda)^\op}$
between endofunctors of $(\R_\lambda)^\op$.
This is just the same as the natural transformation 
$E_i F_i|_{(\Omega^*(\R^\op))_{-\lambda}} \oplus
\operatorname{Id}_{(\Omega^*(\R^\op))_{-\lambda}}^{\oplus (-\langle
  h_i,-\lambda\rangle)}
\Rightarrow F_i E_i|_{(\Omega^*(\R^\op))_{-\lambda}}$
from Lemma~\ref{banach2} applied to the
weight $-\lambda$ and the
$\Heis'_{-k}$-module category $\Omega^*(\R^\op)$.
Hence, it is an isomorphism by the previous lemma.
\end{proof}

\subsection{Heisenberg to Kac-Moody}
Now we can prove the main theorem of the section.
Recall that $\R$ is a locally finite Abelian or Schurian
module category over $\Heis_k$. Let $E_i$ and $F_i$ be the
eigenfunctors from $\S$\ref{sendo}, and recall the various
diagrams representing
natural transformations between these functors introduced there. Let $I$ be the spectrum of $\R$
as in $\S$\ref{swd}, and
$\UU(\g)$ be the corresponding Kac-Moody 2-category as in $\S$\ref{songs}.
We also need the weight space decomposition of $\R$ from (\ref{one}).
As in \cite{BK}, there is some freedom in the following theorem as it involves a choice of normalization; for the sake of clarity, we have fixed a particular one.
\begin{Theorem}\label{thm1}
Associated to $\R$,
there is a unique $2$-representation 
$\mathbf{R}:\UU(\g)\rightarrow
\mathfrak{Cat}_\k$
defined on objects by
$\lambda\mapsto \R_\lambda$,
on generating 
1-morphisms by
$E_i 1_\lambda\mapsto E_i|_{\R_\lambda}$
and $F_i 1_\lambda\mapsto F_i|_{\R_\lambda}$, 
and on generating 2-morphisms by
\begin{align*}
\begin{tikzpicture}[baseline = -0.8mm]
	\draw[->,thick] (0.08,-.25) to (0.08,.3);
      \node at (0.08,-0.4) {$\scriptstyle{i}$};
      \node at (0.08,0.02) {$\bull$};
      \node at (0.3,0.02) {$\color{gray}\scriptstyle\lambda$};
\end{tikzpicture}
&\mapsto
\begin{tikzpicture}[baseline = -0.8mm]
	\draw[->] (0.08,-.25) to (0.08,.3);
      \node at (0.08,-0.4) {$\scriptstyle{i}$};
      \node at (0.08,0.02) {$\dt$};
      \node at (0.45,0.02) {$\scriptstyle{x-i}$};
\end{tikzpicture}
,
&
\begin{tikzpicture}[baseline = 1mm]
	\draw[<-,thick] (0.4,0.4) to[out=-90, in=0] (0.1,0);
	\draw[-,thick] (0.1,0) to[out = 180, in = -90] (-0.2,0.4);
      \node at (-0.2,0.55) {$\scriptstyle{i}$};
      \node at (0.6,0.2) {$\color{gray}\scriptstyle\lambda$};
\end{tikzpicture}
&\mapsto
\begin{tikzpicture}[baseline = 1mm]
	\draw[<-] (0.4,0.4) to[out=-90, in=0] (0.1,0);
	\draw[-] (0.1,0) to[out = 180, in = -90] (-0.2,0.4);
      \node at (-0.2,0.55) {$\scriptstyle{i}$};
\end{tikzpicture}\:,
&
\begin{tikzpicture}[baseline = 1mm]
	\draw[<-,thick] (0.4,0) to[out=90, in=0] (0.1,0.4);
	\draw[-,thick] (0.1,0.4) to[out = 180, in = 90] (-0.2,0);
      \node at (-0.2,-0.15) {$\scriptstyle{i}$};
      \node at (0.6,0.2) {$\color{gray}\scriptstyle\lambda$};
\end{tikzpicture}
&\mapsto
\begin{tikzpicture}[baseline = 1mm]
	\draw[<-] (0.4,0) to[out=90, in=0] (0.1,0.4);
	\draw[-] (0.1,0.4) to[out = 180, in = 90] (-0.2,0);
      \node at (-0.2,-0.15) {$\scriptstyle{i}$};
\end{tikzpicture}\:\:,\end{align*}\begin{align*}
\begin{tikzpicture}[baseline = 0]
	\draw[->,thick] (0.28,-.3) to (-0.28,.4);
	\draw[->,thick] (-0.28,-.3) to (0.28,.4);
   \node at (-.3,-.45) {$\scriptstyle{j}$};
   \node at (.3,-.45) {$\scriptstyle{i}$};
   \node at (0.5,0.05) {$\color{gray}\scriptstyle{\lambda}$};
\end{tikzpicture}
&\mapsto
\left\{
\begin{array}{ll}
\begin{tikzpicture}[baseline = 0]
	\draw[->] (0.38,-.4) to (-0.38,.5);
	\draw[->] (-0.38,-.4) to (0.38,.5);
      \node at (0,0.05) {$\diamond$};
	\draw[->,densely dotted] (0.26,-.24) to (-0.19,-.24);
   \node at (-.4,-.55) {$\scriptstyle{i}$};
   \node at (.4,-.55) {$\scriptstyle{i}$};
   \node at (-.4,.65) {$\scriptstyle{i}$};
   \node at (.4,.65) {$\scriptstyle{i}$};
      \node at (-0.26,-0.25) {$\dt$};
      \node at (0.26,-0.25) {$\dt$};
      \node at (1.1,-0.2) {$\scriptstyle{(x_2-x_1+1)^{-1}}$};
\end{tikzpicture}
+\begin{tikzpicture}[baseline = 0mm]
	\draw[->,densely dotted] (0.55,.06) to (0.14,.06);
	\draw[->] (0.58,-.4) to (0.58,.5);
	\draw[->] (0.08,-.4) to (0.08,.5);
      \node at (0.08,-0.55) {$\scriptstyle{i}$};
      \node at (0.58,-0.55) {$\scriptstyle{i}$};
      \node at (0.08,0.05) {$\dt$};
      \node at (0.58,0.05) {$\dt$};
      \node at (1.4,0.1) {$\scriptstyle{(x_2-x_1+1)^{-1}}$};
\end{tikzpicture}
&\text{if $j=i$,}\\
\begin{tikzpicture}[baseline = 0]
	\draw[->] (0.38,-.4) to (-0.38,.5);
	\draw[->] (-0.38,-.4) to (0.38,.5);
	\draw[->,densely dotted] (0.26,-.24) to (-0.19,-.24);
      \node at (0,0.05) {$\diamond$};
   \node at (-.4,-.55) {$\scriptstyle{i+1}$};
   \node at (-.4,.65) {$\scriptstyle{i}$};
   \node at (.4,.65) {$\scriptstyle{i+1}$};
   \node at (.4,-.55) {$\scriptstyle{i}$};
      \node at (-0.26,-0.25) {$\dt$};
      \node at (0.26,-0.25) {$\dt$};
      \node at (.77,-0.2) {$\scriptstyle{x_2-x_1}$};
\end{tikzpicture}
&\text{if $j=i^+$,}\\
-\begin{tikzpicture}[baseline = 0]
	\draw[->] (0.38,-.4) to (-0.38,.5);
      \node at (0,0.05) {$\diamond$};
	\draw[->] (-0.38,-.4) to (0.38,.5);
	\draw[->,densely dotted] (0.26,-.24) to (-0.19,-.24);
   \node at (-.4,-.55) {$\scriptstyle{j}$};
   \node at (.4,-.55) {$\scriptstyle{i}$};
   \node at (.4,.65) {$\scriptstyle{j}$};
   \node at (-.4,.65) {$\scriptstyle{i}$};
      \node at (-0.26,-0.25) {$\dt$};
      \node at (0.26,-0.25) {$\dt$};
      \node at (1.57,-0.2) {$\scriptstyle{(x_2-x_1)(x_2-x_1-1)^{-1}}$};
\end{tikzpicture}
&\text{if $j \neq i,i^+$}
\end{array}\right.
\end{align*}
in the degenerate case,
or
\begin{align*}
\begin{tikzpicture}[baseline = -0.8mm]
	\draw[->,thick] (0.08,-.25) to (0.08,.3);
      \node at (0.08,-0.4) {$\scriptstyle{i}$};
      \node at (0.08,0.02) {$\bull$};
      \node at (0.3,0.02) {$\color{gray}\scriptstyle\lambda$};
\end{tikzpicture}
&\mapsto
\begin{tikzpicture}[baseline = -0.8mm]
	\draw[->] (0.08,-.25) to (0.08,.3);
      \node at (0.08,-0.4) {$\scriptstyle{i}$};
      \node at (0.08,0.02) {$\dt$};
      \node at (0.5,0.02) {$\scriptstyle{\frac{x}{i}-1}$};
\end{tikzpicture}
,
&
\begin{tikzpicture}[baseline = 1mm]
	\draw[<-,thick] (0.4,0.4) to[out=-90, in=0] (0.1,0);
	\draw[-,thick] (0.1,0) to[out = 180, in = -90] (-0.2,0.4);
      \node at (-0.2,0.55) {$\scriptstyle{i}$};
      \node at (0.6,0.2) {$\color{gray}\scriptstyle\lambda$};
\end{tikzpicture}
&\mapsto
\begin{tikzpicture}[baseline = 1mm]
	\draw[<-] (0.4,0.4) to[out=-90, in=0] (0.1,0);
	\draw[-] (0.1,0) to[out = 180, in = -90] (-0.2,0.4);
      \node at (-0.2,0.55) {$\scriptstyle{i}$};
\end{tikzpicture}\:,
&
\begin{tikzpicture}[baseline = 1mm]
	\draw[<-,thick] (0.4,0) to[out=90, in=0] (0.1,0.4);
	\draw[-,thick] (0.1,0.4) to[out = 180, in = 90] (-0.2,0);
      \node at (-0.2,-0.15) {$\scriptstyle{i}$};
      \node at (0.6,0.2) {$\color{gray}\scriptstyle\lambda$};
\end{tikzpicture}
&\mapsto
\begin{tikzpicture}[baseline = 1mm]
	\draw[<-] (0.4,0) to[out=90, in=0] (0.1,0.4);
	\draw[-] (0.1,0.4) to[out = 180, in = 90] (-0.2,0);
      \node at (-0.2,-0.15) {$\scriptstyle{i}$};
\end{tikzpicture}\:\:,\end{align*}\begin{align*}
\begin{tikzpicture}[baseline = 0]
	\draw[->,thick] (0.28,-.3) to (-0.28,.4);
	\draw[->,thick] (-0.28,-.3) to (0.28,.4);
   \node at (-.3,-.45) {$\scriptstyle{j}$};
   \node at (.3,-.45) {$\scriptstyle{i}$};
   \node at (0.5,0.05) {$\color{gray}\scriptstyle{\lambda}$};
\end{tikzpicture}
&\mapsto
\left\{
\begin{array}{ll}
i\begin{tikzpicture}[baseline = 0]
	\draw[->] (0.38,-.4) to (-0.38,.5);
	\draw[-,white,line width=3pt] (-0.38,-.4) to (0.38,.5);
	\draw[->] (-0.38,-.4) to (0.38,.5);
      \node at (0,0.05) {$\diamond$};
	\draw[->,densely dotted] (0.26,-.24) to (-0.19,-.24);
   \node at (-.4,-.55) {$\scriptstyle{i}$};
   \node at (.4,-.55) {$\scriptstyle{i}$};
   \node at (-.4,.65) {$\scriptstyle{i}$};
   \node at (.4,.65) {$\scriptstyle{i}$};
      \node at (-0.26,-0.25) {$\dt$};
      \node at (0.26,-0.25) {$\dt$};
      \node at (1.25,-0.2) {$\scriptstyle{(qx_2-q^{-1}x_1)^{-1}}$};
\end{tikzpicture}
+q^{-1} i\begin{tikzpicture}[baseline = 0mm]
	\draw[->,densely dotted] (0.55,.06) to (0.14,.06);
	\draw[->] (0.58,-.4) to (0.58,.5);
	\draw[->] (0.08,-.4) to (0.08,.5);
      \node at (0.08,-0.55) {$\scriptstyle{i}$};
      \node at (0.58,-0.55) {$\scriptstyle{i}$};
      \node at (0.08,0.05) {$\dt$};
      \node at (0.58,0.05) {$\dt$};
      \node at (1.55,0.1) {$\scriptstyle{(qx_2-q^{-1}x_1)^{-1}}$};
\end{tikzpicture}
&\text{if $j=i$,}\\
q^{-1}i^{-1}\begin{tikzpicture}[baseline = 0]
	\draw[->] (0.38,-.4) to (-0.38,.5);
	\draw[-,white,line width=3pt] (-0.38,-.4) to (0.38,.5);
	\draw[->] (-0.38,-.4) to (0.38,.5);
	\draw[->,densely dotted] (0.26,-.24) to (-0.19,-.24);
      \node at (0,0.05) {$\diamond$};
   \node at (-.4,-.58) {$\scriptstyle{q^2 i}$};
   \node at (-.4,.65) {$\scriptstyle{i}$};
   \node at (.4,.68) {$\scriptstyle{q^2 i}$};
   \node at (.4,-.55) {$\scriptstyle{i}$};
      \node at (-0.26,-0.25) {$\dt$};
      \node at (0.26,-0.25) {$\dt$};
      \node at (.77,-0.2) {$\scriptstyle{x_2-x_1}$};
\end{tikzpicture}
&\text{if $j=i^+$,}\\
-\begin{tikzpicture}[baseline = 0]
	\draw[->] (0.38,-.4) to (-0.38,.5);
	\draw[-,white,line width=3pt] (-0.38,-.4) to (0.38,.5);
	\draw[->] (-0.38,-.4) to (0.38,.5);
      \node at (0,0.05) {$\diamond$};
	\draw[->,densely dotted] (0.26,-.24) to (-0.19,-.24);
   \node at (-.4,-.55) {$\scriptstyle{j}$};
   \node at (.4,-.55) {$\scriptstyle{i}$};
   \node at (.4,.65) {$\scriptstyle{j}$};
   \node at (-.4,.65) {$\scriptstyle{i}$};
      \node at (-0.26,-0.25) {$\dt$};
      \node at (0.26,-0.25) {$\dt$};
      \node at (1.72,-0.2) {$\scriptstyle{(x_2-x_1)(q^{-1}x_2-qx_1)^{-1}}$};
\end{tikzpicture}
&\text{if $j \neq i, i^+$}
\end{array}\right.
\end{align*}
in the quantum case.
\end{Theorem}

\begin{proof}
We need to verify the defining relations (\ref{QHA1})--(\ref{rightadj3}) and (\ref{day1})--(\ref{day3}).

The quiver Hecke algebra relations (\ref{QHA1})--(\ref{QHA3}) follow
from the calculations
performed in \cite{BK}. Note also that our
formulae look different from the ones in \cite{BK} in the quantum
case due to the fact that we are working with a different
normalization for the quadratic relation in the Hecke algebra.
In fact, it is perfectly reasonable to check all of the relations
(\ref{QHA1})--(\ref{QHA3}) from scratch without referring to \cite{BK}
  at all; the diagrammatic formalism
now in place makes this particularly convenient.
To give the flavor of the calculation, we check the quadratic relation
(\ref{QHA2}) in the quantum case.
One first uses (\ref{newdotslide2}) to check that
\begin{align*}
\begin{tikzpicture}[baseline = 0]
	\draw[->,thick] (0.28,-.3) to (-0.28,.4);
	\draw[->,thick] (-0.28,-.3) to (0.28,.4);
   \node at (-.3,.55) {$\scriptstyle{j}$};
   \node at (.3,.55) {$\scriptstyle{i}$};
   \node at (0.5,0.05) {$\color{gray}\scriptstyle{\lambda}$};
\end{tikzpicture}
&\mapsto
\left\{
\begin{array}{ll}
-i\begin{tikzpicture}[baseline = 0]
	\draw[->] (-0.38,-.4) to (0.38,.5);
	\draw[-,white,line width=3pt] (0.38,-.4) to (-0.38,.5);
	\draw[->] (0.38,-.4) to (-0.38,.5);
      \node at (0,0.05) {$\diamond$};
	\draw[->,densely dotted] (0.26,.34) to (-0.18,.34);
   \node at (-.4,-.55) {$\scriptstyle{i}$};
   \node at (.4,-.55) {$\scriptstyle{i}$};
   \node at (-.4,.65) {$\scriptstyle{i}$};
   \node at (.4,.65) {$\scriptstyle{i}$};
      \node at (-0.25,0.33) {$\dt$};
      \node at (0.25,0.33) {$\dt$};
      \node at (1.2,0.3) {$\scriptstyle{(q^{-1}x_2-qx_1)^{-1}}$};
\end{tikzpicture}
+q^{-1}i\begin{tikzpicture}[baseline = 0mm]
	\draw[->,densely dotted] (0.55,.06) to (0.14,.06);
	\draw[->] (0.58,-.4) to (0.58,.5);
	\draw[->] (0.08,-.4) to (0.08,.5);
      \node at (0.08,0.65) {$\scriptstyle{i}$};
      \node at (0.58,0.65) {$\scriptstyle{i}$};
      \node at (0.08,0.05) {$\dt$};
      \node at (0.58,0.05) {$\dt$};
      \node at (1.52,0.1) {$\scriptstyle{(q^{-1}x_2-qx_1)^{-1}}$};
\end{tikzpicture}
&\text{if $j=i$,}\\
-qi^{-1}\begin{tikzpicture}[baseline = 0]
	\draw[->] (-0.38,-.4) to (0.38,.5);
	\draw[-,white,line width=3pt] (0.38,-.4) to (-0.38,.5);
	\draw[->] (0.38,-.4) to (-0.38,.5);
	\draw[->,densely dotted] (0.26,.34) to (-0.18,.34);
      \node at (0,0.05) {$\diamond$};
   \node at (-.4,-.55) {$\scriptstyle{i}$};
   \node at (-.4,.65) {$\scriptstyle{q^{-2} i}$};
   \node at (.4,.65) {$\scriptstyle{i}$};
   \node at (.4,-.55) {$\scriptstyle{q^{-2} i}$};
      \node at (-0.25,0.33) {$\dt$};
      \node at (0.25,0.33) {$\dt$};
      \node at (.75,0.3) {$\scriptstyle{x_2-x_1}$};
\end{tikzpicture}
&\text{if $j=i^-$,}\\
-\begin{tikzpicture}[baseline = 0]
	\draw[->] (-0.38,-.4) to (0.38,.5);
	\draw[-,line width=3pt, white] (0.38,-.4) to (-0.38,.5);
	\draw[->] (0.38,-.4) to (-0.38,.5);
      \node at (0,0.05) {$\diamond$};
	\draw[->,densely dotted] (0.26,.34) to (-0.18,.34);
   \node at (-.4,-.55) {$\scriptstyle{i}$};
   \node at (.4,-.55) {$\scriptstyle{j}$};
   \node at (.4,.65) {$\scriptstyle{i}$};
   \node at (-.4,.65) {$\scriptstyle{j}$};
      \node at (-0.25,0.33) {$\dt$};
      \node at (0.25,0.33) {$\dt$};
      \node at (1.65,0.3) {$\scriptstyle{(x_2-x_1)(qx_2-q^{-1}x_1)^{-1}}$};
\end{tikzpicture}
&\text{if $j \neq i, i^-$.}
\end{array}\right.
\end{align*}
Then, we place the right-hand side of the expression just displayed on top
of the formula for the crossing from the
statement of the theorem, to obtain a natural transformation $\theta$.
This is easily seen to be zero in the case $j=i$ as required. We are left with four cases: $j^-=i\neq j^+$, $j^- \neq i = j^+$,
$j^-=i=j^+$ and $j^- \neq i \neq j^+$. According to (\ref{QHA2}), 
we need to show that $\theta$ equals
$\begin{tikzpicture}[baseline = -1mm]
	\draw[->,densely dotted] (0.45,0.08) to (0.14,0.08);
	\draw[->] (0.48,-.1) to (0.48,.3);
	\draw[->] (0.08,-.1) to (0.08,.3);
      \node at (0.08,-0.25) {$\scriptstyle{j}$};
      \node at (0.48,-0.25) {$\scriptstyle{i}$};
      \node at (0.08,0.07) {$\dt$};
      \node at (0.48,0.07) {$\dt$};
      \node at (.63,0.05) {$\scriptstyle{f}$};
\end{tikzpicture}$
where $f=-q^{-1}i^{-1}(q^{-1}x_2-q
x_1)$,
$qi^{-1}(qx_2-q^{-1}x_1)$, $-i^{-2}(qx_2-q^{-1}x_1)(q^{-1}x_2-qx_1)$
or $1$ in these four cases. Note moreover that
$$
g:=(qx_2 - q^{-1}x_1)(q^{-1}x_2-q x_1) = (x_1-x_2)^2 - z^2 x_1x_2.
$$
Hence, we have that $f=gab$ where
\begin{align*}
a &:= \left\{
\begin{array}{ll}
-qi^{-1}&\text{if $j=i^-$,}\\
-(qx_2-q^{-1}x_1)^{-1}&\text{if $j \neq i^-$,}
\end{array}\right.,
&b &:= \left\{
\begin{array}{ll}
q^{-1}i^{-1}&\text{if $j=i^+$,}\\
-(q^{-1}x_2 - qx_1)^{-1}&\text{if $j\neq i^+$.}
\end{array}\right.
\end{align*}
To complete the analysis, we just have to use
Lemmas~\ref{essential}--\ref{essential2} to see that
$$
\theta = 
\begin{tikzpicture}[baseline = -1mm]
\node at (-.23,.75){$\scriptstyle j$};
\node at (.23,.75){$\scriptstyle i$};
\node at (-.23,-.75){$\scriptstyle j$};
\node at (.23,-.75){$\scriptstyle i$};
\node at (-.35,0){$\scriptstyle i$};
\node at (.35,0){$\scriptstyle j$};
	\draw[->] (-0.23,0) to[out=90,in=-90] (0.23,.6);
	\draw[-,white,line width=3pt] (0.23,0) to[out=90,in=-90] (-0.23,.6);
	\draw[->] (0.23,0) to[out=90,in=-90] (-0.23,.6);
	\draw[-] (0.23,-.6) to[out=90,in=-90] (-0.23,0);
	\draw[-,line width=3pt,white] (-0.23,-.6) to[out=90,in=-90] (0.23,0);
	\draw[-] (-0.23,-.6) to[out=90,in=-90] (0.23,0);
\node at (0,.29){$\diamond$};
\node at (0,-.31){$\diamond$};
	\draw[->,densely dotted] (0.15,.46) to (-0.12,.46);
\node at (.19,.44){$\dt$};
\node at (.8,.44){$\scriptstyle (x_2-x_1)a$};
\node at (-.19,.44){$\dt$};
	\draw[->,densely dotted] (0.15,-.45) to (-0.12,-.45);
\node at (.19,-.46){$\dt$};
\node at (.8,-.46){$\scriptstyle (x_2-x_1)b$};
\node at (-.19,-.46){$\dt$};
\end{tikzpicture}
=
\begin{tikzpicture}[baseline = -1mm]
\node at (-.23,.75){$\scriptstyle j$};
\node at (.23,.75){$\scriptstyle i$};
\node at (-.23,-.75){$\scriptstyle j$};
\node at (.23,-.75){$\scriptstyle i$};
	\draw[->] (-0.23,0) to[out=90,in=-90] (0.23,.6);
	\draw[-,line width=3pt,white] (0.23,0) to[out=90,in=-90] (-0.23,.6);
	\draw[->] (0.23,0) to[out=90,in=-90] (-0.23,.6);
	\draw[-] (0.23,-.6) to[out=90,in=-90] (-0.23,0);
	\draw[-,line width=3pt,white] (-0.23,-.6) to[out=90,in=-90] (0.23,0);
	\draw[-] (-0.23,-.6) to[out=90,in=-90] (0.23,0);
	\draw[-] (-.05,0.4) to (-0.12,.31);
	\draw[-] (.05,0.4) to (0.12,.31);
	\draw[-] (-.05,-0.4) to (-0.12,-.31);
	\draw[-] (.05,-0.4) to (0.12,-.31);
	\draw[->,densely dotted] (0.15,.46) to (-0.12,.46);
\node at (.19,.44){$\dt$};
\node at (.8,.44){$\scriptstyle (x_2-x_1)a$};
\node at (-.19,.44){$\dt$};
	\draw[->,densely dotted] (0.15,-.45) to (-0.12,-.45);
\node at (.19,-.46){$\dt$};
\node at (.8,-.46){$\scriptstyle (x_2-x_1)b$};
\node at (-.19,-.46){$\dt$};
\end{tikzpicture}
-\begin{tikzpicture}[baseline = -1mm]
\node at (-.23,.75){$\scriptstyle j$};
\node at (.23,.75){$\scriptstyle i$};
\node at (-.23,-.75){$\scriptstyle j$};
\node at (.23,-.75){$\scriptstyle i$};
\node at (-.35,0){$\scriptstyle j$};
\node at (.35,0){$\scriptstyle i$};
	\draw[->] (-0.23,0) to[out=90,in=-90] (0.23,.6);
	\draw[-,white,line width=3pt] (0.23,0) to[out=90,in=-90] (-0.23,.6);
	\draw[->] (0.23,0) to[out=90,in=-90] (-0.23,.6);
	\draw[-] (0.23,-.6) to[out=90,in=-90] (-0.23,0);
	\draw[-,line width=3pt,white] (-0.23,-.6) to[out=90,in=-90] (0.23,0);
	\draw[-] (-0.23,-.6) to[out=90,in=-90] (0.23,0);
\node at (0,.29){$\diamond$};
\node at (0,-.31){$\diamond$};
	\draw[->,densely dotted] (0.15,.46) to (-0.12,.46);
\node at (.19,.44){$\dt$};
\node at (.8,.44){$\scriptstyle (x_2-x_1)a$};
\node at (-.19,.44){$\dt$};
	\draw[->,densely dotted] (0.15,-.45) to (-0.12,-.45);
\node at (.19,-.46){$\dt$};
\node at (.8,-.46){$\scriptstyle (x_2-x_1)b$};
\node at (-.19,-.46){$\dt$};
\end{tikzpicture}
=
\begin{tikzpicture}[baseline = -1mm]
\node at (-.23,-.75){$\scriptstyle j$};
\node at (.23,-.75){$\scriptstyle i$};
	\draw[->] (-0.23,-.6) to[out=90,in=-90] (-0.23,.6);
	\draw[->] (0.23,-.6) to[out=90,in=-90] (0.23,.6);
	\draw[->,densely dotted] (0.23,0) to (-0.17,0);
\node at (.23,-.01){$\dt$};
\node at (-.23,-.01){$\dt$};
\node at (1.6,-.01){$\scriptstyle (x_2-x_1)^2ab-z^2x_1x_2ab$};
\end{tikzpicture}
=
\begin{tikzpicture}[baseline = -1mm]
\node at (-.23,-.75){$\scriptstyle j$};
\node at (.23,-.75){$\scriptstyle i$};
	\draw[->] (-0.23,-.6) to[out=90,in=-90] (-0.23,.6);
	\draw[->] (0.23,-.6) to[out=90,in=-90] (0.23,.6);
	\draw[->,densely dotted] (0.23,0) to (-0.17,0);
\node at (.23,-.01){$\dt$};
\node at (-.23,-.01){$\dt$};
\node at (.6,-.01){$\scriptstyle gab$};
\end{tikzpicture},
$$
which is what we wanted as $f=gab$.

The adjunction relation (\ref{rightadj3}) is immediate. 

Finally, we need to check the inversion relation. This depends on
Lemmas~\ref{banach1}--\ref{banach3}.
As usual we just go through the details in the quantum case.
First, we observe using (\ref{newdotslide2}) that
\begin{align}
\begin{tikzpicture}[baseline = 0]
	\draw[<-,thick] (0.28,-.3) to (-0.28,.4);
	\draw[->,thick] (-0.28,-.3) to (0.28,.4);
   \node at (-0.32,-.4) {$\scriptstyle{i}$};
   \node at (-0.32,.5) {$\scriptstyle{j}$};
   \node at (.4,.05) {$\scriptstyle{\color{gray}\lambda}$};
\end{tikzpicture}
&\mapsto
\left\{
\begin{array}{ll}
i\begin{tikzpicture}[baseline = 0]
	\draw[->] (-0.38,-.4) to (0.38,.5);
	\draw[-,white,line width=3pt] (0.38,-.4) to (-0.38,.5);
	\draw[<-] (0.38,-.4) to (-0.38,.5);
      \node at (0,0.05) {$\diamond$};
	\draw[->,densely dotted] (0.23,.3) to (-0.16,.3);
   \node at (-.4,-.55) {$\scriptstyle{i}$};
   \node at (.4,-.55) {$\scriptstyle{i}$};
   \node at (-.4,.65) {$\scriptstyle{i}$};
   \node at (.4,.65) {$\scriptstyle{i}$};
      \node at (-0.23,0.29) {$\dt$};
      \node at (0.23,0.29) {$\dt$};
      \node at (1.2,0.25) {$\scriptstyle{(qx_2-q^{-1} x_1)^{-1}}$};
\end{tikzpicture}
=
-i\begin{tikzpicture}[baseline = 0]
	\draw[<-] (0.38,-.4) to (-0.38,.5);
	\draw[-,white,line width=3pt] (-0.38,-.4) to (0.38,.5);
	\draw[->] (-0.38,-.4) to (0.38,.5);
      \node at (0,0.05) {$\diamond$};
	\draw[->,densely dotted] (0.26,-.24) to (-0.19,-.24);
   \node at (-.4,-.55) {$\scriptstyle{i}$};
   \node at (.4,-.55) {$\scriptstyle{i}$};
   \node at (-.4,.65) {$\scriptstyle{i}$};
   \node at (.4,.65) {$\scriptstyle{i}$};
      \node at (-0.26,-0.25) {$\dt$};
      \node at (0.26,-0.25) {$\dt$};
      \node at (1.25,-0.2) {$\scriptstyle{(q^{-1}x_2-q x_1)^{-1}}$};
\end{tikzpicture}
&\text{if $j=i$,}\\
q^{-1}i^{-1}\begin{tikzpicture}[baseline = 0]
	\draw[->] (-0.38,-.4) to (0.38,.5);
	\draw[,white,line width=3pt] (0.38,-.4) to (-0.38,.5);
	\draw[<-] (0.38,-.4) to (-0.38,.5);
	\draw[->,densely dotted] (0.26,-.24) to (-0.19,-.24);
      \node at (0,0.05) {$\diamond$};
   \node at (-.4,-.55) {$\scriptstyle{i}$};
   \node at (-.4,.68) {$\scriptstyle{q^2 i}$};
   \node at (.4,.65) {$\scriptstyle{i}$};
   \node at (.4,-.58) {$\scriptstyle{q^2 i}$};
      \node at (-0.26,-0.25) {$\dt$};
      \node at (0.26,-0.25) {$\dt$};
      \node at (.77,-0.2) {$\scriptstyle{x_2-x_1}$};
\end{tikzpicture}
&\text{if $j=i^+$,}\\
-\begin{tikzpicture}[baseline = 0]
	\draw[->] (-0.38,-.4) to (0.38,.5);
	\draw[-,white,line width=3pt] (0.38,-.4) to (-0.38,.5);
	\draw[<-] (0.38,-.4) to (-0.38,.5);
      \node at (0,0.05) {$\diamond$};
	\draw[->,densely dotted] (0.26,-.24) to (-0.19,-.24);
   \node at (-.4,-.55) {$\scriptstyle{i}$};
   \node at (.4,-.55) {$\scriptstyle{j}$};
   \node at (.4,.65) {$\scriptstyle{i}$};
   \node at (-.4,.65) {$\scriptstyle{j}$};
      \node at (-0.26,-0.25) {$\dt$};
      \node at (0.26,-0.25) {$\dt$};
      \node at (1.72,-0.2) {$\scriptstyle{(x_2-x_1)(qx_2-q^{-1}x_1)^{-1}}$};
\end{tikzpicture}
&\text{if $j \neq i, i^+$.}
\end{array}
\right.
\end{align}
Using this, we can compute the images of the morphisms in
(\ref{day1})--(\ref{day3}).
The first is equal to the morphism from Lemma~\ref{banach1}
composed on the right by another invertible morphism, hence, it is
invertible.
Similarly, the second is equal to the morphism from
Lemma~\ref{banach2}
 composed on the right by an invertible diagonal matrix.
Finally, the last one is equal to the morphism from
Lemma~\ref{banach3} composed on the left by an invertible diagonal
matrix. This completes the proof.
\end{proof}

\begin{Remark}\label{nocaps}
In the setup of Theorem~\ref{thm1}, the images of the generating 2-morphisms
$\:\begin{tikzpicture}[baseline = 0]
	\draw[-,thick] (0.3,0.2) to[out=-90, in=0] (0.1,-0.1);
	\draw[->,thick] (0.1,-0.1) to[out = 180, in = -90] (-0.1,0.2);
    \node at (0.3,.35) {$\scriptstyle{i}$};
\end{tikzpicture}
\!{\color{gray}\scriptstyle\lambda}\:$ and
$\:\begin{tikzpicture}[baseline = -2]
	\draw[-,thick] (0.3,-0.1) to[out=90, in=0] (0.1,0.2);
	\draw[->,thick] (0.1,0.2) to[out = 180, in = 90] (-0.1,-0.1);
    \node at (0.3,-.25) {$\scriptstyle{i}$};
\end{tikzpicture}
\!{\color{gray}\scriptstyle\lambda}\:$
are uniquely determined by the images of the other
generators thanks to Lemma~\ref{lateaddition}. 
It is not easy to find explicit formulae for these in practice.
Nevertheless, we do 
understand 
how to apply $\mathbf{R}$ to dotted bubbles, although it is
easier for this to go in the other direction; see
(\ref{sewing1})--(\ref{sewing2}) below.
\end{Remark}

\section{Generalized cyclotomic quotients}

The Heisenberg category $\Heis_k$ (resp., the Kac-Moody 2-category
$\UU(\g)$) has some universal cyclic module categories (resp.,
2-representations)
known as generalized cyclotomic quotients (GCQs for short).
In this section, we construct an explicit
isomorphism between Heisenberg and Kac-Moody GCQs, and 
use this to prove a converse to Theorem~\ref{thm1}.

\subsection{Kac-Moody 2-representations}\label{svo}
Let $\UU(\g)$ be the Kac-Moody 2-category as in $\S$\ref{songs}.
Its 2-representation theory has been
developed rather fully in the literature. 
We begin the section by reviewing some of the basic facts established in
\cite{CR, Rou}; see also \cite{BD} which extended some of the results
to the Schurian setting.
Actually, as discussed in the introduction, the history
here is a little convoluted, 
since many of these results were first
established in the setting of (degenerate) affine Hecke algebras.
However, with hindsight, the proofs are most naturally explained in terms of the
representation theory of the nil-Hecke algebra.

Suppose that we are given a locally finite Abelian or 
Schurian 2-representation $(\R_\lambda)_{\lambda \in X}$
of $\UU(\g)$. 
Let
\begin{equation}\label{different}
    \textstyle
    \R :=
    \begin{cases}
        \bigoplus_{\lambda \in X} \R_\lambda & \text{in the locally finite Abelian case}, \\
        \prod_{\lambda \in X} R_\lambda & \text{in the Schurian case},
    \end{cases}
\end{equation}
which is again a locally finite Abelian or Schurian category.
The functors
$E_i|_{\R_\lambda}$
and $F_i|_{\R_\lambda}$
for all $\lambda$
define endofunctors
$E_i$ and $F_i$ of $\R$.
The rightwards and leftwards cups and caps in $\UU(\g)$ define canonical adjunctions 
$(E_i, F_i)$ and $(F_i, E_i)$ for all $i \in I$. In particular, $E_i$ and $F_i$ are
sweet endofunctors of $\R$.
There are also divided power
functors
$E_i^{(r)}, F_i^{(r)}$
such that 
$E_i^r \cong (E_i^{(r)})^{\oplus r!}$ and
$F_i^r \cong (E_i^{(r)})^{\oplus r!}$. These are constructed using
the action of the nil-Hecke algebra on the functors $E_i^r$
and $F_i^r$.
The divided power functors
induce endomorphisms $e_i^{(r)} := [E_i^{(r)}]$ and $f_i^{(r)} := [F_i^{(r)}]$ of
the Grothendieck group 
\begin{equation}\label{wtde}
K_0(\R) = \bigoplus_{\lambda \in  X}K_0(\R_\lambda)
\end{equation}
as defined in $\S$\ref{scars}, making $K_0(\R)$
into an integrable module over the Kostant $\Z$-form for the universal
enveloping algebra of $U(\g)$
with (\ref{wtde}) 
as its weight space decomposition.
This assertion is a consequence of the categorical Serre relations
proved in \cite[Proposition 4.2]{Rou} (see also \cite[Corollary 7]{KL2});
the integrability is \cite[Lemma 3.6]{BD}.

We will assume from now on that the 2-representation $(\R_\lambda)_{\lambda \in X}$ 
is {\em nilpotent}, meaning that the following hold:
\begin{itemize}
\item
For each $\lambda \in X$ and $V \in \R_\lambda$, we have that
$E_i V=F_iV=\mathbf 0$ for all but finitely many $i \in I$.
\item 
The 
endomorphisms
$\mathord{
\begin{tikzpicture}[baseline = -1.5mm]
 	\draw[-,darkg,thick] (1.1,.2) to (1.1,-.2);
 	\draw[<-,thick] (0.68,.2) to (0.68,-.2); 
    \node at (.68,-.35) {$\scriptstyle i$};
     \node at (.68,-.02) {$\bull$};
     \node at (1.3,-.02) {$\darkg\scriptstyle V$};
\node at (.9,0) {$\color{gray}\scriptstyle\lambda$};
\end{tikzpicture}
}:E_i V \rightarrow E_i V$
are nilpotent
for all $i \in I$, $\lambda \in X$ and finitely generated $V \in
\R_\lambda$; equivalently,
all of the 
endomorphisms
$\mathord{
\begin{tikzpicture}[baseline = 0mm]
 	\draw[-,darkg,thick] (1.1,.2) to (1.1,-.2);
 	\draw[->,thick] (0.68,.2) to (0.68,-.2); 
    \node at (.68,.35) {$\scriptstyle i$};
     \node at (.68,.02) {$\bull$};
     \node at (1.3,-.02) {$\darkg\scriptstyle V$};
\node at (.9,0) {$\color{gray}\scriptstyle\lambda$};
\end{tikzpicture}
}$
are nilpotent.
\end{itemize} 
The first property implies that
\begin{align}\label{EF}
E &:= \bigoplus_{i \in I} E_i,
&
F &:= \bigoplus_{i \in I} F_i
\end{align}
are well-defined endofunctors of $\R$.
The canonical
adjunctions $(E_i, F_i)$ and $(F_i, E_i)$ for all $i$ induce adjunctions
$(E, F)$ and $(F, E)$,
hence, these are sweet endofunctors too.
By the second property, it makes sense to define
$\eps_i(V)$ and $\phi_i(V)$
to be the nilpotency degrees of the endomorphisms
$\mathord{
\begin{tikzpicture}[baseline = -2mm]
 	\draw[-,darkg,thick] (1.1,.2) to (1.1,-.2);
 	\draw[<-,thick] (0.68,.2) to (0.68,-.2); 
    \node at (.68,-.35) {$\scriptstyle i$};
     \node at (.68,-.02) {$\bull$};
     \node at (1.3,-.02) {$\darkg\scriptstyle V$};
\node at (.9,0) {$\color{gray}\scriptstyle\lambda$};
\end{tikzpicture}
}$ and 
$\mathord{
\begin{tikzpicture}[baseline = -2mm]
 	\draw[-,darkg,thick] (1.1,.2) to (1.1,-.2);
 	\draw[->,thick] (0.68,.2) to (0.68,-.2); 
    \node at (.68,-.35) {$\scriptstyle i$};
     \node at (.68,.02) {$\bull$};
     \node at (1.3,-.02) {$\darkg\scriptstyle V$};
\node at (.9,0) {$\color{gray}\scriptstyle\lambda$};
\end{tikzpicture}
}$, respectively,
for any finitely generated $V \in \R_\lambda$
and $i \in I$.
In other words, the minimal polynomials of these endomorphisms are
$u^{\eps_i(V)}$ and $u^{\phi_i(V)}$, respectively.

Like in $\S$\ref{swd}, for any $\lambda \in X$ and $V \in \R_\lambda$,
we let $Z_V := Z(\End_\R(V))$. Then, assuming that $V$ is finitely
generated so that all but finitely many bubbles act as zero due to the
assumption of nilpotency, we define 
\begin{align}\label{lll}
\h_{V,i}(u)&:=
\mathord{
\begin{tikzpicture}[baseline = -1mm]
     \node at (0.08,0) {$\scriptstyle\anticlocki(u)$};
 	\draw[-,darkg,thick] (0.68,.2) to (0.68,-.22);
     \node at (0.68,-.37) {$\darkg\scriptstyle{V}$};
     \node at (-0.45,0) {$\color{gray}\scriptstyle{\lambda}$};
\end{tikzpicture}
}
=
\left(\mathord{
\begin{tikzpicture}[baseline = -1mm]
     \node at (0.08,0) {$\scriptstyle\clocki(u)$};
 	\draw[-,darkg,thick] (0.68,.2) to (0.68,-.22);
     \node at (0.68,-.37) {$\darkg\scriptstyle{V}$};
     \node at (-0.45,0) {$\color{gray}\scriptstyle{\lambda}$};
\end{tikzpicture}
}\right)^{-1}\in u^{\langle h_i, \lambda\rangle} + u^{\langle
        h_i,\lambda\rangle-1} Z_V[u^{-1}]
\end{align}
for each $i \in I$.
The following are the Kac-Moody counterparts of Lemmas~\ref{preimpy}--\ref{impy}; see
also \cite[Lemma 3.8]{Wunfurling}.

\begin{Lemma}\label{preimpy2}
Suppose that $i \in I$ and $V$ is a finitely generated object of $\R_\lambda$ for $\lambda \in X$.
\begin{itemize}
\item[(1)]
If $f(u) \in Z_V[u]$ is a monic polynomial 
such that
$\mathord{
\begin{tikzpicture}[baseline = -2mm]
 	\draw[thick,->] (0.08,-.2) to (0.08,.2);
     \node at (0.08,-0.01) {$\bull$};
     \node at (0.08,-0.35) {$\scriptstyle i$};
     \node at (-0.3,-0.01) {$\scriptstyle f(y)$};
 	\draw[-,darkg,thick] (0.52,.2) to (0.52,-.2);
     \node at (0.3,0) {$\color{gray}\scriptstyle{\lambda}$};
     \node at (0.69,0) {$\darkg\scriptstyle{V}$};
\end{tikzpicture}
}=0$, 
then $g(u) := \h_{V,i}(u) f(u)$ is a monic polynomial
in $Z_V[u]$
of degree $\deg f(u)+\langle h_i,\lambda\rangle$ 
such that
$\mathord{
\begin{tikzpicture}[baseline = 0mm]
 	\draw[thick,<-] (0.08,-.2) to (0.08,.2);
     \node at (0.08,0.03) {$\bull$};
     \node at (0.08,0.35) {$\scriptstyle i$};
     \node at (-0.3,0.03) {$\scriptstyle g(y)$};
 	\draw[-,darkg,thick] (0.52,.2) to (0.52,-.2);
     \node at (0.3,0.03) {$\color{gray}\scriptstyle{\lambda}$};
     \node at (0.69,0) {$\darkg\scriptstyle{V}$};
\end{tikzpicture}
}=0$.
\item[(2)]
If $g(u) \in Z_V[u]$ is a monic polynomial such that
$\mathord{
\begin{tikzpicture}[baseline = -1mm]
 	\draw[thick,<-] (0.08,-.2) to (0.08,.2);
     \node at (0.08,0.03) {$\bull$};
     \node at (0.08,0.35) {$\scriptstyle i$};
     \node at (-0.3,0.03) {$\scriptstyle g(y)$};
 	\draw[-,darkg,thick] (0.52,.2) to (0.52,-.2);
     \node at (0.3,0.03) {$\color{gray}\scriptstyle{\lambda}$};
     \node at (0.69,0) {$\darkg\scriptstyle{V}$};
\end{tikzpicture}
}=0$, then $f(u) := \h_{V,i}(u)^{-1} g(u)$ is a monic polynomial
in $Z_V[u]$
of degree $\deg g(u)-\langle h_i,\lambda\rangle$ 
such that
$\mathord{
\begin{tikzpicture}[baseline = -1mm]
 	\draw[thick,->] (0.08,-.2) to (0.08,.2);
     \node at (0.08,0) {$\bull$};
     \node at (0.08,-0.35) {$\scriptstyle i$};
     \node at (-0.3,0) {$\scriptstyle f(y)$};
 	\draw[-,darkg,thick] (0.52,.2) to (0.52,-.2);
     \node at (0.3,0) {$\color{gray}\scriptstyle{\lambda}$};
     \node at (0.69,0) {$\darkg\scriptstyle{V}$};
\end{tikzpicture}
}=0$.
 \end{itemize}
\end{Lemma}

\begin{proof}
Mimic the proof of Lemma~\ref{preimpy} using Lemma~\ref{sure2lemma}.
\end{proof}

\begin{Lemma}\label{impy2}
Let $L \in \R_\lambda$ be an irreducible object. For $i \in I$, 
we have that
\begin{align*}
\h_{L,i}(u) &= u^{\phi_i(L)-\eps_i(L)}.
\end{align*}
In particular, $\phi_i(L)-\eps_i(L)= \langle h_i,\lambda \rangle$.
\end{Lemma}

\begin{proof}
We first apply
Lemma~\ref{preimpy2}(1)
with $f(u) = u^{\eps_i(L)}$ to deduce that $\h_{V,i}(u) u^{\eps_i(L)}$
is a monic polynomial of degree $\eps_i(L)+\langle h_i,\lambda\rangle$
divisible by $u^{\phi_i(L)}$. Hence, $\phi_i(L) \leq \eps_i(L)+\langle
h_i,\lambda\rangle$.
Then we apply Lemma~\ref{preimpy2}(2) with $g(u) = u^{\phi_i(L)}$
to deduce that $\h_{V,i}(u)^{-1} u^{\phi_i(L)}$ is monic of degree
$\phi_i(L)-\langle h_i,\lambda\rangle$ divisible by $u^{\eps_i(L)}$.
Hence, $\eps_i(L) \leq \phi_i(L)-\langle h_i,\lambda\rangle$.
We deduce that both inequalities are equalities and the lemma follows.
\end{proof}

\begin{Corollary}
If $V \in \R_\lambda$ is any finitely generated object, 
all coefficients of $\h_{V,i}(u)$ apart from
the leading one are nilpotent.
\end{Corollary}

\begin{proof}
Lemma~\ref{impy2} shows that the natural transformations defined by 
the bubbles
$\mathord{
\begin{tikzpicture}[baseline = 1.25mm]
  \draw[->,thick] (0.2,0.2) to[out=90,in=0] (0,.4);
  \draw[-,thick] (0,0.4) to[out=180,in=90] (-.2,0.2);
\draw[-,thick] (-.2,0.2) to[out=-90,in=180] (0,0);
  \draw[-,thick] (0,0) to[out=0,in=-90] (0.2,0.2);
   \node at (0.2,0.2) {$\bull$};
   \node at (.4,0.2) {$\scriptstyle{r}$};
   \node at (-0.05,-.15) {$\scriptstyle{i}$};
   \node at (-.38,0.2) {$\color{gray}\scriptstyle{\lambda}$};
\end{tikzpicture}}
$
for $r \geq -\langle h_i,\lambda\rangle$ 
are zero on all irreducible objects $L \in \R_\lambda$.
Hence, they define elements of the Jacobson radical of $\End_\R(V)$.
\end{proof}

The final fundamental result to be mentioned here
reveals some remarkable combinatorics which motivated several
of our earlier notational choices.
Still assuming nilpotency, there is an 
{\em associated crystal} $(\B,\tilde e_i, \tilde f_i, \eps_i,
\phi_i, \operatorname{wt})$ in the general sense of Kashiwara;
more precisely, it is what is called a {\em classical
crystal} in \cite{KKMMNN}. 
This follows by \cite[Proposition 5.20]{CR}
in the locally finite Abelian setting
or \cite[Theorem 4.31]{BD} in the
Schurian case; many of the ideas here go back to the work of Grojnowski \cite{G}.
In more detail, the underlying set $\B$ is the set of isomorphism classes of irreducible objects in
$\R$. The crystal operators
$\tilde e_i, \tilde f_i:\B \rightarrow \B \sqcup
\{\varnothing\}$
are defined on an irreducible object $L \in \R_\lambda$ as follows:
\begin{itemize}
\item
if $E_i L \neq 0$ then $\tilde e_i(L)$ is $\operatorname{hd}(E_i L)
\cong \operatorname{soc}(E_i L)$ (which is irreducible),
else
$E_i L = \varnothing$;
\item
if $F_i L \neq 0$ then $\tilde f_i(L)$ is $\operatorname{hd}(F_i L)
\cong \operatorname{soc}(F_i L)$ (which is irreducible), else
$F_i L = \varnothing$.
\end{itemize}
The weight function $\wt:\B \rightarrow X$ is defined by $\wt(L) :=
\lambda$ for $L \in \R_\lambda$.
The functions $\eps_i, \phi_i:\B \rightarrow \N$ 
take $L \in \B$ to 
the nilpotency degrees of the endomorphisms
$\mathord{
\begin{tikzpicture}[baseline = -2mm]
 	\draw[-,darkg,thick] (1.1,.2) to (1.1,-.2);
 	\draw[<-,thick] (0.68,.2) to (0.68,-.2); 
    \node at (.68,-.35) {$\scriptstyle i$};
     \node at (.68,-.02) {$\bull$};
     \node at (1.3,-.02) {$\darkg\scriptstyle L$};
\node at (.9,0) {$\color{gray}\scriptstyle\lambda$};
\end{tikzpicture}
}$ and 
$\mathord{
\begin{tikzpicture}[baseline = -2mm]
 	\draw[-,darkg,thick] (1.1,.2) to (1.1,-.2);
 	\draw[->,thick] (0.68,.2) to (0.68,-.2); 
    \node at (.68,-.35) {$\scriptstyle i$};
     \node at (.68,.02) {$\bull$};
     \node at (1.3,-.02) {$\darkg\scriptstyle L$};
\node at (.9,0) {$\color{gray}\scriptstyle\lambda$};
\end{tikzpicture}
}$
as above. Part of what it means to say that this is a crystal datum gives that
$\eps_i(L) = \max\{n \in \N\:|\:E_i^{n} L \neq 0\}$ and
$\phi_i(L) = \max\{n \in \N\:|\:F_i^{n} L \neq 0\}$.
Moreover, it is known that the
endomorphism algebras of $E_i L$ and
$F_i L$ are isomorphic to $\k[u] / (u^{\eps_i(L)})$ and $\k[u] / (u^{\phi_i(L)})$,
respectively.

\begin{Remark}
The 2-representations $(\R_\lambda)_{\lambda \in X}$ constructed in 
Theorem~\ref{thm1} are nilpotent. Moreover, the
functors $E_i,
F_i$ and functions $\eps_i, \phi_i$ and $\wt$ as introduced in $\S\S$\ref{sendo}--\ref{swd} are the same
as in the present subsection. Consequently, all of the results
summarized here can be applied to the study of locally finite Abelian
or Schurian 
$\Heis_k$-module categories.
In particular, the description of the endomorphism algebras of $E_i L$ and $F_i L$
just mentioned implies that the homomorphisms (\ref{CRT1})
are actually isomorphisms for irreducible $V$.
\end{Remark}

\subsection{Kac-Moody GCQs}\label{kmgcq}
The next three subsections are concerned with GCQs.
These first appeared on the Kac-Moody side in
\cite[Proposition 5.6]{Wcan}; see also \cite[$\S$4.2]{BD}.
We will only need them under the assumption of nilpotency,
although it can also be useful to consider these categories more
generally; e.g., see \cite{Wunfurling}.
Let $\UU(\g)$ be the Kac-Moody 2-category as in the previous
subsection.
The data needed to define a (nilpotent) GCQ of
$\UU(\g)$ is as follows:
\begin{itemize}
\item a finite-dimensional, commutative, local $\k$-algebra $Z$
with maximal ideal $J$;
\item dominant weights $\mu,\nu \in X^+$;
\item monic polynomials
$\mu_i(u) \in u^{\langle h_i,
    \mu\rangle}+J[u]$, $\nu_i(u) \in u^{\langle h_i,
    \nu\rangle}+J[u]$ for all $i
\in I$.
\end{itemize}
In the important special case
that $Z = \k$, 
the polynomials $\mu_i, \nu_i$ provide no additional data 
beyond that of the dominant weights $\mu,\nu$ since 
we necessarily have that $\mu_i(u) =
    u^{\langle h_i, \mu\rangle}$ and $\nu_i(u) = u^{\langle h_i,
      \nu\rangle}$.
Let $\kappa := \nu - \mu \in X$ and
\begin{align}
\h_i(u)&:=\nu_i(u) / \mu_i(u) \in u^{\langle
             h_i,\kappa\rangle}+u^{\langle h_i,\kappa\rangle-1} J[u^{-1}].
\end{align}
We also need notation for the coefficients of $\h_i(u)$ and its inverse defined
from the expansions
\begin{equation}\label{expansions}
\h_i(u) = \sigma_i(\kappa) \sum_{r \in \Z} \h_i^{(r)} u^{-r-1},
\qquad
\h_i(u)^{-1} = \sigma_i(\kappa)\sum_{r \in \Z} \widetilde\h_i^{(r)} u^{-r-1}.
\end{equation}
Associated to the weight $\kappa$, there is a universal
2-representation
$(\R(\kappa)_\lambda)_{\lambda \in X}$
defined by setting
$\R(\kappa)_\lambda:=
\mathcal{H}om_{\UU(\g)}(\kappa,\lambda)$; the 1- and 2-morphisms in
$\UU(\g)$ act by horizontally composing on the left 
in the obvious way.
Extending scalars, we obtain from this a $Z$-linear 2-representation
$(\R(\kappa)_\lambda \otimes_\k Z)_{\lambda \in X}$.
Let $(\mathcal I_Z(\mu|\nu)_\lambda)_{\lambda \in X}$ be the 
sub-2-representation generated by the 2-morphisms
\begin{equation}
\bigg\{
\mathord{
\begin{tikzpicture}[baseline = -1.2mm]
 	\draw[thick,->] (0.08,-.2) to (0.08,.2);
     \node at (0.08,-0.01) {$\bull$};
     \node at (0.08,-0.35) {$\scriptstyle i$};
     \node at (-0.4,-0.01) {$\scriptstyle \mu_i(y)$};
     \node at (0.3,0) {$\color{gray}\scriptstyle{\kappa}$};
\end{tikzpicture}
},\:
\begin{tikzpicture}[baseline=-.9mm]
\draw[-,thick] (0,-0.18) to[out=180,in=-102] (-.178,0.02);
\draw[-,thick] (-0.18,0) to[out=90,in=180] (0,0.18);
\draw[-,thick] (0.18,0) to[out=-90,in=0] (0,-0.18);
\draw[<-,thick] (0,0.18) to[out=0,in=90] (0.18,0);
   \node at (0,-.32) {$\scriptstyle{i}$};
   \node at (0.18,0) {$\bull$};
   \node at (0.4,0) {$\scriptstyle{y^r}$};
   \node at (-0.4,0) {$\color{gray}\scriptstyle{\kappa}$};
\end{tikzpicture}
-
\h_i^{(r)} 1_{1_\kappa}\:\Big|\:i \in I,
-\langle h_i, \kappa \rangle \leq r < \langle h_i, \mu\rangle
\bigg\}.\label{oneofthe}
\end{equation}
Equivalently, 
by \cite[Lemma 4.14]{BD},
$(\mathcal I_Z(\mu|\nu)_\nu)_{\lambda \in X}$
is generated by the 2-morphisms
\begin{equation}
\bigg\{
\mathord{
\begin{tikzpicture}[baseline = -1.2mm]
 	\draw[thick,<-] (0.08,-.2) to (0.08,.2);
     \node at (0.08,0.02) {$\bull$};
     \node at (0.08,-0.35) {$\scriptstyle i$};
     \node at (-0.4,0.02) {$\scriptstyle \nu_i(y)$};
     \node at (0.3,0) {$\color{gray}\scriptstyle{\kappa}$};
\end{tikzpicture}
},\:
\begin{tikzpicture}[baseline=-.9mm]
\draw[-,thick] (0,-0.18) to[out=180,in=-102] (-.178,0.02);
\draw[->,thick] (-0.18,0) to[out=90,in=180] (0,0.18);
\draw[-,thick] (0.18,0) to[out=-90,in=0] (0,-0.18);
\draw[-,thick] (0,0.18) to[out=0,in=90] (0.18,0);
   \node at (0,-.32) {$\scriptstyle{i}$};
   \node at (-0.18,0) {$\bull$};
   \node at (-0.4,0) {$\scriptstyle{y^r}$};
   \node at (0.4,0) {$\color{gray}\scriptstyle{\kappa}$};
\end{tikzpicture}
-
\widetilde\h_i^{(r)} 1_{1_\kappa}\:\Big|\:i \in I,
\langle h_i, \kappa \rangle \leq r < -\langle h_i, \nu\rangle
\bigg\}.
\end{equation}
The {\em generalized cyclotomic quotient} 
$(\H_Z(\mu|\nu)_\lambda)_{\lambda \in X}$
is the quotient 2-representation. Thus, for $\lambda \in X$, we have that
\begin{equation}
\H_Z(\mu|\nu)_\lambda := \left(\R(\kappa)_\lambda \otimes_\k Z\right) \big/
\mathcal I_Z(\mu|\nu)_\lambda.
\end{equation}
This is the $Z$-linear category with objects that are
1-morphisms $G 1_\kappa:\kappa \rightarrow \lambda$ in $\UU(\g)$,
and morphism space
$\Hom_{\H_Z(\mu|\nu)_\lambda}(G 1_\kappa, G'1_\kappa)$ that
is the quotient of
$\Hom_{\UU(\g)}(G 1_\kappa,G'1_\kappa) \otimes_\k Z$ 
by the $Z$-submodule spanned by all string diagrams from $G 1_\kappa$ to $G'1_\kappa$
which have one of the above generating
2-morphisms appearing on its right-hand boundary.
Note in particular 
by \cite[Lemma 4.14]{BD} again that
\begin{align}
{\scriptstyle\color{gray}\kappa}\:\anticlocki(u) &= \h_i(u) 1_{1_\kappa},&
{\scriptstyle\color{gray}\kappa}\:\clocki(u) &= \h_i(u)^{-1} 1_{1_\kappa}
\end{align}
in $\End_{\H_Z(\mu|\nu)_\kappa}(1_\kappa)$.
It is often convenient to put all of the categories 
$\H_Z(\mu|\nu)_\lambda$ together into a single
$Z$-linear category
\begin{equation}\label{single}
\H_Z(\mu|\nu) 
:= \coprod_{\lambda \in X} \H_Z(\mu|\nu)_\lambda.
\end{equation}
We denote objects in this category simply by words 
in the monoid $\langle E_i,
F_i\rangle_{i \in I}$
generated by the symbols $E_i, F_i\:(i \in I)$, 
such a word $G$ standing for the
1-morphism $G 1_\kappa$.
If $G = G_d\cdots G_1$ with each $G_r \in \{E_i, F_i\:|\:i \in I\}$,
we let
\begin{equation}\label{bluesky}
\wt(G) := \wt(G_1)+\cdots+\wt(G_d)\quad\text{where }
\wt(E_i) = \alpha_i\text{ and }\wt(F_i) = -\alpha_i.
\end{equation}
Then the object $G$ 
belongs to $\H_Z(\mu|\nu)_{\lambda}$ for $\lambda = \kappa+\wt(G)$.

Certain morphism spaces in
$\H_Z(\mu|\nu)$ can be described quite explicitly.
To prepare for this, recall that there is a basis theorem for 2-morphism spaces in
$\UU(\g)$. This was formulated originally as the
{\em nondegeneracy condition} by Khovanov and Lauda in
\cite[$\S$3.2.3]{KL3}. It was proved by them in
finite type A, and it was proved in general in \cite{Wunfurling}; see also \cite{Dupont} for a completely different approach. 

\begin{Lemma}\label{bp}
The quotient of the $Z$-algebra $\End_{\UU(\g)}(1_\kappa) \otimes_\k Z$ 
by the ideal $V$ generated by 
$\Big\{
\begin{tikzpicture}[baseline=-.9mm]
\draw[-,thick] (0,-0.18) to[out=180,in=-102] (-.178,0.02);
\draw[-,thick] (-0.18,0) to[out=90,in=180] (0,0.18);
\draw[-,thick] (0.18,0) to[out=-90,in=0] (0,-0.18);
\draw[<-,thick] (0,0.18) to[out=0,in=90] (0.18,0);
   \node at (0,-.32) {$\scriptstyle{i}$};
   \node at (0.18,0) {$\bull$};
   \node at (0.4,0) {$\scriptstyle{y^r}$};
   \node at (-0.35,0) {$\color{gray}\scriptstyle{\kappa}$};
\end{tikzpicture}
\!\!-
\h_i^{(r)} 1_{1_\kappa},
\begin{tikzpicture}[baseline=-.9mm]
\draw[-,thick] (0,-0.18) to[out=180,in=-102] (-.178,0.02);
\draw[-,thick] (-0.18,0) to[out=90,in=180] (0,0.18);
\draw[-,thick] (0.18,0) to[out=-90,in=0] (0,-0.18);
\draw[<-,thick] (0,0.18) to[out=0,in=90] (0.18,0);
   \node at (0,-.32) {$\scriptstyle{i}$};
   \node at (0.18,0) {$\bull$};
   \node at (0.75,0) {$\scriptstyle{\mu_i(y)y^s}$};
   \node at (-0.35,0) {$\color{gray}\scriptstyle{\kappa}$};
\end{tikzpicture}
\big|\:i \in I,
-\langle h_i, \kappa \rangle \leq r < \langle h_i, \mu\rangle, s \geq 0
\Big\}$
is isomorphic to $Z$.
\end{Lemma}

\begin{proof}
By the nondegeneracy condition, 
$\End_{\UU(\g)}(1_\kappa)$
is a polynomial algebra generated freely by the dotted bubbles
$\begin{tikzpicture}[baseline=-1mm]
\draw[-,thick] (0,-0.18) to[out=180,in=-102] (-.178,0.02);
\draw[-,thick] (-0.18,0) to[out=90,in=180] (0,0.18);
\draw[-,thick] (0.18,0) to[out=-90,in=0] (0,-0.18);
\draw[<-,thick] (0,0.18) to[out=0,in=90] (0.18,0);
   \node at (0,-.32) {$\scriptstyle{i}$};
   \node at (0.18,0) {$\bull$};
   \node at (0.4,0) {$\scriptstyle{y^r}$};
   \node at (-0.4,0) {$\color{gray}\scriptstyle{\kappa}$};
\end{tikzpicture}$
for $i \in I$ and $r \geq - \langle h_i, \kappa\rangle$.
Since $\mu_i(u)$ is monic of degree $\langle h_i, \mu\rangle$, 
factoring out the ideal generated by
$\begin{tikzpicture}[baseline=-1mm]
\draw[-,thick] (0,-0.18) to[out=180,in=-102] (-.178,0.02);
\draw[-,thick] (-0.18,0) to[out=90,in=180] (0,0.18);
\draw[-,thick] (0.18,0) to[out=-90,in=0] (0,-0.18);
\draw[<-,thick] (0,0.18) to[out=0,in=90] (0.18,0);
   \node at (0,-.32) {$\scriptstyle{i}$};
   \node at (0.18,0) {$\bull$};
   \node at (0.75,0) {$\scriptstyle{\mu_i(y)y^s}$};
   \node at (-0.35,0) {$\color{gray}\scriptstyle{\kappa}$};
\end{tikzpicture}$
for $s \geq 0$ reduces to the free polynomial algebra
on generators
$\begin{tikzpicture}[baseline=-1mm]
\draw[-,thick] (0,-0.18) to[out=180,in=-102] (-.178,0.02);
\draw[-,thick] (-0.18,0) to[out=90,in=180] (0,0.18);
\draw[-,thick] (0.18,0) to[out=-90,in=0] (0,-0.18);
\draw[<-,thick] (0,0.18) to[out=0,in=90] (0.18,0);
   \node at (0,-.32) {$\scriptstyle{i}$};
   \node at (0.18,0) {$\bull$};
   \node at (0.4,0) {$\scriptstyle{y^r}$};
   \node at (-0.4,0) {$\color{gray}\scriptstyle{\kappa}$};
\end{tikzpicture}$
for $i \in I$ and $- \langle h_i, \kappa\rangle \leq r < \langle h_i, \mu\rangle$.
Then we tensor over $\k$ with $Z$ and factor out the ideal generated
by the remaining elements
$\begin{tikzpicture}[baseline=-1mm]
\draw[-,thick] (0,-0.18) to[out=180,in=-102] (-.178,0.02);
\draw[-,thick] (-0.18,0) to[out=90,in=180] (0,0.18);
\draw[-,thick] (0.18,0) to[out=-90,in=0] (0,-0.18);
\draw[<-,thick] (0,0.18) to[out=0,in=90] (0.18,0);
   \node at (0,-.32) {$\scriptstyle{i}$};
   \node at (0.18,0) {$\bull$};
   \node at (0.4,0) {$\scriptstyle{y^r}$};
   \node at (-0.35,0) {$\color{gray}\scriptstyle{\kappa}$};
\end{tikzpicture}
\!\!-
\h_i^{(r)} 1_{1_\kappa}$, leaving the algebra $Z$ as the final quotient.
\end{proof}

Let $QH_d$ be the {\em quiver Hecke algebra}.
This is the locally unital $\k$-algebra with local unit provided by the system
$\{1_\bi\:|\:\bi = (i_1,\dots,i_d) \in I^d\}$ of 
mutually orthogonal idempotents,
and  generators
$$
\left\{y_r 1_\bi, \tau_s1_{\bi}\:\big|\: \bi \in I^d, 1 \leq r \leq d, 1
  \leq s < d\right\}.
$$
These generators are subject to the ``local'' relations
represented by
(\ref{QHA1})--(\ref{QHA3}), interpreting
$y_r 1_\bi$ (resp., $\tau_s 1_\bi$) as the string diagram with 
$d$ upwards-oriented 
strings colored $i_1,\dots,i_d$ from
right to left with a dot on the $r$th one (resp., a crossing of the
$s$th and $(s+1)$th ones).
The {\em cyclotomic quiver Hecke algebra}
 $H_d^\mu(Z)$ is the
quotient of the $Z$-algebra $QH_d \otimes_\k Z$ by the two-sided ideal $U$
generated by 
$\left\{\mu_{i_1}(y_1 1_\bi)\:\big|\:\bi \in I^d\right\}$; we interpret $H^\mu_0(Z)$ simply as the
algebra $Z$.
Consider the diagram
\begin{equation}
\begin{diagram}
\node{\!\!\!\!(QH_d \otimes_\k Z)\otimes_Z (\End_{\UU(\g)}(1_\kappa) \otimes_\k Z)}
\arrow{s,l,A}{\pi_1\bar\otimes\pi_2}\arrow{e,t}{\imath_d}\node{
{
\textstyle\bigoplus_{\bi, \bj \in I^d}} \Hom_{\UU(\g)}(E_{\bi}1_\kappa, E_{\bj} 1_\kappa) \otimes_\k Z
}\arrow{s,r,A}{\pi}\\
\node{H_d^\mu(Z)}\arrow{e,b}{\jmath_d}\node{
{\textstyle\bigoplus_{\bi, \bj \in I^d}}
\Hom_{\H_Z(\mu|\nu)}(E_{\bi}, E_{\bj}).}
\end{diagram}\label{buffers}
\end{equation}
The top map $\imath_d$ here is the obvious $Z$-algebra homomorphism
sending $1_\bi \otimes \beta$ to the endomorphism of $E_\bi
1_\kappa := 
E_{i_d}\cdots
E_{i_1}1_\kappa$ induced by $\beta:1_\kappa\Rightarrow 1_\kappa$,
and $y_r 1_\bi
\otimes 1$ and $\tau_s 1_\bi
\otimes 1$ to the 2-morphisms represented by the string diagrams of $y_r 1_\bi$ and
$\tau_s 1_\bi$, respectively. 
The nondegeneracy condition implies that $\imath_d$ is actually an
{isomorphism}.
The right-hand map $\pi$ is the natural quotient map.
The left-hand map $\pi_1\bar\otimes\pi_2$ is 
the product of
the natural quotient map $\pi_1:QH_d \otimes_\k Z \twoheadrightarrow
H_d^\mu(Z)$ with kernel $U$
and the $Z$-algebra homomorphism $\pi_2:\End_{\UU(\g)}(1_\kappa)\otimes_\k Z
\twoheadrightarrow Z$
with kernel $V$ arising from Lemma~\ref{bp}.
The proof of the following lemma is similar to the proof of
\cite[Lemma 8.3]{BCK}.

\begin{Lemma}\label{cqh}
There is a unique isomorphism $\jmath_d$ making the diagram (\ref{buffers})
commute.
\end{Lemma}

\begin{proof}
Let $A := QH_d \otimes_\k Z$ and $B
:= \End_{\UU(\g)}(1_\kappa)\otimes_\k Z$.
As $\imath_d$ is an isomorphism, it suffices to show that 
$\imath_d(\ker \pi_1 \bar\otimes \pi_2) = \ker \pi$.
Note that $\ker \pi_1 \bar\otimes \pi_2 = A \otimes V + U \otimes B$. 
It is obvious that $\pi \circ \imath_d$ sends generators of $A \otimes
V$
and $U \otimes B$ to zero, hence, $\imath_d(\ker \pi_1 \bar\otimes \pi_2)
\subseteq \ker \pi$.
It remains to show that 
$\imath_d^{-1}(\ker \pi) \subseteq \ker \pi_1 \bar\otimes \pi_2$.
By definition
$\ker \pi$ consists of $Z$-linear combinations of
2-morphisms
$\theta:E_\bi 1_\kappa \Rightarrow E_\bj 1_\kappa$
of the form
$$
\theta = 
\mathord{
\begin{tikzpicture}[baseline = -.5mm]
\draw[thick,yshift=-4pt,xshift=-4pt] (.5,0) rectangle ++(8pt,8pt);
\node at (.5,0) {$\scriptstyle \rho$};
\draw[thick,yshift=-4pt,xshift=-4pt] (-.33,0) rectangle ++(16pt,8pt);
\node at (-.2,0.02) {$\scriptstyle \lambda$};
\draw[thick,yshift=-4pt,xshift=-4pt] (-.34,-.5) rectangle ++(32pt,8pt);
\node at (0.1,-.5) {$\scriptstyle \tau$};
\draw[thick,yshift=-4pt,xshift=-4pt] (-.34,.5) rectangle ++(32pt,8pt);
\node at (0.1,.5) {$\scriptstyle \sigma$};
\draw[->,thick] (-.34,0.65) to (-.34,.9);
\draw[->,thick] (.49,0.65) to (.49,.9);
\draw[<-,thick] (-.34,-0.65) to (-.34,-.9);
\draw[<-,thick] (.49,-0.65) to (.49,-.9);
\draw[-,thick] (.49,-0.35) to (.49,-.15);
\draw[-,thick] (.49,0.35) to (.49,.15);
\draw[-,thick] (-.34,-0.35) to (-.34,-.15);
\draw[-,thick] (-.34,0.35) to (-.34,.15);
\draw[-,thick] (-.04,-0.35) to (-.04,-.15);
\draw[-,thick] (-.04,0.35) to (-.04,.15);
\node at (0.52,-1.05) {$\scriptstyle i_1$};
\node at (0.1,-1.05) {$\scriptstyle \cdots$};
\node at (-0.33,-1.05) {$\scriptstyle i_d$};
\node at (0.52,1.05) {$\scriptstyle j_1$};
\node at (0.1,1.05) {$\scriptstyle \cdots$};
\node at (-0.33,1.05) {$\scriptstyle j_d$};
\node at (1,0) {$\scriptstyle \color{gray}\kappa$};
\end{tikzpicture}
}
$$
where $\rho$ is one of the generating 2-morphisms (\ref{oneofthe})
for $\mathcal I_Z(\mu|\nu)$ and $\sigma,\tau,\lambda$ are any
other 2-morphisms in $\UU(\g)$ 
so that the compositions make sense.
We must show for such $\theta$ 
that $\imath_d^{-1}(\theta) \in A \otimes V + U\otimes B$.
If $\rho=
\begin{tikzpicture}[baseline=-.9mm]
\draw[-,thick] (0,-0.18) to[out=180,in=-102] (-.178,0.02);
\draw[-,thick] (-0.18,0) to[out=90,in=180] (0,0.18);
\draw[-,thick] (0.18,0) to[out=-90,in=0] (0,-0.18);
\draw[<-,thick] (0,0.18) to[out=0,in=90] (0.18,0);
   \node at (0,-.32) {$\scriptstyle{i}$};
   \node at (0.18,0) {$\bull$};
   \node at (0.4,0) {$\scriptstyle{y^r}$};
   \node at (-0.4,0) {$\color{gray}\scriptstyle{\kappa}$};
\end{tikzpicture}
-
\h_i^{(r)} 1_{1_\kappa}$ for $-\langle h_i,\kappa\rangle \leq r <
\langle h_i, \mu\rangle$, the inverse image
$\imath_d^{-1}(\theta)$ obviously lies in $A \otimes V$.
Assume instead that $\rho = \mathord{
\begin{tikzpicture}[baseline = -1.2mm]
 	\draw[thick,->] (0.08,-.2) to (0.08,.2);
     \node at (0.08,-0.01) {$\bull$};
     \node at (0.08,-0.35) {$\scriptstyle i$};
     \node at (-0.4,-0.01) {$\scriptstyle \mu_i(y)$};
     \node at (0.3,0) {$\color{gray}\scriptstyle{\kappa}$};
\end{tikzpicture}
}$.
To compute $\imath_d^{-1}(\theta)$, we first ``straighten'' the
diagram $\theta$. Thus, proceeding by induction on the number of
crossings, we use the relations in $\UU(\g)$ to 
slide dotted bubbles to the right-hand edge and to eliminate all other cups
or caps from the diagram, always keeping the generator $\rho$ fixed
on
the right boundary.
This process reduces $\theta$ to a 
$Z$-linear combination of morphisms of the following two types:
\begin{itemize}
\item[(I)]
$
\mathord{
\begin{tikzpicture}[baseline = -.5mm]
\draw[thick,yshift=-4pt,xshift=-4pt] (.5,0) rectangle ++(8pt,8pt);
\node at (.5,0) {$\scriptstyle \rho$};
\draw[thick,yshift=-4pt,xshift=-4pt] (-.34,-.5) rectangle ++(32pt,8pt);
\node at (0.1,-.5) {$\scriptstyle \tau'$};
\draw[thick,yshift=-4pt,xshift=-4pt] (-.34,.5) rectangle ++(32pt,8pt);
\node at (0.1,.5) {$\scriptstyle \sigma'$};
\draw[->,thick] (-.34,0.65) to (-.34,.9);
\draw[->,thick] (.49,0.65) to (.49,.9);
\draw[<-,thick] (-.34,-0.65) to (-.34,-.9);
\draw[<-,thick] (.49,-0.65) to (.49,-.9);
\draw[->,thick] (.49,-0.35) to (.49,-.15);
\draw[<-,thick] (.49,0.35) to (.49,.15);
\draw[->,thick] (-.34,-0.35) to (-.34,.35);
\draw[<-,thick] (.14,0.35) to (.14,-.35);
\node at (-0.1,0) {$\scriptstyle \cdots$};
\node at (0.52,-1.05) {$\scriptstyle i_1$};
\node at (0.1,-1.05) {$\scriptstyle \cdots$};
\node at (-0.33,-1.05) {$\scriptstyle i_d$};
\node at (0.52,1.05) {$\scriptstyle j_1$};
\node at (0.1,1.05) {$\scriptstyle \cdots$};
\node at (-0.33,1.05) {$\scriptstyle j_d$};
\node at (1,-.5) {$\scriptstyle \color{gray}\kappa$};
\draw[thick,yshift=-4pt,xshift=-4pt] (1,-.05) rectangle ++(8pt,10pt);
\node at (1,0) {$\delta$};
\end{tikzpicture}
}$
for $\sigma',\tau'\in \imath_d(A \otimes 1)$
and $\delta \in \imath_d(1 \otimes B)$;
\item[(II)]  $\mathord{
\begin{tikzpicture}[baseline = -.5mm]
\draw[thick,yshift=-4pt,xshift=-4pt] (-.34,0) rectangle ++(32pt,8pt);
\node at (0.1,0) {$\scriptstyle \lambda'$};
\draw[->,thick] (-.34,0.15) to (-.34,.4);
\draw[->,thick] (.49,0.15) to (.49,.4);
\draw[<-,thick] (-.34,-0.15) to (-.34,-.4);
\draw[<-,thick] (.49,-0.15) to (.49,-.4);
\node at (0.52,-.55) {$\scriptstyle i_1$};
\node at (0.1,-.55) {$\scriptstyle \cdots$};
\node at (-0.33,-.55) {$\scriptstyle i_d$};
\node at (0.52,.55) {$\scriptstyle j_1$};
\node at (0.1,.55) {$\scriptstyle \cdots$};
\node at (-0.33,.55) {$\scriptstyle j_d$};
\node at (1,-.4) {$\scriptstyle \color{gray}\kappa$};
\draw[thick,yshift=-4pt,xshift=-4pt] (1,-.05) rectangle ++(8pt,10pt);
\node at (1,0) {$\delta$};
\end{tikzpicture}
}\:
\mathord{\begin{tikzpicture}[baseline = -0.2mm]
  \draw[->,thick] (0,0.3) to[out=180,in=90] (-.3,0);
  \draw[-,thick] (0.3,0) to[out=90,in=0] (0,.3);
 \draw[-,thick] (-.3,0) to[out=-90,in=180] (0,-0.3);
  \draw[-,thick] (0,-0.3) to[out=0,in=-90] (0.3,0);
\filldraw[white,thick,yshift=-4pt,xshift=-4pt] (0.28,-0.05) rectangle ++(7.8pt,7.6pt);
\draw[thick,yshift=-4pt,xshift=-4pt] (0.28,-0.06) rectangle ++(8pt,8pt);
\node at (0.28,-0.05) {$\scriptstyle\rho$};
   \node at (0.35,.4) {$\scriptstyle{y^s}$};
   \node at (0.05,-.42) {$\scriptstyle{i}$};
      \node at (0.16,.25) {$\bull$};
\end{tikzpicture}
}$
for $\lambda'\in \imath_d(A \otimes 1), \delta \in
\imath_d(1 \otimes B)$ and $s \geq 0$.
\end{itemize}
These morphisms 
arise when 
$\rho$ ends up on a propagating strand 
(type I)
or on a dotted bubble (type II)
after straightening.
It remains to observe that the image under $\imath_d^{-1}$
of a type I morphism 
lies in $U \otimes B$, and
the image of 
a
type II morphism lies in $A \otimes V$.
\end{proof}

\begin{Corollary}\label{cqhd}
For $d \geq 0$, we have that
$
\dim \left(\bigoplus_{\bi,\bj \in I^d}
\Hom_{\H_Z(\mu|\nu)}(E_\bi, E_{\bj})\right) = \ell^d d! \dim Z
$
where $\ell := \sum_{i \in I} \langle h_i, \mu\rangle$.
\end{Corollary}

\begin{proof}
It is well known that $\dim H^\mu_d(Z) = \ell^d d! \dim Z$.
For example, this follows by a Shapovalov form calculation 
given the categorification theorem of \cite{KK}.
\end{proof}

\begin{Corollary}\label{maybe}
$\H_Z(\mu|\nu)$ is a finite-dimensional category, i.e., all of its
morphism spaces are finite-dimensional vector spaces over $\k$.
\end{Corollary}

\begin{proof}
Using the biadjunction of $E_i$ and $F_i$, it suffices to show that
$\Hom_{\H_Z(\mu|\nu)}(\varnothing, G)$ is finite-dimensional for any
word $G$. This space is clearly zero unless $G$ has an equal number of
$E$-type letters as $F$-type letters.
Also Corollary~\ref{cqhd} establishes the result if $G=F_{j_1}\cdots
F_{j_d} E_{i_d}\cdots E_{i_1}$, i.e., all the $F$-type letters are to
the left of the $E$-type letters.
The general case then follows by induction on the length, 
using the isomorphisms (\ref{day1})--(\ref{day3}) to establish the induction step.
\end{proof}

\begin{Remark}
In the remainder of the article, we really only need to appeal to
fact that $\ell^d d! \dim Z$ is an {\em upper bound} for the dimension 
in Corollary~\ref{cqhd}.
This follows from the existence of a surjective homomorphism
$\jmath_d$ as in (\ref{buffers}), which follows as above from the
surjectivity of the homomorphism $\imath_d$. The latter assertion is easily proved 
without needing to appeal to the nondegeneracy
condition.
\end{Remark}

It is often useful to work in the larger category
$\mathcal{H}om_\k(\H_Z(\mu|\nu)^{\op}, \Vec)$
of $\k$-linear functors and natural transformations. 
This can be thought of in elementary algebraic terms by
replacing $\H_Z(\mu|\nu)$ with
the locally unital algebra
\begin{align}\label{stupify}
H_Z(\mu|\nu) &:= \bigoplus_{G, G' \in \H_Z(\mu|\nu)}
\Hom_{\H_Z(\mu|\nu)}(G,G').
\end{align}
Multiplication in $H_Z(\mu|\nu)$ is induced by composition in the
category
$\H_Z(\mu|\nu)$, and its local unit $\left\{1_G\:\big|\:G \in 
\langle E_i, F_i\rangle_{i \in I}\right\}$ arises from the identity morphisms of the objects of $\H_Z(\mu|\nu)$.
Then $\mathcal{H}om_\k(H_\Z(\mu|\nu)^{\op}, \Vec)$ is isomorphic to
the category $\lfdrmod H_Z(\mu|\nu)$ of locally finite-dimensional
right modules over this algebra. 
In view of Corollary~\ref{maybe}, $H_Z(\mu|\nu)$ is locally
finite-dimensional, hence, $\lfdrmod H_Z(\mu|\nu)$
is a Schurian category.
Similarly to (\ref{stupify}), we define $H_Z(\mu|\nu)_\lambda$ 
from the category $\H_Z(\mu|\nu)_\lambda$
for each $\lambda \in X$; then (\ref{single})
translates into the algebra
decomposition
\begin{equation}\label{blockde}
H_Z(\mu|\nu) = \bigoplus_{\lambda \in X} H_Z(\mu|\nu)_\lambda.
\end{equation}
The categorical action of $\UU(\g)$ on
$\left(\H_Z(\mu|\nu)\right)_{\lambda \in X}$
extends to make $\left(\lfdrmod
  H_Z(\mu|\nu)_\lambda\right)_{\lambda \in X}$ 
into a Schurian 2-representation.
One way to see this is explained in \cite[Construction 4.26]{BD},
where the extensions of the categorification functors $E_i$ and $F_i$ to
arbitrary $H_Z(\mu|\nu)$-modules are realized
by tensoring with certain bimodules, and the generating 2-morphisms of
$\UU(\g)$ act via explicit bimodule homomorphisms.

Let $P:= 1_\varnothing H_Z(\mu|\nu)$ be the finitely
generated projective $H_Z(\mu|\nu)$-module
associated to the empty word. By Lemma~\ref{cqh}
with $d=0$, we have that $\End_{H_Z(\mu|\nu)}(P) \cong
Z$. As $Z$ is local, it follows that $P$ is a
projective indecomposable module.
Then for any word $G \in \langle E_i, F_i\rangle_{i \in I}$
the module $G P$ obtained by applying the functor $G$
is identified with the right ideal $1_G H_Z(\mu|\nu)$.
These modules for all $G$ give a projective generating 
family for the Schurian category $\lfdrmod H_Z(\mu|\nu)$ such that
\begin{equation}
H_Z(\mu|\nu) = 
\bigoplus_{G, G' \in \langle E_i, F_i\rangle_{i \in I}}
1_{G'} H_Z(\mu|\nu) 1_{G}
\cong\bigoplus_{G, G' \in \langle E_i, F_i\rangle_{i \in I}}
\Hom_{H_Z(\mu|\nu)}(G P, G' P).
\end{equation}

\begin{Remark}\label{usual}
In the special case that $\nu = 0$, the GCQ
$H_Z(\mu|\nu)$ is Morita equivalent to the usual cyclotomic
quotient, that is, the locally unital algebra $\bigoplus_{d \geq 0} H^\mu_d(Z)$; 
see \cite[Theorem 4.25]{R2}.
Recall by \cite{KK} that finitely generated projective modules over this algebra
gives a categorification of the Weyl $\Z$-form of the integrable lowest
weight module $V(-\mu)$ of $\UU(\g)$. 
In general, finitely generated projective $H_Z(\mu|\nu)$-modules can
be used to categorify the tensor product $V(\mu|\nu) :=
V(-\mu)\otimes V(\nu)$
of the integrable lowest weight module $V(-\mu)$ and the
integrable highest weight module $V(\nu)$; see \cite{Wcan}.
This result is not needed below.
\end{Remark}

\subsection{Heisenberg GCQs}\label{massive}
On the Heisenberg side, GCQs
have been defined in
the degenerate case in the introduction of \cite{Bheis}, and in the
quantum case in \cite[$\S$9]{BSWqheis}.
As usual, we will discuss both cases simultaneously according to the
value of $z \in \k$.
The required data is as follows:
\begin{itemize}
\item
a finite-dimensional, commutative, local $\k$-algebra $Z$ with maximal
ideal $J$;
\item
monic polynomials $m(u), n(u) \in Z[u]$,
assuming in addition in the quantum case that $m(0), n(0) \in Z^\times$.
\end{itemize}
 Let $k :=\deg n(u)-\deg m(u)$ and
\begin{align}\label{hat}
\h(u) &:= n(u) / m(u) \in u^k + u^{k-1}Z[\![u^{-1}]\!].
\end{align}
To this data, we are going to associate a left tensor ideal
$\mathcal I_Z(m|n)$ of the strict 
$Z$-linear monoidal category $\Heis_k \otimes_\k Z$.
The precise definition of $\mathcal I_Z(m|n)$ is slightly different in the degenerate
and quantum cases; it will be explained in the next two paragraphs.
Then the {\em generalized cyclotomic quotient} is the quotient category
\begin{equation}\label{hab}
\H_Z(m|n) := (\Heis_k \otimes_\k Z) / \mathcal I_Z(m|n),
\end{equation}
which is itself naturally a $Z$-linear $\Heis_k$-module category.
This quotient category has objects that are words in the monoid $\langle E, F\rangle$, 
and for two such words $G, G'$ the morphism space
$\Hom_{\H_Z(m|n)}(G, G')$ is the quotient of
$\Hom_{\Heis_k}(G,G')\otimes_\k Z$ by the $Z$-submodule defined by the
ideal $\mathcal I_Z(m|n)$.
In both cases, we will have that
\begin{align}\label{bothcases}
\anticlock{\scriptstyle(u)} &= \h(u)1_\unit,
&\clock{\scriptstyle(u)} &= \h(u)^{-1}1_\unit
\end{align}
in $\End_{\H_Z(m|n)}(\unit)$.

Here is the definition of the left tensor ideal $\mathcal I_Z(m|n)$ in
the degenerate case. Define $\h^{(r)}, \widetilde\h^{(r)} \in Z$ from the coefficients of the
formal Laurent series $\h(u), \h(u)^{-1}$ so that
\begin{align}
\h(u) &= \sum_{r \in \Z} \h^{(r)} u^{-r-1},
&
\h(u)^{-1} &= -\sum_{r \in \Z} \widetilde\h^{(r)} u^{-r-1},
\end{align}
this notation being consistent with (\ref{igq})--(\ref{igp}).
Then $\mathcal I_Z(m|n)$ is generated by
\begin{equation}
\Big\{
\mathord{
\begin{tikzpicture}[baseline = -1.2mm]
 	\draw[->] (0.08,-.25) to (0.08,.2);
     \node at (0.08,-0.01) {$\dt$};
     \node at (-0.35,-0.01) {$\scriptstyle m(x)$};
\end{tikzpicture}
},\:
\begin{tikzpicture}[baseline=-.9mm]
\draw[-] (0,-0.18) to[out=180,in=-102] (-.178,0.02);
\draw[-] (-0.18,0) to[out=90,in=180] (0,0.18);
\draw[-] (0.18,0) to[out=-90,in=0] (0,-0.18);
\draw[<-] (0,0.18) to[out=0,in=90] (0.18,0);
   \node at (0.18,0) {$\dt$};
   \node at (0.4,0) {$\scriptstyle{x^r}$};
\end{tikzpicture}
-
\h^{(r)} 1_{\unit}\:\big|\:
-k \leq r < \deg m(u)
\Big\}.\label{anotheroneofthe}
\end{equation}
Equivalently, 
by \cite[Lemma 1.8]{Bheis},
it is generated by
\begin{equation}
\Big\{
\mathord{
\begin{tikzpicture}[baseline = -.9mm]
 	\draw[<-] (0.08,-.2) to (0.08,.25);
     \node at (0.08,0.02) {$\dt$};
     \node at (-0.35,0.02) {$\scriptstyle n(x)$};
\end{tikzpicture}
},\:
\begin{tikzpicture}[baseline=-1mm]
\draw[-] (0,-0.18) to[out=180,in=-102] (-.178,0.02);
\draw[->] (-0.18,0) to[out=90,in=180] (0,0.18);
\draw[-] (0.18,0) to[out=-90,in=0] (0,-0.18);
\draw[-] (0,0.18) to[out=0,in=90] (0.18,0);
   \node at (-0.18,0) {$\dt$};
   \node at (-0.4,0) {$\scriptstyle{x^r}$};
\end{tikzpicture}
-
\widetilde\h^{(r)} 1_{\unit}\:\bigg|\:
k \leq r < \deg n(u)
\Big\}.
\end{equation}
The same lemma implies that (\ref{bothcases}) holds.

\begin{Lemma}\label{bp2}
In the degenerate case, the quotient of the $Z$-algebra $\End_{\Heis_k}(\unit) \otimes_\k Z$ 
by the ideal $V$ generated by 
$\Big\{
\begin{tikzpicture}[baseline=-.9mm]
\draw[-] (0,-0.18) to[out=180,in=-102] (-.178,0.02);
\draw[-] (-0.18,0) to[out=90,in=180] (0,0.18);
\draw[-] (0.18,0) to[out=-90,in=0] (0,-0.18);
\draw[<-] (0,0.18) to[out=0,in=90] (0.18,0);
   \node at (0.18,0) {$\dt$};
   \node at (0.4,0) {$\scriptstyle{x^r}$};
\end{tikzpicture}
\!-
\h^{(r)} 1_\unit,
\:
\begin{tikzpicture}[baseline=-.9mm]
\draw[-] (0,-0.18) to[out=180,in=-102] (-.178,0.02);
\draw[-] (-0.18,0) to[out=90,in=180] (0,0.18);
\draw[-] (0.18,0) to[out=-90,in=0] (0,-0.18);
\draw[<-] (0,0.18) to[out=0,in=90] (0.18,0);
   \node at (0.18,0) {$\dt$};
   \node at (0.7,0) {$\scriptstyle{m(x)x^s}$};
\end{tikzpicture}
\:\big|\:
-k \leq r < \deg m(u), s \geq 0
\Big\}
$
is isomorphic to $Z$.
\end{Lemma}

\begin{proof}
The basis theorem proved in \cite[Theorem 6.4]{BSWk0} implies that
$\End_{\Heis_k}(\unit)$
is a polynomial algebra generated freely by
$\begin{tikzpicture}[baseline=-1mm]
\draw[-] (0,-0.18) to[out=180,in=-102] (-.178,0.02);
\draw[-] (-0.18,0) to[out=90,in=180] (0,0.18);
\draw[-] (0.18,0) to[out=-90,in=0] (0,-0.18);
\draw[<-] (0,0.18) to[out=0,in=90] (0.18,0);
   \node at (0.18,0) {$\dt$};
   \node at (0.4,0) {$\scriptstyle{x^r}$};
\end{tikzpicture}$
for $r \geq -k$.
Given this, the lemma follows similarly to Lemma~\ref{bp}.
\end{proof}

In order to define the left tensor ideal $\mathcal I_Z(m|n)$ in
the quantum case, there is a minor additional complication involving some
choices of square roots: we assume henceforth that we are given distinguished
square roots $\sqrt{c}$ of each $c \in \k^\times$ such that $\sqrt{1/c} = 1/\sqrt{c}$.
The need for this is an artifact of the
choice of normalization
of the quantum Heisenberg category; see Remark~\ref{normalization}.
Given these square roots, we get also distinguished square roots
$\sqrt{c}$ of
all $c \in Z^\times$ lifting the chosen square root of the image of $c$
in $\k = Z / J$.
Then we define $\h^{(r)}, \widetilde\h^{(r)} \in Z$ so that
\begin{align}
\h(u) &=z\sqrt{\frac{n(0)}{m(0)}} \sum_{r \in \Z} \h^{(r)} u^{-r},
&
\h(u)^{-1} &= -z\sqrt{\frac{m(0)}{n(0)}} \sum_{r \in \Z} \widetilde\h^{(r)} u^{-r},
\end{align}
this notation being consistent with (\ref{summer1})--(\ref{summer2}) for $t =
\sqrt{m(0)/n(0)}$.
Then $\mathcal I_Z(m|n)$ is generated by
\begin{equation}
\Big\{
\mathord{
\begin{tikzpicture}[baseline = -1.2mm]
 	\draw[->] (0.08,-.25) to (0.08,.2);
     \node at (0.08,-0.01) {$\dt$};
     \node at (-0.35,-0.01) {$\scriptstyle m(x)$};
\end{tikzpicture}
},\:
\begin{tikzpicture}[baseline=-.9mm]
   \node at (0,0) {$\anticlockplus$};
   \node at (0.3,0) {$\scriptstyle{r}$};
\end{tikzpicture}
\!-
\h^{(r)} 1_{\unit}\:\big|\:
-k \leq r < \deg m(u)
\Big\}.\label{yetanotheroneofthe}
\end{equation}
Equivalently, by
\cite[Lemma 9.2]{BSWqheis},
it is generated by
\begin{equation}
\Big\{
\mathord{
\begin{tikzpicture}[baseline = -.9mm]
 	\draw[<-] (0.08,-.2) to (0.08,.25);
     \node at (0.08,0.02) {$\dt$};
     \node at (-0.35,0.02) {$\scriptstyle n(x)$};
\end{tikzpicture}
},\:
\begin{tikzpicture}[baseline=-.9mm]
   \node at (0,0) {$\clockplus$};
   \node at (-0.3,0) {$\scriptstyle{r}$};
\end{tikzpicture}
\!\!-
\widetilde\h^{(r)} 1_{\unit}\:\big|\:
k \leq r < -\deg n(u)
\Big\},
\end{equation}
and also (\ref{bothcases}) holds.
Recalling from (\ref{groundring}) 
that $\Heis_k$ is defined over the algebra $\K = \k[t,t^{-1}]$ in
the quantum case, and that
$t\unit =  z\begin{tikzpicture}[baseline=-.9mm]
   \node at (0,0) {$\anticlockplus$};
   \node at (0.4,0) {$\scriptstyle{-k}$};
\end{tikzpicture}$ by the defining relations, 
the presence of the generator
$\begin{tikzpicture}[baseline=-.9mm]
   \node at (0,0) {$\anticlockplus$};
   \node at (0.4,0) {$\scriptstyle{-k}$};
\end{tikzpicture}
\!-
\h^{(-k)} 1_{\unit}$ in the definition of $\mathcal I_Z(m|n)$
has the effect of forcing the parameter $t$ to act on any morphism in
$\H_Z(m|n)$ by multiplication by the scalar
$\sqrt{m(0)/n(0)} \in Z^\times$. 
This is necessary for some choice of the square root 
due to the last part of Lemma~\ref{preimpy}.

\begin{Lemma}\label{bp3}
In the quantum case, the quotient of the $Z$-algebra $\End_{\Heis_k}(\unit) \otimes_\k Z$ 
by the ideal $V$ generated by 
$\Big\{
\!\! \begin{tikzpicture}[baseline=-.9mm]
   \node at (0,0) {$\anticlockplus$};
   \node at (0.3,0) {$\scriptstyle{r}$};
\end{tikzpicture}
\!\!-
\h^{(r)} 1_\unit,
\:
\begin{tikzpicture}[baseline=-.9mm]
\draw[-] (0,-0.18) to[out=180,in=-102] (-.178,0.02);
\draw[-] (-0.18,0) to[out=90,in=180] (0,0.18);
\draw[-] (0.18,0) to[out=-90,in=0] (0,-0.18);
\draw[<-] (0,0.18) to[out=0,in=90] (0.18,0);
   \node at (0.18,0) {$\dt$};
   \node at (0.7,0) {$\scriptstyle{m(x)x^s}$};
\end{tikzpicture}
\big|\:
-k \leq r < \deg m(u), s \in \Z
\Big\}
$
is isomorphic to $Z$.
\end{Lemma}

\begin{proof}
By the basis theorem from \cite[Theorem 10.1]{BSWqheis},
$\End_{\Heis_k}(\unit)$
is a free polynomial algebra over $\K$
on generators
$
\begin{tikzpicture}[baseline=-1mm]
  \node at (0,0) {$\anticlockplus$};
   \node at (0.3,0) {$\scriptstyle{r}$};
\end{tikzpicture}$
for $-k < r < \deg m(u)$,
$\begin{tikzpicture}[baseline=-1mm]
\draw[-] (0,-0.18) to[out=180,in=-102] (-.178,0.02);
\draw[-] (-0.18,0) to[out=90,in=180] (0,0.18);
\draw[-] (0.18,0) to[out=-90,in=0] (0,-0.18);
\draw[<-] (0,0.18) to[out=0,in=90] (0.18,0);
   \node at (0.18,0) {$\dt$};
   \node at (0.4,0) {$\scriptstyle{x^s}$};
\end{tikzpicture}$
for $s <0$, and 
$\begin{tikzpicture}[baseline=-1mm]
\draw[-] (0,-0.18) to[out=180,in=-102] (-.178,0.02);
\draw[-] (-0.18,0) to[out=90,in=180] (0,0.18);
\draw[-] (0.18,0) to[out=-90,in=0] (0,-0.18);
\draw[<-] (0,0.18) to[out=0,in=90] (0.18,0);
   \node at (0.18,0) {$\dt$};
   \node at (0.4,0) {$\scriptstyle{x^s}$};
\end{tikzpicture}$
for
$s \geq \deg m(u)$.
Since $m(u)$ is monic, 
factoring out the ideal generated by
$\begin{tikzpicture}[baseline=-.9mm]
\draw[-] (0,-0.18) to[out=180,in=-102] (-.178,0.02);
\draw[-] (-0.18,0) to[out=90,in=180] (0,0.18);
\draw[-] (0.18,0) to[out=-90,in=0] (0,-0.18);
\draw[<-] (0,0.18) to[out=0,in=90] (0.18,0);
   \node at (0.18,0) {$\dt$};
   \node at (0.7,0) {$\scriptstyle{m(x)x^s}$};
\end{tikzpicture}
$ for $s \geq 0$ leaves us with the free polynomial algebra
over $\K$
on generators
$
\begin{tikzpicture}[baseline=-1mm]
  \node at (0,0) {$\anticlockplus$};
   \node at (0.3,0) {$\scriptstyle{r}$};
\end{tikzpicture}$
for $-k < r < \deg m(u)$ and
$\begin{tikzpicture}[baseline=-1mm]
\draw[-] (0,-0.18) to[out=180,in=-102] (-.178,0.02);
\draw[-] (-0.18,0) to[out=90,in=180] (0,0.18);
\draw[-] (0.18,0) to[out=-90,in=0] (0,-0.18);
\draw[<-] (0,0.18) to[out=0,in=90] (0.18,0);
   \node at (0.18,0) {$\dt$};
   \node at (0.4,0) {$\scriptstyle{x^s}$};
\end{tikzpicture}$
for $s <0$. Then, since $m(0)$ is a unit,
factoring out the ideal generated by
$\begin{tikzpicture}[baseline=-.9mm]
\draw[-] (0,-0.18) to[out=180,in=-102] (-.178,0.02);
\draw[-] (-0.18,0) to[out=90,in=180] (0,0.18);
\draw[-] (0.18,0) to[out=-90,in=0] (0,-0.18);
\draw[<-] (0,0.18) to[out=0,in=90] (0.18,0);
   \node at (0.18,0) {$\dt$};
   \node at (0.7,0) {$\scriptstyle{m(x)x^s}$};
\end{tikzpicture}
$ for $s < 0$ leaves us with the free polynomial algebra
over $\K$
on generators
$
\begin{tikzpicture}[baseline=-1mm]
  \node at (0,0) {$\anticlockplus$};
   \node at (0.3,0) {$\scriptstyle{r}$};
\end{tikzpicture}$
for $-k < r < \deg m(u)$.
Finally we tensor over $\k$ with $Z$ and factor out the ideal generated
by the remaining elements
$\begin{tikzpicture}[baseline=-.9mm]
   \node at (0,0) {$\anticlockplus$};
   \node at (0.3,0) {$\scriptstyle{r}$};
\end{tikzpicture}
\!\!-
h^{(r)} 1_\unit$ for $-k \leq r < \deg m(u)$. The first of these with $r=-k$
substitutes $t \in \K$ by
$\sqrt{m(0)/n(0)} \in Z$, leaving a free polynomial algebra
over $Z$ on generators
$
\begin{tikzpicture}[baseline=-1mm]
  \node at (0,0) {$\anticlockplus$};
   \node at (0.3,0) {$\scriptstyle{r}$};
\end{tikzpicture}$
for $-k < r < \deg m(u)$. Then the remaining relations for $-k < r <
\deg m(u)$
evaluate these generators to elements of $Z$.
\end{proof}

Now we proceed like in the previous subsection.
Let $AH_d$ be the {\em affine Hecke algebra}, degenerate or quantum
according to the value of $z$.
This is the $\k$-algebra
with generators $\{x_1,\dots,x_d, s_1,\dots,s_{d-1}\}$ (dots and
crossings) in the
degenerate case or $\{x_1^{\pm 1}, \dots, x_d^{\pm 1}, \tau_1^{\pm 1}, \dots,
\tau_{d-1}^{\pm 1}\}$ (invertible dots and positive/negative crossings)
in the quantum case subject to the ``local'' relations represented by
(\ref{dAHA}) or (\ref{AHA1})--(\ref{AHA2}), respectively.
The {\em cyclotomic Hecke algebra} $H_d^m(Z)$ is the
quotient of the $Z$-algebra $AH_d \otimes_\k Z$ by the two-sided ideal $U$
generated by $m(x_1)$; 
we interpret $H^m_0(Z)$ simply as the algebra $Z$.
Consider the diagram
\begin{equation}
\begin{diagram}
\node{\!\!\!\!(AH_d \otimes_\k Z)\otimes_Z (\End_{\Heis_k}(\unit) \otimes_\k Z)}
\arrow{s,l,A}{\pi_1\bar\otimes\pi_2}\arrow{e,t}{\imath_d}\node{
\End_{\Heis_k}(E^d) \otimes_\k Z}\arrow{s,r,A}{\pi}\\
\node{H_d^m(Z)}\arrow{e,b}{\jmath_d}\node{\End_{\H_Z(a|b)}(E^d).}
\end{diagram}\label{buffers2}
\end{equation}
The top map $\imath_d$ is the evident $Z$-algebra homomorphism.
The basis theorem proved in \cite[Theorem 6.4]{BSWk0} or \cite[Theorem 10.1]{BSWqheis} implies
that this is actually an isomorphism.
The right-hand map $\pi$ is the natural quotient map.
The left-hand map $\pi_1\bar\otimes\pi_2$ is 
the product of
the natural quotient map $\pi_1:AH_d \otimes_\k Z \twoheadrightarrow
H_d^m(Z)$ with kernel $U$
and the $Z$-algebra homomorphism $\pi_2:\End_{\Heis_k}(\unit)\otimes_\k Z
\twoheadrightarrow Z$
with kernel $V$ arising from Lemmas~\ref{bp2}--\ref{bp3}.

\begin{Lemma}\label{cqh2}
There is a unique isomorphism $\jmath_d$ making the diagram (\ref{buffers2})
commute.
\end{Lemma}

\begin{proof}
This is similar to the proof of Lemma~\ref{cqh}, using Lemmas~\ref{bp2}--\ref{bp3}
in place of Lemma~\ref{bp}.
\end{proof}

\begin{Corollary}\label{cqhd2}
$\dim \End_{\H_Z(a|b)}(E^d) = \ell^d d! \dim Z$ where $\ell := \deg m(u)$.
\end{Corollary}

\begin{proof}
This is the dimension of the level $\ell$ cyclotomic Hecke algebra
$H_d^m(Z)$.
\end{proof}

\begin{Corollary}\label{definitely}
$\H_Z(m|n)$ is a finite-dimensional category.
\end{Corollary}

\begin{proof}
This follows by an argument similar to Corollary~\ref{maybe}, using 
(\ref{invrela})--(\ref{invrelb}) in the degenerate case, 
or the analogous inversion relations
in the quantum case.
\end{proof}

As we did in (\ref{stupify}), we switch from now on to using 
algebraic language by viewing the finite-dimensional
category 
$\H_Z(m|n)$ instead as the locally finite-dimensional locally unital
algebra
\begin{equation}\label{stupify2}
H_Z(m|n) := \bigoplus_{G, G' \in \H_Z(m|n)}
\Hom_{\H_Z(m|n)}(G, G'),
\end{equation}
with 
local unit
$\{1_G\:|\:G \in \langle E, F \rangle\}$ and multiplication induced by
composition. 
Then we can consider 
the Schurian category $\lfdrmod H_Z(m|n)$. 
The categorical action of $\Heis_k$ on $\H_Z(m|n)$ extends canonically
to make $\lfdrmod H_Z(m|n)$ into a Schurian $\Heis_k$-module
category.

Let $P := 1_\varnothing H_Z(m|n)$.
As $\End_{H_Z(m|n)}(P) \cong Z$, this is a projective indecomposable module.
Then for any $G \in \langle E, F \rangle$ the projective module $G P$
is identified with the right ideal $1_G H_Z(m|n)$.
These modules for all $G$ give a projective generating family for
$\lfdrmod H_Z(m|n)$
such that
\begin{equation}\label{nupify}
H_Z(\mu|\nu) = 
\bigoplus_{G, G' \in \langle E, F\rangle}
1_{G'} H_Z(\mu|\nu) 1_{G}
\cong\bigoplus_{G, G' \in \langle E, F\rangle}
\Hom_{H_Z(\mu|\nu)}(G P, G' P).
\end{equation}

\begin{Remark}\label{builtlikeatank}
Like in Remark~\ref{usual}, in the case that $n(u) = 1$, 
the GCQ $H_Z(m|n)$ is Morita equivalent to the usual cyclotomic quotient, that is, the 
locally unital algebra $\bigoplus_{d  \geq 0} H^m_d(Z)$.
This is proved in the degenerate case in \cite[Theorem 1.7]{Bheis};
the proof in the quantum case is similar.
\end{Remark}

\subsection{Isomorphisms between GCQs}
\label{sec:gener-cycl-quot}
Fix a finite-dimensional, 
commutative, local $\k$-algebra $Z$ with maximal ideal $J$ and
monic polynomials $m(u), n(u) \in Z[u]$, and let $k := \deg n(u)-\deg m(u)$.
Then define the Heisenberg GCQ $H_Z(m|n)$ as in (\ref{stupify2}).
Let $\bar m(u), \bar n(u) \in \k[u]$ be the reductions of $m(u), n(u)$
modulo $J$.
Let $I$ be the union of the trajectories of the roots of $\bar m(u)$ and $\bar n(u)$ under the
automorphisms $i \mapsto i^\pm$ defined in the introduction.
This gives us the data needed to define the Kac-Moody algebra
$\mathfrak{g}$ with root lattice $X$.
Let $\mu,\nu \in X^+$ be the dominant weights 
defined by declaring that $\langle h_i, \mu\rangle$ and $\langle
h_i,\nu\rangle$ are the multiplicities of $i\in I$ as a root of $\bar m(u)$
and $\bar n(u)$, respectively, and let $\kappa := \nu-\mu$.
Then we apply Corollary~\ref{hens} to the polynomials $m(u), n(u)$ 
to deduce that there are unique monic
polynomials
$\mu_i(u) \in u^{\langle h_i, \mu\rangle} + J[u], \nu_i(u)
\in u^{\langle h_i, \nu\rangle} +
J[u]$ such that
\begin{align}\label{boring}
m(u) &= 
\prod_{i \in I} \mu_i(u-i),
&n(u) &= 
\prod_{i \in I} \nu_i(u-i)\\\intertext{in the degenerate case (here
  $\mu_i(u), \nu_i(u)$ are the polynomials $m_i(u),
  n_i(u)$ produced by Corollary~\ref{hens}), or}
m(u)&=
\prod_{i \in I} i^{\langle h_i, \mu\rangle}\mu_i\big({\textstyle\frac{u}{i}}-1\big),
&
n(u)&=
\prod_{i \in I} i^{\langle h_i, \nu\rangle}\nu_i\big({\textstyle\frac{u}{i}}-1\big)\label{music}
\end{align}
in the quantum case (this time $\mu_i(u), \nu_i(u)$ are
$i^{-\langle h_i,\mu\rangle} m_i(iu), i^{-\langle h_i,\nu\rangle} n_i(iu)$).
These polynomials give us the data needed for the Kac-Moody GCQ 
$H_Z(\mu|\nu)$ according to (\ref{stupify}).
In this subsection, we are going to show that $H_Z(\mu|\nu) \cong H_Z(m|n)$.
In order to do this, we apply the general machinery from
Section 4 to analyze the Schurian
$\Heis_k$-module category $\lfdrmod H_Z(m|n)$.

\begin{Lemma}\label{speccy}
The spectrum of the $\Heis_k$-module category
$\lfdrmod H_Z(m|n)$ is the set $I$ generated by the roots of the
polynomials $m(u)$ and $n(u)$ as above.
\end{Lemma}

\begin{proof}
From (\ref{boring})--(\ref{music}), we get the CRT decomposition
\begin{equation}\label{cars}
    Z[u] / (m(u)) \cong
    \begin{cases}
        \bigoplus_{i \in I} Z[u]/(\mu_i(u-i)) & \text{in the degenerate case}, \\
        \bigoplus_{i \in I} Z[u]/ \left( \mu_i \left( \frac{u}{i} - 1 \right) \right) & \text{in the quantum case}.
    \end{cases}
\end{equation}
Moreover, the image of $(u-i)$ in the $i$th summand 
of this decomposition is 
nilpotent.
From the $d=1$ case of Lemma~\ref{cqh2}, we see that there is an isomorphism 
\begin{align*}
Z[u] / (m(u)) &\stackrel{\sim}{\rightarrow} \End_{H_Z(m|n)} (EP),
&u &
\mapsto \mathord{
\begin{tikzpicture}[baseline = -.8mm]
 	\draw[->] (0.08,-.2) to (0.08,.2);
     \node at (0.08,0) {$\dt$};
 	\draw[-,darkg,thick] (0.45,.2) to (0.45,-.2);
     \node at (0.7,.05) {$\darkg\scriptstyle{P}$};
\end{tikzpicture}
}.
\end{align*}
So from (\ref{cars}), we get induced a decomposition of the module
$EP$ such that $(u-i)$ acts nilpotently on the $i$th summand. It
follows that this summand is simply the generalized
$i$-eigenspace $E_i P$ as defined in $\S$\ref{sendo}.
This shows that $EP = \bigoplus_{i \in I} E_i P$
with $\End_{H_Z(m|n)}(E_i P) \cong Z[u] / (\mu_i(u))$.
Consequently, $E_i P$ is non-zero if and only if $i \in I$ and $\langle h_i, \mu\rangle > 0$.
A similar discussion applies to $FP$: we have that
$FP = \bigoplus_{i \in I} F_i P$
with $\End_{H_Z(m|n)}(F_i P) \cong Z[u] / (\nu_i(u))$.
Consequently, $F_i P$ is 
non-zero if and only if $i \in I$ and $\langle h_i, \nu\rangle > 0$.
In view of Lemma~\ref{closed}, we deduce that the spectrum of
$\lfdrmod H_Z(m|n)$ contains the set $I$.

Conversely, we must show that $I$ is contained in the spectrum
of $\lfdrmod H_Z(m|n)$.
To prove this, we say that $V \in \lfdrmod H_Z(m|n)$ belongs to
  $I$ if $E V = \bigoplus_{i \in I} E_i V$ and $FV = \bigoplus_{i \in I} F_i V$.
We must show that every $V \in \lfdrmod H_Z(m|n)$ belongs to $I$.
As the modules $GP$ for $G \in \langle E, F \rangle$ give a projective
generating family, and these functors are exact, 
it suffices to show that all $GP$ belong to $I$.
To prove this, we proceed by induction on the length of the word $G$.
The base case $G = \varnothing$ follows from
the previous paragraph.
To prove the induction step, it suffices to establish the following:
{\em if $L$ is an irreducible $H_Z(m|n)$-module belonging to $I$ 
then the modules $EL$ and $FL$ also belong to $I$.}
To see this, let $K$ be an irreducible subquotient of either $EL$ or $FL$.
Since $L$ belongs to $I$, all roots of the minimal polynomials
$m_L(u)$ and $n_L(u)$ belong to $I$.
We must show that all roots of 
$m_K(u)$ and $n_K(u)$ also belong to $I$.
In the case that $K$ is a subquotient of $EL$, we argue as follows.
All roots of $m_K(u)$ belong to $I$ due to a well-known observation about the affine 
Hecke algebra $AH_2$; see \cite[Lemma 6.1]{BSWk0} or \cite[Lemma 9.3]{BSWqheis}.
To deduce that all roots of $n_K(u)$ belong to $I$, use the fact that
$n_K(u) = \frac{n_L(u) m_K(u)}{m_L(u)} \times $(a rational
function in $\k(u)$ with zeros and poles in $I$)
thanks to Lemmas~\ref{impy} and
\ref{choose}.
The argument in the case that $K$ is a subquotient of $FL$ is similar.
\end{proof}

With Lemma~\ref{speccy} in hand, we see that the endofunctors $E$
and $F$ of $\lfdrmod H_Z(m|n)$ decompose into eigenfunctors as 
$E = \bigoplus_{i \in I} E_i$ and $F = \bigoplus_{i \in I} F_i$ as in (\ref{pape}). 
Applying (\ref{one}), we also have the weight decomposition
\begin{equation}\label{timetorun}
\lfdrmod H_Z(m|n) = \prod_{\lambda \in X} \lfdrmod H_Z(m|n)_\lambda.
\end{equation}
Applying Theorem~\ref{thm1},
$\left(\lfdrmod H_Z(m|n\right)_\lambda)_{\lambda \in X}$
is a nilpotent 2-representation of $\UU(\g)$.
Let $P$ be the projective indecomposable module $1_\varnothing
H_{m|n}$
and recall that $\kappa = \nu-\mu$.

\begin{Lemma}\label{mower}
The module $P$ belongs to $\lfdrmod H_Z(m|n)_\kappa$.
Moreover, under the categorical action of $\UU(\g)$ just defined, we
have isomorphisms
\begin{align}\label{4thofJulypm}
Z[u] / (\mu_i(u)) &\stackrel{\sim}{\rightarrow} \End_{H_Z(m|n)} (E_iP),
&u &
\mapsto \mathord{
\begin{tikzpicture}[baseline = -1mm]
 	\draw[thick,->] (0.08,-.2) to (0.08,.2);
     \node at (0.08,-0.02) {$\bull$};
     \node at (0.08,-.35) {$\scriptstyle i$};
 	\draw[-,darkg,thick] (0.5,.2) to (0.5,-.2);
     \node at (0.3,0) {$\color{gray}\scriptstyle{\kappa}$};
     \node at (0.75,.05) {$\darkg\scriptstyle{P}$};
\end{tikzpicture}
},\\
Z[u] / (\nu_i(u)) &\stackrel{\sim}{\rightarrow} \End_{H_Z(m|n)} (F_iP),
&u &\mapsto \mathord{
\begin{tikzpicture}[baseline = -.5mm]
 	\draw[<-,thick] (0.08,-.2) to (0.08,.2);
     \node at (0.08,0.03) {$\bull$};
     \node at (0.08,.35) {$\scriptstyle i$};
     \node at (0.3,0.04) {$\color{gray}\scriptstyle{\kappa}$};
 	\draw[-,darkg,thick] (0.5,.2) to (0.5,-.2);
     \node at (0.75,.05) {$\darkg\scriptstyle{P}$};
\end{tikzpicture}
}.\label{5thofJulypm}
\end{align}
Finally, the generating function
from (\ref{lll}) satisfies $\h_{P,i}(u) = \nu_i(u)/\mu_i(u)$.
\end{Lemma}

\begin{proof}
For the first statement, it suffices to show that the unique
irreducible quotient $L$ of $P$ belongs to $\lfdrmod H_Z(m|n)_\kappa$.
By (\ref{bothcases}) and the definition (\ref{hat}), we have that
$\h_P(u) = n(u)/m(u)$,
where $\h_P(u)$ is the generating function
defined by (\ref{jjj}).
Hence, $\h_L(u) = \bar n(u) / \bar m(u)$.
Using the observation immediately following (\ref{chickens})
together with (\ref{boring})--(\ref{music}),
it follows that 
$\langle h_i, \wt(L) \rangle = \langle h_i, \nu\rangle
- \langle h_i, \mu\rangle = \langle h_i, \kappa\rangle.$
Thus, $\wt(L) = \kappa$ as required.

To establish (\ref{4thofJulypm}), the argument from the first paragraph of
the proof of Lemma~\ref{speccy} shows that there is an isomorphism
\begin{align*}
Z[u] / (\mu_i(u)) &\stackrel{\sim}{\rightarrow} \End_{H_Z(m|n)} (E_iP),
&u &
\mapsto 
\left\{
\begin{array}{rl}
\mathord{
\begin{tikzpicture}[baseline = -1mm]
 	\draw[->] (0.08,-.2) to (0.08,.2);
     \node at (0.08,-0.03) {$\dt$};
     \node at (0.08,-0.35) {$\scriptstyle i$};
     \node at (-0.33,-0.02) {$\scriptstyle x-i$};
 	\draw[-,darkg,thick] (0.4,.2) to (0.4,-.2);
     \node at (0.65,0) {$\darkg\scriptstyle{P}$};
\end{tikzpicture}
}&\text{in the degenerate case,}\\
\mathord{
\begin{tikzpicture}[baseline = -1mm]
 	\draw[->] (0.08,-.2) to (0.08,.2);
     \node at (0.08,-.03) {$\dt$};
     \node at (0.08,-0.35) {$\scriptstyle i$};
     \node at (-0.38,-0.02) {$\scriptstyle \frac{x}{i}-1$};
 	\draw[-,darkg,thick] (0.4,.2) to (0.4,-.2);
     \node at (0.65,0) {$\darkg\scriptstyle{P}$};
\end{tikzpicture}
}&\text{in the quantum case.}
\end{array}\right.
\end{align*}
So we get (\ref{4thofJulypm}) using the definition of the action of
$\mathord{
\begin{tikzpicture}[baseline = -2]
	\draw[->,thick] (0.08,-.15) to (0.08,.3);
      \node at (0.08,0.05) {$\bull$};
   \node at (0.08,-.3) {$\scriptstyle{i}$};
\end{tikzpicture}
}
{\color{gray}\scriptstyle\kappa}$
from the statement of Theorem~\ref{thm1}.
The proof of (\ref{5thofJulypm}) is similar.

Finally, we must compute $\h_{P,i}(u) \in u^{\langle
  h_i,\kappa\rangle}+
u^{\langle h_i,\kappa\rangle-1}Z[u^{-1}]$.
Applying Lemma~\ref{preimpy2}(1) with $f(u) = \mu_i(u)$, we see that
$g(u) := \h_{P,i}(u) \mu_i(u)$ is a monic polynomial in $Z[u]$ of
degree $\langle h_i, \nu\rangle$
such that 
$\mathord{
\begin{tikzpicture}[baseline = -1mm]
 	\draw[thick,<-] (0.08,-.2) to (0.08,.2);
     \node at (0.08,0.03) {$\bull$};
     \node at (0.08,0.35) {$\scriptstyle i$};
     \node at (-0.3,0.03) {$\scriptstyle g(y)$};
 	\draw[-,darkg,thick] (0.52,.2) to (0.52,-.2);
     \node at (0.3,0.03) {$\color{gray}\scriptstyle{\kappa}$};
     \node at (0.69,0) {$\darkg\scriptstyle{P}$};
\end{tikzpicture}
}=0$.
Then by (\ref{5thofJulypm}), it follows that the image of $g(u) - \nu_i(u)$ is zero
in $Z[u] / (\nu_i(u))$.
Since $g(u)-\nu_i(u)$ is a polynomial of degree $\langle h_i,\nu\rangle-1$ and 
$1, u, u^2,\dots,u^{\langle h_i,\nu\rangle-1}$ are linearly
independent over $Z$ in this algebra, it follows that $g(u) =
\nu_i(u)$.
Now we have shown that $\h_{P,i}(u) \mu_i(u)=\nu_i(u)$, and
the result follows.
\end{proof}

Finally, we need to pass to an idempotent expansion of the locally unital algebra $H_Z(m|n)$, i.e., we must refine its local unit.
Take $G = G_d \cdots G_1 \in \langle E,F\rangle$.
As $E = \bigoplus_{i \in I} E_i$ and $F = \bigoplus_{i \in I} F_i$,
there is a decomposition 
\begin{equation}\label{disaster}
G = \bigoplus_{\bi \in I^d} G_\bi
\end{equation}
of the endofunctor $G$,
where $G_\bi := (G_d)_{i_d}\cdots (G_1)_{i_1}$ for $\bi = (i_1,\dots,i_d) \in I^d$.
Recalling that $G P = 1_G H_Z(m|n)$, we deduce that the idempotent
$1_G \in H_Z(m|n)$ decomposes as
$1_G = \sum_{\bi \in I^d} 1_{G_\bi}$ for mutually orthogonal idempotents $1_{G_\bi}$
such that $1_{G_\bi} H_Z(m|n)=G_\bi P$.
In this way, we have defined a refinement of the original local unit for the
algebra $H_Z(m|n)$,
with the new local unit being
indexed by the same set $\langle E_i, F_i\rangle_{i \in I}$
as the local unit of $H_Z(\mu|\nu)$.
Note moreover for $G \in \langle E_i, F_i\rangle_{i \in I}$ 
that $G P$ belongs to $\lfdrmod H_Z(m|n)_\lambda$ 
for $\lambda = \kappa + \wt(G)$. 
So the weight space decomposition of $\lfdrmod H_Z(m|n)$ from (\ref{timetorun})
is consistent with the
algebra decomposition
\begin{equation}\label{cup}
H_Z(m|n) = \bigoplus_{\lambda \in X} H_Z(m|n)_\lambda
\quad\text{where}\quad
H_Z(m|n)_\lambda := \hspace{-5mm}\bigoplus_{\substack{G, G' \in \langle E_i,
  F_i\rangle_{i \in I}\\\wt(G) = \wt(G') = \lambda-\kappa}}
\hspace{-5mm}1_{G'} H_Z(m|n)1_G.
\end{equation}
This should be compared with the decomposition (\ref{blockde}) of the
algebra $H_Z(\mu|\nu)$.

\begin{Theorem}\label{thm2}
With notation as above,
there is a unique isomorphism of locally unital $Z$-algebras 
$\theta:H_Z(\mu|\nu)\stackrel{\sim}{\rightarrow} H_Z(m|n)$
defined on generators of the algebra
$H_Z(\mu|\nu)$ involving upwards dots or crossings or rightwards cups
or caps by the formulae in the statement of Theorem~\ref{thm1}, and with
        $\theta(1_G) = 1_G$ for each $G \in \langle E_i,
F_i\rangle_{i \in I}$.
\end{Theorem}

\begin{proof}
As the projective module $P$ belongs to $\lfdrmod
H_Z(m|n)_\kappa$, the categorical action of $\UU(\g)$ on the family
 $\left(\lfdrmod H_Z(m|n)_\lambda\right)_{\lambda \in X}$ 
induces a unique $Z$-linear morphism of
2-representations
$$\left(\R(\kappa)_\lambda\otimes_\k Z\right)_{\lambda \in X}
\rightarrow
\left(\lfdrmod H_Z(m|n)_\lambda\right)_{\lambda \in X}$$
sending $1_\kappa \mapsto P$.
Lemma~\ref{mower} shows that the generators of 
$\mathcal I_Z(\mu|\nu)$ from (\ref{oneofthe}) map to zero, hence, this factors through
the quotient to give a $Z$-linear morphism of 2-representations
$\left(\H_Z(\mu|\nu)_\lambda\right)_{\lambda \in X} \rightarrow 
\left(\lfdrmod H_Z(m|n)_\lambda\right)_{\lambda \in X}.$
Thus, we have constructed a $Z$-linear functor
\begin{equation*}
\Theta:\H_Z(\mu|\nu) \rightarrow \lfdrmod H_Z(m|n)
\end{equation*}
sending $G$ to $G P$ for each $G \in \langle E_i, F_i\rangle_{i \in I}$.
Using (\ref{stupify}) and (\ref{nupify}), it follows that $\Theta$
induces a $Z$-algebra homomorphism
$\theta:H_Z(\mu|\nu) \rightarrow H_Z(m|n)$ sending
$1_G \mapsto 1_G$ for each $G \in \langle E_i, F_i \rangle_{i \in I}$.
By its definition, this may be computed explicitly on 
generators of $H_Z(\mu|\nu)$ involving upwards dots or crossings or
rightwards cups or caps using the formulae from
Theorem~\ref{thm1}; as noted in Remark~\ref{nocaps}, 
we have not given explicit formulae for leftwards cups or caps, but these are not needed.

To show $\theta$ is an isomorphism, we show equivalently that the
functor $\Theta$ is fully faithful, i.e.,
it defines isomorphisms
$\Theta_{G,G'}:\Hom_{\H_Z(\mu|\nu)}(G, G') \stackrel{\sim}{\rightarrow}
\Hom_{H_Z(m|n)}(GP, G'P)$ for all $G, G' \in \langle E_i,
F_i\rangle_{i \in I}$. We first treat the case that $G$ and $G'$ both
belong to $\langle E_i\rangle_{i \in I}$. We may assume that both $G$
and $G'$ have the same length $d \geq 0$, since otherwise both morphism spaces
are zero. Then the $Z$-linear map $\Theta_{G,G'}$ is surjective.
To see this, since every
morphism in $\Hom_{H_Z(m|n)}(GP, G'P)$ is a $Z$-linear combination of
morphisms obtained by composing dots 
and crossings of the form (\ref{interesting}), we just need to
show that all of the latter morphisms are in the image. This follows
on inverting the formulae in Theorem~\ref{thm1} (we will write
the inverse formulae explicitly
explicitly in Theorem~\ref{thm3} below). Then to see that $\Theta_{G,G'}$ is
injective we use the equality of dimensions which follows on comparing
Corollaries~\ref{cqhd} and \ref{cqhd2}.

To establish the fully faithfulness
for more general words $G, G' \in \langle E_i, F_i\rangle_{i \in I}$,
the idea is to reduce to the special case just treated. We proceed by
induction on the sum of the lengths of the words $G$ and $G'$.
Given morphisms
$H,H' \in \Add(\H_Z(\mu|\nu))$, i.e., finite direct sums of words in
$\langle E_i, F_i\rangle_{i \in I}$,
and an isomorphism $\alpha\in\Hom_{\H_Z(\mu|\nu)}(H', H)$
defined by some 2-morphism in $\UU(\g)$,
we can apply
$\Theta$ to obtain an isomorphism $\beta \in
\Hom_{H_Z(m|n)}(H' P, H P)$ such that the following diagram commutes:
$$
\begin{diagram}
\node{\Hom_{\H_Z(\mu|\nu)}(H, G')}
\arrow{s,l}{\alpha^*}
\arrow{e,t}{\Theta_{H,G'}}\node{\Hom_{H_Z(m|n)}(H P, G'P)}
\arrow{s,r}{\beta^*}
\\
\node{\Hom_{\H_Z(\mu|\nu)}(H',G')}
\arrow{e,b}{\Theta_{H',G'}}\node{\Hom_{H_Z(m|n)}(H' P,
G'P)}
\end{diagram}.
$$
The vertical arrows in this diagram are isomorphisms, so we deduce
that 
$\Theta_{H,G'}$ is an isomorphism if and only if $\Theta_{H',G'}$ is an isomorphism.
Using this observation for isomorphisms $\alpha$ obtained from 
the isomorphisms (\ref{day1})--(\ref{day3}), one reduces 
to proving the fully faithfulness in the situation that 
all letters of the form
$F_i\:(i \in I)$ in $G$ appear to the left of all letters of the form
$E_i \:(i \in I)$; this argument also requires the induction hypothesis since
shorter words may arise when (\ref{day2})--(\ref{day3}) are used.
Next, 
we note for any $H \in \langle E_i, F_i\rangle_{i \in I}$
that the following diagram commutes:
$$
\begin{diagram}
\node{\Hom_{\H_Z(\mu|\nu)}(F_i H, G')}
\arrow{s,l}{E_i}
\arrow{e,t}{\Theta_{F_i H, G'}}\node{\Hom_{H_Z(m|n)}(F_i H P, G'P)}
\arrow{s,r}{E_i}
\\
\node{\Hom_{\H_Z(\mu|\nu)}(E_i F_i H, E_iG')}
\arrow{s,l}{\alpha^*}
\arrow{e,b}{\Theta_{E_i F_i H, E_i G'}}\node{\Hom_{H_Z(m|n)}(E_i F_i H P,
E_i G'P)}
\arrow{s,r}{\beta^*}\\
\node{\Hom_{\H_Z(\mu|\nu)}(H, E_iG')}
\arrow{e,b}{\Theta_{H, E_i G'}}\node{\Hom_{H_Z(m|n)}(H P,
E_i G'P)}
\end{diagram},
$$
where $\alpha:H \rightarrow E_i F_i H$ is the morphism in
$\H_Z(\mu|\nu)$
defined by the unit of the adjunction $(F_i, E_i)$
and $\beta:HP \rightarrow E_i F_i HP$ is its image under $\Theta$. The
compositions
down the left edge and down the right edge of this diagram are 
adjunction isomorphisms, so we deduce that $\Theta_{F_i H, G'}$ is an isomorphism if
and only if $\Theta_{H, E_i G'}$ is an isomorphism.
Using this observation, we reduce the proof of fully faithfulness to
the situation that $G \in \langle E_i\rangle_{i \in I}$.
Then we repeat the process to reduce further to the case that
all letters of the form
$F_i\:(i \in I)$ in $G'$ appear to the left of all letters of the form
$E_i\:(i \in I)$.
Finally, using the other adjunction $(E_i, F_i)$ we move all the letters $F_i$ from $G'$
to $G$, putting us into the situation treated in the previous paragraph.
\end{proof}

\begin{Corollary}
Let $\theta^*:\lfdrmod H_Z(m|n)\rightarrow \lfdrmod H_Z(\mu|\nu)$
be the restriction functor arising from the isomorphism $\theta$.
This defines a strongly equivariant isomorphism
between $\left(\lfdrmod H_Z(m|n)_\lambda\right)_{\lambda \in X}$, that is,
the 2-representation
 obtained by applying Theorem~\ref{thm1}
to the Heisenberg GCQ $\lfdrmod H_Z(m|n)$, and
$\left(\lfdrmod
  H_Z(\mu|\nu)_\lambda\right)_{\lambda \in X}$,
that is, the 2-representation 
 arising from the Kac-Moody GCQ.
\end{Corollary}

\begin{Remark}
Bearing in mind Remarks~\ref{usual} and \ref{builtlikeatank},
Theorem~\ref{thm2} can be viewed as a substantial generalization of
the isomorphism theorem from \cite{BK}.
The original isomorphism
$H_d^\mu(Z) \stackrel{\sim}{\rightarrow} H_d^m(Z)$ from \cite{BK} may be recovered from 
Theorem~\ref{thm2} 
using also Lemmas~\ref{cqh} and \ref{cqh2}; actually, one just needs 
the special case $n(u)=1,\nu=0$ of the theorem.
\end{Remark}

\subsection{Kac-Moody to Heisenberg}
Now we can prove the converse to Theorem~\ref{thm1}.
As usual we discuss the degenerate case $z=0$ and the quantum case $z
\neq 0$ simultaneously.
Let $I$ be a subset of $\k$ closed under the automorphisms $i \mapsto i^{\pm}$ defined in the introduction, assuming $0 \notin I$ in the quantum case.
Let $\UU(\g)$ be the  Kac-Moody 2-category associated to this data.
The dotted arrows in the statement of the following theorem should be
interpreted in the same way as was explained after (\ref{brexit}). Now
these dotted arrows can be labelled by any 
power series in $\k[\![y_1,\dots,y_n]\!]$, and make sense due to the assumed nilpotency.

\begin{Theorem}\label{thm3}
Assume that $(\R_\lambda)_{\lambda \in X}$ is a nilpotent 2-representation of $\UU(\g)$ that is either locally finite Abelian or Schurian.
Let $\R$ 
be defined from this as in (\ref{different}).
Assume in addition that $\R$ is of central charge $k \in \Z$, i.e., $\R_\lambda \neq
\mathbf{0}\Rightarrow
\sum_{i \in I} \langle h_i, \lambda \rangle = k$.
Then there is a unique way to make $\R$ into a
$\Heis_k$-module category so 
that $E$ and $F$ act as the endofunctors
(\ref{EF}), and the generating morphisms in
$\Heis_k$ map to natural transformations according to
\begin{align*}
\begin{tikzpicture}[baseline = -0.8mm]
	\draw[->] (0.08,-.25) to (0.08,.3);
      \node at (0.08,0.02) {$\dt$};
\end{tikzpicture}
&\mapsto
\sum_{\substack{\lambda \in X \\i \in I}}\begin{tikzpicture}[baseline = -0.8mm]
	\draw[->,thick] (0.08,-.25) to (0.08,.3);
      \node at (0.08,-0.4) {$\scriptstyle{i}$};
      \node at (0.08,0.02) {$\bull$};
      \node at (0.42,0.02) {$\scriptstyle{y+i}$};
      \node at (0.35,-0.3) {$\color{gray}\scriptstyle\lambda$};
\end{tikzpicture},
&
\begin{tikzpicture}[baseline = 1mm]
	\draw[<-] (0.4,0.4) to[out=-90, in=0] (0.1,0);
	\draw[-] (0.1,0) to[out = 180, in = -90] (-0.2,0.4);
\end{tikzpicture}
&\mapsto
\sum_{\substack{\lambda \in X \\i \in I}}\begin{tikzpicture}[baseline = 1mm]
	\draw[<-,thick] (0.4,0.4) to[out=-90, in=0] (0.1,0);
	\draw[-,thick] (0.1,0) to[out = 180, in = -90] (-0.2,0.4);
      \node at (-0.2,0.55) {$\scriptstyle{i}$};
      \node at (0.6,0.2) {$\color{gray}\scriptstyle\lambda$};
\end{tikzpicture}\:,
&
\begin{tikzpicture}[baseline = 1mm]
	\draw[<-] (0.4,0) to[out=90, in=0] (0.1,0.4);
	\draw[-] (0.1,0.4) to[out = 180, in = 90] (-0.2,0);
\end{tikzpicture}&\mapsto
\sum_{\substack{\lambda \in X \\i \in I}}\begin{tikzpicture}[baseline = 1mm]
	\draw[<-,thick] (0.4,0) to[out=90, in=0] (0.1,0.4);
	\draw[-,thick] (0.1,0.4) to[out = 180, in = 90] (-0.2,0);
      \node at (-0.2,-0.15) {$\scriptstyle{i}$};
      \node at (0.6,0.2) {$\color{gray}\scriptstyle\lambda$};
\end{tikzpicture}
\:\:,\end{align*}\begin{align*}
\begin{tikzpicture}[baseline = 0]
	\draw[->] (0.28,-.3) to (-0.28,.4);
	\draw[->] (-0.28,-.3) to (0.28,.4);
\end{tikzpicture}
&\mapsto
\sum_{\substack{\lambda \in X \\ i \in I}}
\left(\begin{tikzpicture}[baseline = 0]
	\draw[->,thick] (0.38,-.4) to (-0.38,.5);
	\draw[->,thick] (-0.38,-.4) to (0.38,.5);
	\draw[->,thick,densely dotted] (0.26,-.24) to (-0.21,-.24);
   \node at (-.4,-.55) {$\scriptstyle{i}$};
   \node at (.4,-.55) {$\scriptstyle{i}$};
      \node at (-0.26,-0.25) {$\bull$};
      \node at (0.26,-0.25) {$\bull$};
      \node at (.95,-0.2) {$\scriptstyle{y_2-y_1+1}$};
      \node at (0.7,.25) {$\color{gray}\scriptstyle\lambda$};
\end{tikzpicture}
-\begin{tikzpicture}[baseline = 0mm]
	\draw[->,thick] (0.38,-.4) to (0.38,.5);
	\draw[->,thick] (0,-.4) to (0,.5);
      \node at (0,-0.55) {$\scriptstyle{i}$};
      \node at (0.38,-0.55) {$\scriptstyle{i}$};
      \node at (0.55,0) {$\color{gray}\scriptstyle\lambda$};
\end{tikzpicture}\right)\\
&+\sum_{\substack{\lambda \in X \\ i \in I}}
\left(\begin{tikzpicture}[baseline = 0]
	\draw[->,thick] (0.38,-.4) to (-0.38,.5);
	\draw[->,thick] (-0.38,-.4) to (0.38,.5);
	\draw[->,thick,densely dotted] (0.26,-.24) to (-0.21,-.24);
   \node at (-.4,-.55) {$\scriptstyle{i+1}$};
   \node at (.4,-.55) {$\scriptstyle{i}$};
      \node at (-0.26,-0.25) {$\bull$};
      \node at (0.26,-0.25) {$\bull$};
      \node at (1.1,-0.2) {$\scriptstyle{(y_2-y_1+1)^{-1}}$};
      \node at (0.55,.2) {$\color{gray}\scriptstyle\lambda$};
\end{tikzpicture}
+\begin{tikzpicture}[baseline = 0mm]
	\draw[->,thick,densely dotted] (0.55,.06) to (0.12,.06);
	\draw[->,thick] (0.58,-.4) to (0.58,.5);
	\draw[->,thick] (0.08,-.4) to (0.08,.5);
      \node at (0.08,-0.55) {$\scriptstyle{i+1}$};
      \node at (0.58,-0.55) {$\scriptstyle{i}$};
      \node at (0.08,0.05) {$\bull$};
      \node at (0.58,0.05) {$\bull$};
      \node at (1.4,0.1) {$\scriptstyle{(y_2-y_1+1)^{-1}}$};
      \node at (0.85,-.3) {$\color{gray}\scriptstyle\lambda$};
\end{tikzpicture}
\right)\\
&
+
\sum_{\substack{\lambda \in X \\ i,j \in I\\j \neq i, i^+}}
\left(
-\begin{tikzpicture}[baseline = 0]
	\draw[->,thick] (0.38,-.4) to (-0.38,.5);
	\draw[->,thick] (-0.38,-.4) to (0.38,.5);
	\draw[->,thick,densely dotted] (0.26,-.24) to (-0.21,-.24);
   \node at (-.4,-.55) {$\scriptstyle{j}$};
   \node at (.4,-.55) {$\scriptstyle{i}$};
      \node at (-0.26,-0.25) {$\bull$};
      \node at (0.26,-0.25) {$\bull$};
      \node at (2.15 ,-0.2) {$\scriptstyle{(y_2-y_1+j-i-1)(y_2-y_1+j-i)^{-1}}$};
      \node at (0.55,.2) {$\color{gray}\scriptstyle\lambda$};
\end{tikzpicture}
+\begin{tikzpicture}[baseline = 0mm]
	\draw[->,thick,densely dotted] (0.55,.06) to (0.12,.06);
	\draw[->,thick] (0.58,-.4) to (0.58,.5);
	\draw[->,thick] (0.08,-.4) to (0.08,.5);
      \node at (0.08,-0.55) {$\scriptstyle{j}$};
      \node at (0.58,-0.55) {$\scriptstyle{i}$};
      \node at (0.08,0.05) {$\bull$};
      \node at (0.58,0.05) {$\bull$};
      \node at (1.55,0.1) {$\scriptstyle{(y_2-y_1+j-i)^{-1}}$};
      \node at (0.85,-.3) {$\color{gray}\scriptstyle\lambda$};
\end{tikzpicture}
\right)
\end{align*}
in the degenerate case,
or
\begin{align*}
\begin{tikzpicture}[baseline = -0.8mm]
	\draw[->] (0.08,-.25) to (0.08,.3);
      \node at (0.08,0.02) {$\dt$};
\end{tikzpicture}
&\mapsto
\sum_{\substack{\lambda \in X \\i \in I}}\begin{tikzpicture}[baseline = -0.8mm]
	\draw[->,thick] (0.08,-.25) to (0.08,.3);
      \node at (0.08,-0.4) {$\scriptstyle{i}$};
      \node at (0.08,0.02) {$\bull$};
      \node at (0.58,0.02) {$\scriptstyle{i(y+1)}$};
      \node at (0.35,-0.3) {$\color{gray}\scriptstyle\lambda$};
\end{tikzpicture},
&
\begin{tikzpicture}[baseline = 1mm]
	\draw[<-] (0.4,0.4) to[out=-90, in=0] (0.1,0);
	\draw[-] (0.1,0) to[out = 180, in = -90] (-0.2,0.4);
\end{tikzpicture}
&\mapsto
\sum_{\substack{\lambda \in X \\i \in I}}\begin{tikzpicture}[baseline = 1mm]
	\draw[<-,thick] (0.4,0.4) to[out=-90, in=0] (0.1,0);
	\draw[-,thick] (0.1,0) to[out = 180, in = -90] (-0.2,0.4);
      \node at (-0.2,0.55) {$\scriptstyle{i}$};
      \node at (0.6,0.2) {$\color{gray}\scriptstyle\lambda$};
\end{tikzpicture}\:,
&
\begin{tikzpicture}[baseline = 1mm]
	\draw[<-] (0.4,0) to[out=90, in=0] (0.1,0.4);
	\draw[-] (0.1,0.4) to[out = 180, in = 90] (-0.2,0);
\end{tikzpicture}&\mapsto
\sum_{\substack{\lambda \in X \\i \in I}}\begin{tikzpicture}[baseline = 1mm]
	\draw[<-,thick] (0.4,0) to[out=90, in=0] (0.1,0.4);
	\draw[-,thick] (0.1,0.4) to[out = 180, in = 90] (-0.2,0);
      \node at (-0.2,-0.15) {$\scriptstyle{i}$};
      \node at (0.6,0.2) {$\color{gray}\scriptstyle\lambda$};
\end{tikzpicture}
\:\:,\end{align*}
\begin{align*}
\begin{tikzpicture}[baseline = 0]
	\draw[->] (0.28,-.3) to (-0.28,.4);
	\draw[-,line width=4pt,white] (-0.28,-.3) to (0.28,.4);
	\draw[->] (-0.28,-.3) to (0.28,.4);
\end{tikzpicture}
&\mapsto
\sum_{\substack{\lambda \in X \\ i \in I}}
\left(\begin{tikzpicture}[baseline = 0]
	\draw[->,thick] (0.38,-.4) to (-0.38,.5);
	\draw[->,thick] (-0.38,-.4) to (0.38,.5);
	\draw[->,thick,densely dotted] (0.26,-.24) to (-0.21,-.24);
   \node at (-.4,-.55) {$\scriptstyle{i}$};
   \node at (.4,-.55) {$\scriptstyle{i}$};
      \node at (-0.26,-0.25) {$\bull$};
      \node at (0.26,-0.25) {$\bull$};
      \node at (1.5,-0.2) {$\scriptstyle{q(y_2+1)-q^{-1}(y_1+1)}$};
      \node at (0.7,.25) {$\color{gray}\scriptstyle\lambda$};
\end{tikzpicture}
-q^{-1}\begin{tikzpicture}[baseline = 0mm]
	\draw[->,thick] (0.38,-.4) to (0.38,.5);
	\draw[->,thick] (0,-.4) to (0,.5);
      \node at (0,-0.55) {$\scriptstyle{i}$};
      \node at (0.38,-0.55) {$\scriptstyle{i}$};
      \node at (0.55,0) {$\color{gray}\scriptstyle\lambda$};
\end{tikzpicture}\right)\\
&
+\sum_{\substack{\lambda \in X \\ i \in I}}
\left(\begin{tikzpicture}[baseline = 0]
	\draw[->,thick] (0.38,-.4) to (-0.38,.5);
	\draw[->,thick] (-0.38,-.4) to (0.38,.5);
	\draw[->,thick,densely dotted] (0.26,-.24) to (-0.21,-.24);
   \node at (-.4,-.58) {$\scriptstyle{q^2 i}$};
   \node at (.4,-.55) {$\scriptstyle{i}$};
      \node at (-0.26,-0.25) {$\bull$};
      \node at (0.26,-0.25) {$\bull$};
      \node at (1.8,-0.2) {$\scriptstyle{\left(q(y_2+1)-q^{-1}(y_1+1)\right)^{-1}}$};
      \node at (0.55,.2) {$\color{gray}\scriptstyle\lambda$};
\end{tikzpicture}
+qz\begin{tikzpicture}[baseline = 0mm]
	\draw[->,thick,densely dotted] (0.55,.06) to (0.12,.06);
	\draw[->,thick] (0.58,-.4) to (0.58,.5);
	\draw[->,thick] (0.08,-.4) to (0.08,.5);
      \node at (0.08,-0.58) {$\scriptstyle{q^2 i}$};
      \node at (0.58,-0.55) {$\scriptstyle{i}$};
      \node at (0.08,0.05) {$\bull$};
      \node at (0.58,0.05) {$\bull$};
      \node at (2.5,0.1) {$\scriptstyle{(y_2+1)\left(q(y_2+1)-q^{-1}(y_1+1)\right)^{-1}}$};
      \node at (0.85,-.3) {$\color{gray}\scriptstyle\lambda$};
\end{tikzpicture}
\right)\\
&+
\sum_{\substack{\lambda \in X \\ i,j \in I\\j \neq i, i^+}}
\left(
-\begin{tikzpicture}[baseline = 0]
	\draw[->,thick] (0.38,-.4) to (-0.38,.5);
	\draw[->,thick] (-0.38,-.4) to (0.38,.5);
	\draw[->,thick,densely dotted] (0.26,-.24) to (-0.21,-.24);
   \node at (-.4,-.55) {$\scriptstyle{j}$};
   \node at (.4,-.55) {$\scriptstyle{i}$};
      \node at (-0.26,-0.25) {$\bull$};
      \node at (0.26,-0.25) {$\bull$};
      \node at (2.9,-0.2) {$\scriptstyle{\left(q^{-1}j (y_2+1)-qi(y_1+1)\right) \left(j(y_2+1)-i(y_1+1)\right)^{-1}}$};
      \node at (0.55,.2) {$\color{gray}\scriptstyle\lambda$};
\end{tikzpicture}
+z\begin{tikzpicture}[baseline = 0mm]
	\draw[->,thick,densely dotted] (0.55,.06) to (0.12,.06);
	\draw[->,thick] (0.58,-.4) to (0.58,.5);
	\draw[->,thick] (0.08,-.4) to (0.08,.5);
      \node at (0.08,-0.55) {$\scriptstyle{j}$};
      \node at (0.58,-0.55) {$\scriptstyle{i}$};
      \node at (0.08,0.05) {$\bull$};
      \node at (0.58,0.05) {$\bull$};
      \node at (2.2,0.1) {$\scriptstyle{j(y_2+1)\left(j(y_2+1)-i(y_1+1)\right)^{-1}}$};
      \node at (0.85,-.3) {$\color{gray}\scriptstyle\lambda$};
\end{tikzpicture}
\right)
\end{align*}
in the quantum case
with the action of $t \in \K$ chosen so that
$t_L = \sqrt{\prod_{i \in I} (-i)^{-\langle h_i,\lambda\rangle}}$
for all 
irreducible $L \in \R_\lambda$ and $\lambda \in X$.
We also have that
\begin{align}\label{sewing1}
\anticlock(u)&\mapsto \sum_{\lambda \in X}\left(\prod_{i \in I} {\scriptstyle\color{gray}\lambda\:}\anticlocki(u-i)\right),&
\clock(u)&\mapsto \sum_{\lambda \in X}\left(\prod_{i \in I} {\scriptstyle\color{gray}\lambda\:}\clocki(u-i)\right)\\\intertext{in the
  degenerate case, and}
\anticlock(u)&\mapsto \sum_{\lambda\in X}\left(\prod_{i \in I} i^{\langle h_i, \lambda\rangle}
{\scriptstyle\color{gray}\lambda\:}\anticlocki\left(\frac{u}{i}-1\right)\right),&
\clock(u)&\mapsto \sum_{\lambda \in X}\left(\prod_{i \in I} i^{-\langle h_i, \lambda\rangle}
{\scriptstyle\color{gray}\lambda\:}\clocki\left(\frac{u}{i}-1\right)\right)\label{sewing2}
\end{align}
in the quantum case.
\end{Theorem}  

\begin{proof}
Note to start with that the formulae for dots, crossings and rightwards
cups and caps in the statement of the theorem
are equivalent to the ones in Theorem~\ref{thm1}. We have simply
rearranged them to make the Heisenberg morphisms the subjects.

We first explain the proof in the (easier) degenerate case.
The formulae in the theorem give us well-defined natural
transformations
$\begin{tikzpicture}[baseline = -1mm]
	\draw[->] (0.08,-.2) to (0.08,.2);
      \node at (0.08,0) {$\dt$};
\end{tikzpicture}:E \Rightarrow E$,
$\begin{tikzpicture}[baseline = -1mm]
	\draw[->] (0.2,-.2) to (-0.2,.2);
	\draw[->] (-0.2,-.2) to (0.2,.2);
\end{tikzpicture}:E^2\Rightarrow E^2$,
$\,\begin{tikzpicture}[baseline = .75mm]
	\draw[<-] (0.3,0.3) to[out=-90, in=0] (0.1,0);
	\draw[-] (0.1,0) to[out = 180, in = -90] (-0.1,0.3);
\end{tikzpicture}:\operatorname{Id}_\R \Rightarrow
FE$ and $\:\begin{tikzpicture}[baseline = .75mm]
	\draw[<-] (0.3,0) to[out=90, in=0] (0.1,0.3);
	\draw[-] (0.1,0.3) to[out = 180, in = 90] (-0.1,0);
\end{tikzpicture}:EF\Rightarrow \operatorname{Id}_\R$.
We just need to verify that these natural transformations satisfy the
relations (\ref{dAHA})--(\ref{rightadj}) and the 
inversion relations (\ref{invrela})--(\ref{invrelb}).
As every object of $\R$ is a direct limit of finitely generated
objects, and every finitely generated object is a finite direct sum of
indecomposable objects, 
it suffices to check the relations on an indecomposable, finitely generated $V
\in \R_\kappa$ and $\kappa \in X$.
Let $Z := Z_V = Z(\End_\R(V))$, which is a finite-dimensional,
commutative, local $\k$-algebra.
Let $\mu \in X^+$ be defined so that $\langle h_i, \mu \rangle$ is the
nilpotency degree of the endomorphism
$\mathord{
\begin{tikzpicture}[baseline = -1.5mm]
 	\draw[-,darkg,thick] (1.1,.2) to (1.1,-.2);
 	\draw[<-,thick] (0.68,.2) to (0.68,-.2); 
    \node at (.68,-.35) {$\scriptstyle i$};
     \node at (.68,-.02) {$\bull$};
     \node at (1.3,-.02) {$\darkg\scriptstyle V$};
\node at (.9,0) {$\color{gray}\scriptstyle\kappa$};
\end{tikzpicture}
}$.
Let $\mu_i(u) := u^{\langle h_i, \mu\rangle}$.
Let $\nu := \kappa+\mu$ and $\nu_i(u) := \h_{V, i}(u) \mu_i(u) \in
Z[u]$,
which is a polynomial of degree $\langle h_i, \nu\rangle$.
In other words, $\h_i(u) := \nu_i(u) / \mu_i(u)$ is $\h_{V,i}(u)$.
Defining $\h_i^{(r)}$ as in (\ref{expansions}), 
the relations (\ref{oneofthe}) are satisfied in the action of
$\UU(\g)$ on $V$. These are the defining relations of the Kac-Moody
GCQ $\H_Z(\mu|\nu)$, so we get induced a unique $Z$-linear 
morphism of 2-representations $\left(\H_Z(\mu|\nu)\right)_{\lambda \in
  X}
\rightarrow (\R_\lambda)_{\lambda \in X}$ sending $1_\kappa \mapsto
V$.
This gives us a $Z$-linear functor
$\H_Z(\mu|\nu) \rightarrow \R$, $\varnothing \mapsto V$. Hence, using the isomorphism of
Theorem~\ref{thm2}, we get a $Z$-linear functor
$\H_Z(m|n) \rightarrow \R$, $\varnothing \mapsto V$ for $m(u), n(u) \in
Z[u]$ defined as in (\ref{boring}). The assumption that $\R$ is
of central charge $k$ means 
that $\H_Z(m|n)$ is a $\Heis_k$-module category.
The evaluations on $V$ of the natural transformations arising in the
relations to be checked are the images under this functor of
corresponding morphisms in $\H_Z(m|n)$. Since the relations hold for
the latter this does the job.
It just remains to prove (\ref{sewing1}). Again it suffices to check
that this holds when evaluated on the chosen object $V$, that is, we
must show that $\h_V(u) = \prod_{i \in I} \h_{V,i}(u-i)$.
We know already that $\h_{V,i}(u) = \nu_i(u) / \mu_i(u)$, so by the
definition (\ref{boring}) we have that $\prod_{i \in I} \h_{V,i}(u-i)
= n(u) / m(u)$. This equals $\h_V(u)$ due to (\ref{hat}) and
(\ref{bothcases}).

Now consider the quantum case.
The formulas in the statement of the theorem give us 
natural transformations
$\begin{tikzpicture}[baseline = -1mm]
	\draw[->] (0.08,-.2) to (0.08,.2);
      \node at (0.08,0) {$\dt$};
\end{tikzpicture}:E \Rightarrow E$,
$\begin{tikzpicture}[baseline = -1mm]
	\draw[->] (0.2,-.2) to (-0.2,.2);
	\draw[-,line width=4pt,white] (-0.2,-.2) to (0.2,.2);
	\draw[->] (-0.2,-.2) to (0.2,.2);
\end{tikzpicture}:E^2\Rightarrow E^2$,
$\,\begin{tikzpicture}[baseline = .75mm]
	\draw[<-] (0.3,0.3) to[out=-90, in=0] (0.1,0);
	\draw[-] (0.1,0) to[out = 180, in = -90] (-0.1,0.3);
\end{tikzpicture}:\operatorname{Id}_\R \Rightarrow
FE$ and $\:\begin{tikzpicture}[baseline = .75mm]
	\draw[<-] (0.3,0) to[out=90, in=0] (0.1,0.3);
	\draw[-] (0.1,0.3) to[out = 180, in = 90] (-0.1,0);
\end{tikzpicture}:EF\Rightarrow \operatorname{Id}_\R$,
the first two of which are clearly invertible.
As $\Heis_k$ is a
$\K$-linear category rather than a $\k$-linear category, we also need 
to define an invertible natural transformation $t:\R \rightarrow \R$ such that
$t_{EV} = E t_V$ and $t_{FV} = F t_V$ for each $V \in \R$.
Before we do this in general, consider the situation for
an irreducible object $L \in \R_\lambda$. The minimal
polynomials of
$\mathord{
\begin{tikzpicture}[baseline = -2mm]
 	\draw[-,darkg,thick] (1.1,.2) to (1.1,-.2);
 	\draw[<-,thick] (0.68,.2) to (0.68,-.2); 
    \node at (.68,-.35) {$\scriptstyle i$};
     \node at (.68,-.02) {$\bull$};
     \node at (1.3,-.02) {$\darkg\scriptstyle L$};
\node at (.9,0) {$\color{gray}\scriptstyle\lambda$};
\end{tikzpicture}
}$
and
$\mathord{
\begin{tikzpicture}[baseline = 0mm]
 	\draw[-,darkg,thick] (1.1,.2) to (1.1,-.2);
 	\draw[->,thick] (0.68,.2) to (0.68,-.2); 
    \node at (.68,.35) {$\scriptstyle i$};
     \node at (.68,.02) {$\bull$};
     \node at (1.3,-.02) {$\darkg\scriptstyle L$};
\node at (.9,0) {$\color{gray}\scriptstyle\lambda$};
\end{tikzpicture}
}$
are $u^{\eps_i(L)}$ and $u^{\phi_i(L)}$, so in a
$\Heis_k$-action consistent with these formulas,
we have that 
$m_L(u) = \prod_{i \in I} (u-i)^{\eps_i(L)}$ and
$n_L(u) = \prod_{i \in I} (u-i)^{\phi_i(L)}$.
In view of Lemma~\ref{impy}, using also that $\langle h_i, \lambda
\rangle = \phi_i(L)-\eps_i(L)$ by Lemma~\ref{impy2},
it follows that
$t_L^2 = \prod_{i \in I} (-i)^{-\langle h_i, \lambda \rangle}$.
In the statement of the theorem, we have stipulated that
$t_L = \sqrt{\prod_{i \in I} (-i)^{-\langle h_i, \lambda \rangle}}$, thereby
making the same fixed choice of square root as in the definition of
GCQs in $\S$\ref{massive}.
In general, it suffices to define the natural transformation $t$ on objects
$V$ that are finitely
generated and indecomposable; then we can define $t_V$ on an arbitrary
$V \in \R$ by taking direct sums and limits.
Fixing such an object $V \in \R_\kappa$, define $Z, \mu, \nu, \mu_i(u)$
and $\nu_i(u)$ as in the previous paragraph. Then, as before, we get a
$Z$-linear functor $\H_Z(\mu|\nu) \rightarrow \R$, $\varnothing \mapsto
V$.
Composing with the isomorphism from Theorem~\ref{thm2},
this gives us a $Z$-linear functor
$\H_Z(m|n) \rightarrow \R$, $\varnothing \mapsto V$ where
$m(u), n(u) \in Z[u]$ are defined as in (\ref{music}).
On $\H_Z(m|n)$, we know 
that $t$ acts as $$
\sqrt{m(0) / n(0)} =
\sqrt{\textstyle\prod_{i \in I} i^{-\langle h_i, \kappa\rangle}
  \mu_i(-1) / \nu_i(-1)}
\in Z.
$$
Modulo the unique maximal ideal $J$ of $Z$, this expression equals
$\sqrt{\prod_{i \in I} (-i)^{-\langle h_i, \kappa \rangle}}$, which is the
desired action of $t$ on irreducible quotients of $V$.
So we can use this formula to define
the morphism $t_V:V \rightarrow V$, and have the data
needed to define the natural transformation $t$.
We still need to check that 
$t_{EV} = E t_V$ and $t_{FV} = F t_V$
and to verify the other defining relations of $\Heis_k$, namely,
(\ref{AHA2})--(\ref{rightadj2}),
the inversion relation (\ref{invrel1a})--(\ref{invrel1b}),
and the additional relation explained immediately after
(\ref{invrel1b}). But these all follow as in the previous paragraph because they
are true for the action of $\Heis_k$ on $\H_Z(m|n)$.
Finally, to prove (\ref{sewing2}), we argue in the same way as
explained at the end of the previous paragraph, using (\ref{music})
instead of (\ref{boring}).
\end{proof}

\begin{Remark}
Like in Remark~\ref{nocaps}, the actions of the leftwards cups and caps in
$\Heis_k$ are uniquely determined by the actions of the other
generators
due now to \cite[Lemma 5.2]{BSWk0} or
\cite[Lemma 4.3]{BSWqheis}, but it is not easy to find explicit formulae.
\end{Remark}

\end{document}